\documentclass[11pt, a4paper]{book}

\usepackage[utf8]{inputenc} 
\usepackage{lmodern} 
\usepackage[ngerman, english]{babel} 
\usepackage{csquotes} 

\usepackage[unicode, hidelinks]{hyperref} 
\usepackage{enumitem} 

\usepackage{emptypage} 
\usepackage{xcolor}

\usepackage[final]{fixme} 
\fxsetup{theme=color}

\usepackage{wrapfig} 

\usepackage[ruled,linesnumbered]{algorithm2e}

\makeatletter
\renewenvironment{algomathdisplay}
 {\[}
 {\@endalgocfline\vspace{-\baselineskip}\]{\DontPrintSemicolon\;}}
\makeatother

\usepackage{mathtools} 
\usepackage{amsthm, amssymb} 
\usepackage{bbm} 
\usepackage{cancel} 

\usepackage[style=authoryear, backend=biber, seconds=true, date=iso, urldate=iso, maxcitenames=2]{biblatex}
\usepackage{./Macros}

\usepackage{./Title}
\firstLastName{Felix Benning}
\birthdate{27 November 1996}
\birthplace{Nürtingen}
\MatNr{1501817}
\titellineI{High Dimensional Optimization}
\subtitellineI{through the Lens of Machine Learning}
\Supervisor{Prof. Dr. Leif Döring}
\DueDate{20.12.2021}

\begin{document}
	\listoffixmes 
	\frontmatter
	\titelpage


\chapter*{Declaration of Authorship}
\thispagestyle{empty}

I hereby declare that the thesis submitted is my own unaided work. All direct or indirect sources used are acknowledged as references.

\vspace{.5\baselineskip}
This thesis was not previously presented to another examination board and has not been published.

\vspace{4\baselineskip}
\begin{center}
\parbox{.8\textwidth}{City, Date \hfill Signature}
\end{center}



\chapter*{Acknowledgments}
    
I want to thank my colleagues at the Faculty for Mathematics, in particular
my supervisor Prof. Dr. Leif Döring, for the helpful discussions and their open
ears to bounce ideas off of.
For proof reading I want to thank Katharina Enin, Julie Naegelen and especially
Linda Ritzau for their numerous suggestions and corrections.


	\cleardoublepage

	\setcounter{tocdepth}{1}

	\tableofcontents

	{

\chapter{Introduction}

\section{Optimization in Machine Learning}

In statistics or machine learning we usually want to find some model
which ``best describes'' our data. So let \(\model(X)\) be a prediction for
\(Y\), parametrized by weights \(\weights\), where \(Z=(X,Y)\) is drawn from a
distribution \(\dist\) of real world examples. We formalize the prediction error
of our parameter \(\weights\) with a loss function \(l(\weights, z)\), e.g. 
a squared loss
\begin{align}\label{eq: square loss}
	\loss(\weights, z) := (\model(x) - y)^2.
\end{align}
Then we want to minimize the expected loss over all (possible) examples, i.e.
``(theoretical) risk''
\begin{align*}
	\Loss(\weights) := \E_{Z\sim \dist} [\loss(\weights, Z)].
\end{align*}
In general, we do not have access to this expected loss. Instead we usually
assume to be in possession of only a sample \(\sample=(Z_1, \dots,
Z_{\sampleSize})\) of data, drawn (independently) from distribution \(\dist\).
This gives rise to the ``empirical risk''
\begin{align}\label{eq: empirical risk}
	\Loss_\sample(\weights)
	:= \frac{1}{\sampleSize}\sum_{k=1}^\sampleSize \loss(\weights, Z_k)
	= \E_{Z\sim \dist_\sample}[\loss(\weights, Z)],
\end{align}
where \(\dist_\sample\) is the empirical distribution generated by sampling
uniformly from the set \(\sample\). The empirical risk usually approaches
the theoretical risk as \(\sampleSize\to\infty\), since some law of large number
is usually applicable (e.g. if we sample from \(\sample\) independently). Thus the
next best thing to minimizing \(\Loss\) is often to minimize \(\Loss_\sample\).
In the case of a linear model
\begin{align*}
	\model(x) = \langle x, \weights \rangle
\end{align*}
minimizing the empirical risk (\ref{eq: empirical risk}) induced by a square
loss (\ref{eq: square loss})
\begin{align*}
	\Loss_\sample(\weights)
	= \frac{1}{\sampleSize}\sum_{k=1}^\sampleSize (\model(X_k)-Y_k)^2
	= \frac{1}{\sampleSize}\| \mathbf{X}^T\weights - \mathbf{Y}\|^2
\end{align*}
can be done analytically, resulting in the well known least squared
regression
\begin{align*}
	\minimum = (\mathbf{X}^T\mathbf{X})^{-1}\mathbf{X}^T\mathbf{Y} \quad \text{where}\quad 
	\begin{aligned}
		\mathbf{X}&=(X_1,\dots,X_\sampleSize)^T\\
		\mathbf{Y}&=(Y_1,\dots,Y_\sampleSize)^T
	\end{aligned}.
\end{align*}
But this is not possible in general. And even if it was possible, it might still
be difficult to come up with the solution, slowing down experimentation
with different types of models. We are therefore interested in (numeric)
algorithms which can minimize arbitrary loss functions.

This train of thought leads us to various ``black box'' models where we assume
no knowledge of the loss function itself but assume we have some form of oracle
telling us things about the loss function. ``Zero order oracles'' allow us to
evaluate the loss function at a particular point, i.e.\ retrieve \(\Loss(\weights)\)
for any \(\weights\) at a fixed cost. ``First order oracles'' provide us not
only with \(\Loss(\weights)\) but also with \(\nabla\Loss(\weights)\), and so on. 

In some sense these black box models are extremely wasteful as we do not utilize
any of the inherent structure of our model. ``Structural optimization'' addresses
this problem and opens up the black box and makes further assumptions about
the loss function. Proximal methods for example assume that some form of convex
regularization function (e.g. distance penalty) is part of our overall loss
function \parencite[e.g.][]{bottouOptimizationMethodsLargeScale2018}:
\begin{align*}
	\Loss(\weights)
	= F(\weights) + \lambda \underbrace{\Omega(\weights)}_{\mathclap{\text{regularizer}}}.
\end{align*}
Nevertheless we will mostly navigate around structural optimization in
this work, only touching on it again with ``mirror descent'' in
Remark~\ref{rem: mirror descent}.

We will also assume that the set of possible weights is a real
vector space \(\reals^\dimension\). Since projections \(\Pi_\Omega\) to convex
subsets \(\Omega\) are contractions \parencite[Lemma
3.1]{bubeckConvexOptimizationAlgorithms2015}, i.e.
\begin{align*}
	\| \Pi_\Omega(x) - \Pi_\Omega(y) \| \le \| x - y \|,
\end{align*}
they can often be applied without any complications to the proofs for unbounded
optimization we will cover, which show a reduction in the distance to the
minimum
\begin{align*}
	\|\weights_{n+1} - \minimum \|
	&= \|\Pi_\Omega(\tilde{\weights}_{n+1}) - \Pi_\Omega(\minimum)\|
	\le \|\tilde{\weights}_{n+1} - \minimum \|\\
	&\le \dots \text{ [unbounded arguments] }\\
	&\le \|\weights_n - \minimum\|.
\end{align*}

The next section will explain why we will mostly skip zero order
optimization.

\section{Zero Order Optimization}

If we assume an infinite number of possible inputs \(\weights\), it does not matter
how often and where we sample \(\Loss\), there is always a different \(\weights\) where
\(\Loss\) might be arbitrarily lower (see various ``No-Free-Lunch Theorems'').
We must therefore start making assumptions about \(\Loss\), to make this
problem tractable.

\subsection{Grid Search}

It is reasonable to assume that the value of our loss \(\Loss\) at some point
\(\weights\) will not differ much from the loss of points in close proximity.
Otherwise we would have to sample every point, which is not possible for continuous
and thus infinite parameter sets. This assumption is continuity of the loss
\(\Loss\). One possible formulation is ``\(\lipConst\)-Lipschitz continuity''
\begin{align*}
	| \Loss(\weights_1) - \Loss(\weights_2) | < \lipConst | \weights_1 - \weights_2 |.
\end{align*}
And while Lipschitz continuity is sufficient to find an \(\lossError\)-solution
\(\hat{\weights}\) inside a unit (hyper-)cube
\begin{align*}
	\Loss(\hat{\weights}) \le \inf_{ 0\le\weights_i\le 1} \Loss(\weights) + \lossError
\end{align*}
in a finite number of evaluations of the loss, it still requires a full
``grid search'' (i.e.\ tiling our parameter set with a grid into (tiny) cells
and trying out one point for every cell) with time complexity \parencite[pp.
12,13]{nesterovLecturesConvexOptimization2018}
\begin{align*}
	\left(\left\lfloor \frac{\lipConst}{2\lossError}\right\rfloor + 1\right)^\dimension
\end{align*}
where \(\dimension\) is the dimension of \(\weights\). The fact that grid search is
sufficient can be seen by tiling the space into a grid and asking the oracle
for the value of the loss function at each center. If the minimum is in
some tile, the Lipschitz continuity forces the center of that tile to be \(\lossError\)
close to it \parencite[cf.][p. 11]{nesterovLecturesConvexOptimization2018}.

The fact that this is necessary can be seen by imagining a ``resisting
oracle'' which places the minimum in just the last grid element where we look
\parencite[cf.][p. 13]{nesterovLecturesConvexOptimization2018}. The exponential
increase of complexity in the dimension makes the problem intractable for even
mild dimensions and precisions (e.g. \(\dimension=11\), \(\lossError=0.01\), \(\lipConst=2\)
implies millions of years of computation).

\subsection{Other Algorithms}

There are more sophisticated zero order optimization techniques than grid
search like ``simulated annealing'' \parencite[e.g.][]{bouttierConvergenceRateSimulated2019}
or ``evolutionary algorithms'' \parencite[e.g.][]{heConditionsConvergenceEvolutionary2001}.
But in the end they all scale relatively poor in the dimension. Calculating the
\mbox{(sub-)gradient} on the other hand scales only linearly in the dimension.
Calculating the second derivative (Hessian) requires \(O(\dimension^2)\)
operations, and doing something useful with it (e.g. solving a linear
equation) usually causes \(O(\dimension^3)\) operations. For this reason first
order methods seem most suited for high dimensional optimization, as their
convergence guarantees are independent of the dimension. But there are attempts
to make second order methods viable by avoiding the calculation of the Hessian.
These methods are usually called ``Quasi Newton methods'' of which we will 
sketch the most common methods in Chapter~\ref{chap: other methods}.

\section{Related Work}

A similar review was conducted by
\textcite{bottouOptimizationMethodsLargeScale2018} with a stronger focus on
second order methods but only passing interest in
momentum methods which are quite successful in practice.
The second half of \textcite{nesterovLecturesConvexOptimization2018} introduces
more ideas for structural optimization, the first half covers both first order
and second order methods extensively but does not consider stochasticity much
and provides little intuition.
A review with heavier focus on geometric methods (Section~\ref{sec: geometric methods}) 
was conducted by \textcite{bubeckConvexOptimizationAlgorithms2015}.

\section{Readers Guide}

In Chapter~\ref{chap: gradient descent} we will review the well
established method of gradient descent which can minimize deterministic
functions. It might be tempting to skip this chapter, but note that this chapter 
introduces important building blocks we will use in subsequent chapters. We
prove convergence of gradient descent mostly as motivation for these building blocks.
The convergence proofs themselves are actually very short.

The complexity bounds in Section~\ref{sec: complexity bounds} on the other hand,
are not required for subsequent chapters. Neither is the loss surface analysis
in Section~\ref{sec: loss surface} and backtracking in Section~\ref{sec:
backtracking}, but they motivate assumptions we make.

While gradient descent is applicable to \(\Loss_\sample\), i.e.\ allows us
to minimize empirical risk (\ref{eq: empirical risk}), we will question this
approach in Chapter~\ref{chap: sgd} where we show that it is possible to
minimize \(\Loss\) itself, provided we have an infinite supply of samples. If we
start to reuse samples we would just end up minimizing \(\Loss_\sample\) again.

In Chapter~\ref{chap: momentum} we will discuss methods to speed up
gradient descent. This chapter is mostly independent of Chapter~\ref{chap: sgd},
so they can be read in any order.

Chapter~\ref{chap: other methods} provides a review of all the methods we did
not have time to go over in detail. This review should provide enough intuition
to understand their importance, but we will skip convergence proofs entirely
and sometimes only sketch the ideas.


}
	\mainmatter 
	{

\chapter{Gradient Descent (GD)}\label{chap: gradient descent}

One relatively obvious way to utilize the gradient information
\(\nabla\Loss(\weights)\), provided by first order oracles, is to incrementally
move in the direction of steepest descent
\parencite{cauchyMethodeGeneralePour1847}, i.e.
\begin{align}
	\label{eq: gradient descent}
	\tag{gradient descent}
	\weights_{n+1} = \weights_n - \lr\nabla \Loss(\weights_n).
\end{align}
Where \(\lr\) denotes the ``learning rate''. A useful way to look at this
equation is to notice that it is the discretization of an ordinary differential
equation (ODE)
\begin{align*}
	\frac{d}{dt}\weights_n \approx \frac{\weights_{n+1} - \weights_n}{\lr}
	= - \nabla \Loss(\weights_n),
\end{align*}
if we view the learning rate \(\lr\) as a time delta between times
\(t_n\) and \(t_{n+1}\) with \(t_n = n\lr\). As then for
\(\lr\to 0\) we arrive at the ODE in time \(t\)
\begin{align}\label{eq: gradient flow}
	\tag{gradient flow}
	\dot{\weights}(t) = \frac{d}{dt}\weights(t)= -(\nabla_\weights \Loss)(\weights(t)).
\end{align}
That we are taking the gradient with respect to \(\weights\), will from now on be
implied. With the \ref{eq: gradient flow} equation and the fundamental theorem of
analysis we get\footnote{
	if you are familiar with Lyapunov functions you probably recognize this argument
}
\begin{align}\label{eq: gradient integral}
	\Loss(\weights(t_1)) - \Loss(\weights(t_0))
	&= \int_{t_0}^{t_1} \nabla \Loss(\weights(s)) \cdot \dot{\weights}(s) ds
	= \int_{t_0}^{t_1} -\|\nabla \Loss(\weights(s))\|^2 ds
	\le 0,
\end{align}
which implies a decreasing loss sequence
\begin{align*}
	\Loss(\weights(t_0)) \ge \Loss(\weights(t_1)) \ge \dots \ge \inf_\weights \Loss(\weights) \ge 0
\end{align*}
implying convergence of \(\Loss(\weights(t))\). But it does not necessitate a
convergent \(\weights(t)\) nor convergence to \(\inf_\weights
\Loss(\weights)\). This highlights why there are three different types of
convergence we are interested in
\begin{description}
	\item[Convergence of the Weights] to a (the) minimum \(\minimum\)
	\begin{align*}
		\|\weights_n - \minimum\| \to 0 \qquad (n\to\infty),
	\end{align*}
	\item[Convergence of the Loss]
	\begin{align*}
		|\Loss(\weights_n) - \Loss(\minimum)|
		=\Loss(\weights_n) - \Loss(\minimum) \to 0,
	\end{align*}
	\item[Convergence of the Gradient]
	\begin{align*}
		\|\nabla\Loss(\weights_n)\|
		= \|\nabla\Loss(\weights_n) - \nabla\Loss(\minimum)\| \to 0.
	\end{align*}
\end{description}

\section{Visualizing the 2nd Taylor Approximation}\label{sec: visualize gd}

To build some intuition what leads to convergence of the weights let us consider
more simple cases. For the second order Taylor approximation \(T_2\Loss\) of the
loss \(\Loss\)
\begin{align}\label{eq: taylor approx}
	\Loss(\weights+x) \approx T_2\Loss(\weights+x)
	= \Loss(\weights) + x^T \nabla \Loss(\weights) + \tfrac12 x^T \nabla^2 \Loss(\weights) x,
\end{align}
we can use the first order condition, i.e.\ set the first derivative to zero
\begin{align}\label{eq: Taylor approx derivative}
	\nabla T_2\Loss(\weights+x) = \nabla \Loss(\weights) + \nabla^2\Loss(\weights) x \xeq{!} 0
\end{align}
to find the minimum 
\begin{align}\label{eq: newton minimum approx relative coords}
	\hat{x} = -(\nabla^2 \Loss(\weights))^{-1}\nabla \Loss(\weights).
\end{align}
This assumes our Hessian \(\nabla^2 \Loss(\weights)\) exists and is
positive definite (all eigenvalues are positive), which also means it is invertible.
Plugging (\ref{eq: newton minimum approx relative coords}) into
\begin{align*}
	\nabla T_2\Loss(\weights+x) = \nabla^2 \Loss(\weights)(x-\hat{x}).
\end{align*}
and comparing it to (\ref{eq: Taylor approx derivative})
best explains where it comes from.
Similarly we can rewrite the original Taylor approximation (\ref{eq: taylor
approx}) using (\ref{eq: newton minimum approx relative coords})
\begin{align*}
	T_2\Loss(\weights+x)
	&\lxeq{(\ref{eq: newton minimum approx relative coords})} \Loss(\weights) - x^T \nabla^2 \Loss(\weights) \hat{x} + \tfrac12 x^T \nabla^2 \Loss(\weights) x \\
	&= \underbrace{
		\Loss(\weights) - \tfrac12 \hat{x} \nabla^2 \Loss(\weights) \hat{x}
	}_{
		=: c(\weights) \text{ (const.)}
	} + \tfrac12 (x-\hat{x})^T \nabla^2 \Loss(\weights)(x-\hat{x}).
\end{align*}
So far we have expressed the location of the minimum \(\hat{x}\) of the Taylor
approximation only relative to \(\weights\). To obtain its absolute coordinates
\(\theta:=\weights+x\) we define\footnote{
	We now motivated the Newton-Raphson Method in passing, which is a second
	order method utilizing the second derivative.
}
\begin{align}\label{eq: newton minimum approx}
	\tag{Newton-Raphson step}
	\hat{\weights} := \weights + \hat{x}
	\xeq{(\ref{eq: newton minimum approx relative coords})}
	\weights -(\nabla^2 \Loss(\weights))^{-1}\nabla \Loss(\weights).
\end{align}
Using \(x-\hat{x}=\theta-\hat{\weights}\), we therefore obtain the notion of a
paraboloid centered around vertex \(\hat{\weights}\)
\begin{align}\label{eq: paraboloid approximation of L}
	\Loss(\theta)
	= \overbrace{
		\tfrac12 (\theta- \hat{\weights})^T \nabla^2 \Loss(\weights) (\theta-\hat{\weights})
		+ c(\weights)
	}^{T_2\Loss(\weights+x)} + o(\|\theta-\weights\|^2).
\end{align}
\begin{wrapfigure}{O}{0.65\textwidth}
	\centering
	\def\svgwidth{0.65\textwidth}
	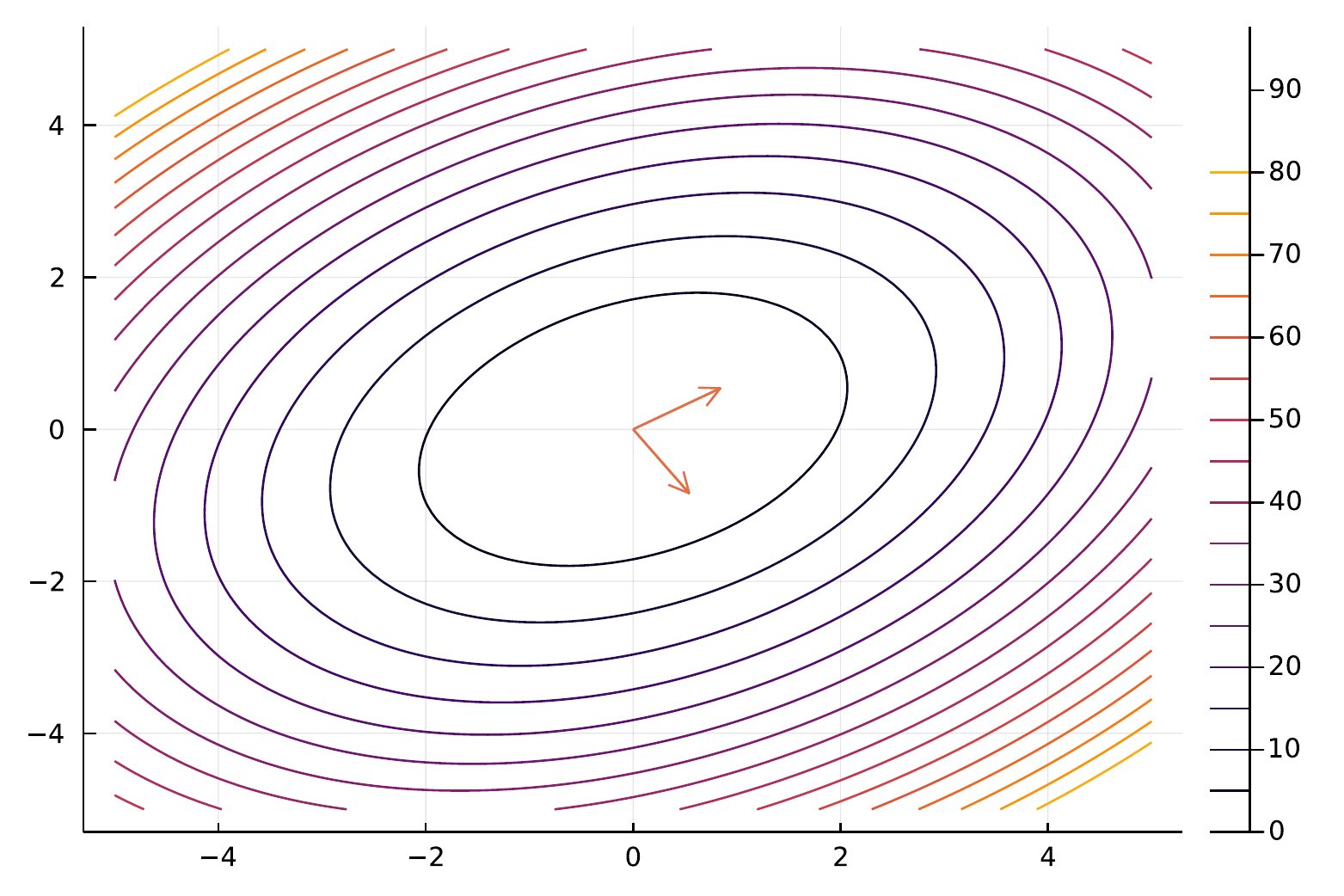
	\caption{Assuming center \(\hat{\weights}=0\), eigenvalues \(\hesseEV_1=1,
	\hesseEV_2=2\), and the respective eigenvectors \(v_1=(\sin(1), \cos(1))\)}
	\label{fig: 2d paraboloid}
\end{wrapfigure}
To fully appreciate Figure~\ref{fig: 2d paraboloid}, we now only need to realize
that the diagonizability of \(\nabla^2 \Loss(\weights)\)
\begin{align}\label{eq: diagnalization of the Hesse matrix}
	V \nabla^2 \Loss(\weights) V^T
	= \diag(\hesseEV_1,\dots,\hesseEV_\dimension), \qquad V=(v_1,\dots, v_\dimension)
\end{align}
implies, that once we have selected the center \(\hat{\weights}\) and the
direction of one eigenspace in two dimensions, the other eigenspace has to be
orthogonal to the first one. This is due to orthogonality of eigenspaces with
different eigenvalues of symmetric matrices. The direction of the
second eigenspace is therefore fully determined by the direction of the second
eigenspace.

Before we move on, it is noteworthy that we only really needed the positive
definiteness of \(\nabla^2 \Loss(\weights)\) to make it invertible, such that the
vertex \(\hat{x}\) is well defined. If it was not positive definite but still
invertible, then \(\hat{x}\) would not be a minimum but all the other arguments
would still hold.
In that case the eigenvalues might be negative as well as positive representing
a saddle point (cf. Figure~\ref{fig: visualize saddle point gd}), or all
negative representing a maximum.

Taking the derivative of the representation (\ref{eq: paraboloid approximation
of L}) of \(\Loss\) we can write the gradient at \(\weights\) as
\begin{align}\label{eq: hesse representation of gradient}
	\nabla \Loss(\weights)
	=  \nabla^2 \Loss(\weights)(\weights-\hat{\weights})
	\ (= -\nabla^2 \Loss(\weights)\hat{x}).
\end{align}
But note that the vertex \(\hat{\weights}\) depends on \(\weights\), so we need
to index both by \(n\) to rewrite \ref{eq: gradient descent}
\begin{align*}
	\weights_{n+1} &= \weights_n - \lr\nabla \Loss(\weights_n)
	= \weights_n - \lr\nabla^2 \Loss(\weights_n)(\weights_n - \hat{\weights}_n).
\end{align*}
Subtracting \(\hat{\weights}_n\) from both sides we obtain the following
transformation 
\begin{align}\label{eq: Matrix GD Formulation}
	\weights_{n+1} - \hat{\weights}_n
	&= (\identity - \lr\nabla^2 \Loss(\weights_n) ) (\weights_n - \hat{\weights}_n),
\end{align}
where \(\identity\) denotes the identity.
Taking a closer look at this transformation matrix, we can use the
diagonalization (\ref{eq: diagnalization of the Hesse matrix}) to see
\begin{align*}
	\identity - \lr\nabla^2 \Loss(\weights_n)
	&= V(\identity - \lr\cdot\diag(\hesseEV_1,\dots,\hesseEV_\dimension) )V^T \\
	&= V\cdot\diag(1-\lr\hesseEV_1, \dots,1-\lr\hesseEV_\dimension)V^T.
\end{align*}
Now, if we assume like \textcite{gohWhyMomentumReally2017} that the second
taylor approximation is accurate, and thus that \(\nabla^2 \Loss(\weights)=H\) is a
constant, then \(\hat{\weights}_n = \minimum\) is the true minimum (if all
eigenvalues are positive). And by induction we get
\begin{align}
	\weights_n - \minimum
	= V\cdot\diag[(1-\lr\hesseEV_1)^n,\dots,(1-\lr\hesseEV_\dimension)^n] V^T (\weights_0 - \minimum).
\end{align}
Decomposing the distance to the minimum \(\weights - \minimum\) into the
eigenspaces of \(H\), we can see that each component scales exponentially on its
own 
\begin{align*}
	\langle \weights_n -\minimum, v_i\rangle
	= (1-\lr\hesseEV_i)^n \langle \weights_0 - \minimum, v_i\rangle.
\end{align*}
This is beautifully illustrated with interactive graphs in
\citetitle{gohWhyMomentumReally2017} by \citeauthor{gohWhyMomentumReally2017}.

\section{Assumptions on the Loss \texorpdfstring{\(\Loss\)}{ℒ}}

\subsection{Negative Eigenvalues}\label{subsec: Negative Eigenvalues}

If our eigenvalues are not all positive and \(\hesseEV_i<0\), then
\(1-\lr\hesseEV_i\) is greater than one. This repels this component
\(\langle \weights_0 - \minimum, v_i\rangle\) away from our vertex \(\minimum\),
which is a good thing, since \(\minimum\) is a maximum in this component. Therefore
we will walk down local minima and saddle points no matter the learning
rate, assuming this component was not zero to begin with. In case of a maximum
this would mean that we would start right on top of the maximum, in case of a
saddle point it implies starting on the rim such that one can not roll down
either side.

Assuming a stochastic initialization (with zero probability on zero measure
sets) we would start in these narrow equilibria with probability zero. So
being right on top of them is of little concern. But being close to zero in such
a component still means slow movement away from these equilibria, since we are
merely multiplying factors to this initially small distance. This fact,
illustrated by Figure~\ref{fig: visualize saddle point gd}, is a possible
explanation for the common observation in deep learning of ``long plateaus on
the way down when the error hardly change[s], followed by sharp drops''
\parencite{sejnowskiUnreasonableEffectivenessDeep2020} when the exponential
factor ramps up.
\begin{figure}[h]
	\centering
	\def\svgwidth{1\textwidth}
	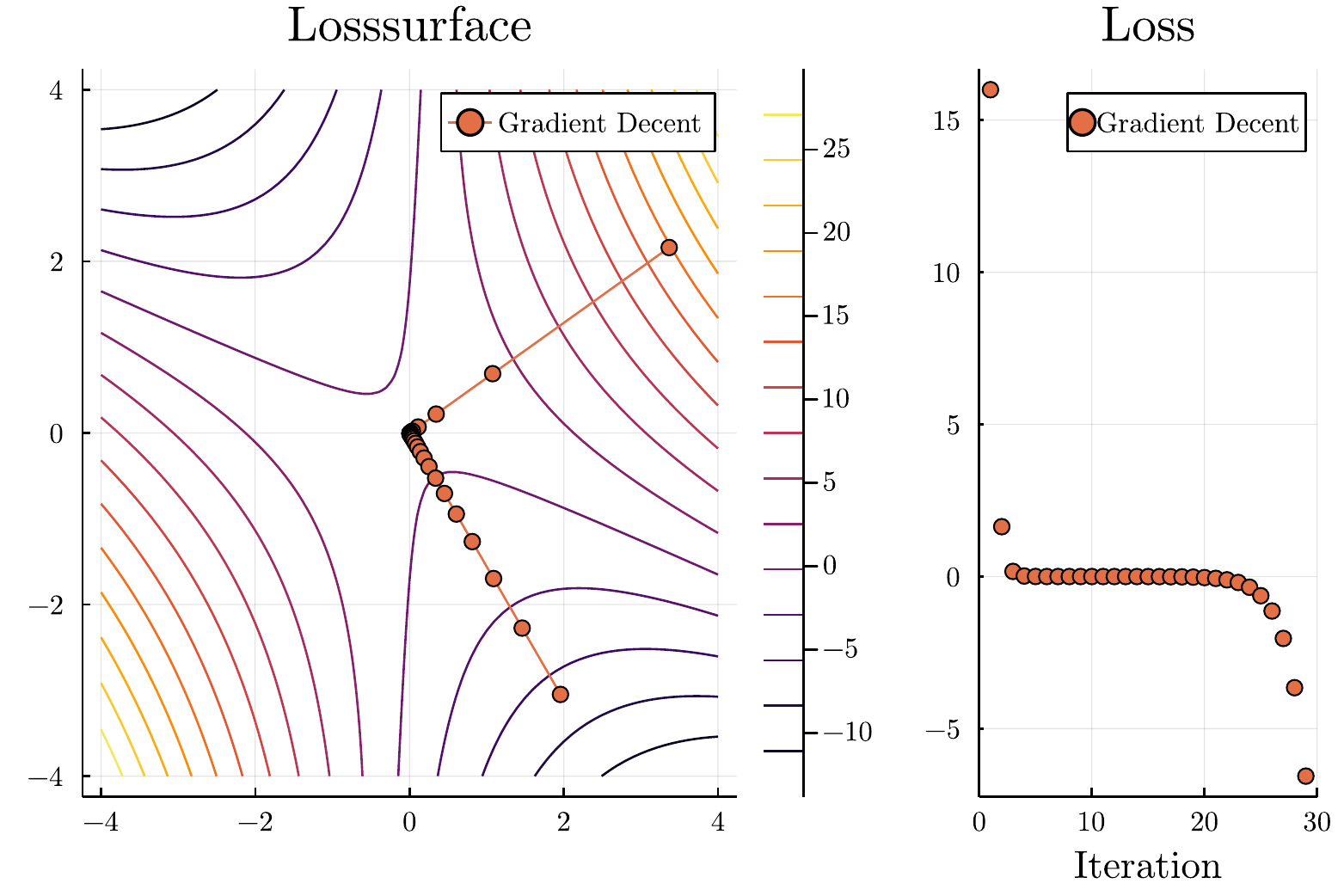
	\caption{Losssurface of \(\Loss\) with vertex at zero and eigenvalues of its
	Hessian being \(\hesseEV_1=-1, \hesseEV_2=2\). GD with start
	\(0.001v_1+4v_2\) and learning rate \(\lr=0.34\)}
	\label{fig: visualize saddle point gd}
\end{figure}

\subsection{Necessary Assumptions in the Quadratic Case}
\label{subsec: necessary assumptions in the quadratic case}

Since negative eigenvalues repel, let us consider strictly positive eigenvalues
now and assume that our eigenvalues are already sorted
\begin{align}
	0 < \hesseEV_1 \le \dots \le \hesseEV_\dimension.
\end{align}
Then for positive learning rates \(\lr\) all exponentiation bases are smaller
than one
\begin{align*}
	1-\lr\hesseEV_\dimension \le \dots \le 1-\lr\hesseEV_1 \le 1.
\end{align*}
But to ensure convergence we also need all of them to be larger than \(-1\).
This leads us to the condition
\begin{align}\label{eq: learning rate restriction (eigenvalue)}
	0< \lr < \tfrac{2}{\hesseEV_\dimension}.
\end{align}
Selecting \(\lr = \tfrac{1}{\hesseEV_\dimension}\) reduces the exponentiation
base (``rate'') of the corresponding eigenspace to zero ensuring convergence in
one step. But if we want to maximize the convergence rate of \(\weights_n\),
this is not the best selection.

We can reduce the rates of the other eigenspaces if we increase the learning
rate further, getting their rates closer to zero. At this point we are of course
overshooting zero with the largest eigenvalue, so we only continue with this
until we get closer to \(-1\) with the largest eigenvalue than we are to \(1\)
with the smallest. In other words the overall convergence rate is
\begin{align*}
	\rate(\lr)=\max_{i} |1-\lr\hesseEV_i| = \max\{|1-\lr\hesseEV_1|, |1-\lr\hesseEV_\dimension|\}.
\end{align*}
This is minimized when
\begin{align*}
	1-\lr\hesseEV_1 = \lr\hesseEV_\dimension -1,
\end{align*}
implying an optimal learning rate of
\begin{align}\label{eq: optimal SGD lr eigenvalue representation}
	\lr^* = \frac{2}{\hesseEV_1 + \hesseEV_\dimension}.
\end{align}
If \(\hesseEV_1\) is much smaller than \(\hesseEV_\dimension\), this leaves \(\lr\)
at the upper end of the interval in (\ref{eq: learning rate restriction
(eigenvalue)}). The optimal convergence rate
\begin{align*}
	\rate(\lr^*)
	= 1 - \frac{2}{1+\condition}
	\qquad \condition:=\hesseEV_\dimension/\hesseEV_1 \ge 1
\end{align*}
becomes close to one if the ``condition number'' \(\condition\) is large.
If all the eigenvalues are the same on the other hand, the condition number
becomes one and the rate is zero, implying instant convergence. This is not
surprising if we recall our visualization in Figure~\ref{fig: 2d paraboloid}.
When the eigenvalues are the same, the contour lines are concentric circles and
the gradient points right at the center.

\subsection{Generalizing Assumptions}

To prove convergence in a more general setting, we need to formulate the assumptions
we made about the eigenvalues of the Hessian without actually requiring
the existence of a Hessian if possible. 
\begin{definition}
	We call a matrix \(A\) ``positive definite'', if
	\begin{align*}
		\langle x, Ax \rangle \ge 0 \quad \forall x
	\end{align*}
	which is equivalent to \(A\) having only non-negative eigenvalues, as can be seen
	by testing the definition with eigenvectors and utilizing linearity to generalize.
	We denote
	\begin{align*}
		A \precsim B :\iff B-A \text{ is positive definite}.
	\end{align*}
\end{definition}

\subsubsection{Convexity}

Instead of non-negative eigenvalues we can demand ``convexity'' which does
not require the existence of a Hessian, but is equivalent to its positive
definiteness (non-negative eigenvalues) if it exists.
\begin{definition}[Convexity]\label{def: convexity}
	A function \(f:\reals^\dimension\to \reals\) is called ``convex'', if 
	\begin{align}\label{eq: convex combination definition}
		f(x + \hesseEV(y-x)) \le f(x) + \hesseEV (f(y)-f(x))
		\qquad \forall x,y\in\reals^\dimension,\ \forall\hesseEV\in[0,1].
	\end{align}
	It can be shown \parencite[e.g.][Prop.
	1.1]{bubeckConvexOptimizationAlgorithms2015} that this definition is
	equivalent to non-emptyness of the set of subgradients for all \(x\)
	\begin{align}\label{eq: subgradient definition}
		\partial f(x) \coloneqq \{
			g \in\reals^\dimension: f(x) + \langle g, y-x\rangle \le f(y)
			\quad \forall y\in\reals^\dimension
		\}.
	\end{align}
	If \(f\) is differentiable and convex, then the gradient \(\nabla f\) is in
	this set of subgradients \parencite[e.g.][Prop.
	1.1]{bubeckConvexOptimizationAlgorithms2015} and therefore makes it non-empty.
	This means that for differentiable \(f\) the statement
	\begin{align*}
		f(x) + \langle\nabla f(x), y-x\rangle \le f(y)\qquad \forall x,y\in\reals^\dimension
	\end{align*}
	is another possible definition. If \(f\) is twice differentiable, then another
	equivalent statement \parencite[e.g.][Theorem
	2.1.4]{nesterovLecturesConvexOptimization2018} is the positive definiteness
	of \(\nabla^2 f\)
	\begin{align*}
		0\precsim\nabla^2 f.
	\end{align*}
\end{definition}
\begin{figure}[h]
	\centering
	\def\svgwidth{1\textwidth}
	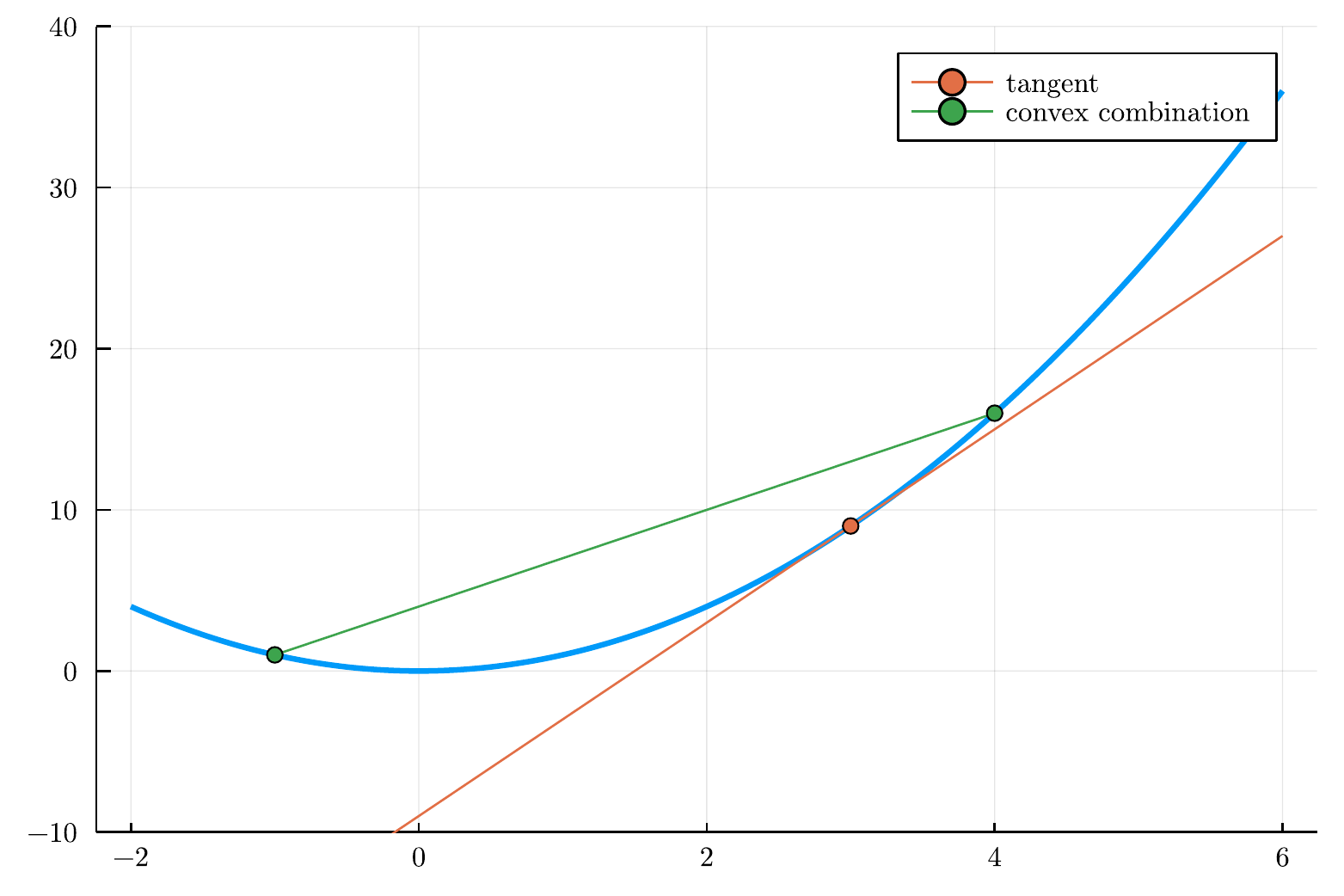
	\caption{
		A visualization of convexity definitions for a function plotted in blue.
		We can see that a convex combination of two function values will stay
		above the convex combination of its inputs
		(\ref{eq: convex combination definition}) and that there is a slope inducing
		a line which touches the function in some point and stays below it
		(\ref{eq: subgradient definition}).
	}
	\label{fig: visualize convexity definition}
\end{figure}
\subsubsection{Strong Convexity and Lipschitz Continuous Gradient}

If you recall how the rate of convergence is related to the condition number
\(\condition:=\hesseEV_\dimension/\hesseEV_1\), then it is not surprising that
the lack of a lower and upper bound will have us struggle to get convergence, as
\(\hesseEV_\dimension\) can be arbitrarily large and \(\hesseEV_1\) arbitrarily
small. To get similar convergence results we have to assume something like
\begin{align}\label{eq: upper lower bound on hesse matrix}
	\lbound \identity \precsim \nabla^2 \Loss(\weights) \precsim \ubound \identity, \qquad \lbound,\ubound >0.
\end{align}
We will see that the lower bound is not necessary to get convergence of the
loss function, but it is necessary to get convergence of the weights. It is
also necessary for a linear\footnote{
	In numerical analysis  ``exponential rate of convergence'' is usually called
	``linear convergence'' because the relationship
	\[
		\|\weights_{n+1} - \minimum\| \le c \|\weights_n - \minimum\|
	\]
	for some constant \(c\) is linear. In general, convergence order \(q\) implies
	\[
		\|\weights_{n+1} -\minimum\| \le c\|\weights_n - \minimum\|^q.
	\]
	This is likely because quadratic convergence (\(q=2\)) already results in an ``actual
	rate of convergence'' of \(c^{\left(2^n\right)}\) which is double exponential and
	repeated application of the exponential function is annoying to write, so
	writers prefer the recursive notation above which motivates this seemingly
	peculiar naming. Slower convergence orders are referred to as ``sub-linear''
	convergence.
}
rate of convergence. When it comes to the upper bound on the other hand recall
that we use it to bound the learning rate, ensuring that
\(1-\lr\hesseEV_\dimension\) is larger than \(-1\). If we did not have an upper
bound, then we could not guarantee this ``stability'' of \ref{eq: gradient descent}.

As we do not want to assume the Hessian exists, let us integrate (\ref{eq:
upper lower bound on hesse matrix}) and use the constants to make sure that we
have equality of the functions and derivatives in \(\weights\). Then we get
\begin{align}\label{eq: bregman divergence upper and lower bound}
	\frac{\lbound}{2}\| \theta - \weights\|^2
	\le \underbrace{
		\Loss(\theta) - \Loss(\weights) - \langle \nabla \Loss(\weights), \theta-\weights\rangle
	}_{=:\bregmanDiv{\Loss}(\theta, \weights)\quad\text{(Bregman Divergence)}}
	\le \frac{\ubound}{2}\| \theta - \weights\|^2.
\end{align}
The ``Bregman Divergence'' is the distance between the real function and its first
Taylor approximation. If we move the first Taylor approximation from the
middle to the sides on the other hand, this represents a quadratic upper and lower
bound on \(\Loss\) as illustrated in Figure~\ref{fig: visualize strong
convexity}. For \(\lbound=0\) we get the definition of convexity again.
\begin{definition}\label{def: strong convexity}
	A function \(\Loss\) is \(\lbound\)-\emph{strongly convex} for some
	\(\lbound>0\), if for all \(x,y\in\reals^\dimension\)
	\begin{align*}
		\Loss(x) + \langle \nabla \Loss(x), y-x\rangle + \frac{\lbound}{2}\|y-x\|^2 \le \Loss(y).
	\end{align*}
\end{definition}

The upper bound of \ref{eq: bregman divergence upper and lower bound} can be
motivated in another way: First notice that the operator norm is equal to the
largest absolute eigenvalue
\begin{align*}
	\|A\| := \sup_{\|x\| =1} \|Ax\|
	= \sup_{\|x\| =1} \sqrt{\langle Ax, Ax\rangle}
	= \sup_{\|x\| =1} \sqrt{\sum_{i=1}^\dimension \hesseEV_i^2 x_i^2},
	= \max_{i=1,\dots,\dimension} |\hesseEV_i|,
\end{align*}
where we have used the orthonormal basis of eigenvectors \((v_i)_i\) with
eigenvalues \((\hesseEV_i)_i\) of a (symmetric) matrix A to represent a point
\(x\in\reals^\dimension\)
\begin{align*}
	x = \sum_{i=1}^{d}x_i v_i.
\end{align*}
As we have no negative eigenvalues in the convex case, the upper bound on the
Hessian is equivalent to a bound on its operator norm. And since a
bounded derivative is equivalent to Lipschitz continuity of the function itself
(cf. Lemma~\ref{lem-appendix: lipschitz and bounded derivative}), another way
to express our upper bound is by requiring Lipschitz continuity of the gradient
\(\nabla \Loss\). Under convexity, Lipschitz continuity and our upper bound are
in fact equivalent.
\begin{figure}[h]
	\centering
	\def\svgwidth{1\textwidth}
	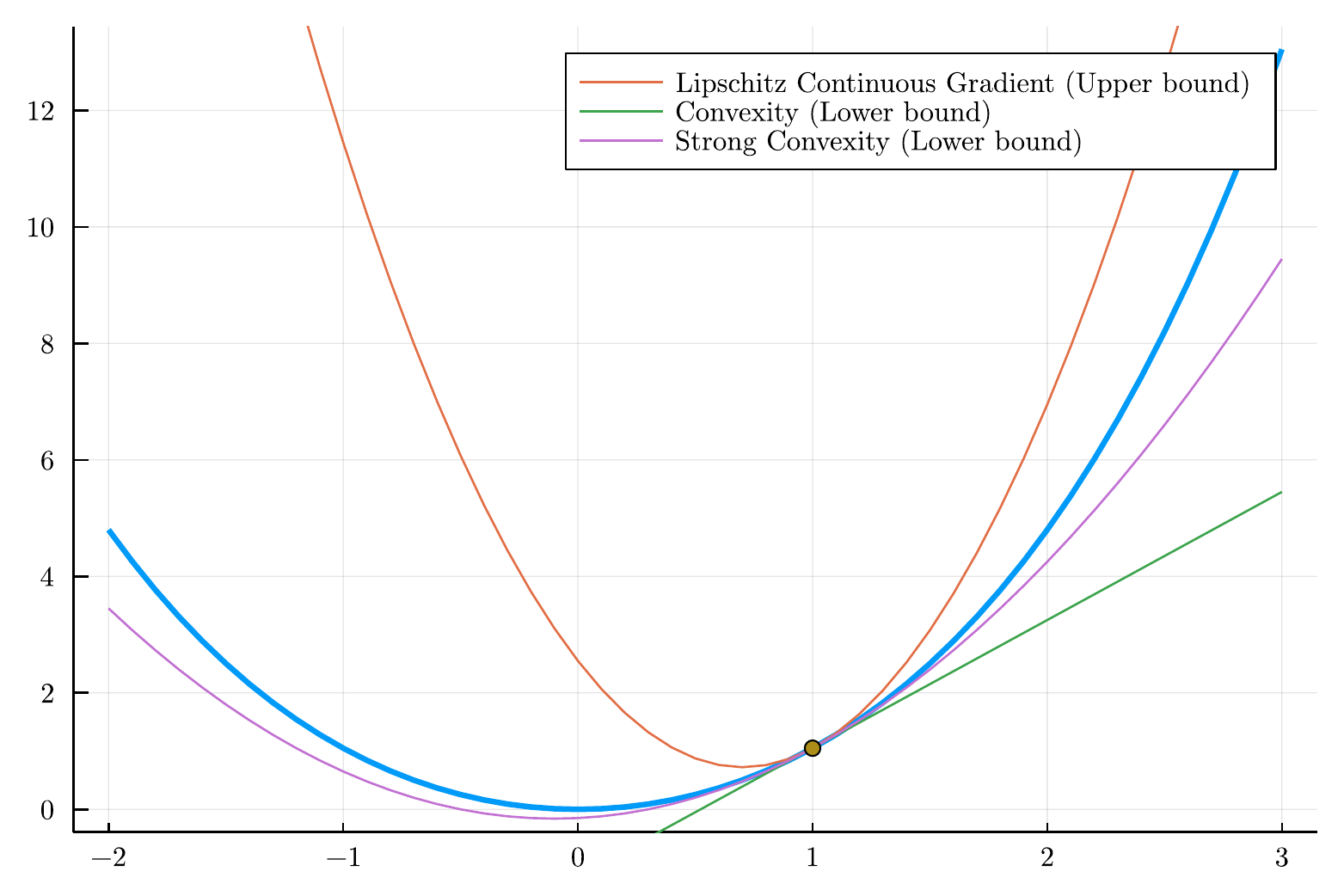
	\caption{
		Strong convexity provides a tighter lower bound than convexity
		while Lipschitz continuity of the gradient provides an upper bound.
		Visualized for a function plotted in blue.
	}
	\label{fig: visualize strong convexity}
\end{figure}
\begin{lemma}[{\cite[Lemma~1.2.3,2.1.5]{nesterovLecturesConvexOptimization2018}}]
	\label{lem: Lipschitz Gradient implies taylor inequality}
	If \(\nabla \Loss\) is \(\ubound\)-Lipschitz continuous, then
	\begin{align*}
		| \Loss(y) - \Loss(x) - \langle \nabla \Loss(x), y-x\rangle |
		\le \tfrac{\ubound}2 \|y-x\|^2
	\end{align*}
	If \(\Loss\) is convex, then the opposite direction is also true.
\end{lemma}
\begin{proof}
	See Appendix \ref{Appdx-lem: Lipschitz Gradient implies taylor inequality}.
\end{proof}
\begin{definition}
	Finally let us borrow some notation from \citeauthor{nesterovLecturesConvexOptimization2018}:
	\begin{description}
		\item[{\(\lipGradientSet[q,p]{\ubound}\)}] the set of \(q\) times
		differentiable, convex functions with an \(\ubound\)-Lipschitz continuous
		\(p\)-th derivative
		\item[{\(\strongConvex[q,p]{\lbound}{\ubound}\)}] the \(\lbound\)-strongly convex
		functions from \(\lipGradientSet[q,p]{\ubound}\)
	\end{description}
\end{definition}

\section{Lipschitz Continuity of the Gradient}
\label{sec: lipschitz continuity of the Gradient}

Using \(x\le|x|\) on Lemma~\ref{lem: Lipschitz Gradient implies taylor inequality}
explains Lipschitz continuity of the gradient as a distance penalty

\begin{align}\label{eq: lin approx + distance penalty notion}
	\Loss(\theta)
	= \overbrace{
		\Loss(\weights_n)+ \langle\nabla\Loss(\weights_n), \theta-\weights_n\rangle 
	}^{\text{linear approximation}}
	+ \overbrace{\underbrace{\bregmanDiv{\Loss}(\theta,\weights_n)}_{
		\xle{(\ref{eq: bregman divergence upper and lower bound})}
		\tfrac\ubound{2} \|\theta-\weights_n\|^2 
	}}^{\text{distance penalty}},
\end{align}
penalizing the fact that we do not have enough information about the loss if we
move farther away from our current point. Our limit on the change of the gradient
allows us to create this worst case upper bound though. Gradient descent
with learning rate \(\frac1\ubound\) minimizes this upper bound.
\begin{lemma}[{\cite[Theorem~2.1.5]{nesterovLecturesConvexOptimization2018}}]
	\label{lem: smallest upper bound}
	Let \(\nabla \Loss\) be \(\ubound\)-Lipschitz continuous, then for any \(\weights\in\reals^d\)
	\begin{align}
		\weights - \tfrac{1}{\ubound}\nabla \Loss(\weights) 
		&= \arg\min_{\theta}\{
			\Loss(\weights)+ \langle\nabla \Loss(\weights), \theta-\weights\rangle + \tfrac{\ubound}{2}\|\theta-\weights\|^2 
		\}\\
		\label{eq: min approximation}
		\Loss\left(\weights-\tfrac{1}{\ubound}\nabla\Loss(\weights)\right)
		&\le \min_{\theta} \{
			\Loss(\weights)+ \langle\nabla \Loss(\weights), \theta-\weights\rangle
			+ \tfrac{\ubound}{2}\|\theta-\weights\|^2 
		\} \\ \nonumber
		&= \Loss(\weights) - \tfrac{1}{2\ubound} \|\nabla \Loss(\weights)\|^2
	\end{align}
	More generally, we have for all \(\lr\in(0, \tfrac{2}{\ubound})\)
	\begin{align*}
		\Loss(\weights - \lr\nabla\Loss(\weights)) - \Loss(\weights)
		\le-\lr(1-\tfrac{\ubound}2 \lr)\|\nabla\Loss(\weights)\|^2.
	\end{align*}
\end{lemma}
\begin{proof}
	The direction is irrelevant for the distance penalty term \(\|\theta-\weights\|^2\).
	So if we keep the length of the vector \(\theta-\weights\) constant at \(r\)
	we can first optimize over the direction and then later optimize over its
	length separately. Now the Cauchy-Schwarz inequality implies the optimal
	direction is along the gradient vector:
	\begin{align*}
		\big\langle 
			\nabla\Loss(\weights),
			-\tfrac{r}{\|\nabla\Loss(\weights)\|}\nabla\Loss(\weights)
		\big\rangle
		&= - \|\Loss(\weights)\| r\\
		&\lxle{\text{C.S.}}
		\min_{\theta} \{\langle \nabla \Loss(\weights), \theta-\weights \rangle : \|\theta-\weights\|=r \}.
	\end{align*}
	So we only need to find the right step size to optimize the entire equation
	\begin{align*}
		&\min_{\theta} \{
			\Loss(\weights)+ \langle\nabla \Loss(\weights), \theta-\weights\rangle
			+ \tfrac{\ubound}{2}\|\theta-\weights\|^2 
		\} \\
		&= \min_{\lr} \{ \Loss(\weights) + \langle \nabla \Loss(\weights), -\lr \nabla \Loss(\weights)\rangle
		+ \tfrac{\ubound}{2}\|\lr\nabla \Loss(\weights)\|^2\}.
	\end{align*}
	And a simple expansion of terms results in
	\begin{align*}
		\Loss(\weights-\lr\nabla\Loss(\weights))
		&\lxle{(\ref{eq: lin approx + distance penalty notion})}
		\Loss(\weights) + \langle \nabla \Loss(\weights), -\lr \nabla \Loss(\weights)\rangle
		+ \tfrac{\ubound}{2}\|\lr\nabla \Loss(\weights)\|^2\\
		&=\Loss(\weights) - \underbrace{\lr(1-\tfrac{\ubound}2\lr)}_{=\tfrac{2}{\ubound}\alpha(1-\alpha)}\|\nabla\Loss(\weights)\|^2
	\end{align*}
	using the reparametrization
	\begin{align*}
		\lr=\tfrac{2}{\ubound}\alpha\qquad \alpha \in (0,1).
	\end{align*}
	Therefore the learning rate \(\lr=\tfrac1\ubound\), right in the middle of the
	interval \((0, \tfrac2\ubound)\) of guaranteed decrease, minimizes our upper bound.
\end{proof}
\begin{remark}[Markovian Optimality of GD]
	What we have seen is that \ref{eq: gradient descent} is optimal in a ``Markovian''
	(history independent) sense, with regard to a certain upper bound. As every
	step optimizes the decrease of our upper bound, regardless of the history
	leading up to this point. But it is not optimal over \emph{multiple} steps, as
	it might sometimes be sensible to incorporate past gradient information into
	our current step (Chapter~\ref{chap: momentum} Momentum).
\end{remark}

\begin{remark}[Mirror Descent]\label{rem: mirror descent}
	The basic idea of ``Mirror Descent''
	\parencite[e.g.][]{guptaAdvancedAlgorithmsFall2020,chenLargeScaleOptimizationData2019,bubeckConvexOptimizationAlgorithms2015}
	is: If we had a better upper bound on the Bregman divergence than Lipschitz
	continuity of the gradient 
	\begin{align*}
		\min_y f(y)
		&= \min_y \{f(x) + \langle f(x), y-x\rangle +\bregmanDiv{f}(y,x)\} \\
		&\le \min_y\{f(x) + \langle f(x), y-x\rangle +\tfrac{\ubound}{2} \|y-x\|^2\},
	\end{align*}
	then we could also improve on \ref{eq: gradient descent} by taking the minimum of this
	improved upper bound.
\end{remark}

\subsection{Convergence of the Gradient}

\begin{theorem}[{\cite[Theorem~4.10]{bottouOptimizationMethodsLargeScale2018}}]
	\label{thm: convergence of the gradient (only lip cont)}
	For a loss function with \(\ubound\)-Lipschitz gradient and finite lower bound
	(i.e.\ \(\inf_{\weights}\Loss(\weights) >-\infty\)) \ref{eq: gradient descent} with
	learning rate \(\lr\in(0,\tfrac2\ubound)\) induces convergence of the gradient.
	The convergence rate is at least \(o(\tfrac1n)\). More precisely, the average
	gradient converges with rate \(O(\tfrac1n)\)
	\begin{align*}
		\frac1{n}\sum_{k=0}^n \|\nabla\Loss(\weights_k)\|^2
		&\le \frac{\Loss(\weights_0) - \inf_{\weights}\Loss(\weights)}{n\lr(1-\tfrac{\ubound}2\lr)}.
	\end{align*} 
\end{theorem}
\begin{proof}
	Let the weights \(\weights_n\) be generated with \ref{eq: gradient descent} with learning rate \(\lr\).
	Then we have	
	\begin{align*}
		\sum_{k=0}^n \|\nabla\Loss(\weights_k)\|^2
		&\xle{(\ref{eq: loss implies gradient convergence})}
		\sum_{k=0}^n \frac{\Loss(\weights_k) - \Loss(\weights_{k+1})}{\lr(1-\tfrac{\ubound}2\lr)}
		\xeq{\text{telescoping}} \frac{\Loss(\weights_0) - \Loss(\weights_{n+1})}{\lr(1-\tfrac{\ubound}2\lr)}
		\\
		&\le \frac{\Loss(\weights_0)
		- \inf_{\weights}\Loss(\weights)}{\lr(1-\tfrac{\ubound}2\lr)}.
	\end{align*}
	Taking \(n\to\infty\) this implies summability, and thus convergence of the
	gradients. Since the harmonic series \(\sum_{k=0}^n \tfrac1{k}\) is divergent,
	the gradients must necessarily converge faster.
\end{proof}

Since we cannot guarantee convergence to a (global) minimum as convergence of
the gradient only implies convergence to a critical point, it is the weakest
form of convergence.

\begin{theorem}[Convergence Chain]\label{thm: convergence chain}
	If the loss \(\Loss\) has Lipschitz continuos gradient, then we have
	\begin{align*}
		\text{weight convergence}
		\implies \text{loss convergence}
		\implies \text{gradient convergence}.
	\end{align*}
	More precisely,
	\begin{align}
	\label{eq: weight implies loss convergence}
		\tfrac{1}{2\ubound}\|\nabla\Loss(\weights)\|^2
		&\le \Loss(\weights) - \Loss(\minimum)\\
	\label{eq: loss implies gradient convergence}
		&\le \tfrac{\ubound}{2}\|\weights-\minimum\|^2.
	\end{align}
\end{theorem}
\begin{proof}
	The first claim follows from Lemma~\ref{lem: smallest upper bound} (\ref{eq: min approximation})
	\begin{align*}
		\Loss(\minimum) \le \Loss(\weights -\tfrac{1}{\ubound}\nabla\Loss(\weights))
		\le \Loss(\weights) - \tfrac{1}{2\ubound}\|\nabla\Loss(\weights)\|^2.
	\end{align*}
	Claim (\ref{eq: loss implies gradient convergence}) follows from Lemma~\ref{lem:
	Lipschitz Gradient implies taylor inequality} using
	\(\nabla\Loss(\minimum)=0\):
	\begin{align*}
		\Loss(\weights) - \Loss(\minimum)
		= |\Loss(\weights) - \Loss(\minimum) - \langle \nabla\Loss(\minimum), \weights-\minimum\rangle|
		&\le \tfrac{\ubound}{2}\|\weights-\minimum\|^2
		\qedhere
	\end{align*}
\end{proof}

\subsection{Discussion}

Lipschitz continuity of the gradient provides us with a ``trust bound''
\(\lr\in(0,\frac{2}{\ubound})\) on the gradient. It limits the rate of change of the
gradient, and thus allows us to formulate and optimize over an upper bound of
the loss function, resulting in learning rate \(\frac{1}{\ubound}\).

If the gradient can change faster, we can trust it less and have to take smaller
steps. Even with infinitesimal steps (i.e.\ in the \ref{eq: gradient flow} case)
we need at least uniform continuity of \(\nabla\Loss\) such that the inequality
\begin{align}\label{eq: bounded gradient integral}
	\int_{t_0}^\infty \|\nabla \Loss(\weights(s))\|^2 ds
	&= \Loss(\weights(t_0)) - \lim_{t\to\infty} \Loss(\weights(t)) \\
	&\le \Loss(\weights(t_0)) - \inf_{\weights} \Loss(\weights) < \infty \nonumber
\end{align}
derived from (\ref{eq: gradient integral}) is sufficient for a convergent
\(\|\nabla \Loss(\weights(t))\|\), otherwise the gradients could be arbitrarily
large or small. The discrete version of this integral inequality is the inequality
from Theorem~\ref{thm: convergence of the gradient (only lip cont)}
\begin{align*}
	\frac1{n}\sum_{k=0}^n \|\nabla\Loss(\weights_k)\|^2
	&\le \frac{\Loss(\weights_0) - \inf_{\weights}\Loss(\weights)}{n\lr(1-\tfrac{\ubound}2\lr)}.
\end{align*} 
Note that if we use \(\tfrac{1}{\ubound}\) as the time increment \(\lr\)
between the \(\weights_k\), we only lose the factor \(\tfrac12\) in our
discretization of the continuos version (\ref{eq: bounded gradient integral}) so
the bound is quite tight.

Another motivation for Lipschitz continuity is the fact that the gradient defines
the \ref{eq: gradient flow} ODE, and it is very common to use Lipschitz
continuity to argue for existence, uniqueness and stability of ODEs.

If the gradient does not converge, neither can the weights, since their updates
(i.e., the difference between \(\weights_{n+1}\) and \(\weights_n\))
are proportional to the gradient. So convergence of the gradient is necessary.

\subsubsection{Gradient Convergence does not imply Weight Convergence}

If we only have that the derivative \(\dot{\weights}(t) = -\nabla \Loss(\weights(t))\)
converges, then the logarithm is an obvious example where the other direction
is wrong. Convergence of the gradient is therefore necessary but not
sufficient.

If the series of unsquared gradient norms were finite as well on the other hand,
the \(\weights_n\) would be a Cauchy sequence, i.e.
\begin{align*}
	\|\weights_n - \weights_m \|
	\le \sum_{k=m}^{n-1} \|\weights_{k+1} - \weights_k\|
	= \lr \sum_{k=m}^{n-1} \|\nabla \Loss(\weights_k)\|.
\end{align*}
This provides us with some intuition how a situation might look like when the
gradient converges but not the sequence of \(\weights_n\). The gradient would
behave something like the harmonic series as its squares converges and the weights
\(\weights_n\) would behave like the partial sums of the harmonic series which
behaves like the logarithm in the limit\footnote{
\(
	\log(n+1) = \int_1^{n+1} \tfrac{1}x dx \le \sum_{k=0}^n \tfrac1k \le 1+ \log(n)
\)
}.
So for
\begin{align*}
	\Loss(\weights) = \exp(-\weights)
\end{align*}
the \ref{eq: gradient flow} ODE is solved by
\begin{align*}
	\weights(t) &= \log(t),\\
	\nabla \Loss(t) &= -\tfrac1t,
\end{align*}
which provides intuition how ``flat'' a minima has to be to cause such behavior.
The minimum at infinity has an infinitely wide basin which flattens out
more and more. If we wanted such an example in a bounded space we would have
to try and coil up such an infinite slope into a spiral, which spirals outwards
to avoid convergence. A spiral is of course not convex, so this hints
at why we will need a notion of increasing or infinite dimensions to
provide complexity bounds in Section~\ref{sec: complexity bounds} essentially
spiralling through the dimensions.

Since our example is one dimensional, we can also immediately see the eigenvalue
\begin{align*}
	\nabla^2 \Loss(\weights(t)) = \exp(-\weights(t)) = \tfrac1{t},
\end{align*}
which decreases towards zero as \(t\to\infty\), stalling the movement towards
the minimum. But we do have convexity in this example, so we can get convergence
of the loss as we will see in the next section.

\section{Convergence with Convexity}\label{sec: convex convergence theorems}

We have used Lemma~\ref{lem: smallest upper bound} for an upper bound on the
sum of gradients, using the fact that the distance of the current loss to the
minimal loss is bounded. But if we want to ensure convergence of the loss, we
want to flip this around and find a lower bound on the gradient to ensure
the loss decreases at a reasonable speed.

One way to ensure such a lower bound on the gradient is using convexity.

\begin{theorem}[{\cite[Theorem~2.1.14]{nesterovLecturesConvexOptimization2018}}]
	\label{thm: convex function GD loss upper bound}
	Let \(\Loss\in\lipGradientSet{\ubound}\), let \(\minimum\) be the unique minimum.
	Then \ref{eq: gradient descent} with learning rate \(0 < \lr < 2/\ubound\) results in
	\begin{align}\label{eq: convex function loss upper bound}
		\Loss(\weights_n) - \Loss(\minimum)
		\le \frac{2\ubound\|\weights_0 - \minimum\|}{4 + n\ubound\lr(2-\ubound\lr)}.
	\end{align}
	The optimal rate of convergence 
	\begin{align*}
		\Loss(\weights_n) - \Loss(\minimum)
		\le \frac{2\ubound\|\weights_0-\minimum\|}{4+n}
		\in O\Big(\frac{\ubound\|\weights_0-\minimum\|}{n}\Big)
	\end{align*}
	is achieved for \(\lr=\tfrac1{\ubound}\).
\end{theorem}
\begin{proof}
	To lower bound our gradient we use convexity and the Cauchy-Schwarz inequality
	\begin{align*}
		\Loss(\weights_n)-\Loss(\minimum)
		\xle{\text{convexity}} -\langle \nabla\Loss(\weights_n), \minimum-\weights_n\rangle
		\xle{\text{C.S.}} \|\weights_n - \minimum\| \|\nabla\Loss(\weights_n)\|.
	\end{align*}
	Assuming we have
	\begin{align*}
		\|\weights_n - \minimum \| \le \|\weights_0 - \minimum\|,
	\end{align*}
	this provides us with a lower bound 
	\begin{align}\label{eq: lower bound gradient size}
		\|\nabla\Loss(\weights)\|
		\ge \frac{\Loss(\weights_n)-\Loss(\minimum)}{\|\weights_n-\minimum\|}
		\ge \frac{\Loss(\weights_n)-\Loss(\minimum)}{\|\weights_0-\minimum\|}
	\end{align}
	Now we need to make sure, that our weights do not move away from our minimum
	\(\minimum\). Utilizing convexity again, we will prove in Lemma~\ref{lem:
	bregmanDiv lower bound} that
	\begin{align*}
		\langle\nabla\Loss(\weights) - \nabla\Loss(\minimum), \weights-\minimum\rangle
		\ge \tfrac1{\ubound} \|\nabla\Loss(\weights) - \nabla\Loss(\minimum)\|^2.
	\end{align*}
	Plugging the recursive definition of \ref{eq: gradient descent} into \(\weights_{n+1}\)
	with \(\nabla \Loss(\minimum)=0\) in mind results in
	\begin{align}
		\nonumber
		\|\weights_{n+1} - \minimum\|^2
		&=\|\weights_n -\lr\nabla\Loss(\weights_n) - \minimum\|^2\\
		&= \|\weights_n - \minimum\|^2
		- 2\lr\underbrace{
			\langle \nabla\Loss(\weights_n), \weights_n - \minimum\rangle
		}_{
			\ge \tfrac{1}\ubound \|\nabla \Loss (\weights_n) - \nabla\Loss(\minimum)\|^2
			\mathrlap{\quad \text{using Lemma~\ref{lem: bregmanDiv lower bound}}}
		} + \lr^2\|\nabla \Loss(\weights_n)\|^2
		\\
		&\le \|\weights_n - \minimum\|^2 - 
		\underbrace{\lr}_{>0}\underbrace{(\tfrac{2}{\ubound}-\lr)}_{>0}
		\|\nabla\Loss(\weights_n)\|^2.
		\label{eq: decreasing weight difference}
	\end{align}
	Now that we justified our lower bound of the gradient, we can bound the
	convergence of the loss using Lemma~\ref{lem: smallest upper bound}
	\begin{align*}
		\Loss(\weights_{n+1}) - \Loss(\minimum)
		&\lxle{(\ref{eq: min approximation})} \Loss(\weights_n)-\Loss(\minimum)
		- \overbrace{\lr(1-\lr\tfrac{\ubound}2)}^{=:\xi} \|\nabla \Loss(\weights_n)\|^2\\
		&\lxle{(\ref{eq: lower bound gradient size})}
		\Loss(\weights_n)-\Loss(\minimum)
		\underbrace{
			\left(1 - \tfrac{\xi}{\|\weights_0-\minimum\|^2}(\Loss(\weights_n)-\Loss(\minimum))\right)
		}_{\text{``diminishing contraction term''}}.
	\end{align*}
	The convergence rate of such a recursion can be bounded (which we will do in
	Lemma~\ref{lem: upper bound on diminishing contraction}). This results in 
	\begin{align*}
		\Loss(\weights_n)-\Loss(\minimum)
		&\le \frac{1}{
			\frac{1}{\Loss(\weights_0)-\Loss(\minimum)} + \tfrac{\xi}{\|\weights_0-\minimum\|^2}n
		}.
	\end{align*}
	Since our upper bound is increasing in \(\Loss(\weights_0)-\Loss(\minimum)\)
	we can use
	\begin{align*}
		\Loss(\weights_0) - \Loss(\minimum)
		&\xle{(\ref{eq: weight implies loss convergence})}
		\frac{\ubound}{2}\|\weights_0-\minimum\|^2
	\end{align*}
	to get an upper bound dependent only on \(\|\weights_0 - \minimum\|^2\)
	\begin{align*}
		\Loss(\weights_n)-\Loss(\minimum)
		&\le \frac{\|\weights_0-\minimum\|^2}{
			\frac{2}{\ubound} + \xi n
		}
		= \frac{2\ubound\|\weights_0-\minimum\|^2}{
			4 + n 2\ubound \xi
		}.
	\end{align*}
	Finally we obtain our claim (\ref{eq: convex function loss upper bound})
	using our definition of \(\xi\).
\end{proof}
Now to the Lemma we promised, to ensure that our weights do not move away further
from the minimum.
\begin{lemma}[{\cite[Theorem~2.1.5]{nesterovLecturesConvexOptimization2018}}]\label{lem: bregmanDiv lower bound}
	For \(f\in\lipGradientSet{\ubound}\) we have
	\begin{subequations}
	\begin{align}
		\bregmanDiv{f}(y,x)
		&\ge \tfrac{1}{2\ubound}\|\nabla f(x) - \nabla f(y)\|^2,\\
		\label{eq: bregmanDiv lower bound b}
		\langle \nabla f(x) - \nabla f(y), x-y\rangle
		&\ge \tfrac{1}{\ubound}\|\nabla f(x) - \nabla f(y)\|^2.
	\end{align}	
	\end{subequations}	
\end{lemma}
\begin{proof}
	Since \(y\mapsto \langle\nabla f(x), y-x\rangle\) is linear, we know that
	\begin{align*}
		\phi(y):=\bregmanDiv{f}(y,x) = f(y)-f(x)-\langle \nabla f(x), y-x\rangle 
	\end{align*}
	still has \(\ubound\)-Lipschitz gradient
	\begin{align}\label{eq: bregman divergence gradient}
		\nabla\phi(y) = \nabla f(y) - \nabla f(x).
	\end{align}
	This gradient is equal to the change of gradient from \(x\) to \(y\) and 
	therefore represents the error rate we are making assuming a constant gradient
	and linearly approximating \(f\), which is by definition the derivative of the size of
	\(\bregmanDiv{f}(y,x)\). Now this error rate can be translated
 	back into a real error using the stickiness of the gradient due to
	its Lipschitz continuity. i.e.\ by Lemma~\ref{lem: smallest upper bound} we know that
	\begin{align*}
		\bregmanDiv{f}(x,x) = 0
		\xeq{f\text{ convex}} \min_z\phi(z)
		\xle{(\ref{eq: min approximation})} \phi(y) - \tfrac{1}{2\ubound}\|\nabla \phi(y)\|^2.
	\end{align*}
	Now (\ref{eq: bregman divergence gradient}) implies our first result
	\begin{align*}
		\tfrac{1}{2\ubound}\|\nabla f(y) - \nabla f(x)\|^2
		\le \phi(y) =\bregmanDiv{f}(y,x).
	\end{align*}
	The second inequality follows from adding \(\bregmanDiv{f}(x,y)\) to
	\(\bregmanDiv{f}(y,x)\).
\end{proof}

\begin{lemma}[Diminishing Contraction]
	\label{lem: upper bound on diminishing contraction}
	Let \(a_0 \in [0, \frac{1}{q}]\) for \(q>0\), and assume for a sequence
	\((a_n)_{n\in\naturals_0}\)
	\begin{align}\label{eq: diminishing contraction}
		0\le a_{n+1} \le (1-q a_n)a_n \quad \forall n \ge 0.
	\end{align}
	Then we have
	\begin{align*}
		a_n \le \frac{1}{nq + \frac1{a_0}}\le\frac1{(n+1)q}.
	\end{align*}
\end{lemma}
\begin{proof}
	The proof is an adaptation of a technique used in \textcite[Theorem
	2.1.14]{nesterovLecturesConvexOptimization2018}. We start by dividing the
	reordered contraction (\ref{eq: diminishing contraction}), i.e.
	\begin{align*}
		a_n \ge a_{n+1} + qa_n^2
	\end{align*}
	by \(a_na_{n+1}\) resulting in
	\begin{align*}
		\frac1{a_{n+1}}
		\ge \frac1{a_n} + q \underbrace{\frac{a_n}{a_{n+1}}}_{\ge 1}
		\ge \frac1{a_n} + q,
	\end{align*}
	where we have used that the sequence is monotonously decreasing which
	can be proven by induction. This allows us to lower bound the telescoping
	sum
	\begin{align*}
		\frac1{a_n} - \frac1{a_0}
		= \sum_{k=0}^{n-1} \frac1{a_{k+1}} - \frac1{a_k}
		\ge nq
	\end{align*}
	which (after some reordering) results in the first inequality of our claim.
	The second inequality follows from our assumption \(a_0 \le \frac{1}{q}\).
\end{proof}
\begin{remark}
	This bound is tight (see appendix Lemma~\ref{lem-appendix: diminishing contraction}).
\end{remark}

\subsection{Discussion}

Convexity allows us to find explicit rates not just limiting rates. In case
of Theorem~\ref{thm: convergence of the gradient (only lip cont)} we do not
know how many times the gradients got really low (e.g. close to a saddle
point) until they actually stay low. This is all eaten up by the
constant in \(o(\tfrac1n)\).

As we have seen in Subsection~\ref{subsec: Negative Eigenvalues}, negative
eigenvalues act as a repelling force, preventing the transformation (\ref{eq:
Matrix GD Formulation}) from being a contraction and pushing us into convex
regions.

More specifically if we assume that our loss function \(\Loss\) goes to infinity when
our parameters go to infinity, then (using the fact that we are descending down
the loss) we get a bounded (compact) area 
\begin{align*}
	\{\weights : \Loss(\weights) \le \Loss(\weights_0)\} = \Loss^{-1}([0, \Loss(\weights_0)])
\end{align*}
in which we will stay. We can use that boundedness to argue that we will
eventually end up in a convex region if we are pushed out of the non-convex
regions.

That such a convex region exists follows from the existence of a minimum of
\(\Loss\) in that compact region which necessitates non-negative
eigenvalues. And continuity of the second derivative allows us to extend that to 
a local ball around the minimum. Eigenvalues equal to zero
are a bit of an issue, but if they become negative in some epsilon ball
then moving in that direction would lead us further down. This is a contradiction
to our assumption that we created a ball around the minimum of \(\Loss\).

The statement that we are pushed out of the non-convex regions is a bit
dubious, as we could start right on top of local maxima or rims of saddle points.
But as we are ultimately interested in machine learning applications with
stochastic losses, these zero measure areas are of little concern. On the other
hand we already found out in Subsection~\ref{subsec: Negative Eigenvalues},
starting close to such a feature
increases the amount of time it takes for us to escape that area. This means
we cannot really provide an upper bound on the time it takes to end up in
a locally convex area, as we can arbitrarily increase the escape time by
setting the starting point closer and closer to the local maxima or rim of a
saddle point.

So the best we can do without actually getting into probability theory is to
hand-wave this starting phase away with the unlikelihood of it taking too long
due to the exponential repulsion from negative eigenvalues. The probability
theory to make more precise statements is still in active development and will
require a metastability analysis of the stochastic differential equation (SDE)
induced by stochastic gradient descent
\parencite[e.g.][]{bovierMetastabilityPotentialTheoreticApproach2015,nguyenFirstExitTime2019}
resulting in exit time bounds as well as an upper bound on the distance of SGD
to this limiting SDE
\parencite[e.g.][]{liStochasticModifiedEquations2017,ankirchnerApproximatingStochasticGradient2021}.

This difficulty is the reason why we skip this starting phase and only
provide a convergence analysis once we enter the final convex area.

\subsection{Convergence without Lipschitz Continuous Gradient}\label{subsec: subgradient method}

We might need Lipschitz continuity of the gradient for convergence of the
gradients. But if we can not have convergence of the parameters \(\weights\) we
might not care as much about the convergence of the gradient either. In that
case it turns out that we can even get convergence of the loss if we only assume
convexity and \(\lipConst\)-Lipschitz continuity of \(\Loss\) itself
(boundedness of the gradient), i.e.
\begin{align*}
	\Loss \in \lipGradientSet[0,0]{\lipConst}.
\end{align*}
Since we can not guarantee a monotonic decrease with subgradients, the name
``descent'' is generally avoided and we have to keep a running minimum or
average all the weights. A proper treatment is provided in \textcite[Section
2.2.3]{nesterovLecturesConvexOptimization2018} or \textcite[Section
2.1]{bubeckConvexOptimizationAlgorithms2015}. But we will also use the same proof
technique in a more general stochastic setting in Section~\ref{sec: SGD with
Averaging} where this deterministic setting is a special case. The convergence
rate is then only
\begin{align*}
	O\Big(\frac{\|\weights_0-\weights_*\|\lipConst}{\sqrt{n}}\Big).
\end{align*}

\section{Convergence with Strong Convexity}\label{sec: Strong Convexity}

Given that we have proven (\ref{eq: decreasing weight difference}), i.e.
\begin{align}\label{eq: weight iteration}
	\|\weights_{n+1}-\minimum\|^2
	\le \|\weights_n - \minimum\|^2 - \lr(\tfrac2{\ubound} -\lr)\|\nabla\Loss(\weights_n)\|^2
\end{align}
for convex functions and provided a lower bound for the gradient, it might seem
like we should be able to prove a convergence statement for the weights as
well. But the lower bound on our gradient (\ref{eq: lower bound gradient size})
uses the current loss delta
\begin{align*}
	\Loss(\weights_n)	- \Loss(\minimum).
\end{align*}
And if our loss function is really flat, then the loss difference might already
be really small even if we are still far away in parameter space.

Strong convexity prevents this by requiring a minimum of curvature (cf.
Figure~\ref{fig: visualize strong convexity}). With it we can provide a lower
bound of the gradient using the weight difference directly
\begin{align}
	\nonumber
	\|\nabla\Loss(\weights)\|\|\weights -\minimum\|
	&\lxge{\text{C.S.}} \langle \nabla\Loss(\weights) -\nabla\Loss(\minimum), \weights - \minimum \rangle\\
	\nonumber
	&= \bregmanDiv{\Loss}(\weights, \minimum) + \bregmanDiv{\Loss}(\minimum, \weights)\\
	\label{eq: strong convexity implies PL}
	&\lxge{\text{strong convexity}} \lbound \|\weights-\minimum\|^2.
\end{align}
The resulting lower bound on the gradient 
\begin{align*}
	\|\nabla\Loss(\weights)\|\ge \lbound\|\weights -\minimum\|
\end{align*}
allows us to prove convergence of the weights with convergence of the gradients.
More specifically we can bound our weight iteration (\ref{eq: weight iteration})
\begin{align*}
	\|\weights_{n+1}-\minimum\|^2
	\le \left(1-\lr\left(\tfrac2{\ubound} -\lr\right)\lbound\right)
	\|\weights_n - \minimum\|^2.
\end{align*}
In other words we are closing the loop of the convergence type implications in
Theorem~\ref{thm: convergence chain}. Optimizing over the learning rate \(\lr\)
(similarly to Lemma~\ref{lem: smallest upper bound}) results in
\(\lr=\tfrac1\ubound\) and convergence rate
\begin{align*}
	\|\weights_{n+1}-\minimum\|^2
	\le (1-\tfrac{\lbound}{2\ubound})\|\weights_n - \minimum\|^2
	\le (1-\tfrac{1}{2\condition})^{n+1}\|\weights_0-\minimum\|^2
\end{align*}
which is in the same ballpark as we got in the quadratic problem in
Subsection~\ref{subsec: necessary assumptions in the quadratic case}. But with a
bit more work we can actually achieve the same rates.
\begin{theorem}[{\cite[Theorem~2.1.15]{nesterovLecturesConvexOptimization2018}}]
	\label{thm: gd strong convexity convergence rate}
	If \(\Loss\in\strongConvex{\lbound}{\ubound}\), then for learning rate
	\[0 \le \lr \le \tfrac{2}{\ubound+\lbound}\]
	\ref{eq: gradient descent} generates a sequence \(\weights_n\) satisfying
	\begin{subequations}
	\begin{align}
		\label{eq: gd strong convexity convergence rate 1}
		\|\weights_n - \minimum\|
		&\le \left(
			1- 2\lr\frac{\ubound\lbound}{\ubound+\lbound}
		\right)^{\tfrac{n}2}
		\|\weights_0 - \minimum\|, \\
		\label{eq: gd strong convexity convergence rate 2}
		\Loss(\weights_n) - \Loss(\minimum)
		&\le \frac\ubound{2} \left(
			1- 2\lr\frac{\ubound\lbound}{\ubound+\lbound}
		\right)^{n}
		\|\weights_0 - \minimum\|^2.
	\end{align}
	\end{subequations}
	In particular for \(\lr=\tfrac2{\ubound+\lbound}\), we achieve the optimal
	rate of convergence
	\begin{subequations}\label{eq: gd strong convexity optimal rate}
	\begin{align}
		\|\weights_n - \minimum\|
		&\le \left(
			1- \frac{2}{1+\condition}
		\right)^n
		\|\weights_0 - \minimum\|, \\
		\Loss(\weights_n) - \Loss(\minimum)
		&\le \frac\ubound{2} \left(
			1- \frac{2}{1+\condition}
		\right)^{2n}
		\|\weights_0 - \minimum\|^2,
	\end{align}
	\end{subequations}
	where \(\condition:=\tfrac\ubound\lbound\) is the condition number.
\end{theorem}
\begin{proof}
	This proof starts like the proof for convex functions, cf. (\ref{eq:
	decreasing weight difference})
 \begin{align*}
		\|\weights_{n+1} - \minimum\|^2
		&= \| \weights_n -\minimum - \lr\nabla\Loss(\weights_n)\|^2\\
		&= \| \weights_n - \minimum\|^2
		- 2\lr\langle \nabla\Loss(\weights_n), \weights_n - \minimum\rangle
		+ \lr^2 \| \nabla\Loss(\weights_n)\|^2.
	\end{align*}
	But now we need to find a tighter lower bound for the scalar product in the
	middle. Previously we used the lower bound on the Bregman Divergence from
	Lemma~\ref{lem: bregmanDiv lower bound}. Now the lower bound
	from the strong convexity property is also available. But we can not use
	both, suggesting we lose sharpness in our bound. And indeed, we will prove in
	Lemma~\ref{lem: bregmanDiv lower bound (strongly convex)} that this bound can
	be improved to be
 	\begin{align*}
		&\langle \nabla\Loss(\weights_n) -\nabla\Loss(\minimum), \weights_n - \minimum\rangle	\\
		&\ge \tfrac{\ubound\lbound}{\ubound+\lbound}\|\weights_n - \minimum\|^2
		+ \tfrac{1}{\ubound+\lbound}
		\|\nabla\Loss(\weights_n) - \nabla\Loss(\minimum)\|^2.
	\end{align*}
	Using this lower bound we get
 \begin{align*}
		\|\weights_{n+1} - \minimum\|^2
		\le \left(1- 2\lr \tfrac{\ubound\lbound}{\ubound+\lbound}\right)
		\|\weights_n - \minimum\|^2
		+ \underbrace{\lr}_{\ge 0}
		\underbrace{\left(\lr-\tfrac{2}{\ubound+\lbound}\right)}_{\le0}
		\|\nabla\Loss(\weights_n)\|^2,
	\end{align*}
	immediately proving (\ref{eq: gd strong convexity convergence rate
	1}) with induction. Note that we could additionally use our lower bound on the
	gradient. But this does not represent a significant improvement,
	especially for the optimal rates where we max out the learning rate, making
	the term on the right in our inequality equal to zero. Equation~(\ref{eq: gd
	strong convexity convergence rate 2}) follows from our upper bound on the
	loss using weights, cf. Theorem~\ref{thm: convergence chain} (\ref{eq: weight
	implies loss convergence})
	\begin{align*}
		\Loss(\weights_n) - \Loss(\minimum)
		\le \tfrac\ubound{2} \| \weights_n - \minimum\|^2.
	\end{align*}
	To get (\ref{eq: gd strong convexity optimal rate}) we simply have to plug
	in the assumed optimal learning rate \(\lr=\tfrac2{\ubound+\lbound}\) into
	our general rate and show that the rate is still positive. This means that
	we have not overshot zero, implying that the upper bound on our learning
	rate (we assume to be optimal) is binding:
	\begin{align*}
		\left(1- 4 \frac{\ubound\lbound}{(\ubound+\lbound)^2}\right)
		= \frac{(\ubound^2 + 2\ubound\lbound + \lbound^2) - 4 \ubound\lbound}{(\ubound+\lbound)^2}
		= \left(\frac{\ubound-\lbound}{\ubound+\lbound}\right)^2
		= \left(1-\frac{2\lbound}{\ubound+\lbound}\right)^2.
	\end{align*}
	For the representation in (\ref{eq: gd strong convexity optimal rate}) we
	just have to divide both enumerator and denominator by \(\lbound\).
\end{proof}

Now as promised we still have to show the improved lower bound.

\begin{lemma}[{\cite[Theorem~2.1.12]{nesterovLecturesConvexOptimization2018}}]
	\label{lem: bregmanDiv lower bound (strongly convex)}
	If \(f\in\strongConvex{\lbound}{\ubound}\), then for any
	\(x,y\in\reals^\dimension\) we have
	\begin{align*}
		\langle \nabla f(x) - \nabla f(y), x-y\rangle 
		\ge \tfrac{\lbound\ubound}{\lbound+\ubound} \| x-y\|^2
		+ \tfrac{1}{\lbound+\ubound}\|\nabla f(x) -\nabla f(y)\|^2.
	\end{align*}
\end{lemma}
\begin{proof}
	Similarly to Lemma~\ref{lem: bregmanDiv lower bound} we want to
	add \(\bregmanDiv{f}(x,y)\) and \(\bregmanDiv{f}(y,x)\) together to
	get a lower bound for the scalar product equal to this sum. But since we
	already have a lower bound due to strong convexity
	\begin{align}\label{eq: strong convexity implies bregmanDiv lower bound}
		\bregmanDiv{f}(y,x) = f(y) - f(x) -\langle\nabla f(x), y-x\rangle
		\ge \tfrac{\lbound}2 \|y-x\|^2,
	\end{align}
	we first have to ``remove'' this lower bound from \(f\) to apply our other
	lower bound. So we define
	\begin{align*}
		g_x(y) := f(y) - \tfrac{\lbound}2\|y-x\|^2, \qquad
		\nabla g_x(y) = \nabla f(y) - \lbound(y-x),
	\end{align*}
	which still has positive Bregman Divergence, i.e.\ is convex 
	\begin{align*}
		\bregmanDiv{g_x}(y,z)
		&= g_x(y) - g_x(z) - \langle\nabla g_x(z), y-z\rangle \\
		&= 
		\begin{aligned}[t]
			& f(y) - f(z) -\langle\nabla f(z), y-z\rangle \\
			&- \tfrac\lbound{2}\underbrace{\|y-x\|^2}_{
				=\|y-z\|^2 + \mathrlap{2\langle y-z, z-x\rangle + \|z-x\|^2}
			}
			+ \tfrac\lbound{2}\|z-x\|^2
			+ \lbound\langle z-x, y-z\rangle
		\end{aligned}\\
		&= \bregmanDiv{f}(y,z) - \tfrac{\lbound}2\|y-z\|^2
		\xge{(\ref{eq: strong convexity implies bregmanDiv lower bound})} 0,
	\end{align*}
	but has its strong convexity property removed.
	It also has \((\ubound-\lbound)\)-Lipschitz continuous gradient, as can be
	seen by plugging
	\begin{align}
		\label{eq: bregmanDiv upper bound}
		\bregmanDiv{f}(y,z) \xle{\text{Lemma \ref{lem: bregmanDiv lower bound}}} \tfrac{\ubound}2\|y-z\|^2,
	\end{align}
	into the previous equation, and applying Lemma~\ref{lem: Lipschitz Gradient
	implies taylor inequality} again to get \((\ubound-\lbound)\)-Lipschitz
	continuity from our upper bound on the Bregman Divergence.  Therefore
	\(g_x\in\lipGradientSet{\ubound-\lbound}\) and if \(\ubound -\lbound>0\) we
	can apply Lemma~\ref{lem: bregmanDiv lower bound} to get
	\begin{align}
		\bregmanDiv{f}(y,x)
		&= \bregmanDiv{g_x}(y,x) + \tfrac{\lbound}2 \|y-x\|^2
		\nonumber \\
		&\ge \tfrac{1}{2(\ubound-\lbound)} \|\nabla g_x(x) - \nabla g_x(y)\|^2
		+ \tfrac{\lbound}2 \|y-x\|^2 
		\nonumber \\
		\label{eq: improved lower bound strong convex bregman Divergence}
		&\ge \tfrac{1}{2(\ubound-\lbound)}
		\|\nabla f(x) - \lbound x - (\nabla f(y)-\lbound y)\|^2
		+ \tfrac{\lbound}2 \|y-x\|^2.
	\end{align}
	Since the \(\|y-x\|^2\) lower bound improves the convergence rate in
	our convergence theorem (\ref{thm: gd strong convexity convergence rate}), it
	is helpful to view it as the ``good'' part of our lower bound and view the
	application of the gradient lower bound as a fallback.

	The case \(\ubound=\lbound\), which we have to cover separately, offers some
	more insight into that. In this case our lower bound  (\ref{eq:
	strong convexity implies bregmanDiv lower bound}) makes (\ref{eq: bregmanDiv
	upper bound}) an equality. This means we could get
	\begin{align*}
		\langle \nabla f(x) - \nabla f(y), x-y\rangle
		= \bregmanDiv{f}(x,y) + \bregmanDiv{f}(y,x) = \lbound \|y-x\|^2.
	\end{align*}
	But instead we apply Lemma~\ref{lem: bregmanDiv lower bound} to half of it
	which results in the statement of this lemma.

	Now we just have to finish the case \(\ubound>\lbound\). In equation
	(\ref{eq: improved lower bound strong convex bregman Divergence}) we have
	already removed the dependence on our helper function \(g_x\) and will now
	use its symmetry to add together the mirrored Bregman Divergences to get
	\begin{align*}
		&\langle \nabla f(x) - \nabla f(y), x-y\rangle
		= \bregmanDiv{f}(x,y) + \bregmanDiv{f}(y,x) \\
		&\ge \tfrac{1}{\ubound-\lbound}
		\underbrace{\|\nabla f(x) - \lbound x - (\nabla f(y)-\lbound y)\|^2}_{
			=\|\nabla f(x) - \nabla f(y)\|^2
			- 2\lbound \langle \nabla f(x) - \nabla f(y), x-y\rangle
			\mathrlap{+ \lbound^2 \| x-y\|^2}
		}
		+ \lbound \|y-x\|^2.
	\end{align*}
	Moving the scalar product to the left we get
	\begin{align*}
		\overbrace{\frac{ \ubound+\lbound }{ \ubound-\lbound }}^{
			= (1+\tfrac{2\lbound}{\ubound-\lbound})
		}
		&\langle \nabla f(x) - \nabla f(y), x-y\rangle \\
		&\ge \tfrac{1}{\ubound-\lbound}\|\nabla f(x) - \nabla f(y)\|^2
		+ \underbrace{(\tfrac{\lbound^2}{\ubound-\lbound}-\lbound)}_{
			=\frac{\ubound\lbound}{ \ubound-\lbound }
		} \|y-x\|^2.
	\end{align*}
	Dividing by the factor on the left finishes this proof.
 \end{proof}

\subsection{Discussion}

\textcite{karimiLinearConvergenceGradient2020} show that virtually all
attempts to generalize strong convexity are special cases of the
Polyak-\L{}ojasiewicz condition
\begin{align}
	\label{eq: PL condition}\tag{PL}
	\|\nabla \Loss(\weights) \|^2 \ge c(\Loss(\weights)-\inf_{\theta}\Loss(\theta)).
\end{align}
This condition does not require convexity but so called ``invexity'', i.e.
\begin{align*}
	\Loss(\theta) \ge \Loss(\weights) + \nabla\Loss(\weights)\eta(\weights, \theta)
\end{align*}
for some vector valued function \(\eta\). Convexity is a special case for \(\eta(\weights,
\theta)=\weights-\theta\). Considering Lemma~\ref{lem: smallest upper bound}
and Theorem~\ref{thm: convergence chain}, it is perhaps not too surprising that a
lower bound on the gradient for a given loss difference allows for
convergence proofs. In fact the convergence proof is just two lines:
Using Lemma~\ref{lem: smallest upper bound} and the \ref{eq: PL condition}-assumption we get
\begin{align*}
	\Loss(\weights_{n+1})- \Loss(\weights_n) \le -\tfrac1{2\ubound} \|\nabla\Loss(\weights_n)\|^2
	\le -\tfrac{c}{2\ubound}(\Loss(\weights_n) - \Loss(\minimum)),
\end{align*}
subtracting \(\Loss(\minimum)-\Loss(\weights_n)\) from both sides we finally get
\begin{align*}
	\Loss(\weights_{n+1}) - \Loss(\minimum)
	\le (1-\tfrac{c}{2\ubound})(\Loss(\weights_n) - \Loss(\minimum)).
\end{align*}
But while convexity comes with a geometric intuition this \ref{eq: PL
condition}-condition does not.

And since we have built intuition on how non-convex regions
repel iterates into convex regions, we now have intuition how general smooth
functions behave. It is more difficult to build this intuition for non-invex
regions. So we do not really gain much intuition by this generalization. But
this generalization highlights the attributes of (strong)-convexity we
actually need. And we have basically already proven this condition with
(\ref{eq: strong convexity implies PL}). We only need to use (\ref{eq: weight
implies loss convergence}) from Theorem~\ref{thm: convergence chain}, i.e.
weight convergence implies loss convergence.

\section{Loss Surface}\label{sec: loss surface}

A pain point of arguing that we will end up in \emph{some} convex area
(as we are repelled from non-convex regions), is that we will only 
converge to some \emph{local} minima. This problem has been addressed by
\textcite{pascanuSaddlePointProblem2014} who summarize a distribution analysis
by \textcite{brayStatisticsCriticalPoints2007} of critical points in random
Gaussian fields.

Intuitively the dimension increases the number of eigenvalues
of the Hessian, making it exponentially less likely that \emph{all} of them are
either positive or negative. In other words: In high dimension almost all
critical points are saddle points! But the lower the loss at the critical point,
the likelier it is in fact a minimum. This means that on a Gaussian random field
critical points which are local minima are likely close to the global minimum.

\textcite{pascanuSaddlePointProblem2014}
then confirm this hypothesis empirically by analyzing the share of saddle
points in all critical points found at certain loss levels of a loss function
induced by training models on a reduced MNIST dataset. So it \emph{appears} that
random Gaussian fields approximate deep learning problems sufficiently well to
argue that local minima will generally be close enough to the global minima,
since all other critical points are overwhelmingly likely to be saddle points.
Further theoretical justification for neuronal networks in particular, is
provided by \textcite{choromanskaLossSurfacesMultilayer2015}.

\textcite{garipovLossSurfacesMode2018} found in an empirical study that local
minima can be connected by a simple path (e.g. linear spline with just one
free knot) with near-constant loss. In other words, the local minima do not only
appear to be of similar height (in loss), they also appear to be connected
by a network of valleys. They then propose purposefully making SGD unstable
with higher learning rates once it converged, to explore these valleys and find
other minima to create ensemble models with (e.g. bagging, boosting, etc.). 

\textcite{izmailovAveragingWeightsLeads2019} find that averaging the weights
of these minima instead of creating ensembles from them seems to work as well.
In a convex area this is no surprise as a convex combination of points
has smaller loss than the linear combination of losses (see also
Section~\ref{sec: SGD with Averaging}).

\subsection{Second Order Methods}

The assumption of strong convexity leads some to suggest second order
methods. In particular the Newton-Raphson method which has much better
convergence guarantees. But while we do assume convexity to be able to bound
convergence rates, we still want to use our method on non-convex functions. And
the Newton-Raphson method has a bunch of issues in this case.

As we can see in Section~\ref{sec: visualize gd} (\ref{eq: newton minimum
approx}), the method jumps right to the vertex of a quadratic function. This is
nice if the vertex is a minimum (which convexity guarantees) but it is not so
nice if the vertex is a saddle point or maximum. Additionally the method becomes
unstable (or impossible) for small (or zero) eigenvalues, since then the Hessian
is hard to numerically invert (or no
longer invertible).
On top of these fundamental problems, the storage of the Hessian requires
\(O(\dimension^2)\) space already and its inversion takes \(O(\dimension^3)\)
computation time, which is infeasible for high dimensional data. Especially
considering that \ref{eq: gradient descent} will do \(O(\dimension^2)\) steps
while the Newton-Raphson method does one inversion. Lastly, there are issues
obtaining the Hessian in a stochastic setting.

So the difficulty in crafting a working second order method lies in addressing
all these concerns at once. We will discuss some attempts in Chapter~\ref{chap:
other methods}.

\section{Backtracking}\label{sec: backtracking}

In general we do not know the parameters \(\ubound\) or \(\lbound\). While
the learning rate \(\tfrac1\ubound\) might not be optimal as it treats the largest
eigenvalue preferential over the smallest, it results in decent progress and
is much more stable as we will find out in Section~\ref{sec: nesterov momentum convergence}.

And with an idea first proposed by \textcite{armijoMinimizationFunctionsHaving1966}
we can actually emulate this rate without knowing \(\ubound\). The idea is the
following: We start with some (optimistic) initial guess \(\lr_0 := \tfrac{1}{\ubound_0}>0\) and check
whether
\begin{align}
	\label{eq: armijo's condition}
	\tag{Armijo's Condition}
	\Loss\left(\weights - \lr_0\nabla\Loss(\weights)\right)
	\le \Loss(\weights) - \tfrac{\lr_0}{2}\|\nabla\Loss(\weights)\|^2
\end{align}
is satisfied. If that is not the case, we try \(\lr_1 := \delta\lr_0\)
for some \(\delta\in(0,1)\). We continue testing the same equation with
\begin{align*}
	\lr_n = \delta \lr_{n-1} = \delta^n \lr_0
\end{align*}
until the equation is satisfied and then we use this particular \(\lr_n\).
Lemma~\ref{lem: smallest upper bound} implies that if \(\nabla\Loss\) is
\(\ubound\)-Lipschitz continuous, then for any \(\lr_n < \tfrac{1}{\ubound}\)
this equation is satisfied, since the function is also
\(\tfrac{1}{\lr_n}\)-Lipschitz continuous in that case.

\begin{lemma}
	The number of iterations \(n_A\) to find an appropriate learning rate for some
	\(\Loss\in\lipGradientSet{\ubound}\) starting with guess \(\ubound_0>0\) and
	reduction factor \(\delta\), is bounded by
	\begin{align*}
		n_A \le \begin{cases}
			1 & \ubound_0 \ge \ubound \qquad \text{``too pessimistic''}, \\
			\left\lceil \log\left(\frac{\ubound}{ \ubound_0 }\right)\log\left(\frac1\delta\right) \right\rceil
			& \ubound_0 < \ubound.
		\end{cases}
	\end{align*}
\end{lemma}
\begin{proof}
	If \(\ubound_0\) is too pessimistic (larger than \(\ubound\)), then the test
	will immediately succeed and the only overhead is a single evaluation every step.
	In the other case we only need
	\begin{align*}
		\lr_{n_A} = \frac{\delta^{n_A}}{\ubound_0} \xle{!} \lr = \frac{1}{\ubound},
	\end{align*}
	or equivalently
	\begin{align*}
		n_A \log(\delta) \le \log\left(\tfrac{ \ubound_0 }{ \ubound }\right).
	\end{align*}
	As \(\delta<1\) dividing by \(\log(\delta)\) flips the inequality
	\begin{align*}
		n_A \ge \log\left(\tfrac{ \ubound_0 }{ \ubound }\right)\log(\delta)
		= \log\left(\tfrac{\ubound}{ \ubound_0 }\right)\log\left(\tfrac1\delta\right).
	\end{align*}
	Selecting the
	smallest \(n_A\) such that this equation is satisfied results in our claim.
\end{proof}

\textcite[Lemma 3.1]{truongBacktrackingGradientDescent2019} prove that for any
compact region of an arbitrary \(C^1\) loss function, an upper bound for the
number of iterations to find a suitable \(\lr_{n_A}\) exists. They use this
result to prove convergence (without rates) of \ref{eq: gradient descent} to
some critical point for general \(C^1\) functions.\fxnote{present result here?}

They also propose a more sophisticated rule (two-way backtracking) which uses
the previous learning rate as a starting point for the next iteration. This
gambles on the fact that the Lipschitz constant does not change much in a
neighborhood of our current iterate. Experimental results seem promising.

\section{Complexity Bounds}\label{sec: complexity bounds}

Now that we found convergence rates for \ref{eq: gradient descent} on (strongly) convex
problems, the question is, can we do better? i.e.\ how fast could an algorithms
possibly be? To tackle this problem let us first make an assumption about what
this algorithm can do. Unrolling our \ref{eq: gradient descent} algorithm
\begin{align*}
	\weights_n = \weights_0 - \lr\sum_{k=0}^{n-1} \nabla \Loss(\weights_k)
\end{align*}
we can see that even if we allowed custom learning rates \(\lr_n\) for every
iteration \(n\), i.e.
\begin{align*}
	\weights_n = \weights_0 - \sum_{k=0}^{n-1} \lr_k \nabla \Loss(\weights_k)
\end{align*}
we would still end up in the linear span of all gradients, shifted by \(\weights_0\).
And since we are in the class of first order optimization methods, where we are
only provided with the function evaluation itself and the gradient, an obvious
assumption for a class of optimization methods might be
\begin{assumption}[{\cite[Assumption 2.1.4]{nesterovLecturesConvexOptimization2018}}]
	\label{assmpt: parameter in linear hull of gradients}
	The \(n\)-th iterate of the optimization method is contained in the span of all
	previous gradients shifted by the initial starting point
	\begin{align*}
		\weights_n \in \linSpan\{\nabla \Loss(\weights_k) : 0\le k \le n-1\} + \weights_0.
	\end{align*}
\end{assumption}
So how can we utilize this assumption to construct a function which is difficult to
optimize?  \textcite{gohWhyMomentumReally2017} provides an intuitive
interpretation for a loss function taken from \textcite[Section
2.1.2]{nesterovLecturesConvexOptimization2018} which we will further modify into
the following example.

\begin{figure}[h]
	\centering
	\def\svgwidth{1\textwidth}
	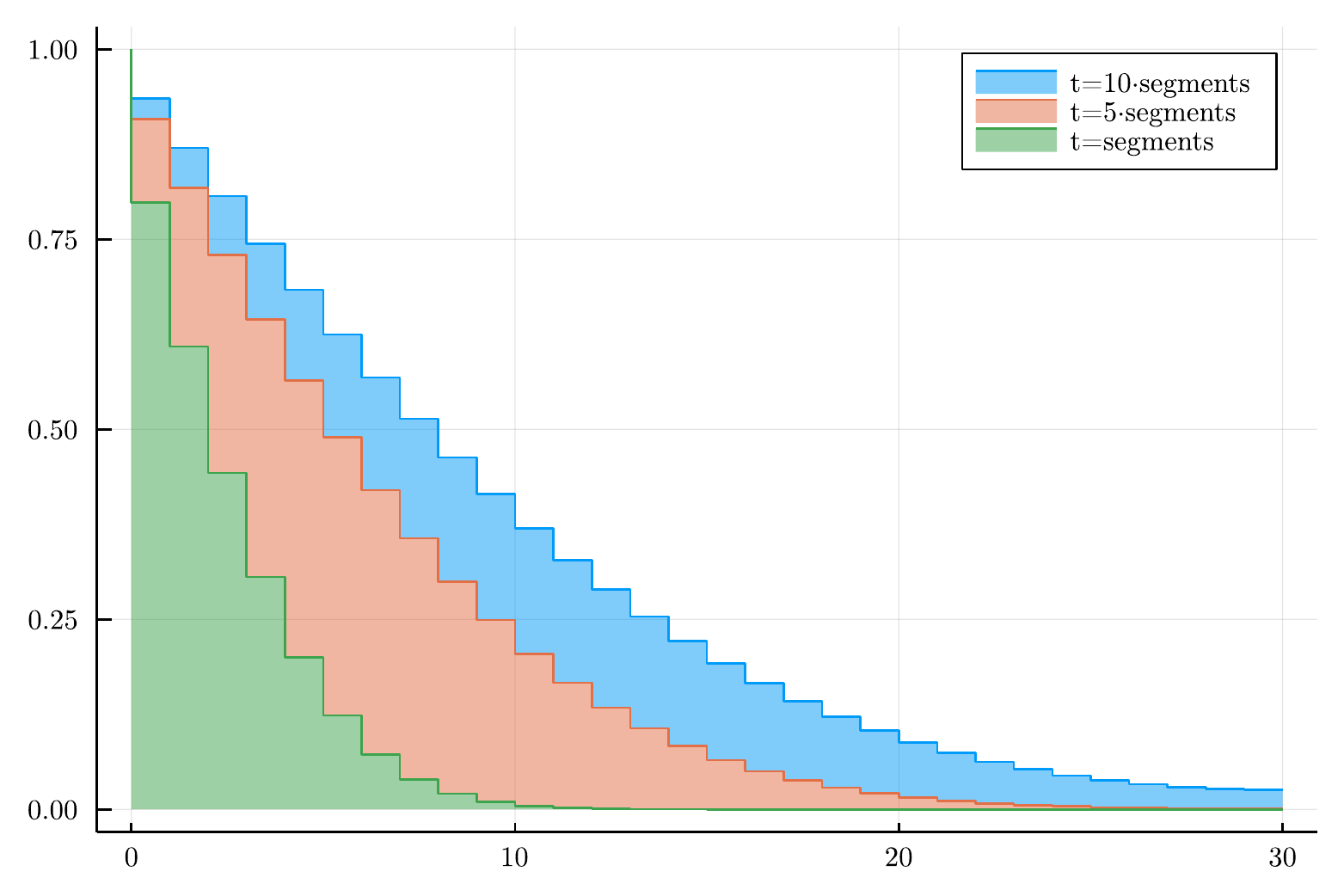
	\caption{
		Even after \(\dimension\) steps (when all segments have a non-zero value)
		\ref{eq: gradient descent} is still far off from the solution. Even when
		using the learning rate \(\tfrac1\ubound\) as is the case here.
	}
	\label{fig: visualize coloring problem}
\end{figure}
Consider a cold, zero temperature, 1-dimensional rod which is heated from one side
by an infinite heat source at temperature one. Then the second segment of the rod
will only start to get hot once the first segment of it has heated up. Similarly
the \(n\)-th segment will only increase in temperature once the \((n-1)\)-th segment
is no longer at zero temperature. If the heat transfer would only depend on the
current heat level difference, one could thus express a simplified heat level
recursion for the \(i\)-th segment as follows\footnote{
	Taking first the learning rate to zero and afterwards the space discretization, would
	result in the well known differential equation for heat with boundary
	conditions.
}: 
\begin{align*}
	\weights_{n+1}^{(i)}
	&= \weights_{n}^{(i)}
	+ \lr 
	\begin{cases}
		[\weights_{n}^{(i-1)} - \weights_{n}^{(i)}] + [\weights_{n}^{(i+1)} - \weights_{n}^{(i)}]
		&  1 < i < \dimension, \\
		[1 - \weights_{n}^{(1)}] + [\weights_{n}^{(2)} - \weights_{n}^{(1)}]
		& i = 1, \\
		[\weights_{n}^{(i-1)} - \weights_{n}^{(i)}]
		& i = \dimension.
	\end{cases}
\end{align*}
For the loss function
\begin{align}\label{eq: naive complexity counterexample loss}
	\tilde{\Loss}(\weights)
	:= \frac12 (\weights^{(1)}-1)^2
	+ \frac12 \sum_{k=1}^{\dimension-1} (\weights^{(k)}-\weights^{(k+1)})^2
\end{align}
this is just the \ref{eq: gradient descent} recursion. Now here is the crucial insight:

Since any weight is only affected once its neighbors are
affected, the second weight can not be affected before the second step since
the first weight is only heated in the first step. And since the second weight
will be unaffected until the second step, the third weight will be unaffected
until the third step, etc.

Therefore the \(\dimension - n\) last components will still be zero in
step \(n\), because the linear span of all the gradients so far is still
contained in the subspace \(\reals^n\) of \(\reals^\dimension\). Formalizing
this argument results in the following theorem inspired by \textcite[Theorem
2.1.7]{nesterovLecturesConvexOptimization2018} and \textcite{gohWhyMomentumReally2017}.

\begin{remark}
	This type of loss function is not unlikely to be encountered in real
	problems. E.g. in reinforcement learning with sparse rewards the estimation
	of the value function over states requires the back-propagation of this
	reward through the states that lead up to it. ``Eligibility traces''
	\parencite[Chapter 12]{suttonReinforcementLearningIntroduction2018}
	are an attempt to speed this process up.
\end{remark}
\begin{theorem}[{\cite[Theorem~2.1.7]{nesterovLecturesConvexOptimization2018}}]
	\label{thm: convex function complexity bound}
	For any \(\weights_0\in\reals^d\), any \(n\in\naturals\) such that 
	\[0\le n\le \tfrac12 (d-1),\]
	there exists \(\Loss\in\lipGradientSet[\infty,1]{\ubound}\)
	such that for any first order method \(\firstOrderMethod\)
	satisfying Assumption~\ref{assmpt: parameter in linear hull of gradients}
	we have
	\begin{subequations}
	\begin{align}
		\Loss(\weights_n) - \Loss(\minimum)
		&\ge \frac{\ubound \|\weights_0 - \minimum\|^2}{16(n+1)^2}, \\
		\|\weights_n -\minimum\|^2 
		&\ge \frac12 \|\weights_0 - \minimum\|^2.
	\end{align}
	\end{subequations}
	where \(\minimum = \arg\min_\weights \Loss(\weights)\) is the unique minimum.
\end{theorem}
\begin{proof}
	First, since we could always define
	\(\Loss(\weights):=\tilde{\Loss}(\weights-\weights_0)\)
	we can assume without loss of generality that \(\weights_0=\mathbf{0}\). 
	
	Second, we can define \(\tilde{\dimension} := 2n +1 \le \dimension	\)	
	and interpret \(\weights_0 = \mathbf{0} \in\reals^\dimension\) as an element of
	\(\reals^{\tilde{\dimension}}\). If our constructed loss simply
	ignores the dimensions between \(\tilde{\dimension}\) and \(\dimension\) then
	this part of the gradients will then be zero keeping all \(\weights_n\) in
	\(\reals^{\tilde{\dimension}}\). We can therefore assume without loss of
	generality that 
	\begin{align}\label{eq: w.l.o.g dim=2n+1}
		\dimension=2n+1.
	\end{align}

	Third, we want to make sure that our loss function is actually convex. To
	see this note that we can write the gradient as
	\begin{align}\label{eq: naive complexity example loss gradient}
		\nabla \tilde{\Loss}(\weights)
		= \graphLaplacian\weights + (\weights^{(1)} -1) \stdBasis_1
		= (\graphLaplacian + \stdBasis_1 \stdBasis_1^T)\weights - \stdBasis_1,
	\end{align}
	where \(\stdBasis_1=(1,0,\dots,0)^T\) is the first standard basis
	vector and \(\graphLaplacian\) is the ``Graph Laplacian``\footnote{
		It might seem a bit excessive to introduce the concept of graphs when
		we could have simply calculated the derivative by hand to get the same
		result. But this graph view is very suggestive how one would generalize
		this ``heat spread'' problem to more complicated structures than a
		1-dimensional rod. In particular the connectedness of graphs is related
		to its condition number \parencite{gohWhyMomentumReally2017}. And this
		provides some more intuition for how ``worst case problems'' look like.
	} for the 
	\fxnote{tikz graph}{undirected graph} \(\graph=(\vertices, \edges)\) with 
	\begin{align*}
		\vertices = \{1,\dots, \dimension\},
		\qquad
		\edges = \{(i, i+1) : 1\le i \le \dimension-1\},
	\end{align*}
	\begin{align*}
		\graphLaplacian^{(i,j)} 
		&= 
		\begin{cases}
			[\text{degree of vertex i}] & i=j\\
			-1 & (i,j)\in\edges \text{ or } (j,i)\in\edges\\
			0 & \text{ else}
		\end{cases}
		\\
		&=
		\begin{pmatrix*}[r]
			1 & -1 & 0  & \cdots & \cdots & 0 \\
			-1 & 2 & -1 & \ddots &  &  \vdots \\ 
			0 & -1 & 2 & \ddots & \ddots & \vdots \\
			\vdots & \ddots & \ddots & \ddots & -1 & 0 \\
			\vdots &  & \ddots & -1 & 2 & -1 \\
			0 & \cdots & \cdots & 0  & -1 & 1
		\end{pmatrix*}\in \reals^{\dimension\times\dimension}
	\end{align*}
	This means our loss function is quadratic with
	\begin{align*}
		H := \nabla^2\tilde{\Loss}(\weights) = (A_G + \stdBasis_1 \stdBasis_1^T)
	\end{align*}
	And the Hessian is positive definite because
	\begin{align*}
		\langle \weights , H \weights\rangle
		&= \langle \weights, (A_G + \stdBasis_1 \stdBasis_1^T) \weights \rangle
		= \langle \weights, \stdBasis_1\rangle^2
		+ \langle \weights, A_G \weights\rangle
		\\
		&= (\weights^{(1)})^2
		+ \sum_{k=1}^{\dimension-1}(\weights^{(k)}-\weights^{(k+1)})^2
		\ge 0.
	\end{align*}
	The last equality can be verified by taking the derivative of both sides and
	realizing they are the same, which is sufficient because they are both zero	
	in the origin. We simply used the fact that our original loss function
	is essentially already a quadratic function, except for the constant influx
	at the first tile. And that constant in the derivative has no bearing on the
	Hessian.

	To calculate the operator norm of \(H\) we can use the equality above and	
	\((a-b)^2 \le 2(a^2 + b^2)\) for any real \(a,b\) to get
	\begin{align*}
		\langle \weights, H\weights\rangle
		\le (\weights^{(1)})^2 + 2\sum_{k=1}^{d-1}(\weights^{(k)})^2 + (\weights^{(k+1)})^2 
		\le 4 \sum_{k=1}^{d-1} (\weights^{(k)})^2
		= 4 \|\weights\|^2.
	\end{align*}
	This immediately implies that the largest eigenvalue is smaller than 4. To
	obtain a convex function with a Lipschitz continuos gradient with constant
	\(\ubound\), we can now simply define
	\begin{align*}
		\Loss(\weights):= \frac{\ubound}{4}\tilde{\Loss}(\weights)
		\xeq{(\ref{eq: naive complexity counterexample loss})} \frac\ubound{8}\left[
			(\weights^{(1)}-1)^2
			+ \sum_{k=1}^{\dimension-1} (\weights^{(k)}-\weights^{(k+1)})^2
		\right].
	\end{align*}
	Further, we know that \(\minimum = (1,\dots,1)^T\) achieves a loss of zero which is
	the unique minimum as the loss function is positive and convex. Since we
	assumed w.l.o.g. that \(\weights_0=\mathbf{0}\) we get
	\begin{align}\label{eq: initial distance}
		\|\minimum - \weights_0\|^2 = \dimension.
	\end{align}
	We also know
	that Assumption~\ref{assmpt: parameter in linear hull of gradients} implies
	\begin{align*}
		\weights_n \in \linSpan\{\nabla\Loss(\weights_k): 0 \le k \le n-1\}
		\subseteq \reals^n \subseteq \reals^\dimension,
	\end{align*}
	which immediately results in the second claim of the theorem
	\begin{align*}
		\|\minimum - \weights_n\|^2
		&\ge \sum_{k=n+1}^d \underbrace{(\minimum^{(k)}- \weights_n^{(k)})^2}_{=1}
		= d-n \\
		&\lxeq{(\ref{eq: w.l.o.g dim=2n+1})} n+1
		\xeq{(\ref{eq: w.l.o.g dim=2n+1})} \tfrac{n+1}{2n+1} d
		\ge \tfrac12 d
		\xge{(\ref{eq: initial distance})} \tfrac12 \|\minimum - \weights_0\|^2. 
	\end{align*}

	To get the bound on the loss function, we have to be a bit more subtle.
	On \(\reals^n\subseteq \reals^\dimension\) the modified loss function
	\begin{align}\label{eq: sink loss}
		\Loss_n(\weights) := \frac\ubound{8}\left[
			(\weights^{(1)} -1)^2
			+ (\weights^{(n)} - 0)^2
			+ \sum_{k=1}^{n-1} (\weights^{(k)} - \weights^{(k+1)})^2
		\right]
	\end{align}
	is equal to \(\Loss(\weights)\) because \(\weights^{(i)}=0\) for \(i>n\).
	Therefore we have
	\begin{align*}
		\Loss(\weights_n) - \underbrace{\inf_{\weights}\Loss(\weights)}_{=0}
		= \Loss(\weights_n) = \Loss_n(\weights_n)
		\ge \inf_{\weights}\Loss_n(\weights).
	\end{align*}
	Having a closer look at \(\Loss_n\) will provide us with a lower bound
	without having to know anything about \(\firstOrderMethod\). Since \(\Loss_n\)
	is similarly convex, setting its derivative equal to zero will actually provide us
	with its minimum. And similarly to \(\Loss\) itself, \(\Loss_n\) is a
	quadratic function, so the gradient is just an affine function with a
	constant Hessian. Therefore finding the minimum only requires solving
	a linear equation. One can just go through the same motions as we did
	before to obtain the equation:
	\begin{align}\label{eq: optimality condition for sink loss}
		A_n \weights - \stdBasis_1 \xeq{!} 0, \qquad
		A_n =\begin{cases}
			2 & i=j\le n, \\
			-1 & i=j+1\le n \text{ or } j=i+1\le n, \\
			0 & \text{else}.
		\end{cases}
	\end{align}
	But at least for solving it, it helps to have an intuition what \(\Loss_n\)
	actually represents. Recall that in our motivation, \(\Loss\) represents just
	a single heat source at interval \(\weights^{(1)}\) at constant
	temperature one, which slowly heats up our rod to this level. Now
	\(\Loss_n\) not only has a source, it also has a sink with constant
	temperature zero at \(n\), which cools down the rod from this other side.

	It is therefore quite intuitive that the equilibrium solution
	\(\hat{\weights}_n\) (optimal solution for \(\Loss_n\)) should be a linearly
	decreasing slope
	\begin{align*}
		\hat{\weights}_n^{(i)} = \begin{cases}
			1 - \tfrac{i}{n+1} & i \le n+1,\\
			0	& i \ge n+1.
		\end{cases}
	\end{align*}	
	And this is in fact the solution of our linear equation, (\ref{eq: optimality
	condition for sink loss}), as can be verified by calculation. Since
	\(\weights_n \in \reals^n\) on which \(\Loss\) and \(\Loss_n\) are equal
	we finally get
	\begin{align*}
		\Loss(\weights_n)-\inf_\weights\Loss(\weights)
		&=\Loss(\weights_n) = \Loss_n(\weights_n) \ge \Loss_n(\hat{\weights}_n)
		= \Loss(\hat{\weights}_n)\\
		&= \frac{\ubound}{8}\left[
			\left(-\tfrac1{n+1}\right)^2 + \left(1-\tfrac{n}{n+1}\right)^2
			+ \sum_{k=1}^{n-1}\left(\tfrac{k+1}{n+1}-\tfrac{k}{n+1}\right)^2
		\right]\\
		&= \frac{\ubound}{8}\sum_{k=0}^n \tfrac{1}{(n+1)^2}
		=\frac{\ubound}{8(n+1)}
		\xeq{(\ref{eq: initial distance})} \frac{L\|\weights_0 - \minimum\|^2}{8(n+1)d}\\
		&\lxge{(\ref{eq: w.l.o.g dim=2n+1})}
		\frac{L\|\weights_0 - \minimum\|^2}{16(n+1)^2}.
		\qedhere
	\end{align*}
\end{proof}

\subsection{Complexity Bounds for Strongly Convex Losses}

For a strongly convex loss we simply add a regularization term:
\begin{align*}
	\Loss(\weights)
	&= \frac{\ubound -\lbound}{8} \left[
		(\weights^{(1)}-1)^2
		+ \sum_{k=1}^{\dimension-1} (\weights^{(k)}-\weights^{(k+1)})^2
	\right]
	+ \frac\lbound{2} \| \weights \|^2\\
	&= \frac{\ubound - \lbound}{4} \tilde{\Loss}(\weights)
	+ \frac{\lbound}{2}\| \weights \|^2, \qquad \lbound>0.
\end{align*}
From the proof of Theorem~\ref{thm: convex function complexity bound} we know
about \(\tilde{\Loss}\) that
\begin{align*}
	0 \precsim \nabla^2\tilde{\Loss} \precsim 4\identity.
\end{align*}
Therefore we know that \(\Loss\in\strongConvex[\infty,1]{\lbound}{\ubound}\) because
\begin{align*}
	\lbound\identity
	\precsim \nabla^2\Loss
	&= \frac{\ubound - \lbound}{4} \nabla^2 \tilde{\Loss} + \lbound\identity\\
	&\precsim (\ubound - \lbound + \lbound)\identity = \ubound \identity.
\end{align*}
Using \(\lbound(\condition -1)=\ubound-\lbound\) where
\(\condition=\tfrac\ubound\lbound\) is our condition number and (\ref{eq: naive
complexity example loss gradient}), we can set the gradient to zero
\begin{align*}
	0 &\lxeq{!} \nabla \Loss(\weights)
	=  \frac{\lbound(\condition -1)}{4} \nabla\tilde{\Loss}(\weights) - \lbound \weights\\
	&= \left[
		\frac{\lbound(\condition-1)}{4}(A_G + \stdBasis_1\stdBasis_1^T) - \lbound\identity
	\right]\weights - \frac{\lbound(\condition-1)}{4}\stdBasis_1.
\end{align*}
Assuming \(\condition>1\), this condition can be rewritten as
\begin{align*}
	0 =  \left[
		A_G + \stdBasis_1\stdBasis_1^T + \frac{4}{\condition-1}\identity
	\right] - \stdBasis_1
	\weights.
\end{align*}
\(\condition=1\) implies \(\Loss=\lbound\|\cdot\|^2\) and
\(\minimum=\mathbf{0}\).
All entries on the diagonal except for the last dimension (which only has one
connection and no source or sink) are equal to
\begin{align*}
	2+\frac{4}{\condition-1} = 2\frac{\condition +1}{\condition-1}.
\end{align*}
This results in the system of equations
\begin{subequations}
\label{eq: solution to the strongly convex coloring problem}
\begin{align}
	0&=2\frac{\condition+1}{\condition -1}\weights^{(1)} - \weights^{(2)} -1, \\
	0&=2\frac{\condition+1}{\condition -1}\weights^{(i)}
	- \weights^{(i+1)} - \weights^{(i-1)}
	&& 2\le i <d, \\
	\label{eq: solution to the strongly convex coloring problem c}
	0&= \left(2\frac{\condition+1}{\condition -1} -1 \right)\weights^{(d)}
	- \weights^{(d-1)}.
\end{align}
\end{subequations}
Defining \(\weights^{(0)} :=1\) we can unify the first two equations. Flipping
the sign and dividing by \(\weights^{(i-1)}\) results in
\begin{align*}
	0&= 
	\frac{\weights^{(i+1)}}{\weights^{(i)}}\frac{\weights^{(i)}}{\weights^{(i-1)}}
	- 2\frac{\condition+1}{\condition -1}\frac{\weights^{(i)}}{\weights^{(i-1)}}
	+ 1  && i <d.
\end{align*}
This means that for a solution \(q\) of the quadratic equation
\begin{align}\label{eq: coloring solution quadratic equation}
	0&= q^2 - 2\frac{\condition+1}{\condition -1}q + 1,
\end{align}
\(\weights^{(i)} := q^i\) will be a solution to the first \(d-1\) equations, as
\begin{align*}
	\underbrace{\frac{\weights^{(i+1)}}{\weights^{(i)}}}_{=\frac{q^{i+1}}{q^i}=q}
	\underbrace{\frac{\weights^{(i)}}{\weights^{(i-1)}}}_{=q}
	- 2\frac{\condition+1}{\condition-1}
	\underbrace{\frac{\weights^{(i)}}{\weights^{(i-1)}}}_{=q} +1 = 0.
\end{align*}
Unfortunately neither of the solutions
\begin{align}\label{eq: solutions to strong convexity complexity quadratic problem}
	q_{1/2} &= \frac{\condition+1}{\condition -1} \pm 
	\sqrt{\left(\tfrac{\condition+1}{\condition-1}\right)^2 -1}
	=\frac{\condition +1 \pm \sqrt{4\condition}}{\condition -1}\\
	&= \frac{(\sqrt{\condition}\pm 1)^2}{(\sqrt{\condition}-1)(\sqrt{\condition}+1)}
\end{align}
results in a solution to the last equation, (\ref{eq: solution to the strongly
convex coloring problem c}). This is why we will formulate the following
theorem only for \(\dimension=\infty\)\footnote{
	For large dimensions the real solution seems to be almost
	indistinguishable from the limiting case (e.g. \(\dimension=40,
	\condition=10\) results in a maximal error of the order \(10^{-11}\). The
	error in the first 35 dimensions is even of the order \(10^{-13}\)).
}.
We formalize \(\dimension=\infty\) as the sequence space with euclidean norm
\begin{align*}
	\sequenceSpace := \Big\{
		f:\naturals \to \reals \mid \| f\|^2 = \sum_{k\in\naturals} f(k)^2< \infty
	\Big\}.
\end{align*}
\begin{theorem}[{\cite[Theorem 2.1.13]{nesterovLecturesConvexOptimization2018}}]
	\label{thm: strong convexity complexity bound}
	For any \(\weights_0\in\sequenceSpace\) there exists
	\(\Loss\in\strongConvex[\infty,1]{\lbound}{\ubound}\) such that for any
	first order method \(\firstOrderMethod\) satisfying Assumption~\ref{assmpt:
	parameter in linear hull of gradients}, we have
	\begin{subequations}
	\begin{align}
		\|\weights_n - \minimum\|
		&\ge \left(1-\frac2{1+\sqrt{\condition}}\right)^n \|\weights_0 - \minimum\| \\
		\Loss(\weights_n) - \Loss(\minimum),
		&\ge \tfrac{\lbound}{2}
		\left(1-\frac2{1+\sqrt{\condition}}\right)^{2n} \|\weights_0 - \minimum\|^2.
	\end{align}
	\end{subequations}
	where \(\minimum\) is the unique minimum of \(\Loss\).
\end{theorem}
\begin{proof}
	Since our factor becomes zero for \(\condition=1\) we can assume without loss
	of generality that \(\condition>1\).
	
	Further, we can again assume without loss of generality that
	\(\weights_0=\mathbf{0}\). Of the two solutions (\ref{eq: solutions to
	strong convexity complexity quadratic problem}) to the linear equations
	(\ref{eq: solution to the strongly convex coloring problem}) induced by the
	two solutions \(q_{1/2}\) of (\ref{eq: coloring solution quadratic equation})
	only
	\begin{align*}
		\minimum^{(i)} = \left(\frac{\sqrt{\condition}-1}{\sqrt{\condition}+1}\right)^i
		= \left(1 - \frac{2}{1+\sqrt{\condition}}\right)^i =: q^i
	\end{align*}
	is an element of \(\sequenceSpace\) and therefore the unique minimum. Since 
	\(\weights_n\) stays in the subspace \(\reals^n\) by assumption we know
	\begin{align*}
		\|\weights_n -\minimum\|^2
		\ge \sum_{k=n+1}^\infty (\minimum^{(k)})^2
		= q^{2n} \sum_{k=0}^\infty q^{2k}
		= q^{2n} \| \minimum\|^2
		= q^{2n} \|\weights_0 - \minimum\|^2.
	\end{align*}
	Taking the root results in the first claim. The second claim immediately
	follows from the definition of strong convexity (Definition~\ref{def: strong convexity})
	and \(\nabla\Loss(\minimum) = 0\):
	\begin{align*}
		\Loss(\weights_n) - \Loss(\minimum)
		&\ge \langle \nabla\Loss(\minimum), \weights_n -\minimum\rangle
		+\tfrac\lbound{2} \|\weights_n -\minimum\|^2
		\qedhere
	\end{align*}
\end{proof}

In view of Lemma~\ref{lem: smallest upper bound} which states that \ref{eq:
gradient descent} is in some sense optimal (minimizes the upper bound induced by
Lipschitz continuity of the gradient), it might be surprising how far away
\ref{eq: gradient descent} is from these lower bounds. Maybe these lower bounds are not
tight enough? This is not the case. In Chapter~\ref{chap: momentum} we will get
to know a family of methods which do achieve these rates of convergence.

So can we explain Lemma~\ref{lem: smallest upper bound}? Well, first of all it
is of course just an upper bound we are minimizing. But the more important issue
is that we are taking a very local perspective. We only care about minimizing
the next loss, using nothing but the current gradient. In
Assumption~\ref{assmpt: parameter in linear hull of gradients} on the other
hand, we consider the linear span of all gradients so far. This distinction is
important as it will make the ``momentum method'' discussed in
Chapter~\ref{chap: momentum} seem much more obvious.

Lastly since this generalization of assumptions improved our rate of convergence
we might ask if we could generalize our assumption further. Heuristics like
the well known Adagrad, Adadelta and RMSProp (see Section~\ref{sec: heuristics
for adpative learning rates}) use different learning rates for different
components of the gradients for example. Their iterates can thus
be expressed by
\begin{align*}
	\weights_n = \weights_0 + \sum_{k=0}^{n-1} H_k \nabla\Loss(\weights_k),
\end{align*}
where \(H_k\) are diagonal matrices. But at least upper triangle matrices do
not change the fact that \(\weights_n\) is contained in subspace \(\reals^n\).
Since we only use this fact, one could thus generalize Assumption~\ref{assmpt:
parameter in linear hull of gradients} to
\begin{assumption}[{\cite{gohWhyMomentumReally2017}}]
	\label{assmpt: parameter in generalized linear hull of gradients}
	The \(n\)-th iterate of the optimization method can be expressed as
	\begin{align*}
		\weights_n = \weights_0 + \sum_{k=0}^{n-1} H_k \nabla\Loss(\weights_k),
	\end{align*}
	where \(H_k\) are upper triangular matrices.
\end{assumption}

Both complexity bounds continue to hold with this assumption in place of
Assumption~\ref{assmpt: parameter in linear hull of gradients}. And since one
could just switch the relation \((\weights_n-\weights_{n+1})^2\) to
\((\weights_n+\weights_{n+1})^2\) to force a flip from positive to negative
at this point, it is unclear how one should extrapolate to components not
available in any gradient. But making further assumptions about
\(\Loss(\weights)\) might enable us to find better bounds.


 }
	{

\chapter{Stochastic Gradient Descent (SGD)}\label{chap: sgd}

So far we have only discussed the minimization of a function we can get accurate
gradient information from. If our loss function is stochastic, the best we can
do so far is to minimize the empirical risk (\ref{eq: empirical risk}), i.e.
\begin{align*}
	\Loss_\sample(\weights) = \frac{1}{\sampleSize}\sum_{k=1}^\sampleSize \loss(\weights, Z_k)
	= \E_{Z\sim \dist_\sample}[\loss(\weights, Z)]
\end{align*}
for some sample \(\sample=(Z_1,\dots,Z_\sampleSize)\), inducing the empirical
distribution \(\dist_\sample\). Now if we try to minimize \(\Loss_\sample\),
resulting in a guess \(\hat{\weights}\) of its minimum \(\minimum^{(\sample)}\),
to get close to the minimum \(\minimum\) of \(\Loss\), we are really applying a
type of triangle
inequality \parencite[e.g.][]{bottouOptimizationMethodsLargeScale2018}
\begin{align*}
	&\Loss(\hat{\weights}) - \Loss(\minimum)\\
	&=
	\underbrace{\Loss(\hat{\weights}) - \Loss_\sample(\hat{\weights})}_{
	\le \sup_{\weights}|\mathrlap{\Loss_\sample(\weights)-\Loss(\weights)|}
	}
	{
	+ \Loss_\sample(\hat{\weights}) - \Loss_\sample(\minimum^{(\sample)})
	+ \underbrace{
		\scriptstyle
		\Loss_\sample(\minimum^{(\sample)}) - \Loss_\sample(\minimum)
	}_{\le 0\ (\text{def. }\minimum^{(\sample)})}
	}
	+ \underbrace{
		\Loss_\sample(\minimum)- \Loss(\minimum)
	}_{\le \sup_{\weights}|\Loss_\sample(\weights)-\Loss(\weights)|}
	\\
	&\le \underbrace{\Loss_\sample(\hat{\weights}) - \Loss_\sample(\minimum^{(\sample)})}_{
		\text{apply GD results}
	}
	+ 2\underbrace{\sup_{\weights}|\Loss_\sample(\weights)-\Loss(\weights)|}_{\text{see statistical learning}}.
\end{align*}
Bounding the distance between the theoretical and empirical risk uniformly is
the topic of statistical learning. In general this distance can be bounded with
some measure -- most prominently the Vapnik-Chervonenkis (VC) dimension -- of how
expressive the hypothesis class \(\{\model\}\) is
\parencite{vapnikOverviewStatisticalLearning1999}. The more dimensions our
weight vectors \(\weights\) have, the larger the VC-dimension usually is. But
VC dimensions try to remove ``duplicates'', i.e.\ parameters which do not
actually allow for different models. Chaining linear functions for example only
results in a linear function and therefore does not increase the expressiveness
of our hypothesis class. Nevertheless, deep learning models should still be
highly over-parametrized according to this theory, and have no business
generalizing to unseen examples. Yet they seem to.

To solve this mystery one has to make the upper bounds more algorithm specific
and avoid taking the supremum over \emph{all} possible parameters \(\weights\).
Instead we want to consider the weights that our algorithm would actually pick.
For an algorithm \(\algo\) for example,
\textcite{hardtTrainFasterGeneralize2016} use the upper bound
\begin{align*}
	&\Loss(\algo(\sample)) - \Loss(\minimum)\\
	&= \underbrace{\Loss(\algo(\sample)) - \Loss_\sample(\algo(\sample))}_{
		\text{generalization error}
	}
	+ \underbrace{\Loss_\sample(\algo(\sample)) - \Loss_\sample(\minimum^{(\sample)})}_{\text{optimization error (GD)}}
	+ \Loss_\sample(\minimum^{(\sample)}) - \Loss(\minimum).
\end{align*}
In their analysis in expectation over \(\sample\) (and \(\algo\) if it is a
randomized algorithm), they use
\begin{align*}
	\Loss(\minimum) &= \inf_{\weights}\E[\loss(\weights, Z)]
	= \inf_{\weights}\E\underbrace{
		\left[\frac{1}{\sampleSize}\sum_{k=1}^\sampleSize\loss(\weights, Z_k)\right]
	}_{\ge \Loss_\sample(\minimum^{(\sample)})}
	\ge \E[\Loss_\sample(\minimum^{(\sample)})]
\end{align*}
to get rid of the last term and focus on the expected generalization error
\begin{align*}
	\E[\Loss(\algo(\sample)) - \Loss_\sample(\algo(\sample))].
\end{align*}
Their bounds are much better in the sense that they are independent of
the (VC-)dimension of parameters. They only depend on the amount of data and
the sum of learning rates of all iterations in the convex case. Recall the sum of
all learning rates is equal to the time we let gradient flow run if we view
gradient descent as a discretization of this flow.\footnote{
	They also discuss the strongly convex case with compact parameter set, where
	one can get rid of this time dependency with the fact that gradient descent
	has contractive properties, i.e.\ starting from two different points and
	making one gradient step results in weights being closer to each other. 
	In a compact and thus bounded set, the length of the path of movement gradient
	flow takes is thus finite. This provides some intuition why the length
	of time can be discussed away here.

	Additionally they discuss the non-convex case where the relation to the time worsens,
	i.e.\ the bound is no longer proportional to the time spent in gradient flow
	but rather worsens faster than our time increases.
} And if we do not run our optimizer for long, it is perhaps not surprising that
we do not get far from the start and thus do not have time to overfit.

Stochastic gradient descent avoids splitting the error into generalization and
optimization error entirely which will provide us with even tighter bounds. The
price for this improved bound, is the requirement to never reuse data, i.e.\ we
need a source of independent \(\dist\) distributed random variables which we use
for a stochastic gradient
\begin{align*}
	\nabla_\weights\loss (\weights, X,Y), \qquad (X,Y)\sim\dist,
\end{align*}
of which the expected value\footnote{
	We need to be able to swap the expectation with differentiation. A sufficient 
	condition is continuity of \(\nabla_\weights\loss\) in \(\weights\), which allows us
	to use the mean value theorem
	\begin{align*}
		\frac{\partial}{\partial \weights_k}\E[\loss(\weights, X,Y)]
		&= \lim_{n\to\infty}
		\int\frac{\loss(\weights+\frac{\stdBasis_k}{n}, X,Y)-\loss(\weights,X,Y)}{\frac{ 1 }{ n }}d\Pr
		\\
		&=\lim_{n\to\infty} \int\frac{\partial}{\partial \weights_k}\loss(\xi_n, X,Y)d\Pr
		\qquad \xi_n \in \left[\weights, \weights + \frac{ \stdBasis_k }{ n }\right].
	\end{align*}
	Then we use the boundedness of a continuous function on the compact interval
	\([\weights, \weights + \frac{ \stdBasis_k }{ n }]\) to move the limit in using
	dominated convergence.
} is the gradient of the risk \(\Loss\).

If we were to reuse data, then we might argue we are really picking samples
from \(\dist_\sample\) not \(\dist\) and thus our results of convergence would
apply to \(\Loss_\sample\) instead of \(\Loss\). This would force us to use
these triangle inequality upper bounds again. Which, as of writing, are often
much worse. So much worse, for example, that (even when the number of iterations go to
infinity) the upper bound is worse than if we had stopped the iteration at the
number of samples \(\sampleSize\) and used the bounds without the triangle
inequality\footnote{
	\textcite{hardtTrainFasterGeneralize2016} use their bounds to prove that
	the averaging algorithm discussed in Section~\ref{sec: SGD with Averaging}
	Theorem~\ref{thm: optimal averaging rates} converges with rate
	\begin{align*}
		\Loss(\overline{\Weights}_n) - \Loss(\minimum)
		\le \Loss(\overline{\Weights}_n) - \E[\Loss_\sample(\minimum^{(\sample)})]
		\le \frac{\|\weights_0 -\minimum^{(\sample)}\|\sqrt{\lipConst^2 + \stdBound^2}}{\sqrt{\sampleSize}}
		\underbrace{\sqrt{\frac{\sampleSize+2n}{n}}}_{\to \sqrt{2} \quad (n\to\infty)}
	\end{align*}
}.

Of course intuitively, reusing a couple of examples should not harm
generalization that much. And since all the iterations up to \(\sampleSize\)
do not overfit in some sense, they might lead us into the vicinity of a well
generalizing minimum already. If that is also a minimum of \(\Loss_\sample\)
further optimization with regard to \(\Loss_\sample\) might only barely hurt
generalization. As the ``true'' model \(\model[\minimum]\) should also perform
well on \(\Loss_\sample\), this seems very plausible.

So even if it might not be realistic, we will assume access to an infinite
source of independent \(\dist\) distributed random variables \(((X_k,Y_k), k\ge 1)\).
We can always interpret them as drawn from \(\dist_\sample\) if appropriate.

Note that if one wants to minimize \(\Loss\), then full gradients are not
possible in the first place. But for \(\Loss_\sample\) one might compare
convergence rates of SGD and GD and deem GD on \(\nabla\Loss_\sample\) more
appropriate.

But while it is difficult to quantize due to the step change between no-reuse
and reuse of random variables, using the full gradient \(\nabla\Loss_\sample\)
probably hurts generalization. Empirically it does in fact
seem that SGD leads to wider\footnote{
\textcite{hochreiterFlatMinima1997} first suggested that flatter minima might
	generalize better since they require less precision in the weights (less
	information entropy) and are thus \emph{simpler} models.
}
(better generalizing) minima \parencite{liVisualizingLossLandscape2018}.
We will try to build intuition for this using mini-batches. I.e. we
will assume that our batches are independent within and between each other,
which is fine for small batch sizes with regard to \(\sampleSize\) but breaks
down if batch sizes approach full batches. For full batches our analysis would
only be valid for one gradient step which is obviously ridiculous.

To shorten notation we will assume that the gradient is always meant
with respect to \(\weights\) unless otherwise stated.
We define ``stochastic gradient descent'' as
\begin{align}
	\tag{SGD}
	\Weights_{n+1}
	&= \Weights_n - \lr_n\nabla\loss(\Weights_n, X_{n+1},Y_{n+1})\\
	\nonumber
	&= \Weights_n - \lr_n\nabla\Loss(\Weights_n)
	+ \lr_n\underbrace{
		[\nabla\Loss(\Weights_n) - \nabla\loss(\Weights_n, X_{n+1}, Y_{n+1})]
	}_{=:\martIncr_{n+1}}.
\end{align}
Where \(\martIncr_n\) are martingale increments for the filtration
\begin{align*}
	\filtration_n :=\sigma((X_k, Y_k): k\le n ),
	\qquad (X_k,Y_k)\stackrel{\text{iid}}{\sim}\dist,
\end{align*}
since \(\Weights_n\) is \(\filtration_n\)-measurable which is independent of
\((X_{n+1},Y_{n+1})\), i.e.\ 
\begin{align}\label{eq: conditional independence of martingale increments}
	\E[\martIncr_{n+1}\mid \filtration_n]
	= \E[\nabla\Loss(\Weights_n) - \nabla\loss(\Weights_n, X_{n+1}, Y_{n+1})\mid \filtration_n]
	= 0.
\end{align}
\section{ODE View}

The first way to view SGD is through the lense of approximating the ODE with
integral equation
\begin{align*}
	\weights(t) =  \weights_0 -\int_{0}^{t}\nabla\Loss(\weights(s))ds.
\end{align*}
Recall its Euler discretization is gradient descent 
\begin{align*}
	\weights_{t_{n+1}}:=\weights_{n+1}
	= \weights_n - \lr_n \nabla\Loss(\weights_n),
\end{align*}
where
\begin{align*}
	0=t_0 \le \dots \le t_N= T \quad \text{with} \quad t_{n+1}-t_n=:\lr_n
	\quad \text{and}\quad \lr:=\max_n \lr_n.
\end{align*}
A constant learning rate \(\lr\) represents an equidistant discretization.
Since \(\nabla\Loss\) is Lipschitz continuous, the local discretization error
\begin{align*}
	\|\weights(\lr)  - \weights_\lr \|
	&= \Big\|
		\Big(\weights_0 - \int_0^\lr\nabla\Loss(\weights(s))ds\Big)
		- (\weights_0-\lr\nabla\Loss(\weights_0))
	\Big\|\\
	&\le \int_0^\lr \underbrace{\|\nabla\Loss(\weights_0) - \nabla\Loss(\weights(s))\|}_{
		\le \ubound \|\weights(s) - \weights(0)\| \in O(s) \mathrlap{\quad\text{(Taylor approx., Lip. \(\nabla\Loss\))}}
	}ds 
\end{align*}
is of order \(O(\lr^2)\). 
Using the discrete Gr\"onwall's inequality (Lemma~\ref{lem-appendix: discrete
gronwall}) and stability of the ODE (due to Lipschitz continuity of
\(\nabla\Loss\) and the continuous Gr\"onwall inequality) the global
discretization error is therefore of order 1, i.e.\ 
\begin{align*}
	\max_{n} \|\weights(t_n) - \weights_n\| \in O(\lr),
\end{align*}
\fxnote{refer to book instead of course?}{as is usually covered in ``Numerics of ODE's'' courses.}
Due to \((a+b)^2\le2(a^2+b^2)\), and therefore 
\begin{align*}
	\max_n\E[\|\weights(t_n)-\Weights_n\|^2]
	\le 2(\underbrace{\max_n \|\weights(t_n)-\weights_n\|^2}_{O(\lr^2)}
	+\max_{n}\E[\|\weights_n - W_n\|^2]),
\end{align*}
it is enough to bound the distance between SGD and GD.
\begin{theorem}\label{thm: distance SGD vs GD}
	In general we have
	\begin{align*}
		\E\|\weights_n-\Weights_n\|^2
		\le \sum_{k=0}^{n-1}\lr_k^2\E\|\martIncr_{k+1}\|^2\exp\left(
			\sum_{j=k+1}^{n-1}\lr_j^2\ubound^2 + 2\lr_j\ubound
		\right),
	\end{align*}
	which can be simplified for \(\lr:=\max_n\lr_n\) and bounded martingale
	increment variances \(\E\|\martIncr_n\|^2 \le \stdBound^2\) to
	\begin{align*}
		\max_{n}\E\|\weights_n	-\Weights_n\|^2
		\le \lr T\stdBound^2\exp[T(\lr\ubound^2 + 2\ubound)] \in O(\lr)
	\end{align*}
\end{theorem}
Now before we get to the proof let us build some intuition why this result
is not surprising. Unrolling SGD we can rewrite it as
\begin{align}\label{eq: unrolled SGD}
	\Weights_{t_n}
	= \weights_0 - \sum_{k=0}^{n-1} \lr_k \nabla\Loss(\Weights_k)
	+ \sum_{k=0}^{n-1} \lr_k\martIncr_{k+1}.
\end{align}
This is the same recursion as in GD except for the last term. Now for an
equidistance grid we have \(\lr=\frac{T}{N}\). Intuitively the last term in \(\Weights_T\)
should therefore disappear due to some law of large numbers for martingales
\begin{align}\label{eq: mean of martingale increments}
	\sum_{k=0}^{N-1}\lr\martIncr_{k+1} = \frac{T}{N}\sum_{k=1}^N\martIncr_k.
\end{align}
And since the variance of a sum of martingale increments is the sum of variances
as mixed terms disappear due to conditional independence (\ref{eq: conditional
independence of martingale increments})
\begin{align*}
	\E\left\|\sum_{k=1}^N\martIncr_k\right\|^2 = \sum_{k=1}^N \E\|\martIncr_k\|^2,
\end{align*}
the variance of a mean decreases with rate \(O(\frac{1}{N})=O(\lr)\), which is the rate
we get in Theorem~\ref{thm: distance SGD vs GD}.
\begin{proof}[Proof (Theorem~\ref{thm: distance SGD vs GD})]
	Using the conditional independence (\ref{eq: conditional independence of martingale increments})
	we can get rid of a all the mixed terms with
	\(\E[\langle \cdot, \martIncr_{n+1}\rangle\mid\filtration_n]=0\) as everything
	else is \(\filtration_n\)-measurable:
	\begin{align*}
		&\E\|\Weights_{n+1}-\weights_{n+1}\|^2
		= \E\|\Weights_n - \weights_n
		+ \lr_n(\nabla\Loss(\weights_n)-\nabla\Loss(\Weights_n))
		+ \lr_n\martIncr_{n+1}\|^2\\
		&\lxeq{(\ref{eq: conditional independence of martingale increments})}
		\begin{aligned}[t]
			&\E\|\Weights_n-\weights_n\|^2 + \lr_n^2\E\|\martIncr_{n+1}\|^2\\
			&+\underbrace{
				\lr_n^2\E\|\nabla\Loss(\weights_n)-\nabla\Loss(\Weights_n)\|^2
				+ 2\lr_n\E\langle\Weights_n-\weights_n,
				\nabla\Loss(\weights_n)-\nabla\Loss(\Weights_n)\rangle.
			}_{
				\le 0 \qquad
				\text{if \(\Loss\) was convex (Lemma~\ref{lem: bregmanDiv lower bound}) and }
				\lr_n<\frac2\ubound
			}
		\end{aligned}
	\end{align*}
	But since we do not want to demand convexity of \(\Loss\) just yet, we will
	have to get rid of these terms in a less elegant fashion. Using Lipschitz
	continuity and the Cauchy-Schwarz inequality we get
	\begin{align*}
		\|\nabla\Loss(\weights_n)-\nabla\Loss(\Weights_n)\|^2
		&\le \ubound^2\|\weights_n-\Weights_n\|^2\\
		\langle\Weights_n-\weights_n,
		\nabla\Loss(\weights_n)-\nabla\Loss(\Weights_n)\rangle,
		&\le \ubound \|\weights_n-\Weights_n\|^2,
	\end{align*}
	which leads us to
	\begin{align*}
		\E\|\Weights_{n+1}-\weights_{n+1}\|^2
		\le (1+\lr_n^2\ubound^2 + 2\lr_n\ubound)\E\|\Weights_n-\weights_n\|^2
		+ \lr_n^2\E\|\martIncr_{n+1}\|^2.
	\end{align*}
	Now we apply the discrete Gr\"onwall inequality (Lemma~\ref{lem-appendix:
	discrete gronwall}) to get
	\begin{align*}
		\E\|\Weights_n-\weights_n\|^2
		\le \sum_{k=0}^{n-1}\lr_k^2\E\|\martIncr_k\|^2
		\underbrace{
			\prod_{j=k+1}^{n-1}(1+\lr_j^2\ubound^2 + 2\lr_j\ubound)
		}_{
			\le \exp\left(\sum_{j=k+1}^{n-1}\lr_j^2\ubound^2 + 2\lr_j\ubound\right)
		}
	\end{align*}
	using \(1+x\le\exp(x)\) for the first claim. Assuming bounded martingale
	increment variances \(\E\|\martIncr_n\|^2\le\stdBound^2\) and \(\lr=\max_n\lr_n\)
	we therefore have 
	\begin{align*}
		\E\|\Weights_n-\weights_n\|^2
		&\le \lr\stdBound^2\underbrace{\sum_{k=0}^{n-1}\lr_k}_{\le T}
		\exp\Bigg((\lr\ubound^2+2\ubound)\underbrace{\sum_{j=k+1}^{n-1}\lr_j}_{\le T}\Bigg).
		\qedhere
	\end{align*}
\end{proof}

\subsection{Excursion: Stochastic Approximation}

Theorem~\ref{thm: distance SGD vs GD} can feel somewhat unnatural as it requires
us to completely restart SGD with smaller learning rates to get behavior closer
to GD and its ODE. If the ODE has an attraction point/area/etc.\ one might
instead want to prove convergence to that within one run of SGD using a
decreasing learning rate sequence. If we have a look at (\ref{eq: unrolled SGD})
it is quite intuitive that we will need
\begin{align*}
	\sum_{k=0}^\infty \lr_k = \lim_{n\to\infty}\sum_{k=0}^{n-1} \lr_k
	= \lim_{n\to\infty}t_n = \infty.
\end{align*}
Otherwise we cannot really claim that SGD should behave similar to the asymptotic
behavior of an ODE as we cannot let \(t\to\infty\) in the ODE. Additionally
we somehow need to do variance reduction to get rid of
\begin{align*}
	\sum_{k=m}^{m-1} \lr_k\martIncr_{k+1}
\end{align*}
to argue that from starting point \(\Weights_{t_m}\) onwards SGD behaves like the ODE
and therefore ends up in its attraction areas. So we also need
\(\lr_k \to 0\). How fast the learning rate needs to vanish depends on the
amount of (bounded) moments we have on \(\martIncr_k\) \parencite[p.
110]{kushnerStochasticApproximationAlgorithms1997}.
If we only have bounded second moments as we have assumed in Theorem~\ref{thm:
distance SGD vs GD}, then we get the classic stochastic approximation conditions
on the learning rate
\begin{align*}
	\sum_{n=0}^\infty \lr_n = \infty \qquad \text{and} \qquad \sum_{n=0}^\infty \lr_n^2 <\infty,
\end{align*}
i.e.\ learning rates \(\lr_n\) which decrease faster than \(O(\frac1{ \sqrt{n} })\) but
not faster than \(O(\frac1n)\). See \textcite[ch.
5]{kushnerStochasticApproximationAlgorithms1997} for a proof of
convergence with probability one of the recursion
\begin{align*}
	\Weights_{n+1} = \Weights_n - \lr_n(g(\Weights_n) + \martIncr_{n+1})
\end{align*}
towards the asymptotic areas of the limiting
ODE
\begin{align*}
	\dot{\weights} = g(\weights)
\end{align*}
under those assumptions.

Here \(g\) can be, but needs not be, a gradient (we have not used that in
Theorem~\ref{thm: distance SGD vs GD} either). And this direction of
generalization is necessary to craft convergence proofs for Reinforcement
Learning where we do not get unbiased estimates of some gradients towards the
minimum, but rather unbiased estimates of a coordinate of the Bellman Operator
(Fixed Point Iteration towards the optimum). See also ``coordinate descent''
(Section~\ref{sec: coordinate descent}).

\section{Quadratic Loss Functions}\label{sec: quadratic loss SGD}

While we have proven that SGD behaves roughly like GD for small learning rates, we
are not really interested in how close SGD is to GD but rather how good it
is at optimizing \(\Loss\). To build some intuition, let us consider a
quadratic loss function again before we get to general convex functions.
Using the same trick we used in Section~\ref{sec: visualize gd} 
\begin{align}\label{eq: hessian gradient reminder}
	\nabla\Loss(\weights)
	= \nabla^2\Loss(\weights)(\weights-\minimum)
	= H(\weights-\minimum)
\end{align}
we can rewrite SGD (by induction using the triviality of \(n=0\)) as
\begin{align}
	\nonumber
	&\Weights_{n+1}-\minimum\\
	\nonumber
	&\lxeq{(\ref{eq: hessian gradient reminder})}
	\Weights_n - \minimum - \lr_n H(\Weights_n-\minimum) + \lr_n\martIncr_{n+1}\\
	\label{eq: recursive SGD formula in the quadratic case}
	&=(1-\lr_n H)\underbrace{(\Weights_n - \minimum)}_{
		\xeq{\text{ind.}} (\weights_n - \minimum)
		\mathrlap{+ \sum_{k=0}^{n-1}\lr_k\left(\prod_{i=k+1}^{n-1}(1-\lr_iH)\right)\martIncr_{k+1}}
	} + \lr_n\martIncr_{n+1}\\
	\nonumber
	&= \underbrace{(1-\lr_n H)(\weights_n - \minimum)}_{=\weights_{n+1}-\minimum\ (\implies\text{ind. step})}
	+ \sum_{k=0}^n\lr_k\left(\prod_{i=k+1}^n(1-\lr_iH)\right)\martIncr_{k+1}\\
	\label{eq: unrolled SGD weights (general quadratic loss case)}
	&=\left(\prod_{k=0}^n(1-\lr_kH)\right)(\weights_0-\minimum)
	+ \sum_{k=0}^n\lr_k\left(\prod_{i=k+1}^n(1-\lr_iH)\right)\martIncr_{k+1}.
\end{align}
Here we can see that SGD splits into two parts, deterministic GD and a sum of
martingale increments.

\subsection{Variance Reduction}\label{subsec: variance reduction}

One particularly interesting case is \(H=\identity\). Why is that interesting?
Considering that
\begin{align*}
	\E[\tfrac12\|\weights-X\|^2]=:\E[\loss(\weights,X)]=:\Loss(\weights)
\end{align*}
is minimized by \(\minimum=\E[X]\), and we have
\begin{align*}
	\E[\nabla\loss(\weights,X)]
	=\E[\weights-X]=\weights- \minimum=\nabla\Loss(\weights),
\end{align*}
we necessarily also have
\begin{align*}
	\Loss(\weights) = \tfrac12\|\weights-\minimum\|^2 + \text{const}.
\end{align*}
This implies \(H=\nabla^2\Loss=\identity\). So this case is essentially
trying to find the expected value of some random variables with as few
iterations and data as possible.

\subsubsection{Constant Learning Rates}

For constant learning rates \(\lr=\frac{T}N<1\) we have
\begin{align*}
	\Weights_n - \minimum
	&= (1-\lr)^n(\weights_0 -\minimum)+ \sum_{k=0}^{n-1}\lr(1-\lr)^{n-1-k}
	\smash{\overbrace{\martIncr_{k+1}}^{X_{k+1}-\minimum}}\\
	&= \left((1-h)^n\weights_0 + \sum_{k=0}^{n-1}\lr(1-\lr)^{n-1-k}X_{k+1}\right)-\minimum.
\end{align*}
This implies we do not weight our \(X_k\) equally, but rather use an
exponential decay giving the most recent data the most weight.

The first interesting thing to highlight is the fact that the convergence
rate of GD actually worsens with finer discretization \(N\)
\begin{align*}
	(1-h)^N = (1-\tfrac{T}N)^N \nearrow \exp(-T) \qquad (N\to\infty).
\end{align*}
This is surprising as more information, more gradient evaluations should help
convergence, which seems not to be the case. To be fair, it does not worsen
much, if the learning rate is small to begin with. But a learning rate of one
allows instant convergence while \ref{eq:
gradient flow} only converges for \(T\to\infty\). So even if we would get
the additional gradient evaluations for free, a finer discretization is worse!

Another view to make this more intuitive is the analysis we did when selecting
optimal learning rates for GD, where the rates of the eigenspaces improved until
we moved \(\lr\hesseEV\) beyond one. From then on they start to deteriorate
again\footnote{
	and quickly, since we
	do not only increase the absolute value of the factor, but also multiply fewer 
	of them together. While in the case \(\lr\hesseEV<1\) these things
	counteract each other.
}.
But this means that, at least for quadratic functions, gradient descent
converges faster than gradient flow!

Also note how we can reobtain the rates from Theorem~\ref{thm: distance
SGD vs GD} by throwing the exponential decay away
\begin{align*}
	\E\|\Weights_N-\weights_N\|^2
	&\xeq{(\ref{eq: unrolled SGD weights (general quadratic loss case)})} \E\left\|\sum_{k=0}^{N-1}\lr(1-\lr)^{N-1-k}\martIncr_{k+1}\right\|^2\\
	&= \sum_{k=0}^{N-1}\lr^2\underbrace{(1-\lr)^{2(N-1-k)}}_{\le1}
	\underbrace{\E\|\martIncr_{k+1}\|^2}_{\le \stdBound^2}\\
	&\le \lr^2 N \stdBound^2 = \lr T\stdBound^2 = \frac{T^2}{N}\stdBound^2
\end{align*}
We can do the same thing with the general case (\ref{eq: unrolled SGD weights
(general quadratic loss case)}) if we take \(\lr<2/\ubound\) for \(|H|\le\ubound\).

Together with the upper bound on GD, we can thus see by
\begin{align*}
	\E\|\Weights_N - \minimum\|^2
	&= \E\|\weights_N - \minimum\|^2 + \E\|\Weights_N - \weights_N\|^2\\
	&\le \exp(-2T)\|\weights_0 - \minimum\|^2 + \frac{T^2}{N}\stdBound^2,
\end{align*}
how we need to increase \(T\) with \(N\) at a roughly logarithmic rate to
minimize our \(L^2\) norm for a given \(N\). Therefore our convergence rate
would be roughly of order
\begin{align*}
	O(\tfrac{\log(N)}N).
\end{align*}
This is a bit worse than the rate \(O(\frac1N)\) we can achieve with the mean
\begin{align*}
	\E\|\overline{X}_N - \E[X]\|^2 = \frac{\E\|X-\E[X]\|^2}{N},
\end{align*}
which is the best linear unbiased estimator (BLUE) for the expected value.
Note that making only one step and increasing the batch size (Section~\ref{sec:
batch learning}) to our sample size, realizes the mean. 

\subsubsection{Decreasing Learning Rates}

So is batch learning better? Well, if we use a particular learning rate, we can
actually realize the mean with SGD as well. Using \(\lr_n=\frac1{n+1}\), we get
by induction
\begin{align}
	\label{eq: getting rid of w_0 first step}
	\Weights_1 &= \weights_0 - \frac{1}{0+1}(\weights_0 - X_1) = X_1,\\
	\label{eq: telescoping product}
	\Weights_{n+1}
	&= \Weights_n - \tfrac{1}{n+1}(\Weights_n - X_{n+1})
	\xeq{\text{ind.}}\smash{\underbrace{\left(1-\tfrac1{n+1}\right)}_{=\frac{n}{n+1}}}
	\frac1n\sum_{k=1}^n X_k + \tfrac{1}{n+1}X_{n+1}\\
	\nonumber
	&= \frac{1}{n+1}\sum_{k=1}^{n+1}X_k.
\end{align}

\subsection{Optimal Rates for General Quadratic Functions}

Since we know we can get optimal rates in the case \(\condition=1\), we might
ask ourselves whether this is possible in the general case as well. So let us
try. First we use the fact that a basis change to a different
orthonormal basis \(V=(v_1,\dots, v_\dimension)\) is an isometry\footnote{
	Proof: \(\|x\|^2 = x^Tx = x^TVV^Tx = (V^T x)^T V^Tx = \|V^T x\|^2\)
}
and thus
\begin{align*}
	\E\|\Weights_{n+1} - \minimum\|^2
	&= \E\|V^T(\Weights_{n+1} - \minimum)\|^2
	= \sum_{j=1}^\dimension \E\langle\Weights_{n+1}-\minimum, v_j\rangle^2.
\end{align*}
So if \(V\) are the eigenvectors with sorted eigenvalues
\(\hesseEV_1 \le \ldots\le\hesseEV_\dimension\) of our Hessian \(H\), we get
\begin{align*}
	\langle \Weights_{n+1}-\minimum, v_j \rangle
	&\xeq{(\ref{eq: recursive SGD formula in the quadratic case})} (1-\lr_n\lambda_j)\langle\Weights_n-\minimum, v_j\rangle
	+ \lr_n\langle\martIncr_{n+1},v_j\rangle.
\end{align*}
Using conditional independence of \(\Weights_n\) and \(\martIncr_{n+1}\) again,
we get
\begin{align}
	\label{eq: expected square error in quadratic case}
	\E\|\Weights_{n+1} - \minimum\|^2
	&= \sum_{j=1}^\dimension (1-\lr_n \hesseEV_j)^2
	\E\langle\Weights_n-\minimum, v_j\rangle^2
	+ \lr_n^2 \E\langle \martIncr_{n+1}, v_j\rangle^2.
\end{align}
Optimizing over \(\lr_n\) by taking the derivative of the term above with
respect to \(\lr_n\) would result in
\begin{align*}
	\lr_n^* &= \frac{
		\sum_{j=1}^\dimension \hesseEV_j \E\langle \Weights_n-\minimum, v_j\rangle^2
	}{
		\sum_{j=1}^\dimension\hesseEV_j^2\E\langle\Weights_n-\minimum, v_j\rangle^2
		+ \E\langle\martIncr_{n+1}, v_j\rangle^2
	}\\
	&= \frac{
		\E\langle \Weights_n - \minimum, \nabla\Loss(\Weights_n)\rangle
	}{
		\E\|\nabla\Loss(\Weights_n)\|^2 + \E\|\martIncr_{n+1}\|^2
	}
	= \frac{
		\E\langle \Weights_n - \minimum, \nabla\Loss(\Weights_n)\rangle
	}{
		\E\|\nabla\loss(\Weights_n)\|^2
	}.
\end{align*}
But while we can estimate the divisor, we do not know the scalar product between
the true gradient and the direction to the minimum. If we knew that, we could
probably get to the minimum much faster.

But we do not, which might be surprising since we got ``optimal'' rates with
a similar approach for non-stochastic gradient in Subsection~\ref{subsec:
necessary assumptions in the quadratic case}. But now, even if we assumed
\(\E\|\martIncr_{n+1}\|^2=0\), we get these weird optimal rates. What happened?

Well, in Subsection~\ref{subsec: necessary assumptions in the quadratic case}
we did not actually optimize over the euclidean norm. Instead we looked at the
reduction factor for every eigenspace and tried to minimize their maximum.
This does not take the actual size of the eigen-components
\(\E\langle\Weights_n-\minimum, v_j\rangle^2\) of our distance to the optimum
\(\minimum\) into account. If one of these components was very small, then
we could sacrifice its reduction factor for a smaller factor in the other
eigenspaces.

We did not do that in Subsection~\ref{subsec: necessary assumptions
in the quadratic case} because we used constant learning rates which would
sacrifice one eigenspace for another indefinitely. But no matter how big the
initial difference, if the rate of one eigenspace is smaller than another, then
eventually the error in the eigenspace with the smaller rate will be
negligible compared to the one with the larger rate. This means that
asymptotically only the maximum of these factors matter.

Now if we assume rotationally invariant noise, i.e.
\begin{align*}
	\E\langle\martIncr_{n+1},v_j\rangle^2 = \E\langle\martIncr_{n+1}, v_i\rangle^2
	\qquad \forall i,j
\end{align*}
then due to
\begin{align*}
	(1-\lr_n \hesseEV_j)^2
	\le \max \{ (1-\lr_n\hesseEV_1)^2, (1-\lr_n\hesseEV_\dimension)^2 \},
\end{align*}
the error of the eigenspaces in-between will eventually be negligible (cf.
(\ref{eq: expected square error in quadratic case})). So we only really have a
trade off between the smallest and largest eigenvalue again.
If we use the learning rates to favour one over the other, then this would
make up for any differences that might have existed after a finite time and
from then on both errors would be similar in size which would mean we should
not treat them differently anymore. This motivates why it is reasonable to
neglect differences between the size of eigen-components at least for the
asymptotic rate.

Now we need to translate this intuition into actual mathematics. For that, we
define a new norm
\begin{align}
	\|X\|_{V,\infty}
	:= \sqrt{\max_{i=1,\dots,\dimension} \E\langle X, v_i\rangle^2}.
\end{align}
Note that without the expectation this would just be the infinity norm with regard
to the basis \(v_1,\dots,v_\dimension\). Now positive definiteness and
scalability of the norm are trivial, and the triangle inequality follows from the
triangle inequality of the \(L^2\) norm \(\|X\|_{L^2} := \sqrt{\E\|X\|^2}\)
\begin{align*}
	\|X+Y\|_{V,\infty}
	= \max_{j=1,\dots,\dimension} \|\langle X, v_j\rangle + \langle Y, v_j\rangle\|_{L^2}
	\le \|X\|_{V,\infty} + \|Y\|_{V,\infty}.
\end{align*}
As we have used the \(L^2\) norm until now, let us consider its relation
to this new norm. As we work with a finite dimensional real vector space, the
norms are equivalent, more specifically we have
\begin{align*}
	\|X\|_{V,\infty}^2
	= \max_{i=1,\dots,\dimension} \E\langle X, v_i\rangle^2
	\le \sum_{i=1}^\dimension \E\langle X, v_i\rangle^2
	= \|X\|_{L^2}
	\le \dimension \|X\|_{V,\infty}^2.
\end{align*}
Though \(\dimension\) is potentially quite large. But as we have argued in the
limiting case all the eigenspaces except the one with the largest and smallest
eigenvalue vanish, so then the factor comes closer to \(2\) independent of the
dimension \(\dimension\).

Armed with this norm, we can now tackle our problem again.

\begin{theorem}\label{thm: optimal rates quadratic case}
	Let \(\Loss\) be a quadratic loss function with eigenvalues
	\(\hesseEV_1\le\ldots\le\hesseEV_\dimension\) of the Hessian and assume
	bounded variance of the martingale increments
	\begin{align*}
		\E\|\martIncr_n\|_{V,\infty}^2
		\le \tilde{\stdBound}^2 = \stdBound^2 \hesseEV_1\hesseEV_\dimension,
		\qquad \forall n\in\naturals,
	\end{align*}
	then we obtain optimal learning rates (w.r.t \(\|\cdot\|_{V,\infty}\)) of
	\begin{align*}
		\lr_n^*
		= \begin{cases}
			\frac{2}{\hesseEV_1+\hesseEV_\dimension}
			& \frac{\stdBound^2}{\|\Weights_n - \minimum\|_{V,\infty}^2}
			\le \frac{\condition-1}{2\condition}, \\
			\frac1{\hesseEV_\dimension(\condition^{-1} + \frac{\stdBound^2}{\|\Weights_n-\minimum\|_{V,\infty}^2})}
			& \frac{\stdBound^2}{\|\Weights_n - \minimum\|_{V,\infty}^2}
			\ge \frac{\condition-1}{2\condition}.
		\end{cases}
	\end{align*}
	We can interpret this result as follows: If the noise is negligible compared to the
	distance to the minimum, we want to use the same learning rates as in the
	classical case -- simply ignoring the existence of noise. After this ``transient
	phase'' we are in the ``asymptotic'' phase and want to select a learning rate
	which decreases with the distance to the minimum. For the asymptotic rate of
	convergence only the second case matters but the transient phase might be much
	more important for practical applications. Only for \(\condition=1\) we are
	in this asymptotic phase immediately.
	
	Assuming that \(n_0\) is the time we enter the ``asymptotic'' phase,
	i.e.\ the first index where \(\frac{\stdBound^2}{\|\Weights_n -
	\minimum\|_{V,\infty}^2} \ge
	\frac{\condition-1}{2\condition}\),
	then using the optimal learning rates our rates of convergence are
	\begin{align*}
		\|\Weights_n-\minimum\|_{V,\infty}^2
		\le \begin{cases}
			\left(1-\tfrac{2}{1+\condition}\right)^n\|\Weights_0-\minimum\|_{V,\infty}^2
			& n \le n_0,\\
			\frac{\stdBound^2\condition(\condition+1)}{(n+1-n_0)(\condition-1)}
			& n \ge n_0.
		\end{cases}
	\end{align*}
\end{theorem}
\begin{proof}
	Using the same
	transformations (in particular conditional independence) as in (\ref{eq:
	expected square error in quadratic case}), we get
	\begin{align}
		\nonumber
		\|\Weights_{n+1} -\minimum\|_{V,\infty}^2
		&= \max_{j=1,\dots,\dimension} \E\langle\Weights_{n+1} -\minimum, v_j\rangle^2\\
		\nonumber
		&= \max_{j=1,\dots,\dimension} (1-\lr_n\hesseEV_j)^2\E\langle\Weights_n -\minimum, v_j\rangle^2
		+ \lr_n^2 \E\langle \martIncr_{n+1}, v_j \rangle^2\\
		\nonumber
		&\le \max_{j=1,\dots,\dimension} (1-\lr_n\hesseEV_j)^2\|\Weights_n -\minimum\|_{V,\infty}^2
		+ \lr_n^2 \|\martIncr_{n+1}\|_{V,\infty}^2\\
		\label{eq: upper bound weird norm weight distance}
		&= \max_{j=1,\dimension} (1-\lr_n\hesseEV_j)^2\|\Weights_n -\minimum\|_{V,\infty}^2
		+ \lr_n^2 \underbrace{\|\martIncr_{n+1}\|_{V,\infty}^2}_{\le \tilde{\stdBound}^2}
	\end{align}
	if we assume bounded variance. So let us optimize over this bound. Since
	we are looking at a maximum over parabolas in \(\lr_n\), the first order
	condition would be enough as this is a convex function. But while the maximum is
	continuous, we have a discontinuity in the derivative 
	\begin{align*}
		\frac{d}{d\lr_n}
		= \begin{cases}
		(2\lr_n \hesseEV_1^2 - 2\hesseEV_1)\|\Weights_n-\minimum\|_{V,\infty}^2
		+ 2 \lr_n \tilde{\stdBound}^2
		& \lr_n \le \frac{2}{\hesseEV_1+\hesseEV_\dimension}\\
		(2\lr_n \hesseEV_\dimension^2 - 2\hesseEV_\dimension)\|\Weights_n-\minimum\|_{V,\infty}^2
		+ 2 \lr_n \tilde{\stdBound}^2
		& \lr_n \ge \frac{2}{\hesseEV_1+\hesseEV_\dimension}
		\end{cases},
	\end{align*}
	which might prevent the monotonously increasing derivative from ever being equal
	to zero. In fact we need either
	\begin{align*}
		\frac{2}{\hesseEV_1+\hesseEV_\dimension} \xge{!} \lr_n
		&= \frac{
			\hesseEV_1 \|\Weights_n-\minimum\|_{V,\infty}^2
		}{
			\hesseEV_1^2 \|\Weights_n-\minimum\|_{V,\infty}^2 + \tilde{\stdBound}^2
		}\\
		&= \frac{1}{
			\hesseEV_1 + \frac{\tilde{\stdBound}^2}{
				\hesseEV_1\|\Weights_n-\minimum\|_{V,\infty}^2
			}
		}
	\end{align*}
	or flipping enumerator and denominator on both sides, subtracting and
	multiplying by \(\hesseEV_1\), we need equivalently
	\begin{align*}
		\frac{\tilde{\stdBound}^2}{
			\|\Weights_n-\minimum\|_{V,\infty}^2
		}
		\ge\frac{\hesseEV_1^2(\condition-1)}{2}.
	\end{align*}
	This also covers the case \(\condition=1\), allowing us to assume
	\(\hesseEV_1 < \hesseEV_\dimension\) without loss of generality in the next
	case. Here we want 
	\begin{align*}
		\frac{2}{\hesseEV_1+\hesseEV_\dimension} \xle{!} \lr_n
		&= \frac{1}{
			\hesseEV_\dimension
			+ \frac{\tilde{\stdBound}^2}{
				\hesseEV_\dimension\|\Weights_n-\minimum\|_{V,\infty}^2
			}
		},
	\end{align*}
	which similarly simplifies to
	\begin{align*}
		\frac{\hesseEV_1+\hesseEV_\dimension}{2}
		\xge{!} \hesseEV_\dimension
		+ \underbrace{\frac{\tilde{\stdBound}^2}{
			\hesseEV_\dimension\|\Weights_n-\minimum\|_{V,\infty}^2
		}}_{\ge 0}
		\ge \hesseEV_\dimension,
	\end{align*}
	which is a contradiction for \(\hesseEV_1 < \hesseEV_\dimension\). So this
	case can be discarded.

	If the first condition is not satisfied either, then the optimal learning
	rate is at the discontinuity (the same as in the classical case). So in
	summary we have
	\begin{align*}
		\lr_n^* = \begin{cases}
			\frac{2}{\hesseEV_1+\hesseEV_\dimension}
			&\frac{\tilde{\stdBound}^2}{
				\|\Weights_n-\minimum\|_{V,\infty}^2
			}
			\le \frac{\hesseEV_1^2(\condition-1)}{2}\\
			\left(
				\hesseEV_1 + \frac{\tilde{\stdBound}^2}{
					\hesseEV_1\|\Weights_n-\minimum\|_{V,\infty}^2
				}
			\right)^{-1}
			&\frac{\tilde{\stdBound}^2}{
				\|\Weights_n-\minimum\|_{V,\infty}^2
			}
			\ge \frac{\hesseEV_1^2(\condition-1)}{2}
		\end{cases}.
	\end{align*}
	Plugging in the definition of \(\tilde{\stdBound}\) results in the version
	of our theorem.

	Before we plug our learning rate into our upper bound (\ref{eq: upper bound
	weird norm weight distance}), let us do some preparation 
	\begin{align}
		\nonumber
		&\|\Weights_{n+1} -\minimum\|_{V,\infty}^2\\
		\nonumber
		&\le \max_{j=1,\dimension} (1-\lr_n\hesseEV_j)^2\|\Weights_n -\minimum\|_{V,\infty}^2
		+ \lr_n^2 \tilde{\stdBound}^2\\
		\label{eq: upper bound useful for rate of convergence}
		&= \max_{j=1,\dimension}\left(1- 2\hesseEV_j\lr_n + \lr_n^2\hesseEV_j^2
			+\lr_n^2\tfrac{\tilde{\stdBound}^2}{\|\Weights_n-\minimum\|_{V,\infty}^2}
		\right)\|\Weights_n-\minimum\|_{V,\infty}^2\\
		\nonumber
		&= \max_{j=1,\dimension}\Bigg(1- \lr_n \hesseEV_j\Big[2 - \lr_n \left(
			\hesseEV_j + \tfrac{\tilde{\stdBound}^2}{\hesseEV_j\|\Weights_n-\minimum\|_{V,\infty}^2}
		\right)\Big] \Bigg)\|\Weights_n-\minimum\|_{V,\infty}^2.
	\end{align}
	Since the maximum is attained by the eigenspace with the smallest eigenvalue \(\hesseEV_1\)
	in the asymptotic case, our bound reduces to
	\begin{align*}
		\|\Weights_{n+1} -\minimum\|_{V,\infty}^2
		&\le (1- \lr_n^* \hesseEV_1[2-\lr_n^*(\lr_n^*)^{-1}])\|\Weights_n-\minimum\|_{V,\infty}^2\\
		&= \left(1- \tfrac{1}{1+ \tfrac{\tilde{\stdBound}^2}{\hesseEV_1^2\|\Weights_n-\minimum\|^2}}\right)
		\|\Weights_n-\minimum\|_{V,\infty}^2\\
		&= \left(1- \tfrac{
			\hesseEV_1^2 \|\Weights_n-\minimum\|_{V,\infty}^2
		}{
			\tilde{\stdBound}^2
			\left(\tfrac{\hesseEV_1^2\|\Weights_n-\minimum\|^2}{\tilde{\stdBound}^2}+1\right)
		}\right)
		\|\Weights_n-\minimum\|_{V,\infty}^2\\
		&\le \Bigg(
		1- \underbrace{\tfrac{
			\hesseEV_1^2
		}{
			\tilde{\stdBound}^2
			\left(\tfrac{2}{\condition -1}+1\right)
		}}_{
			=\tfrac{\hesseEV_1^2(\condition-1)}{\tilde{\stdBound}^2(\condition+1)}
			\mathrlap{=\tfrac{\condition-1}{\stdBound^2\condition(\condition+1)}}
		}
		\|\Weights_n-\minimum\|_{V,\infty}^2
		\Bigg)
		\|\Weights_n-\minimum\|_{V,\infty}^2,
	\end{align*}
	where we have used our lower bound on the variance-distance fraction
	\begin{align*}
		\frac{\tilde{\stdBound}^2}{
			\hesseEV_1^2\|\Weights_n-\minimum\|_{V,\infty}^2
		}
		\ge \frac{\condition-1}{2}
	\end{align*}
	characterizing the asymptotic phase. Using Lemma~\ref{lem: upper bound on diminishing contraction}
	this implies our claim.

	In the transient phase we can apply
	\begin{align*}
		\frac{\tilde{\stdBound}^2}{
			\|\Weights_n-\minimum\|_{V,\infty}^2
		}
		\le \frac{\hesseEV_1^2(\condition-1)}{2}
	\end{align*}
	directly to (\ref{eq: upper bound useful for rate of convergence}) to get
	\begin{align*}
		\|\Weights_{n+1}-\minimum\|_{V,\infty}^2
		&= \max_{j=1,\dimension}\left((1- \hesseEV_j\lr_n)^2
			+\lr_n^2\tfrac{\tilde{\stdBound}^2}{\|\Weights_n-\minimum\|_{V,\infty}^2}
		\right)\|\Weights_n-\minimum\|_{V,\infty}^2\\
		&\le \max_{j=1,\dimension}\left((1- \hesseEV_j\lr_n)^2
			+\lr_n^2\tfrac{\hesseEV_1^2(\condition-1)}{2}
		\right)\|\Weights_n-\minimum\|_{V,\infty}^2.
	\end{align*}
	And for the learning rate
	\(\lr_n^*=\tfrac{2}{\hesseEV_1+\hesseEV_\dimension}
	=\tfrac{2}{\hesseEV_1(1+\condition)}\),
	we have
	\begin{align*}
		(1-\hesseEV_\dimension\lr_n^*)^2
		= (1-\hesseEV_1\lr_n^*)^2
		= \left(1-\tfrac{2}{1+\condition}\right)^2
	\end{align*}
	and therefore
	\begin{align*}
		\|\Weights_{n+1}-\minimum\|_{V,\infty}^2
		&\le \left(\left(1-\tfrac{2}{1+\condition}\right)^2
			+\left(\tfrac{2}{\hesseEV_1(1+\condition)}\right)^2\tfrac{\hesseEV_1^2(\condition-1)}{2}
		\right)\|\Weights_n-\minimum\|_{V,\infty}^2\\
		&=\left(
			\left(\tfrac{\condition-1}{\condition+1}\right)^2 + 2 \tfrac{\condition-1}{(\condition+1)^2}
		\right)\|\Weights_n-\minimum\|_{V,\infty}^2\\
		&= \tfrac{\condition-1}{(\condition+1)^2}(\condition-1 + 2)\|\Weights_n-\minimum\|_{V,\infty}^2\\
		&= \left(1-\tfrac{2}{1+\condition}\right)\|\Weights_n - \minimum\|_{V,\infty}^2.
		\qedhere
	\end{align*}
\end{proof}
\begin{remark}
	The condition
	\begin{align*}
		\E\|\martIncr\|_{V,\infty}^2
		\le \tilde{\stdBound}^2 = \stdBound^2 \hesseEV_1\hesseEV_\dimension
	\end{align*}
	might seem strange. But this definition allows \(\stdBound\) to be scale
	invariant. I.e.\ when we scale \(\loss(\weights, Z)\) we scale \(\hesseEV_j\)
	linearly, as well as \(\E\|\martIncr_n\|^2\) quadratically. So the
	bound \(\stdBound\) is scale invariant while \(\tilde{\stdBound}\) is not.
	And with scale invariance we get convergence rates dependent only on the
	condition number.
\end{remark}
\begin{remark}
	For \(\condition=1\) we enter the asymptotic phase immediately as GD would
	converge in one step. This is the case in Subsection~\ref{subsec:
	variance reduction}.
\end{remark}

As we argued leading up to this theorem, the bounds
we used are quite tight. But guessing the correct learning rates is difficult.
Before we try to address this guessing problem, let us see if we can even
generalize this to non-quadratic functions.

\section{Strongly Convex Functions}

\begin{lemma}\label{lem: SGD bound with noise}
	Let \(\Loss\in\strongConvex{\lbound}{\ubound}\), then using SGD with learning
	rates \(\lr_n\le\tfrac2{\ubound+\lbound}\) results in
	\begin{align}\label{eq: SGD bound with noise}
		\E\|\Weights_{n+1} - \minimum\|^2
		\le \left(1-2\lr_n\tfrac{\ubound\lbound}{\ubound+\lbound}\right)
		\E\|\Weights_n - \minimum\|^2 + \lr_n^2 \underbrace{\E\|\martIncr_{n+1}\|^2}_{\le\stdBound^2\ubound\lbound}.
	\end{align}
\end{lemma}
\begin{proof}
	While we will not tether ourselves to \ref{eq: gradient descent} \(\weights_n\) in
	this section, we still split off the noise
	\begin{align*}
		\|\Weights_{n+1} - \minimum\|^2
		&= \|\Weights_n -\minimum - \lr_n\nabla\Loss(\weights_n) +\lr_n \martIncr_{n+1}\|^2\\
		&= \begin{aligned}[t]
			&\|\Weights_n - \minimum - \lr_n\nabla\Loss(\Weights_n)\|^2\\
			&+ 2\langle \Weights_n - \minimum - \lr_n\nabla\Loss(\Weights_n), \lr_n \martIncr_{n+1}\rangle\\
			&+ \lr_n^2 \|\martIncr_{n+1}\|^2.
		\end{aligned}
	\end{align*}
	The scalar product will disappear once we apply expectation due to conditional
	independence again. For the first part we do the same as in the
	classical case in Theorem~\ref{thm: gd strong convexity convergence rate}:
	\begin{align}
		\nonumber
		&\|\Weights_n - \minimum - \lr_n\nabla\Loss(\Weights_n)\|^2\\
		\label{eq: getting rid of the scalar product (SGD)}
		&= \|\Weights_n - \minimum\|^2
		- 2\lr_n\underbrace{\langle\nabla\Loss(\Weights_n), \Weights_n-\minimum\rangle}_{
			\xge{\text{Lem.~\ref{lem: bregmanDiv lower bound (strongly convex)}}}
			\tfrac{\ubound\lbound}{\ubound+\lbound}\|\Weights_n - \minimum\|^2
			\mathrlap{+ \tfrac{1}{\ubound+\lbound}\|\nabla\Loss(\Weights_n)\|^2}
		}
		+ \lr_n^2 \|\nabla\Loss(\Weights_n)\|^2\\
		\nonumber
		&\le \left(1-2\lr_n\tfrac{\ubound\lbound}{\ubound+\lbound}\right)
		\|\Weights_n - \minimum\|^2
		+ \lr_n(\lr_n - \tfrac{2}{\ubound+\lbound})
		\|\nabla\Loss(\Weights_n)\|^2.
	\end{align}
	For \(\lr_n\le\tfrac{2}{\ubound+\lbound}\) we can drop the second summand
	as in Theorem~\ref{thm: gd strong convexity convergence rate} to obtain our
	claim once we apply expectation.
\end{proof}

Now we can use this recursion to obtain a similar theorem as in the quadratic
case. This theorem was inspired by
\textcite{nemirovskiRobustStochasticApproximation2009} but introduces the split
in transient and asymptotic phase and improves the rate of convergence with the
application of Lemma~\ref{lem: bregmanDiv lower bound (strongly convex)}.

\begin{theorem}[Optimal Rates]\label{thm: optimal rates SGD}
	For a \(\Loss\in\strongConvex{\lbound}{\ubound}\) with
	\begin{align*}
		\E\|\martIncr_n\|^2 \le \stdBound^2\ubound\lbound,
	\end{align*}
	optimal learning rates according to the bounds of Lemma~\ref{lem: SGD bound with noise} are
	\begin{align*}
		\lr_n
		= \begin{cases}
			\frac{2}{\ubound+\lbound}
			& \stdBound^2 \le \frac{\E\|\Weights_n - \minimum\|}{2}
			\qquad \text{``transient phase''},\\
			\frac{\E\|\Weights_n -\minimum\|^2}{\stdBound^2(\ubound+\lbound)}
			& \stdBound^2 \ge \frac{\E\|\Weights_n - \minimum\|}{2}
			\qquad \text{ ``asymptotic phase''}.
		\end{cases}
	\end{align*}
	If \(n_0\) is the start of the asymptotic phase, then we can bound the
	convergence rates of SGD by
	\begin{align*}
		\E\|\Weights_n - \minimum\|^2
		\le \begin{cases}
			\left(1-\frac{2}{(1+\condition^{-1})(1+\condition)}\right)^n\E\|\weights_0 -\minimum\|^2
			& n < n_0,\\
			\frac{\stdBound^2(1+\condition^{-1})(1+\condition)}{n+1-n_0} & n \ge n_0.
		\end{cases}
	\end{align*}
\end{theorem}
\begin{remark}
	In the asymptotic phase the learning rate is thus bounded by
	\begin{align*}
		\lr_n \le \frac{(1+\condition^{-1})(1+\condition)}{(n+1-n_0)(\ubound + \lbound)},
	\end{align*}
	which is independent of the variance \(\stdBound^2\).
\end{remark}
\begin{remark}\label{rem: expected loss is a special case of L2 weight convergence}
	Note that using Theorem~\ref{thm: convergence chain} we get
	\begin{align*}
		\E[\Loss(\Weights_n)]- \Loss(\minimum)
		&\le \tfrac{\ubound}{2}\E\|\Weights_n-\minimum\|^2
	\end{align*}
	and thus also convergence of the loss in expectation.
\end{remark}

\begin{proof}
	In the classical case we just wanted to select \(\lr_n\) as large as possible,
	i.e.\ \(\lr_n=\tfrac{2}{\ubound+\lbound}\). But in the stochastic setting there
	are adverse effects of large learning rates due to noise. Now as our upper bound
	is a convex parabola in \(\lr_n\), we can just use the first order condition
	\begin{align*}
		\frac{d}{d\lr_n}
		= -2\tfrac{\ubound\lbound}{\ubound+\lbound}
		\E\|\Weights_n - \minimum\|^2 + 2\lr_n \stdBound^2\ubound\lbound
		\xeq{!} 0,
	\end{align*}
	to find the best upper bound
	\begin{align}\label{eq: optimal learning rate SGD}
		\lr_n
		= \min\left\{
			\frac{\E\|\Weights_n - \minimum\|^2}{\stdBound^2(\ubound+\lbound)},
			\frac{2}{\ubound+\lbound}
		\right\}.
	\end{align}
	So we still want to ``max out our learning rate'' and use
	\(\lr_n=\tfrac{2}{\ubound+\lbound}\) if
	\begin{align*}
		2\stdBound^2 \le \E\|\Weights_n - \minimum\|^2.
	\end{align*}
	We find again: the stochastic variance only becomes a problem once we are
	close to the minimum. Before, in the transient phase, we basically want to
	treat our problem like the classical one. Using Lemma~\ref{lem: SGD bound
	with noise} here, we get a linear convergence rate
	\begin{align*}
		\E\|\Weights_{n+1} - \minimum\|^2
		&\le \left(1-2\tfrac{2}{\ubound+\lbound}\tfrac{\ubound\lbound}{\ubound+\lbound}\right)
		\E\|\Weights_n - \minimum\|^2 + \left(\tfrac{2}{\ubound+\lbound}\right)^2
		\underbrace{\stdBound^2\ubound\lbound}_{
			\le\tfrac{\ubound\lbound}{2} \E\|\Weights_n - \minimum\|^2
		}\\
		&\le \underbrace{
			\left(1-\tfrac{2\ubound\lbound}{(\ubound+\lbound)^2}\right)
		}_{
			=1-\frac{2}{(1+\condition^{-1})(1+\condition)}
		}
		\E\|\Weights_n - \minimum\|^2.
	\end{align*}
	Plugging our asymptotic learning rate into inequality (\ref{eq: SGD bound
	with noise}) from Lemma~\ref{lem: SGD bound with noise} on the other hand, we
	get
	\begin{align*}
		&\E\|\Weights_{n+1} - \minimum\|^2\\
		&\le \left(1-2
			\tfrac{\E\|\Weights_n - \minimum\|^2}{\stdBound^2(\ubound+\lbound)}
			\tfrac{\ubound\lbound}{\ubound+\lbound}
		\right)
		\E\|\Weights_n - \minimum\|^2
		+ \left(
			\tfrac{\E\|\Weights_n - \minimum\|^2}{\stdBound^2(\ubound+\lbound)}
		\right)^2
		\stdBound^2\\
		&\le \left(1-
			\tfrac{1}{\stdBound^2(1+\condition^{-1})(1+\condition)}
			\E\|\Weights_n - \minimum\|^2
		\right)
		\E\|\Weights_n - \minimum\|^2.
	\end{align*}
	Using diminishing contraction Lemma~\ref{lem: upper bound on diminishing
	contraction} the recursion above finally implies
	\begin{align*}
		\E\|\Weights_n - \minimum\|^2
		&\le \frac{\stdBound^2(1+\condition^{-1})(1+\condition)}{n+1-n_0}.
		\qedhere
	\end{align*}
\end{proof}

Another way to interpret Lemma~\ref{lem: SGD bound with noise} is to show
convergence to an area for constant learning rates.
\begin{theorem}[{\cite[Theorem 4.6]{bottouOptimizationMethodsLargeScale2018}}]
	\label{thm: SGD converges to area}
	Let \(\Loss\in\strongConvex{\lbound}{\ubound}\) with
	\begin{align*}
		\E\|\martIncr_n\|^2 \le \stdBound \ubound\lbound,
	\end{align*}
	then using SGD with learning rate \(\lr\le\frac{2}{\ubound+\lbound}\) results in
	\begin{align*}
		\limsup_n\E\|\Weights_n-\minimum\|^2
		\le \stdBound^2 \lr\frac{(\ubound+\lbound)}{2},
	\end{align*}
	more specifically we have
	\begin{align*}
		&\E\|\Weights_n-\minimum\|^2\\
		&\le \stdBound^2 \lr\frac{(\ubound+\lbound)}{2}
		+ \underbrace{\left(1-2\lr\frac{\ubound\lbound}{\ubound+\lbound}\right)^n}_{\to 0}
		\left[
			\|\weights_0-\minimum\|^2-\stdBound^2 \lr\frac{(\ubound+\lbound)}{2}
		\right].
	\end{align*}
\end{theorem}
\begin{proof}
	Subtracting \(\stdBound^2 \lr\frac{(\ubound+\lbound)}{2}\) from both sides
	of Lemma~\ref{lem: SGD bound with noise} we get
	\begin{align*}
		&\E\|\Weights_{n+1}-\minimum\|^2 - \stdBound^2 \lr\frac{(\ubound+\lbound)}{2}\\
		&\le \left(1-2\lr\frac{\ubound\lbound}{\ubound+\lbound}\right)
		\E\|\Weights_n-\minimum\|^2
		-\stdBound^2 \lr\frac{(\ubound+\lbound)}{2} + \lr^2\stdBound^2\ubound\lbound\\
		&\le \left(1-2\lr\frac{\ubound\lbound}{\ubound+\lbound}\right)
		\left[
			\E\|\Weights_n-\minimum\|^2-\stdBound^2 \lr\frac{(\ubound+\lbound)}{2}
		\right].
		\qedhere
	\end{align*}
\end{proof}

\section{Simplified Learning Rate Schedules}\label{sec: simplified learning rate schedules}

Selecting a learning rate
\begin{align*}
	\lr_n = \frac{\E\|\Weights_n - \minimum\|^2}{\stdBound^2(\ubound+\lbound)}
\end{align*}
is difficult without very good knowledge of the problem which would likely mean
we do not need to use SGD for optimization. But since we know that the learning
rate should behave like \(O(\tfrac1n)\) we could try to guess the learning rate
and see what happens.

\subsection{Diminishing Learning Rate Schedules}

\begin{theorem}\label{thm: convergence rates for diminishing lr schedules}
	For a quadratic and convex loss \(\Loss\) selecting the learning rate
	schedule
	\begin{align*}
		\lr_n = \frac{a}{n+b}	
	\end{align*}
	for some \(a>0\), \(b\ge a\hesseEV_\dimension\) means SGD behaves like
	\begin{align*}
		\E\|\Weights_n - \minimum\|^2 
		&\in O(n^{-2\hesseEV_1a})\|\weights_0-\minimum\|^2
		+ \tilde{\stdBound}^2 O(n^{-1} + n^{-2\hesseEV_1 a})
	\end{align*}
	In particular \(b\) does not matter. Its lower bound is only required so that
	the learning rate starts below \(\tfrac1{\hesseEV_\dimension}\). Eventually
	that would happen anyway, having no bearing on the asymptotic behavior.
	Selecting \(a\ge\tfrac{1}{2\hesseEV_1}\), ensures a convergence rate of
	\(O(\tfrac1n)\). A possible selection to
	achieve \(O(\tfrac1n)\) convergence is thus for example \(a=1/\hesseEV_1\),
	\(b=a\hesseEV_\dimension\) resulting in
	\begin{align*}
		\lr_n = \frac{1}{\hesseEV_1(n+\condition)}.
	\end{align*}
	For \(\hesseEV_1=\hesseEV_\dimension=1\), this results in the mean, which we
	found is optimal in our expected value search in Subsection~\ref{subsec:
	variance reduction}
\end{theorem}
\begin{proof}
	First recall that for quadratic loss functions we derived the explicit
	representation
	\begin{align*}
	&\Weights_{n+1}-\minimum\\
	&\lxeq{(\ref{eq: unrolled SGD weights (general quadratic loss case)})}
	\left(\prod_{k=0}^n(1-\lr_kH)\right)(\weights_0-\minimum)
	+ \sum_{k=0}^n\lr_k\left(\prod_{i=k+1}^n(1-\lr_iH)\right)\martIncr_{k+1}.
	\end{align*}
	By diagonalizing \(H\) and using the fact that
	\begin{align*}
		&\prod_{k=0}^{n-1} (1-\lr_k V\diag[\hesseEV_1, \dots, \hesseEV_\dimension]V^T)\\
		&= V \diag\left[
			\prod_{k=0}^{n-1}(1-\lr_k\hesseEV_1),
			\dots, \prod_{k=0}^{n-1}(1-\lr_k\hesseEV_\dimension)
		\right]V^T,
	\end{align*}
	we can consider all the eigenspaces separately again
	\begin{align*}
		\langle \Weights_{n+1}-\minimum, v_j \rangle
		= \begin{aligned}[t]
			&\left(\prod_{k=0}^n(1-\lr_k\hesseEV_j)\right) \langle \weights_0-\minimum, v_j\rangle\\
			&+ \sum_{k=0}^n \lr_k \left(\prod_{i=k+1}^n(1-\lr_i\hesseEV_j)\right)\langle\martIncr_{k+1},v_j\rangle.
		\end{aligned}
	\end{align*}
	As the first summand is deterministic and the martingale increments are
	pairwise independent, we can use our usual conditional independence argument
	to get
	\begin{align*}
		\E\|\Weights_n-\minimum\|^2
		&= \E[(V^T [\Weights_n-\minimum])^T(V^T[\Weights_n-\minimum])]\\
		&= \sum_{j=1}^\dimension \E\langle\Weights_n-\minimum, v_j \rangle^2\\
		&= \sum_{j=1}^\dimension
		\begin{aligned}[t]
			&\Bigg[\left(\prod_{k=0}^{n-1}(1-\lr_k\hesseEV_j)\right)^2 \langle \weights_0-\minimum, v_j\rangle^2\\
			&+ \sum_{k=0}^{n-1} \lr_k^2 \left(\prod_{i=k+1}^{n-1}(1-\lr_i\hesseEV_j)\right)^2\E[\langle\martIncr_{k+1},v_j\rangle^2]\Bigg].
		\end{aligned}
	\end{align*}
	Note that since the eigenvalues are all different, we cannot select \(\lr_n=\frac{1}{\hesseEV(n+1)}\)
	to achieve
	\begin{align*}
		1-\lr_n\hesseEV = \frac{n}{n+1},
	\end{align*}
	which would be the product equivalent of a telescoping sum as we have used in
	(\ref{eq: telescoping product}). To make the product go away, we could use
	\(1+x\le \exp(x)\) inside the square, to convert it into a sum. But since \(x\to
	x^2\) is only monotonous for positive \(x\), we need
	\begin{align*}
		1-\lr_k \hesseEV_j \ge 0 \qquad \forall k\ge 0, \forall j \in\{1,\dots,\dimension\}.
	\end{align*}
	This is why we need the requirement \(b\ge a\hesseEV_\dimension\).

	Enacting this idea with the exponential function bound \(1-x\le\exp(-x)\), we
	get
	\begin{align}
		\nonumber
		\left(\prod_{i=k}^{n-1}(1-\lr_i\hesseEV_j)\right)^2
		&\le\exp\left(-\sum_{i=k}^{n-1}\lr_i\hesseEV_j\right)^2\\
		\nonumber
		&\le \exp\left(-2\hesseEV_j a[\log(n+b)-\log(k+b)]\right)\\
		\label{eq: detour over exponential bound}
		&= \left(\frac{n+b}{k+b}\right)^{-2\hesseEV_j a},
	\end{align}
	where we have used
	\begin{align*}
		\sum_{i=k}^{n-1} \lr_i
		= \int_{k}^{n} \frac{a}{\lfloor x\rfloor+b} dx
		\ge \int_{k}^{n} \frac{a}{x+b} dx
		= a(\log(n+b) - \log(k+b))
	\end{align*}
	in the second inequality. Using this detour over the exponential function,
	we can upper bound our deterministic part by
	\begin{align}
		\nonumber
		\sum_{j=1}^\dimension\left(\prod_{k=0}^{n-1}(1-\lr_k\hesseEV_j)\right)^2
		\langle \weights_0-\minimum, v_j\rangle^2
		&\le \sum_{j=1}^\dimension \left(\frac{n+b}{b}\right)^{-2\hesseEV_j a}
		\langle\weights_0-\minimum, v_j\rangle^2\\
		\label{eq: upper bound on convergence rate using smallest eigenvalue 1}
		&\le \left(\frac{n+b}{b}\right)^{-2\hesseEV_1 a}\|\weights_0-\minimum\|^2\\
		\nonumber
		&\in O(n^{-2\hesseEV_1a})\|\weights_0-\minimum\|^2.
	\end{align}
	And bound our noise by
	\begin{align}
		\nonumber
		&\sum_{j=1}^\dimension\sum_{k=0}^{n-1} \lr_k^2
		\left(\prod_{i=k+1}^{n-1}(1-\lr_i\hesseEV_j)\right)^2
		\E[\langle\martIncr_{k+1},v_j\rangle^2]\\
		\nonumber
		&\le \sum_{j=1}^d
			\sum_{k=0}^{n-1} \lr_k^2 \left(\tfrac{n+b}{k+1+b}\right)^{-2\hesseEV_j a}
			\E[\langle\martIncr_{k+1},v_j\rangle^2]
		\\
		\label{eq: upper bound on convergence rate using smallest eigenvalue 2}
		&\le \sum_{k=0}^{n-1} \lr_k^2 \left(\tfrac{n+b}{k+b+1}\right)^{-2\hesseEV_1 a}
		\underbrace{\E[\|\martIncr_{k+1}\|^2]}_{\le \tilde{\stdBound}^2}
		\\
		\nonumber
		&\le
		\tilde{\stdBound}^2a^2 (n+b)^{-2\hesseEV_1 a}
		\underbrace{\sum_{k=0}^{n-1} \frac{(k+b+1)^{2\hesseEV_1 a}}{(k+b)^2}}_{
			\sim \int_b^{n+b} \frac{(x+1)^{2\hesseEV_1 a}}{x^2}dx} \\
		\nonumber
		&\in \tilde{\stdBound}^2 O(n^{-2\hesseEV_1 a})  O(n^{-1+2\hesseEV_1a} + 1)
		= \tilde{\stdBound}^2 O(n^{-1} + n^{-2\hesseEV_1 a}).
	\end{align}
	Putting those two bounds together we obtain our claim.
\end{proof}
\begin{remark}\label{rem: rate is retained for larger eigenvalues}
	In equations (\ref{eq: upper bound on convergence rate using smallest eigenvalue 1})
	and (\ref{eq: upper bound on convergence rate using smallest eigenvalue 2})
	we bounded the convergence rate of all eigenspaces, with the convergence rate
	of the eigenspace with the smallest eigenvalue \(\hesseEV_1\). If we select
	our learning rate schedule wrong, the convergence rate in this eigenspace
	is smaller. But the other eigenspaces might still converge at rate \(O(\tfrac1n)\).
\end{remark}
\begin{remark}
	The integral lower bounds on the sum of learning rates could also be upper bounds
	up to some constant shift, so concerning the asymptotic rate this bound is
	tight. The only other estimation method we used (except replacing all
	eigenvalues with the worst) is bounding \(1+x\) by \(\exp(x)\) which is
	also tight for small \(x\). Since the learning rates are decreasing, this
	bound becomes better and better and should thus asymptotically be tight.
\end{remark}

The next theorem provides a more general statement with uglier requirements on
the learning rates but with more explicit bounds.

\begin{theorem}[{\cite[cf. Theorem 4.7]{bottouOptimizationMethodsLargeScale2018}}]
	For \(\Loss\in\strongConvex{\lbound}{\ubound}\) with
	\begin{align*}
		\E\|\martIncr_n\|^2 \le \stdBound^2 \ubound\lbound,
	\end{align*}
	SGD with step sizes
	\begin{align*}
		\lr_n = \frac{\tilde{a}}{n+b},
		\qquad \tilde{a} =  a\frac{\ubound+\lbound}{2\ubound\lbound},
		\quad a > 1,
		\quad b \ge a\frac{(1+\condition)(1+\condition^{-1})}{4}
	\end{align*}
	converges at rate
	\begin{align*}
		\E\|\Weights_n - \minimum\|^2
		\le \frac{\nu}{n+b},
	\end{align*}
	where
	\begin{align*}
		\nu = \max\Bigg\{
			\frac{\stdBound^2a^2(\condition+1)(1+\condition^{-1})}{a-1},
			\underbrace{
				b\|\weights_0 - \minimum\|^2
			}_{\text{induction start}}
		\Bigg\}.
	\end{align*}
\end{theorem}
\begin{proof}
	We will prove the statement by induction. The induction start is covered by
	definition of \(\nu\). For the induction step we use Lemma~\ref{lem: SGD
	bound with noise} for which we need \(\lr \ge \tfrac{2}{\ubound+\lbound}\)
	(which is ensured by our requirement on \(b\)) and define \(\tilde{n}:=n+b\)
	for
	\begin{align*}
		\E\|\Weights_{n+1}-\minimum\|^2
		&\le \left(1-2\lr_n \frac{\ubound\lbound}{\ubound+\lbound}\right)\frac{\nu}{\tilde{n}}
		+ \lr_n^2 \stdBound^2 \ubound\lbound\\
		&= \frac{\tilde{n} - a}{\tilde{n}}\frac{\nu}{\tilde{n}}
		+ \frac{\stdBound^2}{\tilde{n}^2}\frac{a^2(\ubound+\lbound)^2}{\ubound\lbound}\\
		&= \frac{\tilde{n} - 1}{\tilde{n}^2}\nu
		- \underbrace{\frac{(a -1)\nu + \stdBound^2 a^2(\condition+1)(1+\condition^{-1})}{\tilde{n}^2}}_{
			\ge 0\qquad \text{definition of \(\nu\) and } a>1
		}\\
		&\le \frac{\tilde{n}-1}{(\tilde{n}+1)(\tilde{n}-1)}\nu = \frac{\nu}{n+1+b}.
		\qedhere
	\end{align*}
\end{proof}

\subsection{Piecewise Constant Learning Rates}

As we have seen in Theorem~\ref{thm: optimal rates SGD}, we can achieve
much better results if we keep the learning rate high during the transient
phase. If we use the results above we miss out on the exponential reduction that
constant learning rates provide. So the first idea is to use constant learning
rates until we ``stop making progress''. i.e.\ we have converged to an area like
in Theorem~\ref{thm: SGD converges to area}. But convergence in this theorem is
exponential, meaning we do not actually converge in finite time. And due to
stochasticity it is difficult to tell when we actually stop making progress.
Especially because a plateau in the loss could also be due to a saddle point
region (recall Figure~\ref{fig: visualize saddle point gd}).

Detecting convergence is thus its own research field \parencite[for a recent
approach including an overview of previous work see
e.g.][]{pesmeConvergenceDiagnosticBasedStep2020}. One prominent conjecture is,
that the dot product of the weight updates becomes more random when we have
converged. Since then, we do not follow a straight gradient path anymore, with
gradients pointing in roughly the same direction.

Another problem with diminishing learning rates is: If we reduce learning rates
too fast (e.g. \(\tfrac{1}{n+1}\), when \(\hesseEV_1=\tfrac14\) requires something
like \(\tfrac{2}{n+b}\)), then we have seen in Theorem~\ref{thm: convergence
rates for diminishing lr schedules} that our convergence rate is reduced below
\(O(\tfrac1n)\). Since \(\hesseEV_1\) can be arbitrarily small our convergence rate
\(O(n^{-2\hesseEV_1 a})\) could be arbitrarily slow in fact.

So we want to avoid reducing the learning rate too fast. One possibility is
choosing a learning rate schedule with \(O(n^{-\alpha})\) where \(\tfrac12<\alpha<1\),
which sacrifices convergence speed on purpose to avoid it becoming arbitrarily
bad. Another idea is to recycle the ``convergence detection'' algorithm we
need anyway, to transition from the transient phase into the asymptotic phase.
In other words: Use piecewise constant learning rates.

\textcite{smithDonDecayLearning2018} argue that there is another benefit to
piecewise constant learning rates: ``we note that it is well known in the
physical sciences that slowly annealing the temperature (noise scale) helps the
system to converge to the global minimum, which may be sharp. Meanwhile
annealing the temperature in a series of discrete steps can trap the system in a
`robust' minimum whose cost may be higher but whose curvature is lower. We
suspect a similar intuition may hold in deep learning.''
In other words: for a constant learning rate, the area to which SGD can converge
has a minimal size. So it will disregard all holes (local minima) which are
smaller than this size. Sharply decreasing the learning rate while over a larger
crater will keep SGD in this general area.

We will now follow a proof by \textcite[pp. 27-28]{bottouOptimizationMethodsLargeScale2018}
which proves that piecewise constant learning rates can achieve the optimal
convergence rate of \(O(\tfrac1n)\). As a small improvement, we will 
show it for the \(L^2\) distance of the weights and not the expected loss, which
is a special case (cf.  Remark~\ref{rem: expected loss is a special case of L2
weight convergence}).

Recall from Theorem~\ref{thm: SGD converges to area} that for learning rate
\begin{align*}
	\lr = \frac{2^{-r}}{\ubound+\lbound},
\end{align*}
SGD converges to
\begin{align*}
	\limsup_{n}\E\|\Weights_n - \minimum\|^2
	\le \stdBound^2 \lr \frac{\ubound+\lbound}{2} = \frac{\stdBound^2}{2^{r+1}}.
\end{align*}
As convergence takes an infinite amount of time we only wait until
\begin{align*}
	\E\|\Weights_n -\minimum\|^2 \le \frac{\stdBound^2}{2^r}
\end{align*}
to half the learning rate. More precisely we select 
\begin{align}\label{eq: piecewise constant learning rates}
	\lr_n := \frac{2^{-r_n}}{\ubound+\lbound},
	\qquad	
	r_n := \max\left\{k\in\naturals : \E\|\Weights_n-\minimum\|^2 \le 2^{1-k}\stdBound^2\right\}.
\end{align}
Let the points where we half the learning rate
\(2^{-r}/(\ubound+\lbound)\) be 
\begin{align*}
	N(r) := \min\{k\in\naturals : r_k \ge r\}.
\end{align*}

\begin{theorem}[Piecewise Constant Learning Rates]
	If we select learning rates of SGD according to (\ref{eq: piecewise constant
	learning rates}) for \(\Loss\in\strongConvex{\lbound}{\ubound}\) with
	\begin{align*}
		\E\|\martIncr_n\|^2\le \stdBound^2\ubound\lbound,
	\end{align*}
	then we get a convergence rate at the ``halving points'' of
	\begin{align*}
		\E\|\Weights_{N(r)} - \minimum\|^2 \le \frac{c(\condition)\stdBound^2}{N(r)},
	\end{align*}
	where \(c\) is a constant roughly linear in the condition \(\condition\).
\end{theorem}
\begin{proof}
	Applying the recursion from Theorem~\ref{thm: SGD converges to area} we get
	\begin{align*}
		\frac{\stdBound^2}{2^r} - \frac{\stdBound^2}{2^{r+1}}
		&\le\E\|\Weights_{N(r+1)-1} - \minimum\|^2 - \frac{\stdBound^2}{2^{r+1}}\\
		&\le \left(1-2^{1-r}\frac{\ubound\lbound}{(\ubound+\lbound)^2}\right)^{N(r+1)-1-N(r)}
		\Bigg[
			\underbrace{\E\|\Weights_{N(r)}-\minimum\|^2}_{\le 2^{1-r}\stdBound^2} - \frac{\stdBound^2}{2^r}
		\Bigg].
	\end{align*}
	Dividing by \(\tfrac{\stdBound^2}{2^r}\) we get
	\begin{align*}
		\frac{1}{2}
		&\le \left(1-2^{1-r}\frac{\ubound\lbound}{(\ubound+\lbound)^2}\right)^{N(r+1)-1-N(r)}
	\end{align*}
	and therefore
	\begin{align*}
		\log(\tfrac12)
		&\le (N(r+1)-N(r)-1)\log\left(1-2^{1-r}\frac{\ubound\lbound}{(\ubound+\lbound)^2}\right)\\
		&\le (N(r+1)-N(r)-1)\left(-2^{1-r}\frac{\ubound\lbound}{(\ubound+\lbound)^2}\right),
	\end{align*}
	due to \(\log(1+x)\le x\). Dividing by the negative multiple on the right we
	finally get
	\begin{align*}
		N(r+1)-N(r)-1
		\le \frac{\log(\tfrac12)}{-2^{1-r}\frac{\ubound\lbound}{(\ubound+\lbound)^2}}
		= \frac{\log(2)(\ubound+\lbound)^2}{\ubound\lbound}2^{r-1}.
	\end{align*}
	Therefore we have
	\begin{align*}
		N(r) &= \sum_{k=0}^{r-1} N(k+1)-N(k)
		\le r + \frac{\log(2)(\ubound+\lbound)^2}{\ubound\lbound}
		\underbrace{\sum_{k=0}^{r-1} 2^{k-1}}_{=2^r-1}\\
		&=r+\log(2)(\condition+1)(1+\condition^{-1})(2^r-1) \in O(2^r).
	\end{align*}
	I.e.\ halving the error (and step size) takes double the time. Now with
	\begin{align*}
		N(r) \le c2^r \implies \frac{\log(N(r)) - \log(c)}{\log(2)} \le r
	\end{align*}
	and \(r_{N(r)} = r\), we get
	\begin{align*}
		\E\|\Weights_{N(r)}-\minimum\|^2 
		&\le 2^{1-r_{N(r)}}\stdBound^2\\
		&= \exp\left((1-r_{N(r)})\log(2)\right)\stdBound^2\\
		&\le \exp[\log(2)- (\log(N(r)) - \log(c))]\stdBound^2\\
		&=  \frac{2c(\condition)\stdBound^2}{N(r)},
	\end{align*}
	where \(c(\condition)\) highlights the dependence of \(c\) on the condition
	\(\condition\).
\end{proof}

This approach retains the convergence rate of \(O(\tfrac1n)\) even if we wait too long.
E.g. if we do not wait until we are at twice the asymptotic bound, but \(1.5\)
times the asymptotic bound, then this represents a halving of the distance to
the asymptotic bound. Due to exponential convergence to this bound, this should
take roughly double the amount of time. If we spend double the amount of time at
every learning rate step, then the amount of time to reach some level \(r\)
doubles. This does not change the fact it is \(O(\tfrac1n)\),
which reinforces the notion that it is better to keep learning rates high,
rather than reducing them too quickly. As that would slow progress down below
\(O(\tfrac1n)\).

\section{General Convex Case -- Averaging}\label{sec: SGD with Averaging}

So far, we only discussed \emph{strongly} convex functions. In the classical
convex loss function setting (Section~\ref{sec: convex convergence theorems}),
we used Lemma~\ref{lem: bregmanDiv lower bound} in place of Lemma~\ref{lem:
bregmanDiv lower bound (strongly convex)} to ensure a decreasing distance from
the optimum. But the noise prevents us from doing the same here, like we did in
Lemma~\ref{lem: SGD bound with noise} for strongly convex functions.

\textcite{nemirovskiRobustStochasticApproximation2009} showed that the approach
developed to deal with non Lipschitz continuous gradients and subgradients (cf.
Subsection~\ref{subsec: subgradient method}) can be applied directly to the
stochastic setting. The price for this generality is a greatly reduced
convergence speed.

In this setting we assume \(\Loss \in \lipGradientSet[0,0]{\lipConst}\).
And instead of the Lemmas mentioned above, we simply use the
(subgradient) definition of convexity
\begin{align*}
	\Loss(\minimum)
	\ge \Loss(\Weights_n) + 
	\langle\nabla\Loss(\Weights_n), -(\Weights_n - \minimum)\rangle
\end{align*}
in (\ref{eq: getting rid of the scalar product (SGD)}), i.e.
\begin{align*}
	&\|\Weights_n - \minimum - \lr_n\nabla\Loss(\Weights_n)\|^2\\
	&\le \|\Weights_n - \minimum\|^2
	- 2\lr_n\underbrace{\langle\nabla\Loss(\Weights_n), \Weights_n-\minimum\rangle}_{
		\ge \Loss(\Weights_n) - \Loss(\minimum)
	}
	+ \lr_n^2 \|\nabla\Loss(\Weights_n)\|^2,
\end{align*}
to obtain
\begin{align*}
	\E\|\Weights_{n+1}-\minimum\|^2
	&\le
	\begin{aligned}[t]
		&\E\|\Weights_n -\minimum\|^2
		- 2\lr_n \E[\Loss(\Weights_n) - \Loss(\minimum)]\\
		&+ \lr_n^2 [
			\underbrace{\E\|\nabla\Loss(\Weights_n)\|^2}_{\le \lipConst^2}
			+ \underbrace{\E\|\martIncr_{n+1}\|^2}_{\le \stdBound^2}
		],
	\end{aligned}
\end{align*}
as \(\Loss\in\lipGradientSet[0,0]{\lipConst}\) implies Lipschitz continuity and thus bounded
gradients. Note that splitting \(\nabla\loss\) into \(\nabla\Loss\) and \(\martIncr\)
was of no use here as we end up summing them anyway. \textcite{nemirovskiRobustStochasticApproximation2009}
uses bounded \(\E\|\nabla\loss\|^2\) from the start to get the same result.
Reordering the previous equation leads to
\begin{align}
	\nonumber
	&2\sum_{n=0}^{N-1} \lr_n\E[\Loss(\Weights_n)-\Loss(\minimum)]\\
	\nonumber
	&\le \sum_{n=0}^{N-1} \E\|\Weights_n -\minimum\|^2 - \E\|\Weights_{n+1}-\minimum\|^2
	+ \lr_n^2 [\lipConst^2 + \stdBound^2]\\
	\label{eq: sgd averaging}
	&\le \E\|\weights_0 - \minimum\|^2 + [\lipConst^2+\stdBound^2]\sum_{n=0}^{N-1} \lr_n^2.
\end{align}
Dividing both sides by \(2\sum_{n=0}^{N-1}\lr_n=2T\) turns the term on the left
into a convex combination allowing us to finally get 
\begin{align*}
	\E\left[\Loss\left(\sum_{n=0}^{N-1} \frac{\lr_n}{T}\Weights_n\right)\right]-\Loss(\minimum)
	&\xle{\text{conv.}} \sum_{n=0}^{N-1} \frac{\lr_n}{T}\E[\Loss(\Weights_n)]-\Loss(\minimum)\\
	&\lxle{(\ref{eq: sgd averaging})} \frac{
		\|\weights_0 - \minimum\|^2 + [\lipConst^2+\stdBound^2]\sum_{n=0}^N \lr_n^2
	}{
		2T
	}.
\end{align*}
Now if there was no stochasticity, one could instead use the minimum of
\(\Loss(\Weights_n)\) instead, which is how one obtains the result in
Subsection~\ref{subsec: subgradient method}. But since we do have randomness,
it appears to make sense to average the weights \(\Weights_n\) we collect over
time.
Now we can ask ourselves which learning rates \(\lr_n\) we should pick to
minimize our bound. Due to convexity of the square we have
\begin{align*}
	1 = \left(\frac{1}{N}\sum_{n=0}^{N-1}\frac{N}{T}\lr_n\right)^2
	\xle{\text{conv.}} \frac{1}{N}\sum_{n=0}^{N-1}\left(\frac{N}{T}\lr_n\right)^2
	\le \frac{N}{T^2}\sum_{n=0}^{N-1}\lr_n^2,
\end{align*}
which means that constant learning rates \(\lr=\tfrac{T}N\) are optimal, since
\begin{align*}
	\sum_{n=0}^{N-1}\lr^2 = \frac{T^2}{N} \le \sum_{n=0}^{N-1}\lr_n^2.
\end{align*}
This leads to
\begin{align*}
	\E\Big[
		\Loss\Big(\underbrace{
			\frac{1}{N}\sum_{n=0}^{N-1}\Weights_n
		}_{=:\overline{\Weights}_N}\Big)
	\Big]-\Loss(\minimum)
	&\le \frac{
		\|\weights_0 - \minimum\|^2 + [\lipConst^2+\stdBound^2]N\lr^2
	}{2\lr N}.
\end{align*}
Minimizing over \(\lr\) for a constant number of steps \(N\) tunes the
ODE time \(T\). An increase of \(T\) entails larger discretizations but also
allows for more time following the ODE.
This minimization results in the following theorem.
\begin{theorem}[Optimal Rates]\label{thm: optimal averaging rates}
	If we use SGD with learning rate
	\begin{align*}
		\lr_* = \frac{\|\weights_0-\minimum\|}{\sqrt{\lipConst^2+\stdBound^2}\sqrt{N}}
	\end{align*}
	on \(\Loss\in\lipGradientSet[0,0]{\lipConst}\) with
	\begin{align*}
		\E\|\martIncr_n\|^2 \le \stdBound^2,
	\end{align*}
	then we get a convergence rate of
	\begin{align*}
		\E[\Loss(\overline{\Weights}_N)] -\Loss(\minimum)
		\le \frac{\|\weights_0-\minimum\|\sqrt{\lipConst^2+\stdBound^2}}{\sqrt{N}}
	\end{align*}
	for the weight average \(\overline{\Weights}_N\).
\end{theorem}

While this approach helps us lose the strong convexity assumption, which might
ensure decent convergence rates along eigenspaces of the Hessian with
small or zero eigenvalues, it is in a sense much more dependent on the convexity
of the loss function \(\Loss\) as our averaging of weights is only a
guaranteed improvement because of convexity. While \ref{eq: gradient descent} should work
intuitively on non-convex functions as well, with worse rates of convergence in
eigenspaces with small eigenvalues (cf. Remark~\ref{rem: rate is retained for
larger eigenvalues}), this averaging technique will generally not play nicely
with non-convex functions.

Sure, if we moved into a convex area relatively quickly and stayed there for
most of our learning time, then the first few weights outside this convex
area will be negligible for the average.
But it might still make sense to wait with the averaging, until the time it is
actually needed to resume convergence. Especially since the rate of convergence
is so much worse\footnote{
	At least as far as we know -- \textcite{bachNonstronglyconvexSmoothStochastic2013}
	recover the convergence rate \(O(\tfrac1N)\) only for a very specific model.
}.

\section{Batch Learning}\label{sec: batch learning}

Instead of using \(\nabla\loss(\Weights_n, X_{n+1}, Y_{n+1})\) at time \(n\) as an
estimator for \(\nabla\Loss(\Weights_n)\) we could use the average of a
``batch'' of data \((X^{(i)}_{n+1}, Y^{(i)}_{n+1})_{i=1,\dots,m}\) (independently
\(\dist\) distributed) instead, i.e.
\begin{align*}
	\Weights_{n+1} = \Weights_n
	-\lr_n \underbrace{\frac1{m}\sum_{i=1}^m\nabla\loss(\Weights_n,X^{(i)}_{n+1}, Y^{(i)}_{n+1})}_{
		=:\nabla\loss_{n+1}^m(\Weights_n)
	}.
\end{align*}
leading to the modified martingale increment
\begin{align*}
	\martIncr_n^{(m)}
	:= \nabla\Loss(\Weights_{n-1})
	- \nabla\loss_n^m(\Weights_{n-1})
\end{align*}
with reduced variance
\begin{align*}
	\E\|\martIncr_n^{(m)}\|^2 = \tfrac1m\E\|\martIncr_n^{(1)}\|^2 \le \tfrac1m \tilde{\stdBound}^2.
\end{align*}
While this variance reduction does reduce our upper bound on the distance
between SGD and GD (cf. Theorem~\ref{thm: distance SGD vs GD}), this reduction
comes at a cost: We now have to do \(m\) gradient evaluations per iteration!
We would incur an equivalent cost, if we reduced our discretization size \(\lr\)
to
 \begin{align*}
	\tilde{\lr}:=\frac{\lr}{m}=\frac{T}{Nm}.
\end{align*}
This increases the number of iterations to reach \(T\) to \(Nm\), which implies
\(Nm\) gradient evaluations using SGD without batches.

So which one is better? Increasing the batch size, or reducing the size of the
learning rate? If we have a look at the term (\ref{eq: mean of
martingale increments}) again which is in some sense the difference between SGD
and GD, then we can see that these actions have quite a similar effect
\begin{align*}
	\lr\sum_{k=1}^N\martIncr_k^{(m)}
	&=\frac{T}{N}\sum_{k=1}^N
	[\nabla\Loss(\Weights_{k-1})- \nabla\loss_n^m(\Weights_{k-1})]\\
	&=\frac{T}{Nm}\sum_{k=1}^N\sum_{i=1}^m
	[\nabla\Loss(\Weights_{k-1})-\nabla\loss(\Weights_{k-1},X^{(i)}_k, Y^{(i)}_k)]\\
	&\approx \tilde{\lr} \sum_{k=1}^{Nm}
	\underbrace{
		[\nabla\Loss(\tilde{\Weights}_{k-1})
		-\nabla\loss(\tilde{\Weights}_{k-1}, \tilde{X}_k, \tilde{Y}_k)]
	}_{\widetilde{\martIncr}_k}.
\end{align*}
The difference is, that batch learning stays on some \(\Weights_k\) and collects
\(m\) pieces of information before it makes a big step based on this information,
while SGD just makes \(m\) small steps.

This notion is reflected in the upper bound in Theorem~\ref{thm: distance SGD vs
GD}, where these two actions have the same effect on our distance bound between SGD and
GD\footnote{
	Especially if \(\Loss\) is convex, since we can get rid of the exponential
	term then. As hinted at in the proof of the theorem.
}. But as we have seen in
Subsection \ref{subsec: variance reduction} for constant learning rates, halving
them is ever so slightly worse than simply reducing the variance
with batches as GD converges faster than Gradient Flow. But as GD converges
to Gradient Flow, this effect is less pronounced for smaller learning rates.

In the same section we also see, that diminishing learning rates result in an
equivalent improvement on the variance. Instead of running GD for the same
``time'' in Gradient Flow, and reducing the discretization size, we would
actually increase this ``time'' here. In other words: do more steps with
diminishing learning rates \(O(\tfrac1n)\) resulting in a logarithmic time
increase 
 \[
	\sum_{k=0}^n\tfrac{1}{k}\approx\log(n)
\]
of \ref{eq: gradient flow}.
This effect can best be seen in Theorem~\ref{thm: convergence rates for
diminishing lr schedules}, where increasing the number of steps improves the ``GD
part'' as well as the variance reduction part, while increasing the batch size
has this effect only on the variance. In other words, increasing the number of
steps continues optimization as well as reduces variance, while batches
\emph{only} reduces variance.

Even worse: in the transient case, we found in Theorem~\ref{thm:
optimal rates quadratic case} and \ref{thm: optimal rates SGD} that the variance is
completely irrelevant and increasing the batch size has no positive effect at
all. And while the second of the two phases of SGD \parencite[first observed
by][]{darkenFasterStochasticGradient1991} determines the asymptotic properties
of SGD,  ``in practice it seems that for deeper networks in particular, the first
phase dominates overall computation time as long as the second phase is cut off
before the remaining potential gains become either insignificant or entirely
dominated by overfitting (or both)''  
\parencite{sutskeverImportanceInitializationMomentum2013}.

So in this transient phase, a larger batch size only causes more computation
time (which might be hidden by parallelization as batch learning is easier to
parallelize than SGD) and requires more samples. If we assume independent
identically distributed data, this means we have to stop training much earlier, as
we have used up our sample set, or risk overfitting.

In the asymptotic phase the situation is less clear. With (optimal) diminishing
learning rates the same claim would apply. But correctly selecting those rates
is difficult. Increasing the batch size is no miracle cure though, as we still
have to select (different) learning rates here. In theory they would have to
be \(m\) times larger than the learning rate without batches for a batch
size of \(m\). If the selected learning rates during batch learning are too
small, then this mirrors reducing the learning rate too quickly in the diminishing
case. Since then, we are using up our resources more quickly without making full use of
them. In both cases we would therefore get a reduced convergence rate.

Our solution in Section~\ref{sec: simplified learning rate schedules} to the
learning rate selection problem was piecewise constant learning rates. But
as we have found in Subsection~\ref{subsec: variance reduction}, halving the
learning rate is ever so slightly worse than doubling the batch size. So the
idea by \textcite{smithDonDecayLearning2018}, of doubling the batch size
(instead of halving the learning rate) whenever progress stalls, is probably
optimal.

In the general convex case, where we apply averaging on the other hand, SGD
without batches is the clear winner. Only when \(\lipConst^2\) is negligible
compared to \(\stdBound^2\), is reducing \(\stdBound^2\) roughly equivalent to
increasing the number of steps \(N\) in Theorem~\ref{thm: optimal averaging rates}.
Without optimal rates not only \(\lipConst^2\) has to be negligible, but
\(\|\weights_0 -\minimum\|^2\) as well, as we can see in the statements leading
up to this theorem.

Lastly, batch learning is always a good idea, if we have a virtually infinite
amount of supply of data (i.e.\ we do not have to worry about overfitting), and
we do not feel the cost of batches due to parallelization. In that case the
computational speed stays the same whether we use batches or not, and we do not
have to worry running out of data, so we might as well use them. But in general
there should be better uses for parallelization. One could for example train
the same model with different weight initializations (or entirely different
models) to generate ensemble models with (e.g. bagging, boosting, etc.).

\section{SDE View}

While we got bounds on \(\E\|\Weights_n - \minimum\|^2\) which entails a
distributional bound using something like Chebyshev's inequality, we only
have convergence of the expected value of the loss. It might therefore be
interesting to further develop our understanding of the distribution of \(\Weights_n\).

A more recent approach is to model SGD not as an approximation of an ODE (i.e.
\ref{eq: gradient flow}) but as a stochastic differential equation (SDE), also
known in the literature as ``stochastic modified equation''
\begin{align}
	\tag{SME}
	d\mathcal{\Weights}_t^{\lr}
	= -\nabla\Loss(\mathcal{\Weights}_t^{\lr})dt
	+ \sqrt{\lr \Sigma(\mathcal{\Weights}_t^{\lr})}dB_t
\end{align}
where \(B_t\) is a standard Brownian motion and \(\Sigma\) is a covariance
matrix
\begin{align*}
	\Sigma(\weights)
	&:= \E[(\nabla\loss(w,Z) - \nabla\Loss(w))(\nabla\loss(w,Z) - \nabla\Loss(w))^T]\\
	&= \E[(\nabla\loss(w,Z) - \nabla\Loss(w))^{\otimes 2}]\\
	&= \Cov(\nabla\loss(w,Z) - \nabla\Loss(w)),
\end{align*}
where we introduced \(v^{\otimes 2}:=vv^T\) for later use. It might seem
a bit weird to have our learning rate \(\lr\) appear in an equation which is
supposed to be a \emph{limiting} equation for \(\lr\to 0\). But notice how that
is quite similar to
\begin{align*}
	\frac{1}n\sum_{k=0}^n X_k - \E[X_1] \to \mathcal{N}(0, \tfrac{\sigma^2}{n}),
\end{align*}
for independent identically distributed random variables \(X_k\).
Just that we can multiply this equation by \(\sqrt{n}\) to obtain a static 
target as a limit. But we cannot simply multiply some factor to SGD to get
something similar. So we have to figure out how to deal with this moving target.
Now as we want to prove convergence in distribution, we can just consider the
distributional difference
\begin{align*}
	\E g(\mathcal{\Weights}_T^{\lr}) - \E g(\Weights^{\lr}_{T/\lr}),
\end{align*}
where \(g\) is from a sufficiently large test function set (e.g. all uniformly continuous functions)
to test distributional convergence with, and we pick \(\lr\) from
\begin{align*}
	\mathcal{H} = \{\lr : \tfrac{T}\lr \in \naturals\}
\end{align*}
and let \(\Weights_n^{\lr}\) be SGD with constant learning rate \(\lr\). Now
since we know that SGD converges to \ref{eq: gradient flow} (i.e.
\(\mathcal{\Weights}^{0}_T\)) as \(\lr\to 0\), we know that 
\begin{align*}
	\E g(\mathcal{\Weights}_T^{\lr}) - \E g(\Weights^{\lr}_{T/\lr})
	= \underbrace{
		\E g(\mathcal{\Weights}_T^{\lr}) - \E g(\mathcal{\Weights}_T^0)
	}_{\to 0 \quad (\lr\to 0)}
	+ \underbrace{
		\E g(\mathcal{\Weights}_T^0) - \E g(\Weights^{\lr}_{T/\lr})
	}_{\to 0 \quad (\lr\to 0)}.
\end{align*}
So the question is not \emph{whether} we have convergence, but rather how fast.
The following result is the latest result on this error estimation problem.

\begin{theorem}[{\cite[Theorem 1.1, 1.2]{ankirchnerApproximatingStochasticGradient2021}}]
	Assume there exists some \(C>0\) such that
	\begin{align*}
		|\nabla\loss(\weights, z)| \le C(1+|\weights|)
		\qquad \forall \weights, z,
	\end{align*}
	and that the functions \(\nabla\Loss\) and \(\sqrt{\Sigma}\) are Lipschitz
	continuous in \(C^{\infty}\) such that all their derivatives are bounded.
	Then for all test functions \(g\) with at most polynomial growth, i.e.\ there
	exists \(C>0\), \(m\in\naturals\) s.t. 
	\begin{align*}
		|g(x)| \le C(1+|x|^m),
	\end{align*}
	we have
	\begin{align}
		\tag{GF error}
		\E g(\mathcal{\Weights}^0_T) - \E g(\Weights_{T/\lr}^\lr)
		= \lr \int_0^T \phi_t^g(\mathcal{\Weights}^0_t)dt + O(\lr^2)
	\end{align}
	with
	\begin{align*}
		\phi_t^g(\weights)
		:= \tfrac12 \trace[
			\nabla^2 y_t(\weights)
			(\nabla\Loss(\weights)^{\otimes 2} + { \color{purple} \Sigma(\weights)})
		]
		+ \partial_t\nabla y_t(\weights)^T \nabla\Loss(\weights)
		+ \tfrac12 \partial_t^2 y_t(\weights),
	\end{align*}
	where
	\begin{align*}
		y_t(\weights) := \nabla\Loss(\mathcal{\Weights}_{T-t}^0)
		\qquad \text{for } \mathcal{\Weights}^0_0 = \weights.
	\end{align*}
	And we also have
	\begin{align}
		\tag{SME error}
		\E g(\mathcal{\Weights}^h_T) - \E g(\Weights_{T/\lr}^\lr)
		= \lr \int_0^T \varphi_t^g(\mathcal{\Weights}^{\color{purple}0}_t)dt + O(\lr^2)
	\end{align}
	with
	\begin{align*}
		\varphi_t^g(\weights)
		:= \tfrac12 \trace[
			\nabla^2 y_t(\weights)\nabla\Loss(\weights)^{\otimes 2}
		]
		+ \partial_t\nabla y_t(\weights)^T \nabla\Loss(\weights)
		+ \tfrac12 \partial_t^2 y_t(\weights).
	\end{align*}
\end{theorem}

While this results provides more explicit bounds, there is unfortunately no
proof for
\begin{align*}
	\left| \int_0^T \varphi_t^g(\mathcal{W}^0_t)dt \right|
	\le \left| \int_0^T \phi_t^g(\mathcal{W}^0_t)dt \right|
\end{align*}
yet, which would prove that the SME approximates SGD better than \ref{eq:
gradient flow}.

Also note that the convergence towards a Brownian motion driven SDE requires
bounded variances similar to the central limit theorem. Once variances are not
bounded anymore, the limiting SDE would have to be driven by a heavy tailed
\(\alpha\) stable distribution, becoming a Lévy motion. \textcite{simsekliTailIndexAnalysisStochastic2019}
argue that these heavier tailed distributions fit empirical distributions
generated from common datasets better and also argue that Lévy motions would
converge to flatter minima than Brownian motions. In other words, they argue
that the approximation with a Brownian motion is neither accurate nor does it
have desireable properties.

In general this approach is very young and much more research is needed. But as
one obtains distributional properties of SGD this way, it is plausible that
this approach could result in better distributional statements and an
understanding of the behavior of SGD on non-convex loss functions. In particular
one could try to apply metastability results
\parencite[e.g.][]{bovierMetastabilityPotentialTheoreticApproach2015} to the
limiting stochastic processes to bound exit times from meta stable areas (local
minimas, saddle points, etc.).


}
	{

\chapter{Momentum}\label{chap: momentum}

To understand why the convergence rate is poor when the condition
number is high, we can visualize a high ratio of the lowest to the highest
eigenvalue as a narrow ravine. The gradient points in the direction of the
strongest descent, which points mostly to the opposite of the ravine and only slightly
along its length. This causes our iterate to bounce back and forth between
the walls of the ravine.
\begin{figure}[h]
	\centering
	\def\svgwidth{1\textwidth}
	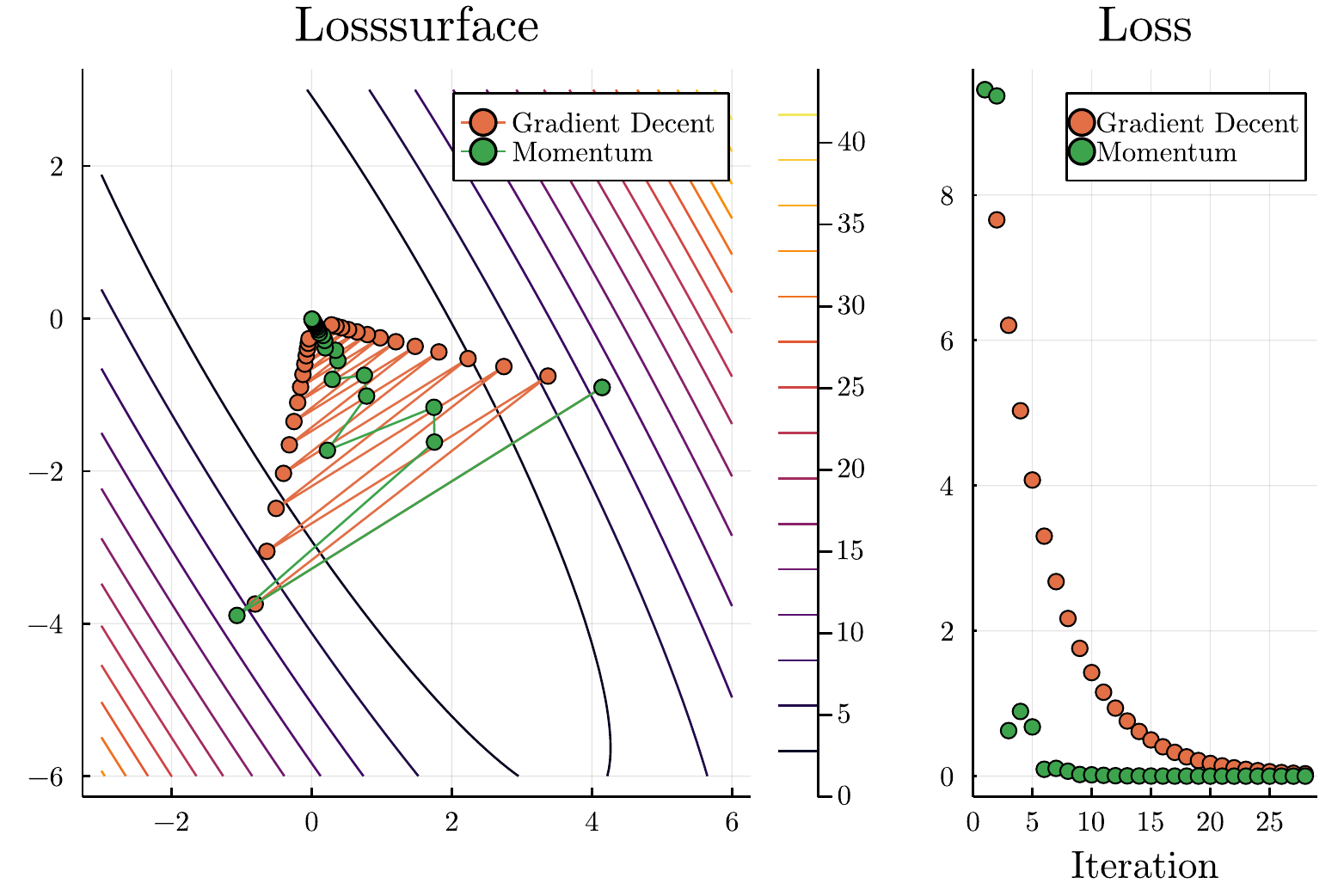
	\caption{Momentum reduces fluctuations and converges faster.}
	\label{fig: visualize bad conditioning}
\end{figure}

As a fix it seems appropriate to average the gradients in some sense to
cancel out the opposing jumps and go straight down the ravine. In other words
we want to build momentum. Now if we move according to the sum (integral) of
the gradients, then our velocity stops being equal to the gradient but instead
becomes the antiderivative of the gradient
\begin{align*}
	\dot{\weights} = -\int\nabla\Loss(\weights)dt.
\end{align*}
So instead of setting our gradient equal to the velocity like in (\ref{eq:
gradient flow}), we want to set the acceleration equal to our gradient
\begin{align*}
	\ddot{\weights} = -\nabla \Loss(\weights).
\end{align*}
But without friction, any potential (height) energy we have, would be conserved.
In other words: when we reach the bottom we have so much velocity, that we would
massively overshoot the minimum. So to dissipate this energy, we also
add a ``friction force'' inversely proportional to our current velocity
\begin{align}\label{eq: acceleration is gradient + friction}
	\ddot{\weights} = -\nabla \Loss(\weights) - \friction \dot{\weights}.
\end{align}
Note that we can only lose as much energy as we dissipate with friction, so in
some sense we want to make it as high as possible to quicken convergence. But
if it is too high, we are not moving fast enough on shallow slopes.

If our friction \(\friction\) is low enough\footnote{
	our momentum parameter \(\momCoeff\) (we will introduce later) high enough
}, then we will find (for heavy ball momentum), that the convergence rate in
every eigenspace is only determined by the friction, i.e.\ the rate we dissipate
energy with. If our friction is higher than some
critical value dependent on learning rate and eigenvalues of the Hessian, then
the eigenvalues matter again. And it will turn out that the optimal momentum
parameter is right at that critical point \parencite[see also ``critical
dampening'', e.g.][]{gohWhyMomentumReally2017}.

The standard way to discretize a second order ODE is to convert it into a first
order ODE
\begin{align*}
	\dot{y} := \begin{pmatrix}
		\dot{\weights}\\
		\ddot{\weights}
	\end{pmatrix}
	= \begin{pmatrix}
		\dot{\weights} \\
		-\nabla \Loss(\weights) - \friction \dot{\weights}
	\end{pmatrix}
	=: g\Big(\begin{pmatrix}
		\weights \\
		\dot{\weights}
	\end{pmatrix}\Big)
	= g(y).
\end{align*}
\begin{subequations}
This allows us to naively discretize our ODE with the Euler discretization
\begin{align}
	\weights_{n+1} &= \weights_n + \lr \momentum_n, \label{eq: naive momentum move}\\
	\momentum_{n+1} &= \momentum_n + \lr [-\nabla \Loss(\weights_n) - \friction \momentum_n]
	\label{eq: naive momentum}\\ \nonumber
	&= (1-\lr\friction)\momentum_n - \lr\nabla \Loss(\weights_n).
\end{align}
\end{subequations}
Here we use \(\momentum\) to denote the velocity \(\dot{\weights}\), or momentum
assuming unit mass.
If we plug the second equation (\ref{eq: naive momentum}) into the first
equation (\ref{eq: naive momentum move}) we get
\begin{align*}
	\weights_{n+1}
	&= \weights_n + \lr [(1-\lr\friction)
	\underbrace{\momentum_{n-1}}_{=\mathrlap{\frac{\weights_n-\weights_{n-1}}{\lr}}}
	- \lr\nabla \Loss(\weights_{n-1})].
\end{align*}
This means we are using gradient information from \(\weights_{n-1}\) to update
\(\weights_{n+1}\). If we instead use the most up to date information
\(\momentum_{n+1}\) instead of \(\momentum_n\) for the \(\weights_{n+1}\) update,
we get the well known ``heavy ball method'' (also known as momentum method) first proposed
by \textcite{polyakMethodsSpeedingConvergence1964} and wonderfully illustrated
by \textcite{gohWhyMomentumReally2017}.

\begin{definition}[Heavy Ball Method]
	\begin{subequations}
	\begin{align}
		\weights_{n+1} &= \weights_n + \lr \momentum_{n+1}, \label{eq: momentum move}\\
		\momentum_{n+1} &= (1-\lr\friction)\momentum_n - \lr\nabla \Loss(\weights_n).
		\label{eq: momentum}
	\end{align}
	\end{subequations}
	An equivalent formulation obtained by plugging (\ref{eq: momentum}) into
	(\ref{eq: momentum move}) but using (\ref{eq: momentum move}) for
	\(\momentum_n\) is
	\begin{align}\label{eq: flat momentum}
		\weights_{n+1}
		&= \weights_n
		+ \underbrace{(1-\lr\friction)}_{
			=:\momCoeff
		}(\weights_n - \weights_{n-1})
		- \underbrace{\lr^2}_{=:\lrSq}\nabla \Loss(\weights_n).
	\end{align}
	In particular we can set ``the momentum coefficient'' \(\momCoeff\) to zero to
	obtain gradient descent again. This is of course an artifact of our
	discretization since \(\lr\to0\) would never allow \(\momCoeff\) to be zero
	in the limit. But actual implementations often use this
	(\(\momCoeff,\lrSq\))-parametrization and thus treat gradient descent as a
	special case.
\end{definition}
Nesterov's momentum is even more aggressive: Instead of using \(p_{n+1}\) for
the \(\weights_{n+1}\) update, it considers the certain ``momentum move''
to calculate an intermediate position \(y_{n+1}\) (cf. Figure~\ref{fig: hinton nesterov momentum})
\begin{align*}
	\weights_{n+1}
	&= \underbrace{\weights_n + \overbrace{\lr [(1-\lr\friction)\momentum_n}^{\text{``momentum move''}}}_{=:y_{n+1}}
	- \lr\nabla \Loss(\weights_n)] \\
	&= y_{n+1} - \lrSq \nabla \Loss(\weights_n).
\end{align*}
It then uses that intermediate position to calculate the gradient instead of the
previous position \(\weights_n\).
\begin{figure}[h]
	\centering
	\includegraphics[scale=0.5]{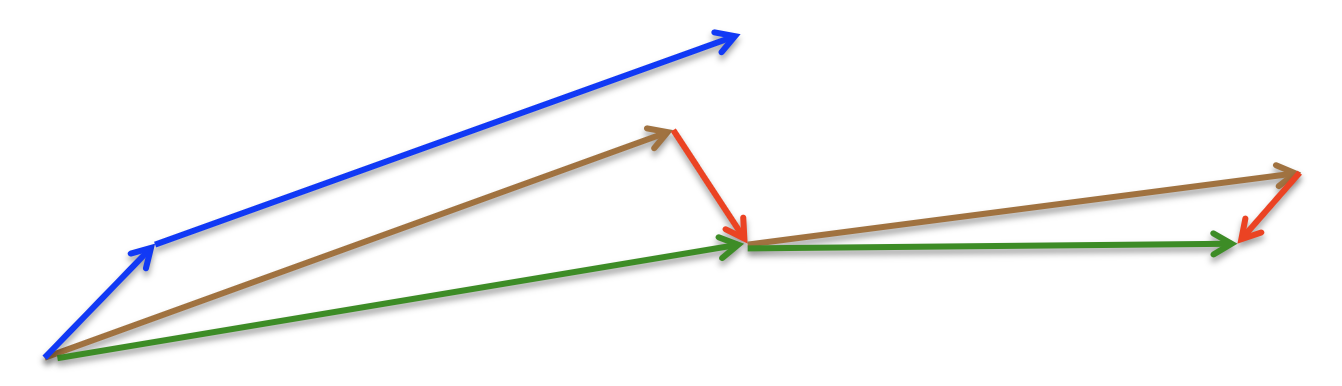}
	\caption{
		\parencite[lecture 6c]{hintonNeuralNetworksMachine2012} The blue arrow
		illustrates the movement of heavy ball momentum, i.e.\ first the small
		``gradient move'', then the large ``momentum move''. Meanwhile the brown
		arrow represents the ``momentum move'' and the red arrow the ``gradient
		move'' for Nesterov's momentum in green. The gradient turns around as we
		move towards the right, and we see that Nesterov's momentum adjusts
		quicker than heavy ball momentum, which is flung outwards.
	}
	\label{fig: hinton nesterov momentum}
\end{figure}
\begin{definition}[Nesterov's Momentum]
	\label{def: nesterov's momentum}
	\begin{subequations}
	\begin{align}
		\weights_{n+1} &= \weights_n + \lr \momentum_{n+1},
		\label{eq: nesterov momentum move}\\
		\momentum_{n+1}
		&= (1-\lr\friction)\momentum_n
		- \lr\nabla L[\weights_n + \lr(1-\lr\friction)\momentum_n]
		\label{eq: nesterov momentum}
		\\ \nonumber
		&= \momCoeff\momentum_n
		- \lr\nabla L[\underbrace{
			\weights_n + \momCoeff(\weights_n - \weights_{n-1})
		}_{= y_{n+1}}]
	\end{align}
	\end{subequations}
	Discarding the momentum term and solely using the intermediate position
	results in the simplified version
	\begin{subequations} \label{eq: nesterov intermediate position version}
	\begin{align}
		\weights_{n+1} &= y_{n+1} - \lrSq \nabla \Loss(y_{n+1}),\\
		y_{n+1}&= \weights_n + \momCoeff(\weights_n - \weights_{n-1}).
	\end{align}
	\end{subequations}
	Dropping the intermediate position as well, results in the analog to (\ref{eq:
	flat momentum})
	\begin{align}
		\weights_{n+1} &= \weights_n + \momCoeff(\weights_n - \weights_{n-1})
		- \lrSq \nabla \Loss(\weights_n + \momCoeff(\weights_n - \weights_{n-1}))
	\end{align}
\end{definition}
\begin{remark}\fxwarning{too informal?}
	This method also known as ``Nesterov's Accelerated Gradient'' (NAG) is said
	to date back to \textcite[in Russian]{nesterovMethodSolvingConvex1983}.
	Fortunately, \citeauthor{nesterovMethodSolvingConvex1983} wrote textbooks
	in English, \citetitle{nesterovLecturesConvexOptimization2018}
	(\citeyear{nesterovLecturesConvexOptimization2018}) being the most recent.
	Unfortunately Chapter 2.2 on optimal methods provides barely any
	intuition at all, and it is advisable to consult other sources (e.g.
	\textcite{dontlooWhatDifferenceMomentum2016}). Section~\ref{sec: nesterov
	momentum convergence} is an attempt to make these proofs somewhat intuitive.
 \end{remark}

\section{Heavy Ball Convergence}\label{sec: heavy ball convergence}

Following the arguments from Section~\ref{sec: visualize gd} in particular
(\ref{eq: hesse representation of gradient}) we can rewrite the momentum
method as
\begin{align*}
	\begin{pmatrix}
		\weights_n - \hat{\weights}_n \\
		\weights_{n+1} - \hat{\weights}_n
	\end{pmatrix}
	&=
	\begin{pmatrix}
		\weights_n - \hat{\weights}_n \\
		\weights_n - \hat{\weights}_n + \momCoeff (\weights_n - \weights_{n-1})
		- \lrSq \nabla\Loss^2(\weights_n)(\weights_n -\hat{\weights}_n)
	\end{pmatrix}\\
	&=
	\begin{pmatrix}
		0\identity_\dimension & \identity_\dimension \\
		-\momCoeff\identity_\dimension
		& (1+\momCoeff)\identity_\dimension -\lrSq \nabla^2\Loss(\weights_n)
	\end{pmatrix}
	\begin{pmatrix}
		\weights_{n-1} - \hat{\weights}_n \\
		\weights_n - \hat{\weights}_n
	\end{pmatrix}
\end{align*}
For readability we will now omit the identity matrix \(\identity\) from the
block matrices which are just a constant multiplied by an identity matrix.
Using the digitalization of the Hessian (\ref{eq: diagnalization of the
Hesse matrix}) again, we get
\begin{align*}
	&\begin{pmatrix}
		0 & 1 \\
		-\momCoeff & 1+\momCoeff -\lrSq \nabla^2\Loss(\weights_n)
	\end{pmatrix}\\
	&=
	\begin{pmatrix}
		V & 0 \\
		0 & V
	\end{pmatrix}	
	\begin{pmatrix}
		0 & 1 \\
		-\momCoeff &
		\diag(1+\momCoeff -\lrSq\hesseEV_1, \dots, 
		1+\momCoeff -\lrSq\hesseEV_\dimension)
	\end{pmatrix}
	\begin{pmatrix}
		V & 0 \\
		0 & V
	\end{pmatrix}^T.
\end{align*}
Reordering the eigenvalues from
\begin{align*}
	\begin{pmatrix}
		v_1 & \cdots & v_\dimension & 0 & \cdots & 0\\
		0 & \cdots & 0 & v_1 & \cdots & v_\dimension
	\end{pmatrix}
\end{align*}
to
\begin{align*}
	\begin{pmatrix}
		v_1 & 0 & \cdots & v_\dimension & 0 \\
		0 & v_1 & \cdots & 0 & v_\dimension
	\end{pmatrix},
\end{align*}
reorders the transformation matrix of the eigenspace into a block diagonal matrix
\begin{align}\label{eq: eigenspace momentum transformation}
	\momMatrix := \begin{pmatrix}
		0 & 1 \\
		-\momCoeff & 1+\momCoeff - \lrSq\hesseEV_1 & \\
		& & \ddots & \\
		& & & 0 & 1 \\
		& & & -\momCoeff & 1+\momCoeff - \lrSq\hesseEV_\dimension \\
	\end{pmatrix}
	=: \begin{pmatrix}
		\momMatrix_1 \\
		& \ddots\\
		&& \momMatrix_\dimension
	\end{pmatrix}
\end{align}
So in contrast to Section~\ref{sec: visualize gd}, we do not get a diagonal
matrix immediately. To achieve a similar eigenvalue analysis as in
Section~\ref{sec: visualize gd}, we first have to determine the eigenvalues
\(\momEV_{i1},\momEV_{i2}\) of each \(\momMatrix_i\) and then ensure that
\begin{align}
	\max_{i=1,\dots, \dimension} \max\{|\momEV_{i1}|,|\momEV_{i2}|\} < 1.
\end{align}
Using the p-q formula on the characteristic polynomial of \(\momMatrix_i\)
results in 
\begin{align*}
	\momEV_{i1/2}
	= \tfrac12 \left(
		1+\momCoeff-\lrSq\hesseEV_i
		\pm \sqrt{(1+\momCoeff-\lrSq\hesseEV_i)^2 - 4\momCoeff}
	\right).
\end{align*}
The analysis of these eigenvalues is very technical and can be found in the
appendix. We will only cover the result here.

\begin{figure}[h]
	\centering
	\def\svgwidth{1\textwidth}
	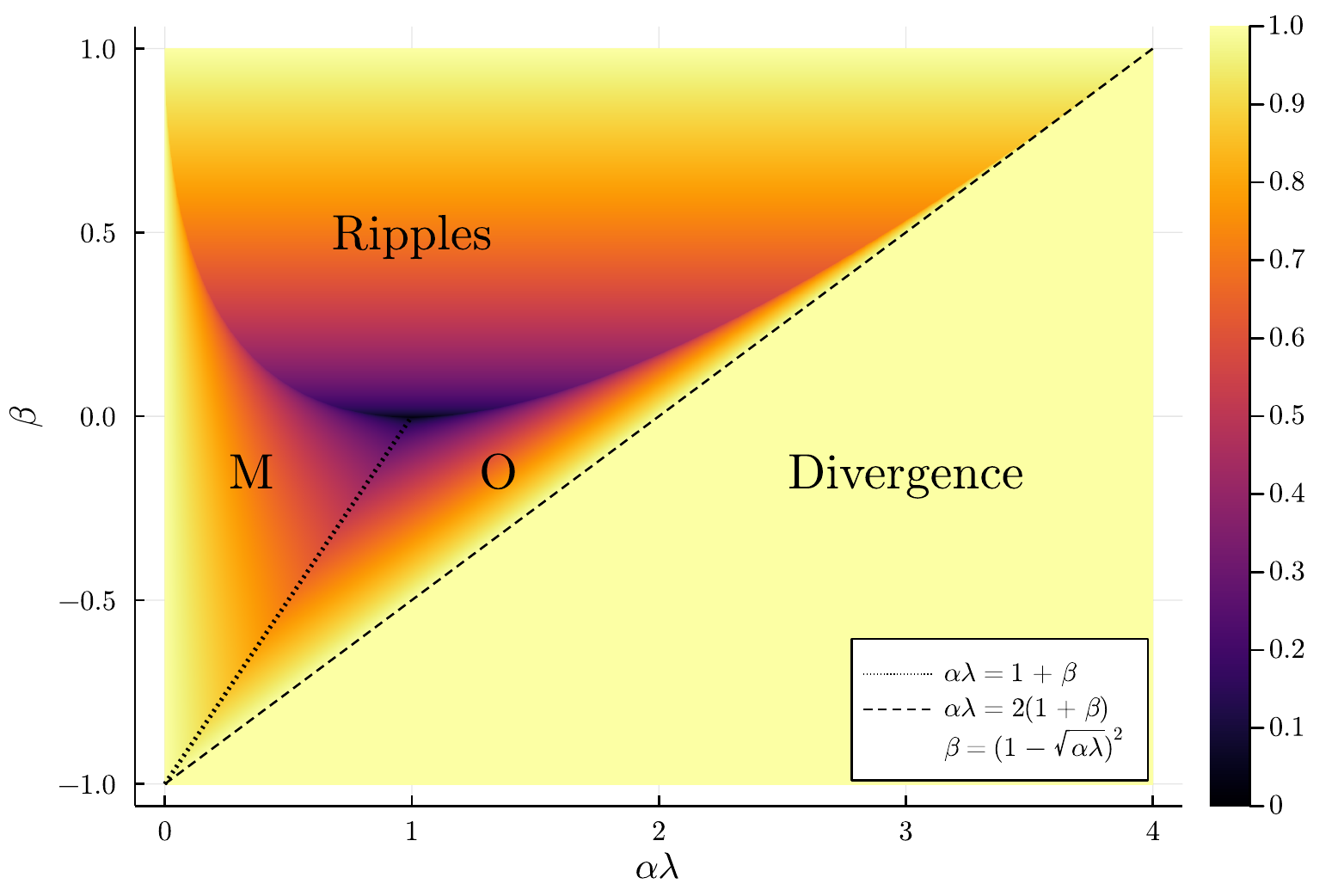
	\caption{
		This is a Heat plot of \(\max\{|\momEV_1|,|\momEV_2|\}\) for the Heavy
		Ball Momentum.  Strictly speaking the ``Monotonic'' (M) and
		``Oscillation'' (O) areas should not extend into negative \(\momCoeff\),
		since \(\momEV_{1/2}\) have opposite signs there. But it still describes
		the behavior of the dominating eigenvalue. The empty third label in the
		legend corresponds to the remarkably sharp border of the ``Ripples'' area.
		Drawing an actual line would have occluded the fact that it is visible by
		itself. Also recall that \(\momCoeff=0\) represents SGD without momentum.
	}
	\label{fig: annotated heavy ball rates}
\end{figure}

\begin{theorem}[\cite{qianMomentumTermGradient1999}]
	\label{thm: momentum - stable set of parameters}
	Let
	\begin{align*}
		\momEV_{1/2}
		= \tfrac12 \left(
			1+\momCoeff-\lrSq\hesseEV \pm \sqrt{(1+\momCoeff-\lrSq\hesseEV)^2 - 4\momCoeff}
		\right),
	\end{align*}
	then 
	\begin{enumerate}
		\item \(\max\{|\momEV_1|,|\momEV_2|\}<1\) if and only if
		\begin{align*}
			0<\lrSq\hesseEV < 2(1+\momCoeff) \qquad \text{and} \qquad |\momCoeff|<1.
		\end{align*}
		\item The complex case can be characterized by either
		\begin{align*}
			0<(1-\sqrt{\lrSq\hesseEV})^2 < \momCoeff < 1
		\end{align*}		
		or alternatively \(\momCoeff>0\) and
		\begin{align*}
			(1-\sqrt{\momCoeff})^2 < \lrSq\hesseEV < (1+\sqrt{\momCoeff})^2,
		\end{align*}
		for which we have \(|\momEV_1|=|\momEV_2|=\sqrt{\momCoeff}\).
		
		As complex
		eigenvalues imply a rotation, this can be viewed as rotating the distance
		to the minimum of the quadratic function into the momentum and back like 
		a pendulum where the friction causes it to eventually end up in the
		minimum. Looking at the distance to the minimum only, we will therefore
		observe a sinus wave (``Ripples'') as the potential energy is transferred
		into velocity and back. In this high momentum case convergence depends
		only on the amount of energy we dissipate through friction and the rate is
		thus equal to the square root of the momentum coefficient \(\beta\)
		regardless of the curvature of the eigenspace.
		\item In the real case we have \(\momEV_1>\momEV_2\) and
		\begin{align}
			\max\{|\momEV_1|, |\momEV_2|\} = \begin{cases}
				|\momEV_1|=\momEV_1 & \lrSq\hesseEV < 1+\momCoeff \\
				|\momEV_2|=-\momEV_2 & \lrSq\hesseEV \ge 1+\momCoeff.
			\end{cases}
		\end{align}
		Restricted to \(1>\momCoeff>0\) this results in two different	
		behaviors. For
		\begin{align*}
			0<\lrSq\hesseEV \le (1-\sqrt{\momCoeff})^2 < 1+\momCoeff,
		\end{align*}
		we have \(1>\momEV_1 > \momEV_2 > 0\) which results in a monotonic
		linear convergence (``Monotonic'' case). For
		\begin{align*}
			1+\momCoeff < (1+\sqrt{\momCoeff})^2\le \lrSq\hesseEV < 2(1+\momCoeff),
		\end{align*}
		on the other hand, we get \(-1 < \momEV_2 < \momEV_1 < 0\) which implies 
		an oscillating convergence (``Oscillation''). In contrast to the ``Ripples''
		we switch the side of the distance to the minimum in every iteration.
	\end{enumerate}
\end{theorem}
\begin{proof}
	Appendix Theorem~\ref{thm-appdx: momentum - stable set of parameters}
\end{proof}

In particular, (positive) momentum extends the available range for
learning rate and eigenvalue products \(\lrSq\hesseEV\), which results in
convergence. As can be seen very well in Figure~\ref{fig: annotated heavy
ball rates}, the wider attraction area allows for moving out the largest
eigenvalues to the right, with larger learning rates shifting the smaller
eigenvalues towards the right as well, into more acceptable convergence rates. If
the smallest and largest eigenvalues are far apart (i.e.\ the condition number
\(\condition\) is high), this should be very useful.

So how much do we gain? In other words, what are the optimal rates of
convergence? 

\begin{lemma}[Optimal Parameter Selection]
	\label{lem: optimal hb parameter selection}
	For a given learning rate \(\lr\), we have
	\begin{align*}
		\min_{\momCoeff}\max_{i=1,\dots, \dimension}\{|\momEV_{i1}|,|\momEV_{i2}|\}
		= \max\{(1-\sqrt{\lrSq\hesseEV_1}), (1-\sqrt{\lrSq\hesseEV_\dimension})\}
	\end{align*}
	with 
	\begin{align*}
		\momCoeff^*=\arg\min_{\momCoeff}\max_{i=1,\dots,\dimension}\{|\momEV_{i1}|,|\momEV_{i2}|\}
		= \max\{(1-\sqrt{\lrSq\hesseEV_1})^2, (1-\sqrt{\lrSq\hesseEV_\dimension})^2\}.
	\end{align*}
	Optimizing over the learning rate \(\lrSq\) as well, we get the optimal rate
	\begin{align*}
		\sqrt{\momCoeff_*} = 1-\frac{2}{1+\sqrt{\kappa}}.
	\end{align*}
	for parameters
	\begin{align*}
		\lr_* = \sqrt{\lrSq_*} = \frac{2}{\sqrt{\hesseEV_1}+\sqrt{\hesseEV_\dimension}},
		\qquad
		\momCoeff_*=\left(1-\frac{2}{1+\sqrt{\kappa}}\right)^2.
	\end{align*}
\end{lemma}

\begin{proof}
	Since the plot in Figure~\ref{fig: annotated heavy ball rates} looks quite
	symmetric, it is relatively intuitive that we only have to consider the largest and
	smallest eigenvalue. For now we can consider it as a lower bound:
	\begin{align}\label{eq: lower bound best heavy ball rates}
		\min_{\momCoeff}\max_{i=1,\dots,\dimension}\{|\momEV_{i1}|,|\momEV_{i2}|\}
		\ge \min_{\momCoeff}
		\max_{i=1,\dimension}\{|\momEV_{i1}|,|\momEV_{i2}|\}
	\end{align}
	Let us inspect the claim
	\begin{align}\label{eq: optimal heavy ball momentum - ev representation}
		\momCoeff^*=\arg\min_{\momCoeff}\max_{i=1,\dimension}\{|\momEV_{i1}|,|\momEV_{i2}|\}
		= \max\{(1-\sqrt{\lrSq\hesseEV_1})^2, (1-\sqrt{\lrSq\hesseEV_\dimension})^2\}.
	\end{align}
	If we do select \(\momCoeff\) like this, we are still (barely) in the complex
	case, therefore the rate of convergence is equal to \(\sqrt{\momCoeff}\),
	i.e.
	\begin{align*}
		\max\{(1-\sqrt{\lrSq\hesseEV_1}), (1-\sqrt{\lrSq\hesseEV_\dimension})\}.
	\end{align*}
	Now let us wiggle around \(\momCoeff\) a bit. If we would increase
	\(\momCoeff\), then we are obviously increasing our rate since it is then
	still equal to \(\sqrt{\momCoeff}\) as we just expanded the size of the
	complex case. Now the question is what happens if we decrease it. By our
	selection of \(\momCoeff\) either the smallest or largest eigenvalue sits
	right at the edge of the real case. If we can show that our absolute
	eigenvalue maximum is increasing if we move further into the real case
	(decrease \(\momCoeff\)), we would have proven that we can only become worse
	off doing so. This is annoyingly technical, so it is covered in
	Lemma~\ref{lem-appdx: just at the border of complex case is best beta} in the
	appendix.

	Now since both the largest and smallest eigenvalue are included in the complex
	case, we know that all eigenvalues in between must also be in the complex case
	which means that (\ref{eq: lower bound best heavy ball rates}) is an equality
	proving the first two claims.

	Now we have to optimize over the learning rate \(\lrSq\). Doing our balancing
	act again, only with \(\sqrt{\hesseEV_1},\sqrt{\hesseEV_\dimension}\) instead
	of \(\hesseEV_1,\hesseEV_\dimension\), we get the same result as in
	(\ref{eq: optimal SGD lr eigenvalue representation}), that is 
	\begin{align*}
		\lr_* = \sqrt{\lrSq_*} = \frac{2}{\sqrt{\hesseEV_1}+\sqrt{\hesseEV_\dimension}},
	\end{align*}
	as this makes both of them equally bad. We can plug this learning rate into
	either equation from inside the maximum of (\ref{eq: optimal heavy ball
	momentum - ev representation}) to get the remaining claims about the optimal
	momentum \(\momCoeff_*\) and optimal convergence rate \(\sqrt{\momCoeff_*}\).
\end{proof}

This convergence rate is in fact \emph{equal} to our lower bound  in
Theorem~\ref{thm: strong convexity complexity bound} (not even up to a
constant!). So on quadratic functions we cannot do any better without violating
our assumptions about possible optimizing algorithms (Assumption~ \ref{assmpt:
parameter in generalized linear hull of gradients}).

\section{From Spectral Radius to Operator Norm}

As the Hessian is symmetric the matrix,
\begin{align*}
	1-\lr \nabla^2\Loss
\end{align*}
is symmetric as well. And since we can find an orthonormal basis \((v_1,\dots,
v_\dimension)\) which diagonalizes the matrix, the operator norm is equal
to the spectral radius
\begin{align*}
	\rho(A)
	:= \max_{i=1,\dots,\dimension} |\hesseEV_i|,
\end{align*}
for eigenvalues \(\hesseEV_i\) of matrix \(A\), because of
\begin{align*}
	\|A\| = \sup_{\|x\| =1} \|Ax\|
	= \sup_{\|x\| =1} \sqrt{\langle Ax, Ax\rangle}
	= \sup_{\|x\| =1} \sqrt{\sum_{i=1}^\dimension \hesseEV_i^2 x_i^2},
	= \rho(A)
\end{align*}
where \(x=\sum_{k=1}^\dimension x_i v_i\). The eigenvalue analysis was therefore
sufficient for the gradient descent case
\begin{align*}
	\|\weights_n-\minimum\|
	\le \underbrace{\|1-\lr\nabla^2\Loss \|^n}_{
		=\rho(1-\lr\nabla^2\Loss)^n
	} \|\weights_0 - \minimum\|.
\end{align*}
In the momentum case on the other
hand, our transformation matrix \(R\) (\ref{eq: eigenspace momentum transformation})
is not symmetric. But the eigenvalue analysis is still useful due to Gelfand's Formula
\parencite{gelfandNormierteRinge1941} i.e.
\begin{align*}
	\rho(R) = \lim_{n\to\infty}\sqrt[n]{\|R^n\|}
\end{align*}
This formula is used by \textcite[p. 38, Lemma 1]{polyakIntroductionOptimization1987}
to argue, that we have for all \(\epsilon>0\)
\begin{align*}
	\|R^n\|\le \text{const } (\rho(R) + \epsilon)^n.
\end{align*}
This still provides exponential convergence to zero for any spectral radius
\(\rho(R)\) smaller one. Therefore the eigenvalue analysis is still useful.

\textcite{kozyakinAccuracyApproximationSpectral2009} already provides much tighter
bounds with more explicit constants, but they still depend on the unknown quantity
\begin{align*}
	\max_{i=1,\dots,\dimension}\frac{\|R_i^2\|}{\|R_i\|^2}.
\end{align*}
Where we can use their bounds for two dimensional matrices as the orthogonality
argument still applies between blocks, so we can consider the operator norm of
one block and take the maximum over all blocks
\begin{align*}
	\|R\| = \max_{i=1,\dots,\dimension}\|R_i\|
	= \max_{i=1,\dots,\dimension}\left\|
	\begin{pmatrix}
		0 & 1 \\
		-\momCoeff & 1+\momCoeff-\lrSq\hesseEV_i
	\end{pmatrix}
	\right\|.
\end{align*}

\subsection{Towards Explicit Bounds}

To build some intuition, notice how \(\lrSq\hesseEV_i=1\) and \(\momCoeff=1\)
lead to
\begin{align*}
	R_i
	=\begin{pmatrix}
		0 & 1 \\
		0 & 0
	\end{pmatrix}
\end{align*}
of which both eigenvalues are zero, but the operator norm is one as the vector
\((0, 1)^T\) does not change in length. Notice how we simply move the more
recent weight up into the older weight slot. If we would apply the matrix again
our weights would fully disappear. In other words: The problem are the cyclical
eigenspaces in the jordan normal form. More precisely, we can find some
basis change \(U\) such that
\begin{align*}
	R_i^n
	= \left(U 
		\begin{pmatrix}
			r_{i1} & \delta\\
			0 & r_{i2}
		\end{pmatrix}
	U^{-1}\right)^n
	= U 
	\begin{pmatrix}
		r_{i1} & \delta\\
		0 & r_{i2}
	\end{pmatrix}^n
	 U^{-1}
\end{align*}
where \(\delta\) is zero when \(r_{i1}\neq r_{i2}\) and can be zero or one if
they are equal. Now if they are both equal to some value \(r_i\), then we
can write
\begin{align*}
	\begin{pmatrix}
		r_i & \delta\\
		0 & r_i
	\end{pmatrix}^n
	&= \left(r_i\identity + 
	\begin{pmatrix}
		0 & \delta\\
		0 & 0 
	\end{pmatrix}\right)^n
	= \sum_{k=0}^n \binom{n}{k} r_i^k \begin{pmatrix}
		0 & \delta \\
		0 & 0
	\end{pmatrix}^{n-k}\\
	&= r_i^n\identity + n r_i^{n-1}\begin{pmatrix}
		0 & \delta \\
		0 & 0
	\end{pmatrix}
	= \begin{pmatrix}
		r_i^n & \delta n r_i^{n-1} \\
		0 & r_i^n
	\end{pmatrix}.
\end{align*}
So together with the case \(\delta=0\), we get
\begin{align*}
	R_i^n = U\begin{pmatrix}
		r_{i1}^n & \delta n r_{i1}^n\\
		0 & r_{i2}^n
	\end{pmatrix}U^{-1},
\end{align*}
Which results in
\begin{align*}
	\|R_i^n\|
	&\le \|U\|
	\left\| \begin{pmatrix}
		r_{i1}^n & \delta n r_{i1}^{n-1}\\
		0 & r_{i2}^n
	\end{pmatrix} \right\|
	\|U^{-1}\|\\
	&\le \|U\|\|U^{-1}\|
	\left(
	\left\| \begin{pmatrix}
		r_{i1}^n & 0\\
		0 & r_{i2}^n
	\end{pmatrix} \right\|
	+ \delta
	\left\| \begin{pmatrix}
		0 &  n r_{i1}^{n-1}\\
		0 & 0
	\end{pmatrix} \right\|
	\right)\\
	&= \|U\|\|U^{-1}\| (\rho(R_i)^n+ \delta n |r_{i1}|^{n-1})\\
	&= \|U\|\|U^{-1}\| (\rho(R_i)+ \delta n)\rho(R_i)^{n-1}\\
	&\in O(n\rho(R_i)^{n}) \subset O((\rho(R_i)+\epsilon)^n).
\end{align*}
Of course we do not have any bounds for \(\|U\|\|U^{-1}\|\) so this is still
not explicit enough. And the fact that \(\delta\) jumps from zero to one (once the
eigenvalues are not equal anymore) does not inspire confidence in the continuity
of \(U\). In fact \textcite[Section 7.1.5]{golubMatrixComputations2013} claim
that ``any matrix [...] that diagonalizes
\begin{align*}
	\begin{pmatrix}
		1+\epsilon & 1 \\
		0 & 1-\epsilon
	\end{pmatrix},
	\qquad \epsilon \ll 1,
\end{align*}
has a 2-norm condition of order \(1/\epsilon\)''. While this is not the same matrix,
this example shows that we have to be careful. Especially since we try to select
our learning rate \(\lrSq\) and momentum coefficient \(\momCoeff\) in such a way
that we are right on the complex border (Lemma~\ref{lem: optimal hb parameter
selection}). So we do in fact strive for the term inside the square root to be zero
\begin{align*}
	\momEV_{i1/2}
	= \tfrac12 \left(
		1+\momCoeff-\lrSq\hesseEV_i
		\pm \sqrt{(1+\momCoeff-\lrSq\hesseEV_i)^2 - 4\momCoeff}
	\right),
\end{align*}
and should therefore assume to be very close to \(r_{i1}=r_{i2}\).

Instead of the Jordan Normal form, we therefore want to use a numerically more
stable transformation, the ``Schur decomposition''.

\subsection{Schur Decomposition}

\begin{theorem}[Schur decomposition, {\cite[Theorem 7.1.3]{golubMatrixComputations2013}}]
	If \(A\in\complex^{\dimension\times \dimension}\), then there exists a unitary
	\(Q\in\complex^{\dimension\times \dimension}\) such that
	\begin{align*}
		Q^* A Q = \diag(\hesseEV_1, \dots, \hesseEV_\dimension) + N
	\end{align*}
	where \(N\) is strictly upper triangular and \(Q^*\) is the conjugate
	transpose of \(Q\).
\end{theorem}

The neat thing about this result is, that unitary matrices are isometries for the
euclidean norm.

\begin{lemma}[Unitary Matrices are Isometries]
	\label{lem: unitary matrices are isometries}
	Unitary matrices \(Q\) are isometries with regard to the euclidean norm
	\begin{align*}
		\|Qx\|_2 = \|x\|_2 \qquad \forall x\in\complex^\dimension
	\end{align*}	
	and also with regard to the induced operator norm
	\begin{align*}
		\|QA\| = \|A\| \quad \text{and} \quad \|AQ\| = \|A\|
		\qquad \forall A \in\complex^{\dimension\times\dimension}
	\end{align*}
\end{lemma}
\begin{proof}
	It is straightforward to show the isometry on euclidean norms
	\begin{align*}
		\|Qx\|^2 = \langle Qx, Qx\rangle = x^* Q^* Q x = x^* x = \|x\|^2.
	\end{align*}
	This then translates to them being isometries for the induced operator norm
	on matrices, since
	\begin{align*}
		\|QA\| = \sup_{\|x\|=1} \|QAx\| = \sup_{\|x\|=1} \|Ax\|
	\end{align*}
	and
	\begin{align*}
		\|AQ\|
		&= \sup_{x} \frac{\|AQx\|}{\|x\|} = \sup_{x} \frac{\|AQx\|}{\|Qx\|}
		= \sup_{y} \frac{\|Ay\|}{\|y\|} = \|A\|.
		\qedhere
	\end{align*}
\end{proof}

This lemma allows us to completely disregard the basis change matrices  in our
problem of estimating the operator norm of \(R\). And as we are only concerned
with \(2\times 2\) matrices, ``strictly upper triangular'' still means only one
entry. So we can deal with it very similarly as with the jordan normal form. The
only issue is, we need to actually know the entry in the upper right corner to
create an upper bound. But as it turns out, calculating the Schur decomposition
for \(2\times 2\) matrices is very simple. We just need one normalized
eigenvector and (any of the two) normalized vectors in its orthogonal
complement.

\begin{lemma}[Explicit Schur Decomposition for Specific Matrix]
	\label{lem: explicit schur decomposition}
	For a real matrix of the form
	\begin{align*}
		R = \begin{pmatrix}
			0 & 1 \\
			-\momCoeff & \xi
		\end{pmatrix},
	\end{align*}
	the unitary operator	
	\begin{align*}
		Q:=\frac{1}{\sqrt{1+|r_1|^2}} \begin{pmatrix}
		1 & \conjugate{r}_1 \\
		r_1 & -1
	\end{pmatrix}
	\end{align*}
	where
	\(
		r_{1/2}
		= \tfrac12 \left(
			\xi \pm \sqrt{\xi^2 - 4\momCoeff}
		\right)
	\)
	are the eigenvalues of \(R\), results in the Schur Decomposition
	\begin{align*}
		Q^* R Q
		&=\begin{pmatrix}
			r_1
			& -\frac{(1+\momCoeff)(1+\conjugate{r}_1^2)}{1+|r_1|^2} \\
			0 &  \frac{\momCoeff\conjugate{r}_1 + r_2}{1+|r_1|^2}
		\end{pmatrix}.
	\end{align*}
	For complex eigenvalues, more specifically \(4\momCoeff\ge\xi^2\), we have
	\begin{align*}
		Q^* R Q
		&=\begin{pmatrix}
			r_1
			& -(1+\conjugate{r}_1^2) \\
			0 & \conjugate{r}_1 
		\end{pmatrix}.
	\end{align*}
	Note that \(r_1\) and \(r_2\) could always be swapped to achieve a more favorable result.
\end{lemma}
\begin{proof}
	Since
	\begin{align*}
		\frac{1}{\sqrt{1+|r_1|^2}} \begin{pmatrix}
			1 \\
			r_1
		\end{pmatrix}
	\end{align*}
	is a normalized eigenvector for eigenvalue \(r_1\) of \(R\), there are
	only two options of opposite direction for an orthonormal second vector
	\begin{align*}
		\pm \frac{1}{\sqrt{1+|r_1|^2}} \begin{pmatrix}
			-\conjugate{r}_1 \\
			1
		\end{pmatrix},
	\end{align*}
	which results in \(Q\).
	
	Recall that any eigenvalue of \(R\), in particular \(r_1\) and
	\(\conjugate{r}_1\in\{r_1, r_2\}\), is a root of the characteristic polynomial
	\begin{align}\label{eq: property of eigenvalues of R}
		\det (r\identity - R) = r^2 - r\xi + \momCoeff = 0.
	\end{align}
	With this we can calculate the Schur decomposition of \(R\):
	\begin{align*}
		Q^* R Q
		&= \frac{1}{1+|r_1|^2} \begin{pmatrix}
			1 & \conjugate{r}_1 \\
			r_{1} & -1
		\end{pmatrix}
		\begin{pmatrix}
			0 & 1 \\
			-\momCoeff & \xi
		\end{pmatrix}
		\begin{pmatrix}
			1 & \conjugate{r}_1 \\
			r_1 & -1
		\end{pmatrix}\\
		&= \frac{1}{1+|r_{1}|^2} \begin{pmatrix}
			1 & \conjugate{r}_1 \\
			r_1 & -1
		\end{pmatrix}
		\begin{pmatrix}
			r_1 & -1 \\
			\smash{\underbrace{-\momCoeff + \xi r_1}_{
				=r_1^2 \ (\ref{eq: property of eigenvalues of R})
			}}
			& -\momCoeff \conjugate{r}_1 - \xi
		\end{pmatrix}
		\vphantom{\underbrace{\begin{pmatrix}1\\1\end{pmatrix}}_{=2}}
		\\
		&=\frac{1}{1+|r_{1}|^2}\begin{pmatrix}
			r_1(1+|r_1|^2)
			& -[\momCoeff \conjugate{r}_1^2  + (1+\xi \conjugate{r}_1)] \\
			0 &  \momCoeff\conjugate{r}_1 + \xi - r_1
		\end{pmatrix}\\
		&\lxeq{(\ref{eq: property of eigenvalues of R})}\begin{pmatrix}
			r_1
			& -\frac{\momCoeff \conjugate{r}_1^2 + 1+ (\conjugate{r}_1^2 + \momCoeff)}{1+|r_1|^2} \\
			0 &  \frac{\momCoeff\conjugate{r}_1 + \xi - r_1}{1+|r_1|^2}
		\end{pmatrix}\\
		&=\begin{pmatrix}
			r_1
			& -\frac{(1+\momCoeff)(1+\conjugate{r}_1^2)}{1+|r_1|^2} \\
			0 &  \frac{\momCoeff\conjugate{r}_1 + r_2}{1+|r_1|^2}
		\end{pmatrix},
	\end{align*}
	where we have used \(\xi - r_1 = r_2\) in the last equation. Now in the complex
	case (and the case \(4\momCoeff = \xi^2\)), we have
	\(|r_1|=|r_2|=\sqrt{\momCoeff}\) and \(\conjugate{r}_1 = r_2\) which results
	in
	\begin{align*}
		Q^* R Q
		&=\begin{pmatrix}
			r_1
			& -(1+\conjugate{r}_1^2) \\
			0 & \conjugate{r}_1 
		\end{pmatrix}.
		\qedhere
	\end{align*}
\end{proof}

\begin{theorem}[Upper Bound on Operator Norm]
	Assuming our momentum coefficient is large enough but smaller one, i.e.
	\begin{align*}
		1>\momCoeff
		\ge \max\{(1-\sqrt{\lrSq\hesseEV_1})^2, (1-\sqrt{\lrSq\hesseEV_\dimension})^2\},
	\end{align*}
	we are in the complex case, i.e.\ \(\rho(R) = \sqrt{\momCoeff}\) and
	\begin{align*}
		\|R^n\|
		&\le \rho(R)^{n-1}(\rho(R)+n[1+\rho(R)^2])\\
		&\le (2n+1)\rho(R)^{n-1}
	\end{align*}
	For all \(\epsilon>0\) we therefore have
	\begin{align*}
		\|R^n\|
		&\le \underbrace{\frac{3\exp(1)/\rho(R)}{\log(\rho(R)+\epsilon)-\log(\rho(R))}}_{
			\approx 3\exp(1)\epsilon^{-1}
		}
		[\rho(R)+\epsilon]^n
		\in O([\rho(R)+\epsilon]^n), \quad \forall n\ge 1
	\end{align*}
	using the first taylor approximation
	\begin{align*}
		\log(\rho(R)+\epsilon) \approx \log(\rho(R)) + \frac{\epsilon}{\log(\rho(R))}.
	\end{align*}
\end{theorem}
\begin{proof}
	First we apply our Lemmas to get
	\begin{align*}
		\|R^n\|
		&=\max_{i=1,\dots,\dimension}	\|R_i^n\|
		\xeq{\ref{lem: unitary matrices are isometries},\ \ref{lem: explicit schur decomposition}}
		\max_{i=1,\dots,\dimension}\left\|\begin{pmatrix}
			r_{i1} & -(1+\conjugate{r}_{i1}^2)\\
			0 & \conjugate{r}_{i1}
		\end{pmatrix}^n\right\|.
	\end{align*}
	Then we use Lemma~\ref{lem-appdx: absolute value inside operator norms} to
	turn the entries into their absolute values which makes the diagonal entries
	equal. As the diagonal now commutes with the nilpotent matrix we can apply the
	binomial theorem 
	\begin{align*}
		\|R^n\|
		&\le \max_{i=1,\dots,\dimension}
		\Bigg\|\left(
		\begin{pmatrix}
			|r_{i1}| & 0\\
			0 & |r_{i1}|
		\end{pmatrix}
		+ \begin{pmatrix}
			0 & |1+\conjugate{r}_{i1}^2|\\
			0 & 0
		\end{pmatrix}
		\right)^n\Bigg\|\\
		&\le \max_{i=1,\dots,\dimension}
		\Bigg\|
		\sum_{k=0}^n \binom{n}{k}
				|r_{i1}|^{n-k}
			\underbrace{
				\left(
				|1+\conjugate{r}_{i1}^2|
				\begin{pmatrix}
					0 & 1\\
					0 & 0
				\end{pmatrix}
				\right)^k
			}_{=0 \ \forall k>1}
		\Bigg\|\\
		&\le \max_{i=1,\dots,\dimension}
		|r_{i1}|^n \|\identity\| + n |r_{i1}|^{n-1}(|1|+|r_{i1}^2|)
		\underbrace{
			\Bigg\|
				\begin{pmatrix}
					0 & 1\\
					0 & 0
				\end{pmatrix}
			\Bigg\|
		}_{=1}
		\\
		&\le \max_{i=1,\dots,\dimension}
		\rho(R_i)^n + n \rho(R_i)^{n-1}(1+\rho(R_i)^2)\\
		&=	\rho(R)^{n-1}(\rho(R) + n [1+\rho(R)^2])\\
		&\le \rho(R)^{n-1}(1 + 2n)\\
		&\le 3n\rho(R)^{n-1} \qquad \forall n\ge 1.
	\end{align*}
	For the second claim let \(\epsilon>0\), then we have
	\begin{align*}
		\frac{\|R^n\|}{[\rho(R)+\epsilon]^n}
		&\le \frac{3n\rho(R)^{n-1}}{[\rho(R)+\epsilon]^n}\\
		&= \frac{3}{\rho(R)}\exp(\log(n) + n[\log(\rho(R))-\log(\rho(R)+\epsilon)])\\
		&\lxle{\max_n} \frac{3}{\rho(R)}\exp
		\left(\log\left(\frac{1}{[\log(\rho(R)+\epsilon)-\log(\rho(R))]}\right) + 1\right)\\
		&= \frac{3\exp(1)/\rho(R)}{\log(\rho(R)+\epsilon)-\log(\rho(R))}.
	\end{align*}
	Here we simply maximized over \(n\) to obtain the last inequality.
\end{proof}

\section{Why Consider Alternatives?}

While Nesterov's Momentum breaks Assumption~\ref{assmpt: parameter in
generalized linear hull of gradients} as it uses gradients at a
different point than the iterates, it does not break it fundamentally. The
assumption applies to the intermediate states \(y_n\) which are not too far
from our actual iterates and would thus allow us to make similar statements
about them with a bit of work.

And since heavy ball momentum can already achieve the optimal convergence
rate of our complexity bounds up to some epsilon due to issues with the operator
norm, why should we consider a different approach which can not be better than
an already optimal method?

\subsubsection{Weakest Link Contraction Factor}

The first issue you
might have noticed is: Since \emph{all eigenspaces} are in the complex case,
all of them have the worst contraction factor \(\sqrt{\momCoeff_*}\). In stochastic
gradient descent on the other hand the contraction factor is \(|1-\lr_*\hesseEV_i|\)
for the eigenspace \(i\), which can be considerably better than the worst
contraction factor 
\[|1-\lr_*\hesseEV_\dimension|=|1-\lr_*\hesseEV_1|.\]
To improve the convergence rates of the worst eigenspaces we have therefore
sacrificed the (possibly excellent) convergence rates of all other eigenspaces.

\subsubsection{No General Convergence at these Rates}

While
\textcite[pp. 65-67]{polyakIntroductionOptimization1987} used the eigenvalue
analysis of \(\nabla\Loss^2(\weights_*)\) to extend this rate of convergence
to a local area around \(\weights_*\), \textcite[pp. 78-79]{lessardAnalysisDesignOptimization2016}
gave an example of an \(\strongConvex{\lbound}{\ubound}\) function where
the heavy ball method with these ``optimal'' parameters is trapped in an attractive
cycle instead of converging. 
\begin{figure}[h]
	\centering
	\def\svgwidth{1\textwidth}
	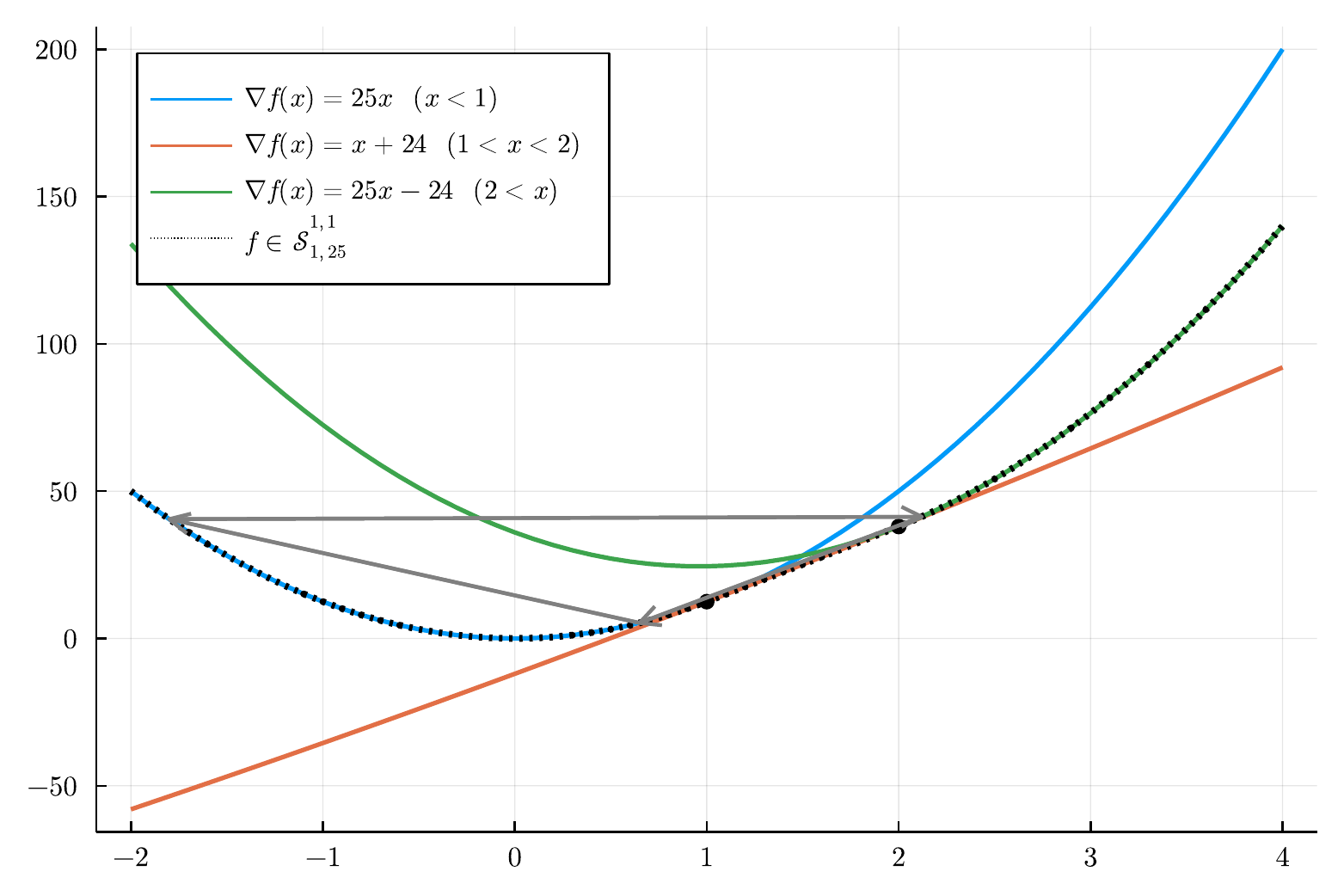
	\caption{
		\citeauthor{lessardAnalysisDesignOptimization2016}'s counterexample.
		Notice that the y-axis is much coarser than the x-axis and gradients are
		therefore much steeper than apparent. At point \(-1.8\) the gradient is
		quite steep and the next parameter is flung across the minimum to
		\(2.12\). Due to the less convex segment on this side, the gradient is
		not quite as steep here, so the iterate does not make it to the other
		side again which would then cause a bouncing back and forth to the
		minimum. Instead it stays on this side, which causes the next gradient
		at \(0.65\) together with the built momentum in this direction to move the
		iterate back out to our first point.
	}
	\label{fig: heavy ball counterexample}
\end{figure}

In their counterexample from Figure~\ref{fig: heavy ball counterexample}, we
know the parameters \(\lbound\) and \(\ubound\) and therefore \(\lrSq_*\)
and \(\momCoeff_*\). They use this to explicitly write out the momentum
recursion with \(\weights_{n+1}, \weights_{n}, \weights_{n-1}\). Assuming a
three-cycle this left them with just three linear equations which results in
the three-cycle displayed in the Figure. Since we start with zero momentum, we
just have to find a point to enter this cycle. Additional to solving this
initialization problem, \textcite[pp.
 93-94]{lessardAnalysisDesignOptimization2016} show that this cycle is
stable under noise (attractive).

\textcite{ghadimiGlobalConvergenceHeavyball2015} salvage some of this wreckage
by proving general convergence results with tighter restrictions on hyper
parameters. In particular for \(\Loss\in\strongConvex{\lbound}{\ubound}\) and
\begin{align*}
	\momCoeff\in[0,1),
	\qquad \lrSq\in \left(0, \frac{2(1-\momCoeff)}{\ubound+\lbound}\right)
\end{align*}
they can guarantee linear convergence with an explicit, but complicated
\fxnote{take a closer look?}{contraction factor}. Notice that this is
considerably weaker than the condition
\begin{align*}
	0<\lrSq\lbound \le \dots \le \lrSq\ubound < 2(1+\momCoeff)
\end{align*}
from Theorem~\ref{thm: momentum - stable set of parameters}. In particular
momentum stops increasing the selection space for learning rates. In fact it
shrinks it!
This rules out rules out our optimal parameter immediately. 

Similarly they prove that for \(\Loss\in\lipGradientSet{\ubound}\) and
\begin{align*}
	\momCoeff\in[0,1),
	\qquad \lrSq\in \left(0, \frac{(1-\momCoeff)}{\ubound}\right]
\end{align*}
that the running minimum of loss evaluation converges to the minimal loss,
with their best guarantee achieved for \(\lrSq =(1-\momCoeff)/\ubound\)
resulting in
\begin{align*}
	\min_{k=0,\dots,n}\Loss(\weights_k) -\Loss(\weights_*)
	\le \frac{1}{1-\momCoeff}\frac{\ubound\|\weights_0-\weights_*\|^2}{2(n+1)}
\end{align*}
which provides a tighter convergence bound than 
Theorem~\ref{thm: convex function GD loss upper bound} for gradient descent 
without momentum \(\momCoeff=0\), but is actually harmed by momentum. While
this bound might not be the best and further improvements could be made
\parencite[see e.g.][]{sunNonErgodicConvergenceAnalysis2019}, it appears that we
always end up with \(O(1/n)\) convergence rates in the general case, which is not
at all optimal, and we
therefore turn towards Nesterov's Momentum.

\section{Nesterov's Momentum Convergence}\label{sec: nesterov momentum convergence}

The Eigenvalue analysis of Nesterov's Momentum starts out very similar, and we
easily obtain the analog eigenspace transformation to (\ref{eq: eigenspace
momentum transformation})
\begin{align*}
	\momMatrix_i = \begin{pmatrix}
		0 & 1\\
		-\momCoeff(1-\lrSq\hesseEV_i) & (1+\momCoeff)(1-\lrSq\hesseEV_i).
	\end{pmatrix}	
\end{align*}
which curiously reduces to the heavy ball version if we use
\begin{align}\label{eq: tilde momentum}
	\tilde{\momCoeff}_i = \momCoeff(1-\lrSq\hesseEV_i)
\end{align}
We therefore immediately get from Theorem~\ref{thm: momentum - stable set of
parameters} that \(\max\{|\momEV_{i1}|,|\momEV_{i2}|\}<1\) if and only if
\begin{align*}
	0<\lrSq\hesseEV_i < 2(1+\tilde{\momCoeff}_i)
	\quad \text{and} \quad
	|\tilde{\momCoeff}_i|<1
\end{align*}
which after some transformations is itself equivalent to
\begin{align*}
	0<\lrSq\hesseEV_i < 1+\frac{1}{1+2\momCoeff}
	\quad \text{and} \quad
	|1-\lrSq\hesseEV_i|<\frac{1}{|\momCoeff|}.
\end{align*}
\begin{figure}[h]
	\centering
	\def\svgwidth{1\textwidth}
	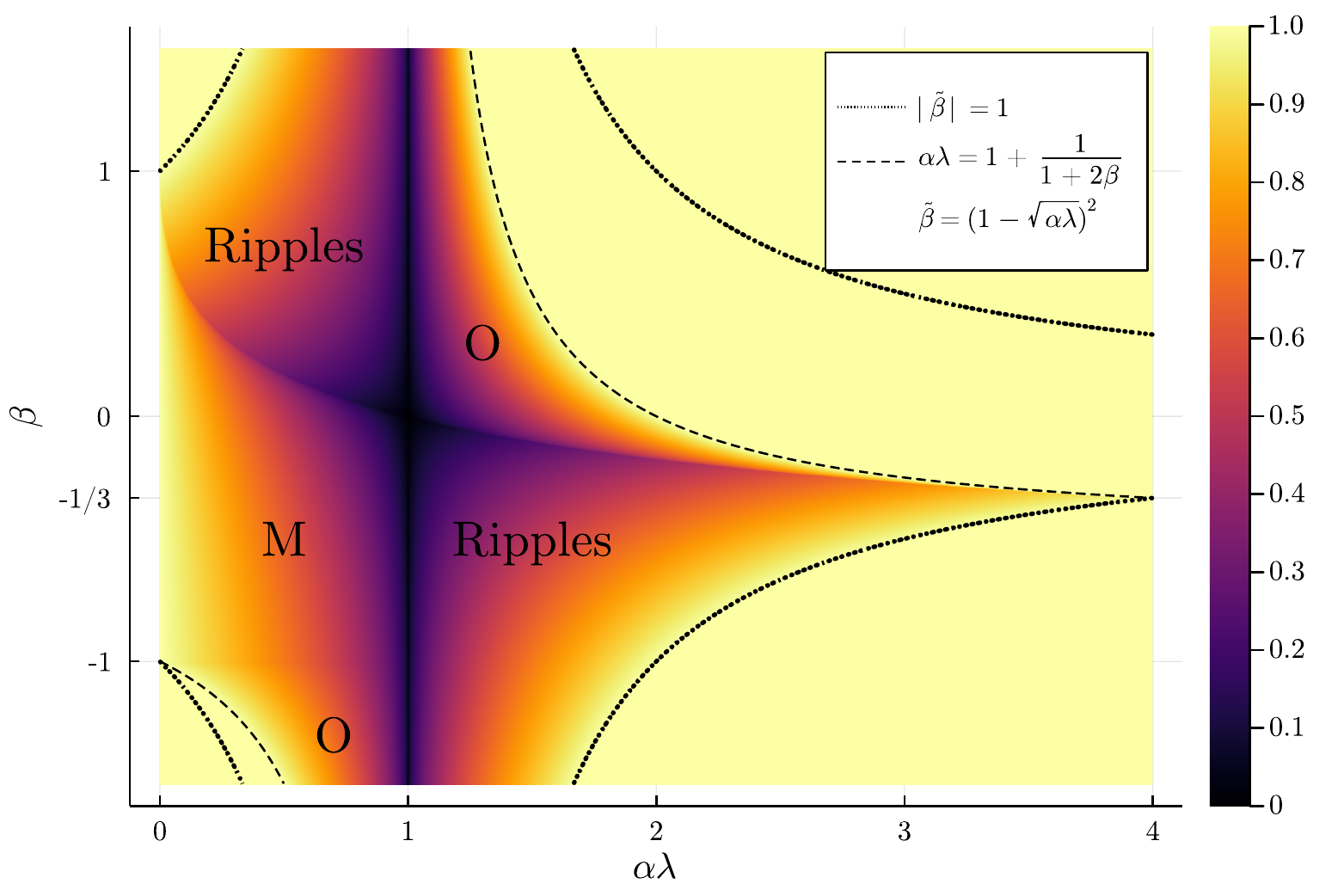
	\caption{
		Heat plot of \(\max\{|\momEV_1|,|\momEV_2|\}\) for the Nesterov Momentum.
		``M'' and ``O'' stand for the ``Monotonic'' and ``Oscillating'' types of 
		real eigenvalues. Note that the separating line of these cases
		\(\lrSq\hesseEV=1+\tilde{\momCoeff}=1+\momCoeff(1-\lrSq\hesseEV)\)
		reduces to \(\lrSq\hesseEV=1\) but since we divide by \(1+\momCoeff\)
		the sides flip at \(\momCoeff=-1\). Again the third label is omitted to
		keep the remarkably sharp border unobstructed.
	}
	\label{fig: annotated nesterov rates}
\end{figure}

In a sense Nesterov's Momentum provides every eigenspace with its own
heavy ball momentum \(\tilde{\momCoeff}_i\) which is a good sign if we want
to avoid the ``Weakest Link Contraction Factor'' issue we had in heavy ball
momentum. And if we are in the complex case the convergence rate is
\(\sqrt{\momCoeff(1-\lrSq\hesseEV_i)}\) which looks like the root of the
gradient descent convergence rate multiplied by a factor we could make smaller
than one.

As we have reduced our problem to the heavy ball problem again, we can also use
the same arguments to bound the operator norm. We will therefore ignore that
problem henceforth.

To ensure all eigenspaces are in the complex case, we first need to ensure that
all \(\tilde{\momCoeff}_i\) are positive which requires
\begin{align*}
	1-\lrSq\hesseEV_i >= 0
\end{align*}
So let us select \(\lrSq= 1/\hesseEV_\dimension\). We just need to ensure
that the smallest eigenvalue is in the complex case. For that, we set it to be
right on the border \(\tilde{\momCoeff}_1=(1-\sqrt{\lrSq\hesseEV_1})^2\),
which results in
\begin{align*}
	\momCoeff = \frac{(1-\sqrt{\lrSq\hesseEV_1})^2}{1-\lrSq\hesseEV_1}
	=\frac{1-\sqrt{\lrSq\hesseEV_1}}{1+\sqrt{\lrSq\hesseEV_1}}
	=\frac{1-\sqrt{1/\condition}}{1+\sqrt{1/\condition}} (< 1).
\end{align*}
Since we have
\begin{align*}
	\tilde{\momCoeff}_1 >= \dots >=\tilde{\momCoeff}_\dimension=0,
\end{align*}
the worst rate of convergence is exhibited by the first eigenspace and is equal to
\begin{align}\label{eq: nesterov convegence rate eigenvalue derivation}
	\sqrt{\tilde{\momCoeff}_1} = \sqrt{(1-\sqrt{\lrSq\hesseEV_1})^2}
	= 1-\sqrt{1/\condition} = 1- \frac{1}{\sqrt{\condition}}.
\end{align}
This rate of convergence is a tiny bit worse than the optimal heavy ball
convergence rate \(1-2/(1+\sqrt{\condition})\), which is also equal to the
complexity bound, but it is quite similar.

Now if you had a look at Figure~\ref{fig: annotated nesterov rates} you might
have noticed that for a fixed learning rate \(\lrSq\), the \(\lrSq\hesseEV_i\)
are \(\dimension\) vertical lines. We made sure
that the rightmost line was right on top of the black line and then moved up
our momentum term until all our rates where in the ripples area. From the
plot you might get the intuition that we can improve convergence if we
increase the learning rate to move the largest eigenvalue into the oscillating
range just like for \ref{eq: gradient descent}. And this is in fact not the optimal
selection of learning rates. \textcite{lessardAnalysisDesignOptimization2016}
claim\footnote{The proof is left to the reader.}, that optimal tuning results in
\begin{align*}
	\lrSq = \frac{4}{3\ubound + \lbound}
	\quad \text{and} \quad
	\momCoeff = \frac{\sqrt{3\condition+1}-2}{\sqrt{3\condition+1}+2}
\end{align*}
with convergence rate
\begin{align*}
	1-\frac{2}{\sqrt{3\condition+1}}.
\end{align*}
But as we have learned with heavy ball momentum, optimal convergence rates are 
of no use if we cannot generalize them.

\subsection{Estimating Sequences}

To prove convergence for \ref{eq: gradient descent}, we proved that gradient descent
minimizes a ``local'' upper bound (Lemma~\ref{lem: smallest upper bound}). Local
in the sense that we only use the gradient information at our current position
and discard all previous gradient information. To prove better convergence rates,
we need to utilize the previous gradients too.

Let us start with an upper bound of the form
\begin{align}\label{eq: Phi_0 (upper bound - kinda)}
	\Phi_0(\weights) = \underline{\Phi}_0 + \tfrac{\lbound_0}{2}\|\weights-z_0\|^2.
\end{align}
One example for such an upper bound is
\begin{align*}
	\weights \mapsto
	\Loss(y_0)
	+ \langle\nabla\Loss(y_0), \weights - y_0\rangle
	+ \tfrac{\ubound}{2}\|\weights -y_0\|^2.
\end{align*}
This is the local upper bound we minimized in Lemma~\ref{lem: smallest upper
bound} to obtain gradient descent. We can write this upper bound as in
(\ref{eq: Phi_0 (upper bound - kinda)}) with \(\lbound_0=\ubound\), by selecting
\begin{align*}
	z_0 = y_0 - \ubound^{-1}\nabla\Loss(y_0).
\end{align*}
Note this is the same technique as we used in (\ref{eq: newton minimum approx}).
Also keep in mind that our upper bound is a quadratic function, which is therefore
equal to its second taylor derivative.

Now since we started with an upper bound, it is trivial to find \(\weights_0\)
such that
\begin{align*}
	\Loss(\weights_0) \le \min_{\weights}\Phi_0(\weights),
\end{align*}
as we can simply select \(\weights_0=z_0\) to get
\begin{align*}
	\Loss(\weights_0) =\Loss(z_0) \le \Phi_0(z_0) = \min_{\weights}\Phi_0(\weights).
\end{align*}
The idea is now, to morph our upper bound \(\Phi_0\) via \(\Phi_n\) slowly into
a lower bound while keeping 
\begin{align}\label{eq: staying ahead of the estimates}
	\Loss(\weights_n) \le \min_{\weights}\Phi_n(\weights)=:\underline{\Phi}_n.
\end{align}
This would force \(\weights_n\) towards the minimum as it needs to stay smaller
than \(\Phi_n\) which changes into a lower bound of \(\Loss\). By controlling
the speed of morphing we will be able to bound the convergence speed. But we
cannot morph too fast otherwise we will not be able to stay ahead of it, i.e.
fulfill (\ref{eq: staying ahead of the estimates}). 

So how do we make this concrete?
The obvious lower bounds we want to utilize, are made up of all the gradients we
collect along the way
\begin{align*}
	\phi_n(\weights)
	:= \Loss(y_n) + \langle\nabla\Loss(y_n), \weights-y_n\rangle
	+ \tfrac{\lbound}{2}\|\weights-y_n\|^2.
\end{align*}
where this includes the non strongly convex case with \(\lbound=0\). So let
us consider
\begin{align*}
	\Phi_1 = \gamma \phi_1 + (1-\gamma)\Phi_0
\end{align*}
For \(\gamma=1\) we would have a lower bound. But if \(\nabla\Loss(y_1)\)
does not happen to be zero, the tangent point \(y_1\) is sitting on a slope and
is therefore not smaller than the minimum of \(\phi_1=\Phi_1\). And since
\(\phi_1\) is a lower bound of \(\Loss\), it will generally be impossible to
satisfy (\ref{eq: staying ahead of the estimates}) with \(\weights_1\) in that
case.
So the first take-away is, that we can generally not select \(\gamma=1\)
unless we are already in the minimum. But this also hints at another issue:
A convex combination of lower bounds
\begin{align*}
	\sum_{k=0}^n \hat{\gamma}_k \phi_k(\weights)
\end{align*}
is also a lower bound. But since they generally only have one tangent point
which is equal and are otherwise strictly lower, it is difficult to find
a \(\weights\) such that \(\Loss(\weights)\) is smaller than the minimum of this
lower bound. So the \(y_k\) already need to be equal to the minimum
of \(\Loss\) such that the \(\phi_k\) start out with no slope and their
minimum is also equal to \(\weights_*\). If we have
\begin{align*}
	\sum_{k=1}^n \hat{\gamma}_k \phi_k(\weights)
	+ \hat{\gamma}_0 \Phi_0(\weights)
\end{align*}
we get a bit more wiggle room in the form of \(\hat{\gamma}_0\). But if we want
to slowly remove our upper bound \(\Phi_0\), we run into a similar problem. Therefore we
also need to slowly remove the old
\(\phi_k\) upper bounds, constructed with \(y_k\) far away from the minimum
\(\weights_*\). This motivates why we might wish to select
\begin{align}\label{eq: estimating sequence recursion}
	\Phi_n(\weights) := \gamma_n \phi_n(\weights) + (1-\gamma_n)\Phi_{n-1}(\weights)
\end{align}
causing an exponential decay not only for the initial upper bound, but also for
our previous lower bounds. For
\begin{align*}
	\Gamma_k^n := \prod_{i=k}^n (1-\gamma_i)
\end{align*}
one can inductively convert the recursion (\ref{eq: estimating sequence
recursion}) into
\begin{align}
	\nonumber
	\Phi_n(\weights)
	&= \sum_{k=1}^n \Gamma_{k+1}^n\gamma_k \phi_k(\weights)
	+ \Gamma_1^n\Phi_0(\weights) \\
	\label{eq: estimating sequence explicit}
	&\le (1-\Gamma_1^n)\Loss(\weights) + \Gamma_1^n\Phi_0(\weights)
\end{align}

\begin{lemma}
	\label{lem: estimating sequence convergence speed}
	If the \(\weights_n\) satisfy (\ref{eq: staying ahead of the estimates}),
	then we have
	\begin{align*}
		\Loss(\weights_n) - \Loss(\weights_*)
		\le \Gamma_1^n (\Phi_0(\weights_*)-\Loss(\weights_*)).
	\end{align*}
	This motivates why we are interested in \(\Gamma_1^n\to 0\). A sufficient
	condition is
	\begin{align*}
		\sum_{k=0}^\infty \gamma_k = \infty
	\end{align*}
\end{lemma}
\begin{proof}
	Both statements are fairly easy to prove:
	\begin{align*}
		\Loss(\weights_n)
		\xle{(\ref{eq: staying ahead of the estimates})}
		\min_{\weights}\Phi_n(\weights)
		&\lxle{(\ref{eq: estimating sequence explicit})} \min_{\weights} \{
		(1-\Gamma_1^n) \Loss(\weights) + \Gamma_1^n\Phi_n(\weights) \}\\ 
		&\le (1-\Gamma_1^n) \Loss(\weights_*) + \Gamma_1^n\Phi_n(\weights_*)
	\end{align*}
	Subtracting \(\Loss(\weights_*)\) from both sides lets us move on to the
	second statement:
	\begin{align*}
		\Gamma_1^n &= \prod_{i=1}^n (1-\gamma_i)
		\le\prod_{i=1}^n \exp(-\gamma_i)
		= \exp\left(-\sum_{i=1}^n\gamma_i\right) \to 0
		\qedhere
	\end{align*}
\end{proof}

Now we just need to make sure that we can actually satisfy (\ref{eq: staying
ahead of the estimates}) while making \(\gamma_k\) as large as possible. For
gradient descent we just minimized the quadratic upper bound \(\Phi_0\). Now
we need to minimized the \(\Phi_n\). For that we want an explicit
representation of \(\Phi_n\).
Fortunately we are only adding up quadratic functions, therefore the function
\(\Phi_{n}\) is quadratic too, and can thus be written in the form
\begin{align}\label{eq: centered Phi representation}
	\Phi_{n}(\weights) = \underline{\Phi}_{n} + \tfrac{\lbound_{n}}{2}\|\weights - z_n\|^2,
\end{align}
where \(z_n\) is the minimum, and our convexity parameter \(\lbound_n\) morphs
from \(\lbound_0\) to \(\lbound\)
\begin{align}\label{eq: lbound recursion}
	\lbound_n =(1-\Gamma_1^n)\lbound + \Gamma_1^n \lbound_0
	= \gamma_n \lbound + (1-\gamma_n)\lbound_{n-1},
\end{align}
because \(\lbound_{n-1}\identity\) is by induction the Hessian of \(\Phi_n\).
Assuming (\ref{eq: staying ahead of the estimates}) holds for \(n-1\), we can
write our minimal value as
\begin{align*}
	\underline{\Phi}_n
	&= \Phi_n(y_n) - \tfrac{\lbound_n}{2}\|y_n - z_n\|^2 \\
	&\lxeq{(\ref{eq: estimating sequence recursion})}
	\gamma_n \underbrace{\phi_n(y_n)}_{
		=\Loss(y_n)
	}
	+ (1-\gamma_n)\underbrace{\Phi_{n-1}(y_n)}_{
		= \underbrace{\underline{\Phi}_{n-1}}_{
			\ge \Loss(\weights_{n-1})
			\mathrlap{\ge \Loss(y_n) + \langle \Loss(y_n), \weights_{n-1} - y_n\rangle}
		} +\mathrlap{\tfrac{\lbound_{n-1}}{2}\|y_n - z_{n-1}\|^2}
	} - \tfrac{\lbound_n}{2}\|y_n - z_n\|^2 \\
	&\ge \Loss(y_n) + \text{``junk''}
	\xge{?} \Loss(y_n) - \tfrac{1}{2\ubound}\|\nabla\Loss(y_n)\|^2
\end{align*}
Now you might recall that \(y_n\) are our intermediate ``information
collection points'', which then define the actual iterate. And
since Lemma~\ref{lem: smallest upper bound}, we know that we can achieve
\begin{align*}
	\Loss(y_n) - \tfrac{1}{2\ubound}\|\nabla\Loss(y_n)\|^2
	\ge \Loss(\weights_n)
\end{align*}
by selecting the learning rate \(\lrSq=1/\ubound\) in
\begin{align*}
	\weights_n = y_n - \lrSq \nabla\Loss(y_n).
\end{align*}
To be able to satisfy (\ref{eq: staying ahead of the estimates}), we therefore
only need to lower bound our ``junk''
\begin{align*}
	J := (1-\gamma_n)\left(
		\langle \Loss(y_n), \weights_{n-1} - y_n\rangle
		+ \tfrac{\lbound_{n-1}}{2}\|y_n - z_{n-1}\|^2
	\right)
	- \tfrac{\lbound_n}{2}\|y_n - z_n\|^2
\end{align*}
by \(- \tfrac{1}{2\ubound}\|\nabla\Loss(y_n)\|^2\). By having a closer look at
\(z_n\) and expressing it in terms of \(z_{n-1}\), the necessary terms fall out
and we obtain the following sufficient conditions
\begin{lemma}[Dumpster Dive]
	The conditions
	\begin{subequations}
	\begin{align}
		\label{eq: morphing speed equation}
		\ubound\gamma_n^2 \le \lbound_n \xeq{\text{def.}} \gamma_n \lbound +  (1-\gamma_n)\lbound_{n-1}\qquad \forall n\ge 1
	\end{align}
	or equivalently for \(\condition_n := \ubound/\lbound_n\) 
	\begin{align}
		\gamma_n^2 &\le \gamma_n \condition^{-1} + (1-\gamma_n)\condition_{n-1}^{-1}
		\qquad \forall n\ge 1
		\label{eq: morphing speed equation using condition}
	\end{align}
	\end{subequations}
	and
	\begin{align}
		y_n
		&= \frac{\lbound_n\weights_{n-1} + \lbound_{n-1}\gamma_n z_{n-1}}{\gamma_n\lbound + \lbound_{n-1}}
	\end{align}
	are sufficient for \(J \ge - \tfrac{1}{2\ubound}\|\nabla\Loss(y_n)\|^2\).
\end{lemma}
\begin{proof}
	See Appendix Lemma~\ref{lem-appendix: dumpster dive}.
\end{proof}
As we are allowed to select \(\gamma_n\) and \(y_n\) in iteration \(n\) we
can ensure that these equations are in fact fulfilled.

\subsection{Convergence Results}

Now while \(\Phi_0\) being an upper bound was nice to build intuition, we do
not actually need it to be an upper bound. The only thing we really need is
for equation (\ref{eq: staying ahead of the estimates}) and (\ref{eq: centered
Phi representation}) to hold for \(n=0\) so that we can do our induction. In
particular picking any \(\weights_0\) and selecting
\begin{align*}
	\Phi_0(\weights):= \Loss(\weights_0) + \tfrac{\lbound}{2}\|\weights_0 - \weights\|^2
\end{align*}
works as well. And in that case we have \(\lbound_n = \lbound\) for all \(n\),
which simplifies (\ref{eq: morphing speed equation using condition}) to
\begin{align}\label{eq: convergence rate for lbound_0=lbound}
	\gamma_{n+1} = \frac1{\sqrt{\condition}} \implies \Gamma_1^n = (1-\tfrac{1}{\sqrt{\condition}})^n.
\end{align}
With Lemma~\ref{lem: estimating sequence convergence speed} in mind this hints
at how we get our convergence rate (\ref{eq: nesterov convegence rate eigenvalue derivation})
we derived with eigenvalues back!
But since \(\lbound=0\) implies \(\condition=\infty\),
which would set our rate of convergence to one, we will continue
to entertain the general case with general \(\lbound_0\) summarized in
Algorithm~\ref{algo: general schema of optimal methods}. 
\begin{algorithm}
	\KwIn{starting point \(\weights_0\), and some \(\lbound_0>0\)}
	\(z_0\leftarrow \weights_0\);
	\(n\leftarrow 1\)\;
	\While{not converged}{
		\label{line: pick gamma}
		Select \(\gamma_n\in(0,1)\) such that
		\begin{algomathdisplay}
			\ubound\gamma_n^2 = \gamma_n\lbound+(1-\gamma_n)\lbound_{n-1}
		\end{algomathdisplay}
		\(\lbound_n\leftarrow \gamma_n\lbound+(1-\gamma_n)\lbound_{n-1}\)\;

		\(y_n\leftarrow 
		\displaystyle
		\frac{\lbound_n\weights_{n-1} + \lbound_{n-1}\gamma_n z_{n-1}}{\gamma_n\lbound + \lbound_{n-1}}\)\;
		\label{line: definition of y_n}

		Select \(\weights_n\) such that 
		\tcp*[f]{e.g. \(\weights_n=y_n-\tfrac1\ubound\nabla\Loss(y_n)\) (Lem~\ref{lem: smallest upper bound})}
		\begin{algomathdisplay}
			\Loss(\weights_n) \le \Loss(y_n) - \tfrac{1}{2\ubound}\|\nabla\Loss(y_n)\|^2
		\end{algomathdisplay}\label{line: selecting next position}
		\(z_n\leftarrow \tfrac{1}{\lbound_n}
			\left[\gamma_n\lbound y_n + (1-\gamma_n)\lbound_n z_{n-1}
			- \gamma_n\nabla\Loss(y_n)\right]
		\)
		\tcp*{(\ref{eq: center of Phi recursion})}

		\(n\leftarrow n+1\)\;
	}
	\caption{
		General Schema of Optimal Methods (by \citeauthor{nesterovLecturesConvexOptimization2018})%
		\label{algo: general schema of optimal methods}
	}
\end{algorithm}
Figure~\ref{fig: gamma path}
visualizes how larger \(\lbound_0\) than \(\lbound\) increase the \(\gamma_n\)
to be larger than zero and thus decrease the \(\Gamma_1^n\) below one. It also
explains why we should make (\ref{eq: morphing speed equation}) an equality.
\begin{figure}[h]
	\centering
	\def\svgwidth{1\textwidth}
	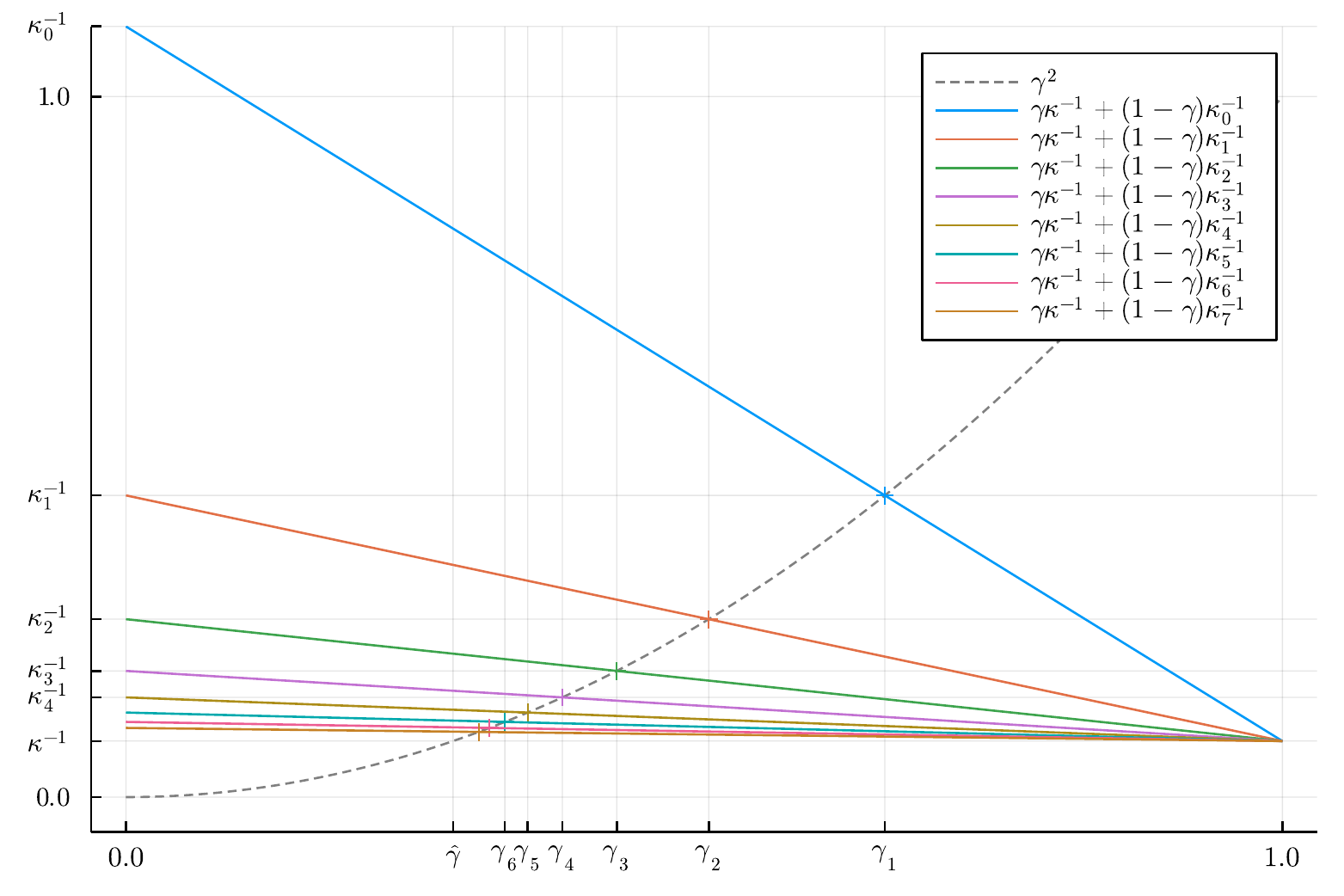
	\caption{
		Picking the largest possible \(\gamma_n\) which satisfy (\ref{eq:
		morphing speed equation using condition}) makes the inequality an
		equality \(\ubound \gamma_n^2 = \lbound_n\),
		and therefore \(\condition_n^{-1} = \gamma_n^2\).
		If we start out with \(\condition_0=\condition\), we therefore have
		\(\gamma_n=\hat{\gamma}=\sqrt{\condition^{-1}}\) for all \(n\).
		In all other cases \(\hat{\gamma}\) still acts as an attraction point.
		%
		For \(\condition^{-1}=0\) i.e.\ \(\lbound=0\)
		we see how it might be helpful to select a different \(\condition_0\)
		even though increasing \(\lbound_0\) (i.e.\ \(\condition_0^{-1}\))
		increases \(\Phi_0\) which is part of the bound on convergence.
	}
	\label{fig: gamma path}
\end{figure}

Before we address convergence rates for general \(\lbound_0\) let us take a
closer look at the strange looking \(y_{n+1}\), which were much more simple in
Definition~\ref{def: nesterov's momentum} and try to get rid of 
additional baggage like \(z_n\) and \(\lbound_n\).

\begin{lemma}\label{lem: streamline general schema of optimal methods}
	If we use the recursion
	\begin{align}\label{eq-lem-assmpt: weigth recursion}
		\weights_n = y_n - \tfrac1\ubound\Loss(y_n)
	\end{align}
	in Algorithm~\ref{algo: general schema of optimal methods} line~\ref{line:
	selecting next position}, then we can simplify \(z_n\) and \(y_n\)
	to
	\begin{align}
		z_n
		&= \weights_{n-1} + \tfrac1{\gamma_n}(\weights_n-\weights_{n-1})
		\qquad \forall n\ge 1\\
		y_{n+1}
		&=\weights_n + \momCoeff_n(\weights_n-\weights_{n-1}) \qquad \forall n\ge 0
	\end{align}
	for \(\weights_{-1}:=\weights_0=z_0\) and
	\begin{align}
		\momCoeff_n
		= \frac{\lbound_n \gamma_{n+1}(1-\gamma_n)}{\gamma_n(\gamma_{n+1}\lbound +\lbound_n)}
		= \frac{\gamma_n(1-\gamma_n)}{\gamma_n^2 + \gamma_{n+1}}
	\end{align}
\end{lemma}
\begin{proof}
	See Appendix Lemma~\ref{lem-appendix: streamline general schema of optimal methods}
\end{proof}

With the streamlined recursion from Lemma~\ref{lem: streamline general schema of optimal methods}
we only need \(\lbound_n\) for the recursion of \(\gamma_{n+1}\) itself at this
point. So we can now use the equivalent recursion (\ref{eq: morphing speed
equation using condition}) with \(\condition_n:=\ubound/\lbound_n\).
By defining \(\gamma_0:=1/\sqrt{\condition_0}\) we can ensure
\begin{align*}
	\gamma_n^2 = \condition_n^{-1} \qquad  \text{for all } n\ge0.
\end{align*}
This in turn allows us to get rid of the \(\condition_n\). Applying the p-q
formula and discarding the negative option, results in the explicit recursion
\begin{align*}
	\gamma_{n+1}
	= \tfrac12 \left[
		-(\gamma_n^2 - \condition^{-1})
		+ \sqrt{(\gamma_n^2 - \condition^{-1})^2 + 4 \gamma_n^2}
	\right].
\end{align*}

Since we get a fixed \(\hat{\gamma}=1/\sqrt{\condition}\) for
\(\lbound_0=\lbound\), the momentum coefficient is fixed in this case and equal to 
\begin{align*}
	\hat{\momCoeff} = \frac{1-1/\sqrt{\condition}}{1+1/\sqrt{\condition}}
\end{align*}
which we recognize from our eigenvalue analysis.

\begin{algorithm}
	\KwIn{starting point \(\weights_0\), and some \(\condition_0>0\)}
	\(\gamma_0 \leftarrow 1/\sqrt{\condition_0}\)
	\tcp*{
		\(\gamma_n^2=\condition_n^{-1}\) in (\ref{eq: morphing speed equation
		using condition}) makes \(\condition_n\) obsolete
	}
	\(\weights_{-1} \leftarrow \weights_0\)\;
	\(n\leftarrow 0\)\;
	\While{not converged}{
		\(\gamma_{n+1} \leftarrow \tfrac12 \left[
		(\condition^{-1} - \gamma_n^2)
		+ \sqrt{(\condition^{-1} - \gamma_n^2)^2 + 4 \gamma_n^2}
		\right]\)\;
		\tcp{\(=1/\sqrt{\condition}\) if \(\condition_0=\condition\)}

		\(\momCoeff_n \leftarrow
		\displaystyle
		\frac{\gamma_n(1-\gamma_n)}{\gamma_{n+1} + \gamma_n^2}\)
		\tcp*{\(= \frac{1-1/\sqrt{\condition}}{1+1/\sqrt{\condition}}\) if \(\condition_0=\condition\)}

		\(y_{n+1} \leftarrow \weights_n + \momCoeff_n(\weights_n-\weights_{n-1})\)
		\tcp*{\(\beta_0\) does not matter!}

		\(\weights_{n+1} \leftarrow y_{n+1} - \tfrac{1}{\ubound}\nabla\Loss(y_{n+1})\)\;

		\(n\leftarrow n+1\)\;
	}
	\caption{Dynamic Nesterov Momentum\label{algo: dynamic nesterov momentum}}
\end{algorithm}

What is left to do is an analysis of convergence, in particular for the case
where \(\condition_0\neq\condition\), which is especially applicable for \(\lbound=0\).
Figure~\ref{fig: gamma path} provides some intuition why it might also be
useful to have \(\condition_0>\condition\) when \(\lbound\neq0\).

\begin{theorem}\label{thm: nesterov momentum convergence rates}
	Let \(\Loss \in \strongConvex{\lbound}{\ubound}\) (includes
	\(\Loss\in\lipGradientSet{\ubound}\) with \(\lbound=0\)). For starting point
	\(\weights_0\) and some \(\condition_0\) (equivalently \(\lbound_0\) or
	\(\gamma_0\)) let \(\weights_n\) be selected via Algorithm~\ref{algo: general schema of optimal methods}
	or \ref{algo: dynamic nesterov momentum}, then for \(\Gamma^n := \Gamma_1^n =
	\prod_{k=1}^n(1-\gamma_k)\) we have
	\begin{align*}
		\Loss(\weights_n) - \Loss(\weights_*)
		&\le \Gamma^n\left[
			\Loss(\weights_0) -\Loss(\weights_*)
			+ \tfrac{\lbound_0}2\|\weights_0-\weights_*\|^2
		\right] \\
		&\le \Gamma^n \tfrac{\ubound + \lbound_0}{2}\|\weights_0-\weights_*\|^2
	\end{align*}
	where we have for \(\condition_0^{-1}\in (\condition^{-1}, 3 + \condition^{-1}]\)
	\begin{align*}
		\Gamma^n
		\le \frac{4\condition^{-1}}{(\condition_0^{-1}-\condition^{-1})\left[
			\exp\left(\frac{n+1}{2\sqrt{\condition}}\right)
			-\exp\left(-\frac{n+1}{2\sqrt{\condition}}\right)
		\right]^2}
		\le \frac{4}{(\condition_0^{-1}-\condition^{-1})(n+1)^2}
	\end{align*}
	and for \(\condition_0 = \condition\)
	\begin{align*}
		\Gamma^n = (1-\sqrt{\condition^{-1}})^n
		\le \exp\left(-\frac{n}{\sqrt{\condition}}\right)
	\end{align*}
\end{theorem}
\begin{proof}
	We can use
	\begin{align*}
		\Phi_0(\weights) := \Loss(\weights_0) + \tfrac{\lbound_0}{2}\|\weights -\weights_0\|^2
	\end{align*}
	to ensure \(\weights_0=z_0\) for Lemma~\ref{lem: streamline general
	schema of optimal methods} needed for Algorithm~\ref{algo: dynamic nesterov
	momentum}. We can then apply Lemma~\ref{lem: estimating sequence
	convergence speed} to obtain
	\begin{align*}
		\Loss(\weights_n)-\Loss(\weights_*)
		&\le \Gamma_1^n(\Phi_0(\weights_*)-\Loss(\weights_*))\\
		&= \Gamma^n\left[
			\Loss(\weights_0) -\Loss(\weights_*)
			+ \tfrac{\lbound_0}2\|\weights_0-\weights_*\|^2
		\right].
	\end{align*}	
	We can bound the original loss with the weight difference
	(Lemma~\ref{lem: Lipschitz Gradient implies taylor inequality})
	\begin{align*}
		\Loss(\weights_0)-\Loss(\weights_*)\le \tfrac{\ubound}{2}\|\weights_0-\weights\|^2
	\end{align*}
	to get the second upper bound.
	Now to the upper bound on \(\Gamma^n\). We have already covered the case
	\(\condition_0=\condition\) in (\ref{eq: convergence rate for lbound_0=lbound}),
	(with \(1+x\le\exp(x)\) in mind). The general case is covered in the
	Lemma~\ref{lem-appendix: convergence rate bounds for estimating sequences}
\end{proof}

\section{Hyperparameter Selection}

Knowing the condition number, one might wish to optimize the convergence rate
of Theorem~\ref{thm: nesterov momentum convergence rates} over \(\condition_0\).
For \(\condition_0^{-1}\in(\condition^{-1}, 3+\condition^{-1}]\) this entails
minimizing
\begin{align}\label{eq: convergence rate optimization}
	\frac{\ubound+\lbound_0}{\condition_0^{-1}-\condition^{-1}}
	= \ubound \frac{1+\condition_0^{-1}}{\condition_0^{-1}-\condition^{-1}}
	= \ubound \frac{1+\frac{1}{\condition_0^{-1}}}{1-\frac{\condition^{-1}}{\condition_0^{-1}}}
\end{align}
Since both, the enumerator decreases with \(\condition_0^{-1}\), and the
denominator increases with \(\condition_0^{-1}\), we want to select
\(\condition_0^{-1}\) as large as possible.

Looking at Figure~\ref{fig: gamma path} this should make us suspicious. There
is nothing special about \(\condition_0^{-1}=3+\condition^{-1}\) when it
comes to \(\Gamma^n\). The \(\gamma_n\) are still increasing afterwards. And
from (\ref{eq: convergence rate optimization}) it seems that the trade-off of
higher \(\lbound_0\) seems to be worth it. So maybe we should have another go
at bounding \(\Gamma^n\). But instead of diving head first into even more
technical calculations, we could also consider the fact that
\begin{align*}
	\Gamma^n
	= \prod_{k=1}^n (1-\gamma_k)
	= (1-\gamma_1)\Gamma_2^n 
\end{align*}
and simply use Lemma~\ref{lem-appendix: convergence rate bounds for estimating sequences}
with \(\tilde{\gamma}_k = \gamma_{k+1}\) to bound \(\Gamma_2^n=\tilde{\Gamma}^{n-1}\)
using the fact that
\begin{align*}
	\tilde{\condition}_0^{-1} = \condition_1^{-1} = \gamma_1^2 \le 1 \le 3+\condition^{-1}.
\end{align*}
With this trick we get\footnote{Once could potentially iterate on this trick to obtain
tighter bounds on \(\Gamma^n\) since we are using our imprecise upper bound only on the
remaining product after factoring out parts of it.}
\begin{align*}
	\Gamma^n
	&\le \frac{4\condition^{-1}(1-\gamma_1)}{(\gamma_1^2-\condition^{-1})\left[
		\exp\left(\frac{n}{2\sqrt{\condition}}\right)
		-\exp\left(-\frac{n}{2\sqrt{\condition}}\right)
	\right]^2}
	\le \frac{4(1-\gamma_1)}{(\gamma_1^2-\condition^{-1})n^2}.
\end{align*}
With Theorem~\ref{thm: nesterov momentum convergence rates} the relevant
constant to optimize is
\begin{align*}
	\frac{4(1-\gamma_1)}{\gamma_1^2-\condition^{-1}}\frac{\ubound+\lbound_0}{2}
	= \frac{2(1-\gamma_1)}{\gamma_1^2-\condition^{-1}}(1+\condition_0^{-1})\ubound
\end{align*}
and we can use the \(\gamma_n\) recursion (\ref{eq: morphing speed equation using condition})
to rewrite \(\condition_0^{-1}\) as
\begin{align*}
	\condition_0^{-1} = \frac{\gamma_1^2 - \gamma_1\condition^{-1}}{1-\gamma_1}.
\end{align*}
Using this representation we want to minimize
\begin{align*}
	c(\gamma_1)
	:=2\ubound\frac{(1-\gamma_1) + (\gamma_1^2 - \gamma_1\condition^{-1})}{\gamma_1^2 -\condition^{-1}}
	\qquad \gamma_1\in(\sqrt{\condition^{-1}}, 1).
\end{align*}
By taking the derivative of \(c\) one can show that it is monotonously
decreasing (Lemma~\ref{lem-appendix: derivative loss difference constant in
gamma}). Taking \(\gamma_1=1\) would imply an infinite
\(\kappa_0^{-1}\) (and \(\lbound_0\)) and therefore seems to prevent us from
doing so at first glance, but all relevant things like the loss difference upper
bound
\begin{align}\label{eq: gamma_1 dependent loss difference}
	\Loss(\weights_n)-\Loss(\weights_*)
	\le c(\gamma_1)
	\frac{\condition^{-1}\|\weights_0-\weights_*\|^2}{\left[
		\exp\left(\frac{n}{2\sqrt{\condition}}\right)
		-\exp\left(-\frac{n}{2\sqrt{\condition}}\right)
	\right]^2}
	\le c(\gamma_1)\frac{\|\weights_0-\weights_*\|}{n^2},
\end{align}
the momentum \(\momCoeff_n\) (except for \(\momCoeff_0\) which does not
matter) and our weight iterates \(\weights_n\) are continuous in \(\gamma_1\) at
one. The limit
\begin{align*}
	c(1)=2\ubound
\end{align*}
is therefore reasonable.

\begin{algorithm}
	\KwIn{
		starting point \(\weights_0\) (condition number \(\condition\),
		lipschitz constant \(\ubound\))
	}
	\(y_1 \leftarrow \weights_0\)\;
	\(\weights_1 \leftarrow y_1 - \tfrac{1}{\ubound}\nabla\Loss(y_1)\)\;

	\(\gamma_1 \leftarrow 1\); \(n\leftarrow 1\)\;
	\While{not converged}{
		\(\gamma_{n+1} \leftarrow \tfrac12 \left[
		(\condition^{-1} - \gamma_n^2)
		+ \sqrt{(\condition^{-1} - \gamma_n^2)^2 + 4 \gamma_n^2}
		\right]\)\;

		\(\momCoeff_n \leftarrow
		\displaystyle
		\frac{\gamma_n(1-\gamma_n)}{\gamma_{n+1} + \gamma_n^2}\)\;

		\(y_{n+1} \leftarrow \weights_n + \momCoeff_n(\weights_n-\weights_{n-1})\)\;

		\(\weights_{n+1} \leftarrow y_{n+1} - \tfrac{1}{\ubound}\nabla\Loss(y_{n+1})\)\;

		\(n\leftarrow n+1\)\;
	}
	\caption{Optimal Nesterov Momentum\label{algo: optimal nesterov momentum}}
\end{algorithm}
\begin{corollary}
	For \(\Loss \in \strongConvex{\lbound}{\ubound}\) (includes
	\(\Loss\in\lipGradientSet{\ubound}\) with \(\lbound=0\)), \(\weights_n\)
	selected by Algorithm~\ref{algo: optimal nesterov momentum}
	we get
	\begin{align*}
		\Loss(\weights_n)-\Loss(\weights_*)
		\le
		\frac{2\lbound\|\weights_0-\weights_*\|^2}{\left[
			\exp\left(\frac{n}{2\sqrt{\condition}}\right)
			-\exp\left(-\frac{n}{2\sqrt{\condition}}\right)
		\right]^2}
		\le\frac{2\ubound\|\weights_0-\weights_*\|}{n^2},
	\end{align*}
	where the term in the middle should be ignored for \(\lbound=0\).
\end{corollary}
\begin{proof}
	Plugging \(c(1)=2\ubound\) into (\ref{eq: gamma_1 dependent loss difference}).
\end{proof}

\subsection{Practical Hyperparameter Selection}

While Algorithm~\ref{algo: optimal nesterov momentum} eliminates another
hyperparameter, we still need to somehow know the Lipschitz constant \(\ubound\)
of the gradient and the condition number \(\condition\).

The Lipschitz constant can be estimated with Backtracking (Section~\ref{sec:
backtracking}). With the condition number we could just be conservative and set
\(\condition^{-1}=0\). This sacrifices linear convergence, but in view of the
asymptotic convergence rates stochasticity causes, this might not be so bad. But
in situations where the ``transient phase'' (before stochasticity kills fast
progression) is significant, this might not be a satisfactory answer.

\begin{figure}[h]
	\centering
	\def\svgwidth{1\textwidth}
	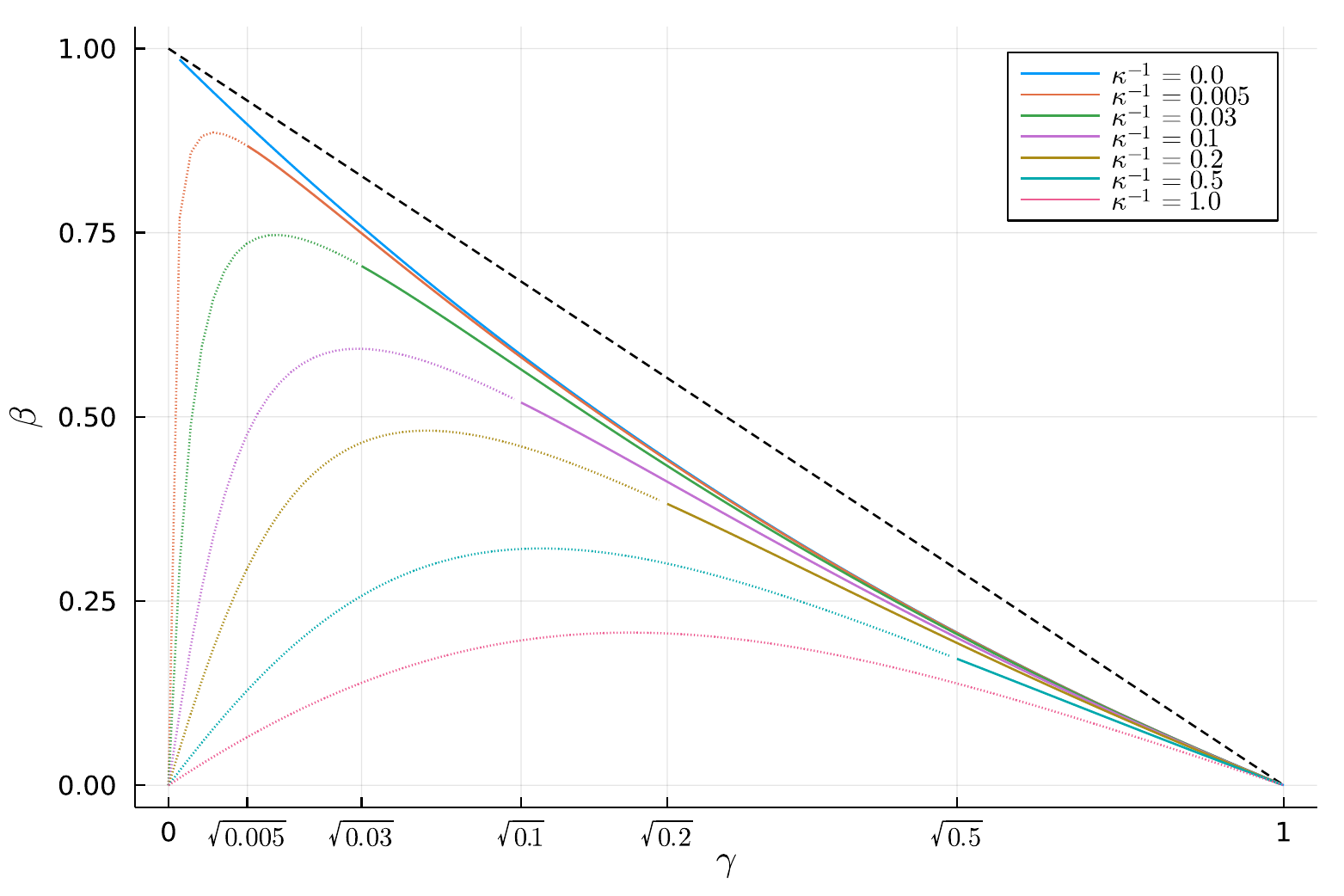
	\caption{
		Since momentum \(\momCoeff_n\) is a function of \((\gamma_n,
		\gamma_{n+1})\), and \(\gamma_{n+1}\) is a function of
		\((\gamma_n,\condition^{-1})\), momentum can be viewed as a function of
		\((\gamma, \condition^{-1})\). This function
		\(\momCoeff(\gamma,\condition^{-1})\) is plotted here over \(\gamma\) for
		different \(\condition^{-1}\). Since the sequence of \(\gamma_n\) stays
		above \(\hat{\gamma}=\sqrt{\condition^{-1}}\) the line is dotted below
		this threshold.
	}
	\label{fig: momentum gamma plot}
\end{figure}

From our eigenvalue analysis, we know that smaller eigenvalues will generally
still converge (even if we neglect them) although at a slower rate. So we
could just guess a condition and see how it goes. And if we have a look at
Figure~\ref{fig: momentum gamma plot} a condition guess will simply stop the
progression of \(\gamma_n\) towards zero at this point and consequently stop the
progression of \(\momCoeff_n\) towards one. If we dislike the resulting
convergence speed, we can probably just increase the condition number on the
fly which causes the
\(\gamma_n\) to fall further and momentum to continue its move towards one.
In Figure~\ref{fig: momentum progression} we can see that it only takes very
few iterations until we converge to
\(\hat{\momCoeff}=\frac{1-\sqrt{\condition^{-1}}}{1+\sqrt{\condition^{-1}}}\)
if \(\condition^{-1}>0\).

\begin{figure}[h]
	\centering
	\def\svgwidth{1\textwidth}
	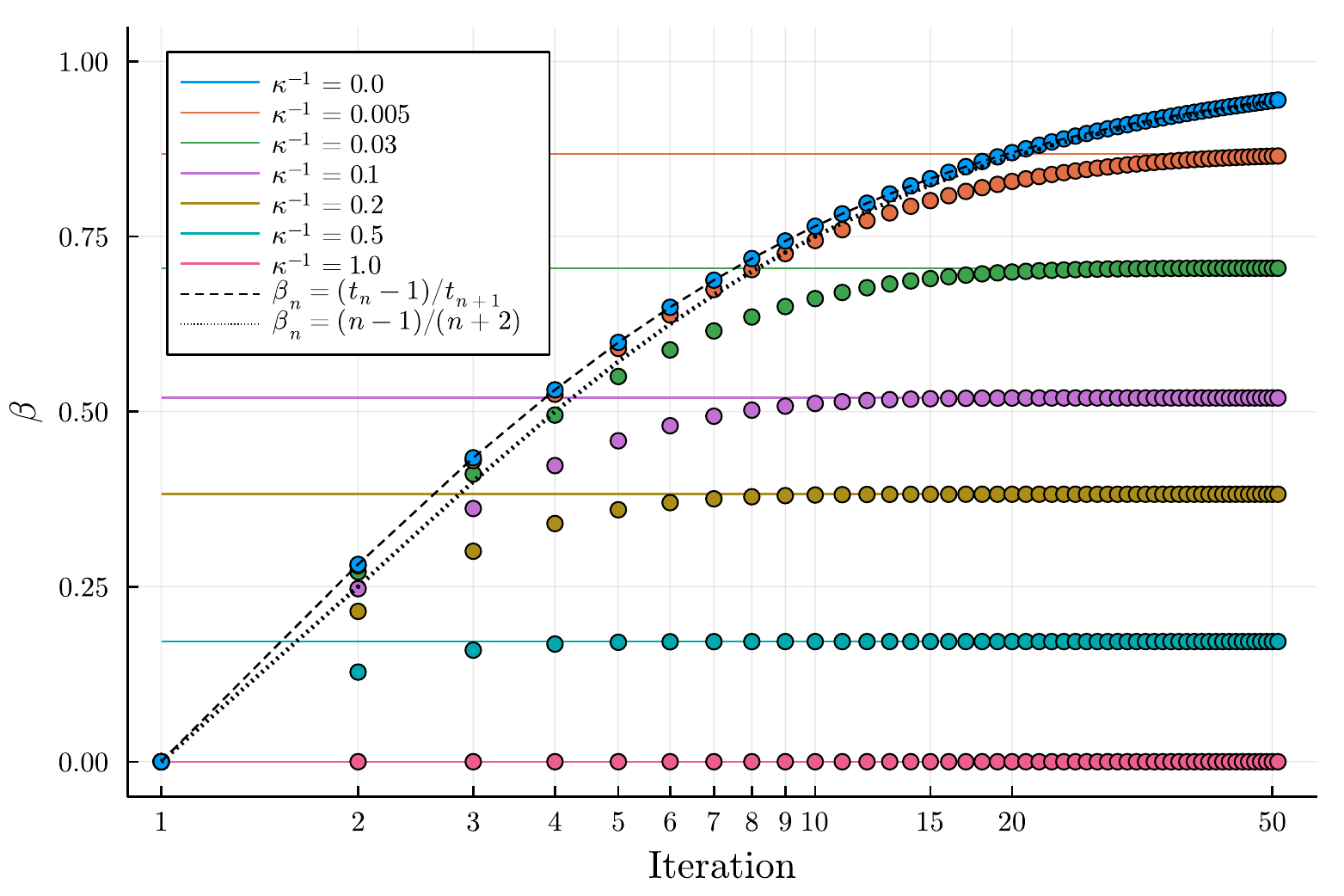
	\caption{
		Starting from \(\gamma_1=1\) the progression of \(\momCoeff_n\) for
		different condition numbers \(\condition\). The heuristics
		\(\momCoeff_n=\frac{t_n-1}{t_{n+1}}\) where \(t_1=1\) and
		\(t_n=\tfrac12(1+\sqrt{1+4t_n^2})\) and \(\momCoeff_n=\frac{n-1}{n+2}\)
		seem to be popular in practice \parencite{rechtOptimization2013}.
	}
	\label{fig: momentum progression}
\end{figure}

What we can also see in Figure~\ref{fig: momentum progression} is that
\begin{align}\label{eq: simplified dynamic nesterov momentum}
	\momCoeff_n = \frac{n-1}{n+2} = 1-\frac{3}{n+2}
\end{align}
is a decent approximation for the \(\momCoeff_n\) progression when \(\condition^{-1}=0\).
If we recall our original motivation, \(\momCoeff\) was \(1-\lr\friction\) where
\(\friction\) was a friction parameter. \textcite{suDifferentialEquationModeling2015}
show that in the limit of \(\lr\to0\), nesterov momentum with momentum
coefficient (\ref{eq: simplified dynamic nesterov momentum}) approximates the
differential equation
\begin{align*}
	\ddot{\weights} + \frac{3}{t}\dot{\weights} + \nabla\Loss(\weights) = 0
\end{align*}
i.e.\ decreasing friction \(\friction=3/t\) over time. They then use this
continuous approximation to derive better convergence rates (\(O(1/n^3)\) instead of
\(O(1/n^2)\)) when this conservative (\(\condition^{-1}=0\)) assumption is used,
but the loss function is strongly convex. They also propose a scheme of
``restarting'' Nesterov's Momentum to achieve linear convergence in the
strongly convex case without knowing the condition number (a restart effectively
resets momentum back to a lower level).

Last but not least recall that \(\lrSq=1/\ubound\) is the \emph{squared}
discretization \(\lr^2\) of the ODE which can deceive you into thinking that the
discretization is similar given that \ref{eq: gradient descent} converges for \(\lr \in
(0,2/\ubound)\). But the implied discretization \(\lr=1/\sqrt{\ubound}\) of
Nesterov's Momentum is much larger in general. This provides some intuition
for where the speedup comes from.

\fxwarning{deal with floats}

\section{Stochasticity}

It has not yet fully crystalized whether or not Momentum is beneficial in a
stochastic setting. Since movement is determined by the integral (sum) of
gradients in this setting, one can argue intuitively, that we might get a sort
of averaging effect, which should reduce variance and be beneficial for convergence
\parencite[e.g.][p. 69]{bottouOptimizationMethodsLargeScale2018}. \textcite{gohWhyMomentumReally2017}
on the other hand claim that ``increasing momentum, contrary to popular
belief, causes the errors to compound'', most likely citing
\textcite{devolderFirstorderMethodsSmooth2014}.

Now as we can already achieve optimal performance (in the BLUE sense) for
variance reduction in SGD (cf. Subsection~\ref{subsec: variance reduction}),
it is also intuitive why momentum might not be beneficial in stochastic settings.
But since it simply introduces additional parameters (\(\momCoeff\) could be zero
to get GD again), it provides a bit more freedom to play with.
\textcite{flammarionAveragingAccelerationThere2015} show that this freedom in
parameters results in a close relation of momentum to the averaging GD, which
we have discussed in Section~\ref{sec: SGD with Averaging}.

}
	{

\chapter{Other Methods}\label{chap: other methods}

\section{Heuristics for Adaptive Learning Rates}
\label{sec: heuristics for adpative learning rates}

There are a number of widely used heuristics which are provided in the
default optimizer selection of machine learning frameworks such as Tensorflow.

As we have only covered \ref{eq: gradient descent} and momentum so far, we have
quite a bit of ground to cover. But at the same time most of the remaining
common optimizers \parencite[as reviewed by e.g.][]{ruderOverviewGradientDescent2017}
can be motivated in a similar way, which might or might not be new.

Recall that the \ref{eq: newton minimum approx}
\begin{align*}
	\weights_{n+1}	= \weights_n - (\nabla^2\Loss(\weights_n))^{-1}\nabla\Loss(\weights_n)
\end{align*}
finds the vertex of a quadratic loss function \(\Loss\) immediately, because we
have
\begin{align*}
	\nabla\Loss(\weights_n) = \nabla^2\Loss(\weights_n)[\weights_n - \hat{\weights}_n],
\end{align*}
where \(\hat{\weights}_n\) was the vertex of the second Taylor approximation.
To better understand how that works, note that for
\begin{align*}
	\nabla^2\Loss(\weights_n) = V \diag(\hesseEV_1,\dots,\hesseEV_\dimension)V^T
\end{align*}
we have
\begin{align*}
	\nabla^2\Loss(\weights_n)^{-1}
	= V \diag(\tfrac1\hesseEV_1,\dots,\tfrac1\hesseEV_\dimension)V^T.
\end{align*}
To simplify things, let us assume a quadratic loss function
\(H=\nabla^2\Loss(\weights_n)\) now.  In this case we found that every
eigenspace scales independently
\begin{align*}
	\langle \weights_{n+1} - \minimum, v_i \rangle
	\le (1-\lr \hesseEV_i)\langle \weights_n - \minimum, v_i \rangle.
\end{align*}
The \ref{eq: newton minimum approx} rescales these steps by dividing every
eigenspace through its respective eigenvalue. For positive eigenvalues this is
all nice and dandy. But for negative eigenvalues, which represent a maximum in
their eigenspace, the \ref{eq: newton minimum approx} multiplies the gradient by
a negative value \(\tfrac1{\hesseEV_i}\) and therefore causes movement towards
this eigenvalue.

Okay, now that we remembered these crucial problems, let us make one egregious
assumption. Let us assume that the standard basis are already the eigenvectors,
i.e.\ \(V=\identity\). Then the components of the gradient of our loss function
are proportional to their respective eigenvalues
\begin{align*}
	\nabla\Loss(\weights_n)^{(i)}
	= \hesseEV_i [\weights_n - \minimum]^{(i)}.
\end{align*}
If we define
\begin{align*}
	G_n^{(i)} = \sum_{k=0}^n [\nabla\Loss(\weights_n)^{(i)}]^2 \qquad i=1,\dots,\dimension,
\end{align*}
then the algorithm
\begin{align}
	\label{eq: almost adagrad}
	\weights_{n+1}^{(i)}
	&= \weights_n^{(i)} - \frac{\lr_n}{\sqrt{G_n^{(i)}}}\nabla\Loss(\weights_n)^{(i)}\\
	\nonumber
	&= \weights_n^{(i)} - \frac{\lr_n}{
		\sqrt{\cancel{\hesseEV_i^2}\sum_{k=0}^n([\weights_k - \minimum]^{(i)})^2}
	}
	\cancel{\hesseEV_i}[\weights_n - \minimum]^{(i)}
\end{align}
also cancels out the eigenvalues. But as we divide through the absolute value of
the eigenvalues in contrast to the \ref{eq: newton minimum approx}, we get to
keep the sign
\begin{align}\label{eq: adagrad motivation contraction}
	[\weights_{n+1}-\minimum]^{(i)} = \left(1-\frac{\lr_n\sign(\hesseEV_i)}{
			\sqrt{\sum_{k=0}^n([\weights_k - \minimum]^{(i)})^2}
		}\right)[\weights_n - \minimum]^{(i)}.
\end{align}
Of course the notion that the standard basis vectors are already the eigenvectors
is pretty ridiculous. But for some reason the following algorithms seem to
work. And they are based on this notion of reweighting the components
by dividing them through the squared past gradients in some sense. The intuition
provided by \textcite{ruderOverviewGradientDescent2017} is, that
``It adapts the learning rate to the parameters, performing smaller updates
(i.e.\ low learning rates) for parameters associated with frequently occurring
features, and larger updates (i.e.\ high learning rates) for parameters
associated with infrequent features.'' But even this intuition seems to assume
that these ``features'' somehow correspond to the coordinates vectors.

\begin{description}
	\item[AdaGrad] \parencite{duchiAdaptiveSubgradientMethods2011}
	We have already motivated most of AdaGrad in (\ref{eq: almost adagrad}).
	But we will use the stochastic gradients
	\begin{align*}
		G_n^{(i)} := \sum_{k=0}^n [\nabla\loss(\weights_n, Z_n)^{(i)}]^2 \qquad i=1,\dots,\dimension,
	\end{align*}
	and introduce a small term \(\epsilon\) \parencite[\(\approx 10^{-8}\)][]{ruderOverviewGradientDescent2017}
	to avoid numerical instability (see also Levenberg-Marquard regularization,
	Subsection~\ref{subsec: levenberg-marquard regularization}). This results in
	the ``AdaGrad'' algorithm
 	\begin{align*}
		\Weights_{n+1}^{(i)}
		&= \Weights_n^{(i)}
		- \frac{\lr_n}{\sqrt{G_n^{(i)}+\epsilon}}\nabla\loss(\Weights_n, Z_n)^{(i)}
		\qquad i=1,\dots,\dimension.
	\end{align*}
	If we were to divide the sum of squared gradients \(G_n^{(i)}\) by the number
	of summands, we would get an estimator for the second moment. This second
	moment consists not only of the average squares of the actual loss function,
	but also of a variance estimator. This means that high variance can also
	reduce the learning rate. This is an additional benefit, as we have found
	when looking at SGD, that we need to reduce the learning rate when the
	variance starts to dominate.

	\item[RMSProp]\parencite[lecture 6e]{hintonNeuralNetworksMachine2012} As the sum of squares increases monotonically, we can see 
	in (\ref{eq: adagrad motivation contraction}) that this quickly reduces
	the learning rate of AdaGrad in every step. So RMSProp actually uses an
	estimator for the second moments, i.e.\ divides \(G_n^{(i)}\) by \(n\). Well,
	that is a bit of a lie. To avoid calculations it uses a recursive exponential
	average
	\begin{align*}
		\hat{\E}[(\nabla\loss(\Weights_n)^{(i)})^2]
		:= \gamma \hat{\E}[(\nabla\loss(\Weights_{n-1})^{(i)})^2]
		+ (1-\gamma)(\nabla\loss(\Weights_n, Z_n)^{(i)})^2.
	\end{align*}
	Another benefit of this exponential average is, that we weigh more recent
	gradients higher. The Second Taylor Approximation is only locally
	accurate after all, so it is probably a good idea to discount older gradients
	further away. Anyway, the algorithm ``RMSProp'' is defined as
	\begin{align*}
		\Weights_{n+1}^{(i)}
		&= \Weights_n^{(i)}
		- \frac{\lr_n}{\text{RMS}(\nabla\loss)_n}\nabla\loss(\Weights_n, Z_n)^{(i)}
		\qquad i=1,\dots,\dimension,
	\end{align*}
	where we denote the root mean square as
	\begin{align*}
		\text{RMS}(\nabla\loss)_n := \sqrt{\hat{\E}[\nabla\loss(\Weights_n)^{(i)}]^2+\epsilon}.
	\end{align*}

	\item[AdaDelta] \parencite{zeilerADADELTAAdaptiveLearning2012} The motivation
	for the independently developed AdaDelta starts exactly like RMSProp, but
	there is one additional step.
	``When considering the parameter updates, [\(\Delta \weights\)], being applied to
	[\(\weights\)], the units should match. That is, if the parameter had some
	hypothetical units, the changes to the parameter should be changes in those
	units as well. When considering SGD, Momentum, or ADAGRAD, we can see that
	this is not the case'' \parencite[p. 3]{zeilerADADELTAAdaptiveLearning2012}.
	They also note that the \ref{eq: newton minimum approx} does have matching
	units. So they consider the case where the Hessian is diagonal (the case we
	have been considering where the eigenvectors are the standard basis), and
	somehow deduce that one should use a similar estimator
	\(\text{RMS}(\Delta\Weights)\) for \(\sqrt{\E[\Delta\Weights^2]}\) as for
	\(\sqrt{\E[\nabla\loss^2]}\) to obtain ``AdaDelta''
	\begin{align*}
		\Weights_{n+1}^{(i)} = \Weights_n^{(i)}
		-\lr_n\frac{\text{RMS}(\Delta\Weights)_{n-1}}{\text{RMS}(\nabla\loss)_n}
		\nabla\loss(\Weights_n, Z_n)^{(i)}.
	\end{align*}

	\item[Adam] \textcite{kingmaAdamMethodStochastic2017} apply the same idea
	as RMSProp to the momentum update, i.e.
	\begin{align*}
		\Weights_{n+1} = \Weights_n + \frac{\lr_n}{\text{RMS}(\nabla\loss)_n} \momentum_{n+1}
	\end{align*}
	where we recall that momentum was defined as
	\begin{align*}
		\momentum_{n+1} = \underbrace{(1-\friction \lr_n)}_{\momCoeff_n} \momentum_n - \lr_n\nabla\Loss(\Weights_n).
	\end{align*}
	Additional to this idea, the paper also introduces the idea of a correction term
	as we have to initialize our exponential averages of the momentum term \(\momentum_n\)
	and \(\text{RMS}(\nabla\loss)_n\). Since these initializations usually are zero
	we have a bias towards zero, which Adam tries to compensate by multiplying
	the term \(\momentum_n\) with \(\frac{1}{1-\momCoeff^n}\) assuming a constant
	momentum coefficient, and similarly for \(\text{RMS}(\nabla\loss)_n\). Then
	this modified momentum and second moment estimator of the gradients are
	used for the update instead, resulting in the \textbf{ada}ptive \textbf{m}omentum method.

	\item[NAdam] \parencite{dozatIncorporatingNesterovMomentum2016}
	uses the same idea as Adam just with Nesterov momentum instead of heavy ball
	momentum.

	\item[AdaMax] \parencite{kingmaAdamMethodStochastic2017} In the same paper
	as Adam, the authors also note that the Root Mean Squared
	\(\text{RMS}(\nabla\loss)\) are basically a \(2\)-norm and one could use a
	different norm for the gradients as well.  Recall that we essentially want
	the absolute eigenvalue of the Hessian in the end. For that, squaring
	the gradient components works just as well as taking the absolute value (i.e.
	using the maximum norm). This idea results in the ``AdaMax'' algorithm.

	\item[AMSGrad] \parencite{reddiConvergenceAdam2019} While we have argued that
	the exponential moving average used for second moment estimation is a good
	thing, \textcite{reddiConvergenceAdam2019} find, that in the cases where
	these methods do not converge it is due to some seldomly occurring directions
	getting discounted too fast. So they suggest using a running maximum instead.
\end{description}

Since all these optimizers make component-wise adjustments which could be written
as a diagonal matrix, Assumption~\ref{assmpt: parameter in generalized linear
hull of gradients} still applies. This means our complexity bounds from
Section~\ref{sec: complexity bounds} still apply. Therefore these algorithms
can not be an improvement over standard momentum in general as Nesterov's
momentum is already asymptotically optimal.

And the theoretical justification for these algorithms is quite thin. There
are some fantastic animations by \textcite{radfordVisualizingOptimizationAlgos2014},
which heavily imply that these ``adaptive algorithms'' are better at breaking
symmetry. But these toy problems are generally axis aligned which means their
eigenspaces will be the eigenvectors. And at that point we know that these
algorithms should behave like Newton's algorithm. So it is no surprise that they
perform really well.

On the other hand one can find examples where these adaptive methods become very
unstable
\parencite[e.g.][]{wilsonMarginalValueAdaptive2018,reddiConvergenceAdam2019}.
But these examples might be unrealistic as well.

In empirical benchmarks such as \textcite{schmidtDescendingCrowdedValley2021} no
optimizer seems to outperform the others consistently.

While this section was mostly concerned with commonly used algorithms with little
theoretical justification, the remaining chapter will walk through approaches
from the theoretical side which are not as widely used in practice as the
methods discussed so far. To be more precise: there are no reference implementations
for the following methods in machine learning frameworks such as Tensorflow.
This proxy for usage is likely anything but perfect, as most of these methods are
more difficult to implement and might simply be deemed too complex to maintain
reference implementations for. But beginners, especially from fields other than
mathematics, are still likely to stick to the default optimizers. So it seems
fair to say that the following methods are not as common in practice.

\section{Geometric Methods}\label{sec: geometric methods}

\subsection{Cutting Plane Methods}

Notice how the \(\dimension-1\) dimensional hyperplane
\begin{align*}
	\mathcal{H}
	:= \{\theta : \langle \nabla\Loss(\weights), \theta - \weights \rangle = 0\}
\end{align*}
separates the set in which direction the loss is increasing
\begin{align*}
	\mathcal{H}_+
	:= \{\theta : \langle \nabla\Loss(\weights), \theta - \weights \rangle > 0\},
\end{align*}
from the set in which direction the loss is decreasing \(\mathcal{H}_-\). While
this is only a local statement, convexity implies that \(\Loss(\weights)\) is
smaller than the loss at any \(\theta\in\mathcal{H}_+\) because
\begin{align*}
	\Loss(\theta)
	\ge \Loss(\weights) + \underbrace{\langle \nabla\Loss(\weights), \theta -\weights\rangle}_{>0}
	> \Loss(\weights)
\end{align*}
Therefore we know the minimum can not be in \(\mathcal{H}_+\), which allows us
to cut away this entire set from consideration.

Cutting plane methods use this fact to iteratively shrink the considered set
until the minimum is found. As convexity is vital here, it is unlikely they
might ever work for non-convex loss functions.

\subsubsection{Center of Gravity Method}

This method simply evaluates the gradient at the center of gravity of the
remaining set, in the hope that the cutting plane method cuts away roughly half
of the remaining set. Therefore convergence speed can be shown to be
linear\footnote{i.e.\ exponential in non numerical-analysis jargon} 
\parencite[e.g.][Theorem 2.1]{bubeckConvexOptimizationAlgorithms2015}.
Unfortunately this requires calculating the center of gravity of an arbitrary
set which is its own difficult computational problem. 

\subsubsection{Ellipsoid Method}

To work around the center of gravity problem, the ellipsoid methods starts out
with an ellipsoid set, makes one cut through the center, and finds a new
ellipsoid around the remaining set. This new ellipsoid is larger than the
set we would actually have to consider, but at least finding its center is
tractable. The convergence speed is still linear \parencite[e.g][Theorem
2.4]{bubeckConvexOptimizationAlgorithms2015}.

\subsubsection{Vaidya's Cutting Plane Method}

This method uses polytopes instead of ellipsoids as base objects to make the
center of gravity search tractable, and seems to be more efficient
computationally \parencite[e.g.][Section
2.3]{bubeckConvexOptimizationAlgorithms2015}.

\subsection{Geometric Descent}

\textcite{bubeckGeometricAlternativeNesterov2015} realized that strong convexity
\begin{align*}
	\Loss(\theta)
	\ge \Loss(\weights) + \langle \nabla\Loss(\weights), \theta-\weights\rangle
	+ \tfrac{\lbound}{2}\|\theta-\weights\|^2
\end{align*}
is equivalent to
\begin{align*}
	\tfrac{\lbound}{2}\|\theta - (\weights - \tfrac1\lbound \nabla\Loss(\weights))\|^2
	\le \frac{\|\nabla\Loss(\weights)\|^2}{2\lbound} - (\Loss(\weights)-\Loss(\theta)),
\end{align*}
which implies that for \(\theta=\minimum\) we know that the minimum
\(\minimum\) is inside the ball
\begin{align*}
	\minimum\in B\left(\weights- \tfrac1\lbound\nabla\Loss(\weights), 
		\frac{\|\nabla\Loss(\weights)\|^2}{\lbound^2}
		- \frac{ \Loss(\weights)-\Loss(\minimum) }{\lbound}
	\right).
\end{align*}
With Lipschitz continuity of the gradient \textcite{bubeckGeometricAlternativeNesterov2015}
further refine this ``confidence ball'' around the minimum and use it to create
an algorithm with the same rate of convergence as Nesterov's momentum.
Unfortunately it looks like this algorithm is much more dependent on (strong)
convexity than the momentum method.

\section{Conjugate Gradient Descent}\label{sec: conjugate gradient descent}

Similar to momentum, conjugate gradient descent is another attempt to solve
high condition problems. One of the most approachable introductions so far was
written by \textcite{shewchukIntroductionConjugateGradient1994}. This is an
attempt at a shorter version. 

Recall that we can write
our loss function as an approximate paraboloid (cf. (\ref{eq: paraboloid
approximation of L}))
\begin{align*}
	\Loss(\theta)
	= (\theta-\weights)^T \nabla^2\Loss(\weights)(\theta-\weights)
	+ c(\weights) + o(\|\theta-\weights\|^2)
\end{align*}
If the second taylor approximation is accurate and we have a constant
positive definite Hessian \(H:=\nabla^2\Loss(\weights)\), then our loss can also
be written as
\begin{align*}
	\Loss(\weights) = \Loss(\minimum) + 
	(\weights-\minimum)^T H(\weights-\minimum)
	= \Loss(\minimum) + \|\weights-\minimum\|_H^2,
\end{align*}
where the norm \(\|\cdot\|_H\) is induced by the scalar product
\begin{align*}
	\langle x, y \rangle_H := \langle x, Hy\rangle.
\end{align*}
In the space \((\reals^d, \|\cdot\|_H)\) our problem therefore has condition one
(\(\condition=1\))!

\subsection{Excursion: The Gradient in \((\reals^\dimension, \|\cdot\|_H)\)}

We have made our problem much simpler but the space we are considering much
more complicated. The question is now, what methods we can realistically use
in this new space?

Naively one might wish to take the gradient with respect to this space. As this
gradient would point straight towards the minimum because we are in the
\(\condition=1\) case. This means, that this gradient should be just as
powerful as the Newton-Raphson method, which would be great. So let us try to
obtain this gradient. From the definition of the gradient we have
\begin{align}
	\label{eq: gradient definition equation}
	0\xeq{!}&\lim_{v\to 0}\frac{
		|\Loss(\weights + v) - \Loss(\weights) - \langle \nabla_H\Loss(\weights), v\rangle_H |
	}{\|v\|_H}\\
	\nonumber
	&=\lim_{v\to 0}\frac{
		|\Loss(\weights + v) - \Loss(\weights) -  \nabla_H\Loss(\weights)^T H v |
	}{\|v\|}
	\frac{\|v\|}{\|v\|_H}\\
	\nonumber
	&=\lim_{v\to 0}\frac{
		|\Loss(\weights + v) - \Loss(\weights) -  \langle H\nabla_H\Loss(\weights), v\rangle |
	}{\|v\|}
	\frac{\|v\|}{\|v\|_H}
\end{align}
Where we have used symmetry of the Hessian \(H\) for the last equation. Now if
we had
\begin{align*}
	H \nabla_H \Loss(\weights) = \nabla\Loss(\weights),
\end{align*}
then the first fraction would go to zero by the definition of the gradient.
The second fraction would be bounded due to norm equivalence, since all norms
are equivalent in final dimensional real vector spaces. Therefore
\begin{align*}
	\nabla_H \Loss(\weights) := H^{-1}\nabla\Loss(\weights)
\end{align*}
is a (and due to uniqueness, the) solution of (\ref{eq: gradient definition
equation}). So we have found a method equivalent to Newton's method.
Unfortunately it \emph{is} Newtons's method. And since we still do not want
to invert the Hessian, we have to think of something else. Fortunately, that
something else only has to work for \(\condition=1\).

\subsection{Conjugate Directions}

Let us take any (normed) direction \(d^{(1)}\) and complete it to an H-orthogonal
basis \(d^{(1)},\dots,d^{(\dimension)}\). H-orthogonal is also known as
``conjugate''. Then we can represent any vector \(\theta\) as a linear
combination of this conjugate basis
\begin{align*}
	\theta = \sum_{k=1}^\dimension \theta^{(k)} d^{(k)}.
\end{align*}
Denoting \(\Delta\weights_n := \weights_n -\minimum\), let us now optimize over
the learning rate for a move in direction \(d^{(1)}\)
\begin{align*}
	\Delta\weights_{1} = \Delta\weights_0 + \lr_0 d^{(1)}
	= (\Delta\weights^{(1)} + \lr_0)d^{(1)}
	+ \sum_{k=2}^\dimension \Delta\weights^{(k)} d^{(k)}.
\end{align*}
Due to conjugacy we have
\begin{align*}
	\Loss(\weights_{1})
	= \Loss(\minimum) + \|\Delta\weights_{1}\|_H^2
	= \Loss(\minimum) + \sum_{k=1}^{\dimension} (\Delta\weights_{1}^{(k)})^2\|d^{(k)}\|^2.
\end{align*}
And since most of the \(\Delta\weights^{(k)}_{1}\) stay constant no matter the
learning rate \(\lr_n\), the only one we actually optimize over is
\((\Delta\weights^{(1)}_{1})^2\).
And we can make that coordinate zero with \(\lr_0 :=-\Delta\weights_0^{(1)}\).

So if we do a line search (optimize the loss over one direction, i.e.\ \(d^{(1)}\)),
then we know that we will never ever have to move in the direction \(d^{(1)}\)
again. The target \(\minimum\) lies H-orthogonal to \(d^{(1)}\).
\begin{figure}[h]
	\centering
	\def\svgwidth{1\textwidth}
	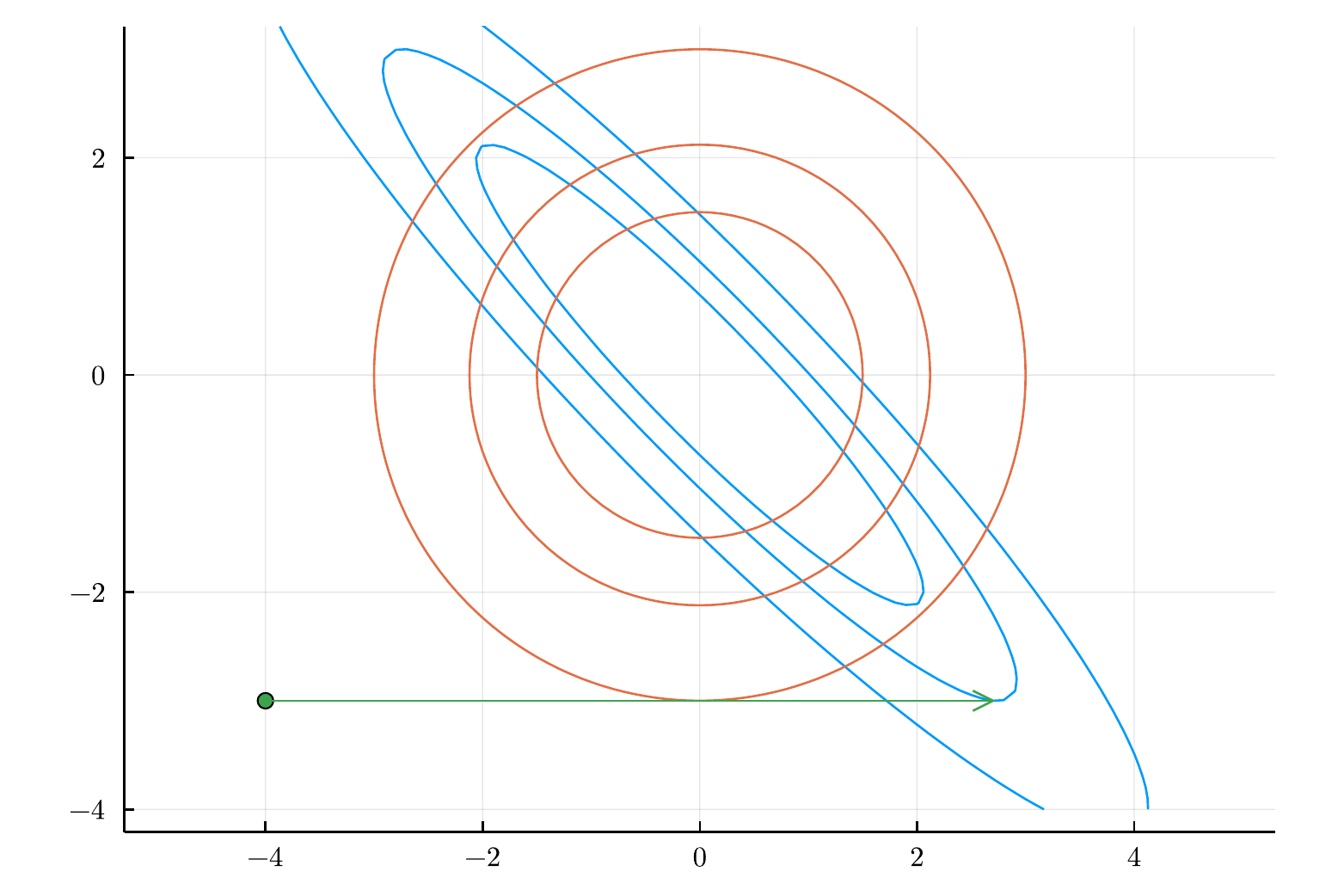
	\caption{
		The method of conjugate directions only works, because in the
		\(\condition=1\) case (brown contour) minimizing the loss in any direction
		eliminates that direction. This does not work for general condition numbers,
		as illustrated with the blue contour lines.
	}
	\label{fig: conjugate direction depends on condition one}
\end{figure}
If we optimize over the next direction the same holds true. In fact we get the
globally optimal solution in the span of search directions using this method 
\begin{align*}
	\Delta\weights_n &= \sum_{k=n+1}^\dimension\Delta\weights^{(k)} d^{(k)}\\
	&= \arg\min \{
		\Loss(\weights) : \weights \in \weights_0 + \linSpan[d^{(1)},\dots, d^{(n)}]
	\} - \minimum.
\end{align*}
In this light, recall that we provided complexity bounds for weights in the
linear span of gradients shifted by \(\weights_0\) in Section~\ref{sec: complexity bounds}
(cf. Assumption~\ref{assmpt: parameter in linear hull of gradients}). So if our
search directions were the gradients, we would obtain optimal convergence for
this class of methods.

\subsection{Simplifying Gram-Schmidt}

Unfortunately our gradients are not conjugate in general. But that is no
trouble, we can just use Gram-Schmidt conjugation. Right?

So here is the problem: We need knowledge of the hessian \(H\) for this
procedure and have to remove the components from all previous directions. This
blows up time complexity making this procedure not worthwhile. But as it turns
out, the gradients do have some additional structure we can exploit to remove
these issues.

First let us summarize what we have so far. We want to obtain the same conjugate
search directions \((d^{(1)},\dots,d^{(\dimension)})\) from
\((\nabla\Loss(\weights_0),\dots,\nabla\Loss(\weights_{\dimension-1}))\)
we would get, if we used Gram-Schmidt conjugation, which also implies
\begin{align*}
	\mathcal{K}_k := \linSpan\{\nabla\Loss(\weights_0),\dots,\nabla\Loss(\weights_k)\}
	= \linSpan\{d^{(1)},\dots,d^{(k+1)}\}.
\end{align*}
And we select our weights with line searches along the search directions
\begin{align*}
	\lr_n := \arg\min_\lr \Loss(\weights_n - \lr d^{(n+1)})
\end{align*}
resulting in
\begin{align*}
	\weights_{n+1} = \weights_n - \lr_n d^{(n+1)}
	= \arg\min\{ \Loss(\weights) : \weights\in\weights_0 + \mathcal{K}_n\}.
\end{align*}
This last condition implies that the gradient \(\nabla\Loss(\weights_n)\)
must necessarily be orthogonal to \(\mathcal{K}_{n-1}\) (unless we have already
converged to \(\minimum\)). Otherwise \(\weights_n\) would not be optimal. This
orthogonality allows us to argue that \(\mathcal{K}_k\) actually has rank
\(k+1\) and its respective generators are both a basis, which was necessary for the
conjugate directions argument.

So if we could somehow translate this euclidean orthogonality into a
H-orthogonality we would be done. And here comes in one crucial insight:
\(\mathcal{K}_n\) is a ``Krylov subspace'', i.e.\ can be generated by repeated
application of an operator. More specifically we have
\begin{lemma}[\(\mathcal{K}_n\) is a \(H\)-Krylov subspace]
	\begin{align*}
		\mathcal{K}_n
		= \{H^1 \Delta\weights_0, \dots, H^{n+1}\Delta\weights_0\}
		= \{H^0 \nabla\Loss(\weights_0), \dots, H^n\nabla\Loss(\weights_0)\}.
	\end{align*}
\end{lemma}
\begin{proof}
	The induction start \(n=0\) is simply due to \(\nabla\Loss(\weights_0) =
	H\Delta\weights_0\) and the induction step \(n-1\to n\) starts with
	\begin{align}
		\nonumber
		\nabla\Loss(\weights_n)= H\Delta\weights_n
		&= H(\Delta\weights_{n-1} - \lr_{n-1} d^{(n)})\\
		\label{eq: represent nabla Loss(weights_n) with prev gradient and d(n)}
		&= \underbrace{\nabla\Loss(\weights_{n-1})}_{\in \mathcal{K}_{n-1}} - \lr_{n-1} H
		\underbrace{d^{(n)}}_{\in \mathcal{K}_{n-1}}.
	\end{align}
	To finish this step, we apply the induction assumption and notice that the
	second summand might now contain the vector \(H^{n+1}\Delta\weights_0\) if the
	coefficient of \(H^{n}\Delta\weights_0\) is not zero. Therefore
	\(\mathcal{K}_{n}\) is contained in the Krylov subspace. And due to the
	number of elements in the generator of the Krylov subspace, it can not have
	higher rank than \(n+1\), and thus cannot be larger.
\end{proof}
This is useful, because it implies that \(H\mathcal{K}_{n-2}\) is contained
in \(\mathcal{K}_{n-1}\) to which \(\nabla\Loss(\weights_n)\) is orthogonal.
Therefore \(\nabla\Loss(\weights_n)\) is already \(H\)-orthogonal to
\(\mathcal{K}_{n-2}\) from the get go. In other words we already have
\begin{align*}
	\langle \nabla\Loss(\weights_n), d^{(k)} \rangle_H = 0 
	\qquad \forall k\le n-1.
\end{align*}
To construct a new conjugate direction \(d^{(n+1)}\) we therefore only have to
remove the previous direction \(d^{(n)}\) component from our gradient! So we
want to select
\begin{align}\label{eq: next search direction}
	d^{(n+1)} = \nabla\Loss(\weights_n)
	- \frac{\langle \nabla\Loss(\weights_n), d^{(n)}\rangle_H}{\|d^{(n)}\|_H^2}d^{(n)}.
\end{align}

\subsection{Evaluating the \(H\)-Dot-Product}

Now, we ``just'' have to find a way to evaluate this factor without actually
knowing the Hessian \(H\). Using the representation (\ref{eq: represent nabla Loss(weights_n) with
prev gradient and d(n)}) of \(\nabla\Loss(\weights_n)\), we get
\begin{align*}
	\langle \nabla\Loss(\weights_n), \nabla\Loss(\weights_n) \rangle
	= \langle \nabla\Loss(\weights_n), \nabla\Loss(\weights_{n-1}) \rangle
	- \lr_{n-1}\langle \nabla\Loss(\weights_n), H d^{(n)}\rangle.
\end{align*}
Using the definition of \(\langle\cdot,\cdot\rangle_H\) and reordering, results in
\begin{align}
	\nonumber
	\langle \nabla\Loss(\weights_n), d^{(n)}\rangle_H
	&= \tfrac1{\lr_{n-1}}\left[
		\langle \nabla\Loss(\weights_n), \nabla\Loss(\weights_{n-1}) \rangle
		- \langle \nabla\Loss(\weights_n), \nabla\Loss(\weights_n) \rangle
	\right]\\
	\label{eq: evaluated H dot product}
	&= \frac{- \|\nabla\Loss(\weights_n)\|^2}{\lr_{n-1}},
\end{align}
where we have used the orthogonality of \(\nabla\Loss(\weights_n)\) to
\(\nabla\Loss(\weights_{n-1})\in\mathcal{K}_{n-1}\).
Using the optimality of the learning rate \(\lr_{n-1}\)
\begin{align*}
	0\xeq{!}\frac{d}{d\lr}\Loss(\weights_{n-1} - \lr d^{(n)})
	&= -\langle \nabla\Loss(\weights_{n-1} - \lr d^{(n)}), d^{(n)} \rangle\\
	&= -\langle H(\Delta\weights_{n-1} -\lr d^{(n)}), d^{(n)} \rangle\\
	&= \lr\|d^{(n)}\|_H^2 - \langle \nabla\Loss(\weights_{n-1}), d^{(n)}\rangle,
\end{align*}
we get
\begin{align}\label{eq: optimal learning rate CG}
	\lr_{n-1}
	= - \frac{\langle \nabla\Loss(\weights_{n-1}), d^{(n)}\rangle}{\|d^{(n)}\|_H^2}.
\end{align}
Using (\ref{eq: evaluated H dot product}) and (\ref{eq: optimal learning rate CG})
we can replace the Gram-Schmidt constant from (\ref{eq: next search direction})
with things we can actually calculate resulting in
\begin{align*}
	d^{(n+1)}
	&= \nabla\Loss(\weights_n)
	- \frac{
		\|\nabla\Loss(\weights_n)\|^2
	}{\langle \nabla\Loss(\weights_{n-1}), d^{(n)}\rangle}d^{(n)}\\
	&= \nabla\Loss(\weights_n)
	- \frac{\|\nabla\Loss(\weights_n)\|^2}{\|\nabla\Loss(\weights_{n-1})\|^2}
	d^{(n)},
\end{align*}
where the last equation is due to \(\nabla\Loss(\weights_{n-1})\) being orthogonal
to \(d^{(n-1)}\in\mathcal{K}_{n-2}\), i.e.
\begin{align*}
	\langle \nabla\Loss(\weights_{n-1}), d^{(n)} \rangle
	\xeq{(\ref{eq: next search direction})} 
	\langle \nabla\Loss(\weights_{n-1}),\nabla\Loss(\weights_{n-1}) - c d^{(n-1)}\rangle
	= \|\nabla\Loss(\weights_{n-1})\|^2.
\end{align*}

\subsection{Relation to Momentum and Discussion}

Unfortunately the assumption of a constant and positive definite Hessian
\(H=\nabla^2\Loss(\weights)\) is fairly difficult to remove here. To do so
we would somehow have to estimate the error we are making, working with the taylor
approximation. This would probably result in some form of estimating sequence
argument again. And we already found that Nesterov momentum is in some sense
optimal then.

In fact for specific selections of the momentum coefficient, the momentum method
\emph{is} the conjugate gradient method. For that consider that the movement
direction \(d^{(n+1)}\) consists of the gradient and the previous direction
\(d^{(n)}\) which would be our momentum. More specifically we have
\begin{align*}
	\lr_{n-1}d^{(n)} = \weights_{n-1} - \weights_n
\end{align*}
and therefore
\begin{align*}
	\weights_{n+1}
	&= \weights_n - \lr_n d^{(n+1)}\\
	&= \weights_n - \lr_n \left[
		\nabla\Loss(\weights_n)
		- \frac{\|\nabla\Loss(\weights_n)\|^2}{\|\nabla\Loss(\weights_{n-1})\|^2}
		d^{(n)},
	\right]\\
	&= \weights_n - \underbrace{\lr_n}_{=:\alpha_n} \nabla\Loss(\weights_n)
		+ \underbrace{\tfrac{\lr_n\|\nabla\Loss(\weights_n)\|^2}{\lr_{n-1}\|\nabla\Loss(\weights_{n-1})\|^2}}_{=:\beta_n}
		(\weights_n-\weights_{n-1}).
\end{align*}
Conjugate gradient can therefore be viewed as a special form of momentum with
a particular selection of parameters.

For a constant Hessian it has the nice property that it converges in
a finite number of steps. And quickly! It is actually a competitive method to
solve linear equations \parencite[Section 10]{shewchukIntroductionConjugateGradient1994},
e.g. by using CG on
\begin{align*}
	\Loss(\weights) = \|A\weights - b\|^2.
\end{align*}
Visualizing CG would be somewhat boring as it converges in two steps on two
dimensional loss surfaces.

\section{Secant Methods}

If we wish to estimate the Hessian with the change of our gradients over time,
we will soon notice that we do not get the evaluation of the Hessian at one point
but rather an average
\begin{align}\label{eq: secant equation}
	\underbrace{\nabla\Loss(\weights_n) - \nabla\Loss(\weights_{n-1})}_{=:\Delta g_n}
	= \underbrace{\int_0^1 \nabla^2\Loss(\weights_{n-1} + t(\weights_n-\weights_{n-1}))dt}_{=:H_n}
	\underbrace{(\weights_n - \weights_{n-1})}_{=:\Delta\weights_n}.
\end{align}
Using this ``average'' \(H_n\) in place of the Hessian \(\nabla^2\Loss(\weights_n)\)
to reshape the gradient, results in the ``secant method''
\begin{align*}
	\tag{secant method}
	\weights_{n+1}	= \weights_n - H_n^{-1}\nabla\Loss(\weights_n).
\end{align*}
While it might seem like this averaging is undesirable at first, it
can sometimes be desireable as it smooths out noisy or fluctuating Hessians
\parencite[e.g.][]{metzGradientsAreNot2021}. In one dimension this is already well
defined, as we can simply define
\begin{align*}
	H_n := \frac{\Loss'(\weights_n)-\Loss'(\weights_{n-1})}{\weights_n-\weights_{n-1}}.
\end{align*}
But there is no unique solution to the secant equation (\ref{eq: secant equation})
in \(\dimension\) dimensions. This is because both the weight difference
\(\Delta\weights_n\) and the gradient difference \(\Delta g_n\) are
\(\dimension\) dimensional vectors, (\ref{eq: secant
equation}) therefore represents \(d\) linear equations to find the
\(\frac{\dimension(\dimension+1)}{2}\) entries of the symmetric matrix \(H_n\).
And for \(\dimension>1\) we have
\begin{align*}
	\frac{\dimension(\dimension+1)}{2} > \dimension.
\end{align*}
So to make \(H_n\) unique we have to add additional requirements. 

\subsection{DFP}

If we already have an estimate of the Hessian \(\hat{H}_{n-1}\) from previous
iterations, it might be sensible to use ``as much as possible'' of
\(\hat{H}_{n-1}\), and change only as much as needed to guarantee \(\hat{H}_n\)
fulfills the secant condition (\ref{eq: secant equation}), i.e.
\begin{align}\label{eq: dfp optimization problem}
	\hat{H}_n &:=
	\arg\min_H \|H-\hat{H}_{n-1}\|
	& \text{subject to}\quad H=H^T,\quad  \Delta g_n= H\Delta\weights_n.
\end{align}
This still leaves us with the selection of the norm. A norm that turns out to 
be useful is the weighted Frobenius norm \(\|\cdot\|_W\) with
\begin{align*}
	\|A\|_W := \|W^{1/2}AW^{1/2}\|_F, \quad\text{where}\quad
	\|A\|_F^2 := \text{trace}(A^TA)=\sum_{i,j=1}^{\dimension}A_{i,j}^2.
\end{align*}
If we select \(W:=H_n^{-1}\), then according to
\textcite{nocedalQuasiNewtonMethods2006} the unique solution of (\ref{eq: dfp
optimization problem}) is
\begin{align}
	\tag{DFP}
	\hat{H}_n =
	\left(\identity-\frac{\Delta g_n \Delta\weights_n^T}{\langle \Delta g_n, \Delta\weights_n\rangle}\right)
	\hat{H}_{n-1}
	\left(\identity-\frac{\Delta g_n \Delta\weights_n^T}{\langle \Delta g_n, \Delta\weights_n\rangle}\right)
	+ \frac{\Delta g_n \Delta g_n^T}{\langle \Delta g_n, \Delta\weights_n\rangle}.
\end{align}
Since we really want to use \(\hat{H}_n^{-1}\) and matrix inversion is expensive,
\textcite{nocedalQuasiNewtonMethods2006} remark that there is also a recursion
for its inverse
\begin{align*}
	\tag{DFP'}
	\hat{H}_n^{-1}
	= \hat{H}_{n-1}^{-1}
	- \frac{
		\hat{H}_{n-1}^{-1}\Delta g_n \Delta g_n^T \hat{H}_{n-1}^{-1}
	}{
		\langle \Delta g_n, \hat{H}_{n-1}^{-1}\Delta g_n\rangle
	}
	+ \frac{
		\Delta\weights_n \Delta\weights_n^T
	}{
		\langle \Delta g_n, \Delta\weights_n\rangle
	}.
\end{align*}
This estimation \(\hat{H}_n\) of \(H_n\) was first proposed by Davidon in 1959
and further examined and developed by Fletcher and Powell. Hence its name.

\subsection{BFGS}

Since we only really want \(\hat{H}_n^{-1}\),
\textcite{broydenConvergenceClassDoublerank1970,fletcherNewApproachVariable1970,goldfarbFamilyVariablemetricMethods1970,shannoConditioningQuasiNewtonMethods1970}
all independently came up with the following improvement to DFP. First we
reorder (\ref{eq: secant equation}) into
\begin{align*}
	H_n^{-1} \Delta g_n = \Delta\weights_n
\end{align*}
and now we look for the inverse directly, instead of looking for \(H_n\) first.
In other words we define
\begin{align*}
	\hat{H}_n^{-1} &:= \arg\min_{H^{-1}} \| H^{-1} - \hat{H}_{n-1}^{-1}\|
	& \text{ subject to}\quad H=H^T,\quad H^{-1}\Delta g_n =\Delta \weights_n.
\end{align*}
For the weighted Frobenius norm with weights \(W=H_n\) this results in the BFGS
recursion
\begin{align}
	\label{eq: bfgs}
	\tag{BFGS}
	\hat{H}_n^{-1} :=
	\left(\identity-\tfrac{\Delta\weights_n \Delta g_n^T}{\langle \Delta g_n, \Delta\weights_n\rangle}\right)
	\hat{H}_{n-1}^{-1}
	\left(\identity-\tfrac{ \Delta\weights_n\Delta g_n^T}{\langle \Delta g_n, \Delta\weights_n\rangle}\right)
	+ \tfrac{\Delta\weights_n \Delta \weights_n^T}{\langle \Delta g_n, \Delta\weights_n\rangle}.
\end{align}
An interesting property of this update is, that it retains strict positive
definiteness of previous estimates as we have for a non-zero vector \(z\)
\begin{align*}
	z^T \hat{H}_n^{-1} z
	&= \underbrace{
		\left(z-\tfrac{
			\langle\Delta\weights_n, z\rangle \Delta g_n^T
		}{
			\langle \Delta g_n, \Delta\weights_n\rangle
		}\right)
		\hat{H}_{n-1}^{-1}
		\left(z-\tfrac{\langle \Delta\weights_n, z\rangle \Delta g_n^T}{\langle \Delta g_n, \Delta\weights_n\rangle}\right)
	}_{
	\ge \begin{cases}
		z^T \hat{H}_{n-1}^{-1} z > 0 & \langle\Delta\weights_n, z\rangle^2 =0 \\
		0 & \text{else}
	\end{cases}
	}
	+ \tfrac{\langle\Delta\weights_n, z\rangle^2}{\langle \Delta g_n, \Delta\weights_n\rangle}.
	\\
	&> 0.
\end{align*}

\subsection{L-BFGS}

Limited memory BFGS is essentially just an implementation detail of BFGS. As
we only want to apply our inverse Hessian \(\hat{H}_n^{-1}\) to some vector,
i.e. the gradient, we do not necessarily need to know or store this matrix if
we can achieve this goal without doing so. And if we use our recursion (\ref{eq:
bfgs}), we can apply our inverse Hessian to the vector \(v\) by calculating
\begin{align*}
	\hat{H}_n^{-1}v =
	\left(\identity-\tfrac{\Delta\weights_n \Delta g_n^T}{\langle \Delta g_n, \Delta\weights_n\rangle}\right)
	\hat{H}_{n-1}^{-1}
	\left(v-\tfrac{ \langle \Delta g_n, v\rangle}{\langle \Delta g_n, \Delta\weights_n\rangle}\Delta\weights_n\right)
	+ \tfrac{\langle \Delta \weights_n, v\rangle}{\langle \Delta g_n, \Delta\weights_n\rangle}\Delta\weights_n.
\end{align*}
Assuming \(\langle \Delta g_n, \Delta\weights_n\rangle\) is precalculated, we
only need to calculate a scalar product and subtract two vectors until we have
reached a point where we need to only apply \(\hat{H}_{n-1}^{-1}\) to a vector.
Going through this recursively and assuming \(\hat{H}_0=\identity\), we only
need \(O(n\dimension)\) operations to calculate \(\hat{H}_n^{-1}v\) this way.
For \(n\le\dimension\) this can be cheaper than matrix vector multiplication
requiring \(O(\dimension^2)\) operations. Unfortunately after some time it
becomes more and more expensive.

So limited memory BFGS simply throws away information older than \(k\) iterations,
effectively restarting the estimation of the inverse Hessian with
\(\hat{H}_{n-k}^{-1}=\identity\), guaranteeing a cost of \(O(k\dimension)=O(\dimension)\).

\section{Global Newton Methods}

So far we have only motivated the Newton-Raphson method on quadratic functions.
Assuming we are in some local area around the minimum, it can be shown that
successive \ref{eq: newton minimum approx}s result in convergence. But these
proofs rely on the fact that the Hessian is locally accurate enough. Therefore
even these proofs implicitly bound the rate of change of the curvature (Hessian).
For global Newton methods we just go about this more explicitly, like we did
with the rate of change of \ref{eq: gradient descent} (Section~\ref{sec:
lipschitz continuity of the Gradient}).

\subsection{Cubic Regularization of Newton's Method}

 In Section~\ref{sec: lipschitz continuity of the Gradient} we used Lipschitz
 continuity of the gradient to get an upper bound on our Loss (\ref{eq: lin
 approx + distance penalty notion})
\begin{align*}
	\Loss(\theta)
	\le \Loss(\weights) + \langle \nabla\Loss(\weights), \theta-\weights \rangle
	+ \tfrac{\ubound}2 \|\theta-\weights\|^2,
\end{align*}
which we then optimized over in Lemma~\ref{lem: smallest upper bound} to find
that gradient descent with learning rate \(\tfrac1\ubound\) is optimal with
regards to this bound. The Lipschitz constant \(\ubound\) bounds the second
derivative and thus the rate of change of the gradient. This allows us to
formulate for how long we can trust our gradient evaluation in some sense.

In the case of the Newton-Raphson method, we need a similar rate of change
bound on the Hessian. In this case we want to bound the rate of change of the
second derivative so we assume a Lipschitz continuous second derivative (or 
in other words bounded third derivative). This similarly results in the upper
bound
\begin{align*}
	\Loss(\theta)
	\le \Loss(\weights) + \langle \nabla\Loss(\weights), \theta-\weights \rangle
	+ \tfrac12\langle \nabla^2\Loss(\weights)(\theta-\weights), \theta-\weights \rangle
	+ \tfrac{M}6\|\theta-\weights\|^3.
\end{align*}
Minimizing this upper bound results in the cubic regularization of Newton's method
\parencite[Section 4.1]{nesterovLecturesConvexOptimization2018}.
The solution of which can unfortunately not be written explicitly.

Since this method is history agnostic again, Nesterov reuses estimating
sequences, to derive an ``accelerated cubic Newton scheme'' \parencite[Section
4.2]{nesterovLecturesConvexOptimization2018} similar to Nesterov momentum, i.e.
the accelerated gradient descent.

\subsection{Levenberg-Marquard Regularization}\label{subsec: levenberg-marquard regularization}

Instead of using a cubic bound motivated by a bounded third derivative, one
could still use a quadratic bound
\begin{align*}
	\Loss(\theta)
	&\le \Loss(\weights) + \langle \nabla\Loss(\weights), \theta-\weights \rangle
	+ \tfrac12\langle \nabla^2\Loss(\weights)(\theta-\weights), \theta-\weights \rangle
	+ \tfrac{M}2\|\theta-\weights\|^2\\
	&= \Loss(\weights) + \langle \nabla\Loss(\weights), \theta-\weights \rangle
	+ \tfrac12\langle [\nabla^2\Loss(\weights)+M\identity](\theta-\weights), \theta-\weights \rangle
\end{align*}
which might perhaps be motivated by a bounded second derivative
\begin{align*}
	\lbound\identity \precsim \nabla^2\Loss(\weights) \precsim \ubound\identity
\end{align*}
which would justify the bound \(M:=\ubound-\lbound\), as we have
\begin{align}
	\label{eq: levenberg-marquard justification}
	\Loss(\theta)
	&\le \Loss(\weights) + \langle \nabla\Loss(\weights), \theta-\weights \rangle
	+ \underbrace{\tfrac{\ubound}2\|\theta-\weights\|^2\mathrlap{.}}_{
		\le \tfrac12\langle
			[\nabla^2\Loss(\weights)+M\identity]
		\mathrlap{
			(\theta-\weights),
				\theta-\weights
			\rangle
		}
	}
\end{align}
Of course in practice one would probably use smaller \(M\), hoping that
the changes are not so drastic that the smallest eigenvalues become the
upper bound. While minimizing the cubic bound was difficult, minimizing this
bound results in
\begin{align}\label{eq: levenberg-marquard regularization}
	\weights_{n+1}
	= \weights_n - [\nabla^2\Loss(\weights_n)+M\identity]^{-1}\nabla\Loss(\weights_n).
\end{align}
For this result one simply replaces \(\nabla^2\Loss(\weights)\) with
\([\nabla^2\Loss(\weights)+M\identity]\) in the derivation of the \ref{eq:
newton minimum approx}.

For very small \(M\) this method behaves like the Newton-Raphson method, while
for large \(M\) the step in every eigenspace is roughly \(\tfrac1M\) and the
method therefore behaves like \ref{eq: gradient descent}. Levenberg-Marquard
regularization (\ref{eq: levenberg-marquard regularization}) can therefore also
be viewed as a way to interpolate between the faster Newton-Raphson method and
the more stable \ref{eq: gradient descent}. 

It moves small eigenvalues further away from zero, making the Hessian easier to invert (without numerical errors). And it can turn negative eigenvalues into
(small) positive eigenvalues, which results in large movements in these
eigenspaces in the correct direction (away from the vertex). On the other hand
it might also fail to make large negative eigenvalues positive, turning them into
small negative eigenvalues causing big jumps towards the vertex in this
eigenspace. So if we pick \(M\) too small, this method can make things worse.
If it is too big, it will just slow down convergence and behave roughly like
gradient descent with a small learning rate. 

\subsubsection{Theoretical Guarantees}

As we have used the Lipschitz continuity to justify its bound (\ref{eq:
levenberg-marquard justification}), we are already starting out with a worse
bound than the one we have used for the convergence of gradient descent. So
this approach does not result in better convergence guarantees. It is
therefore more of a heuristic than theoretically motivated. In his motivation
for the cubic regularization Nesterov states ``Unfortunately,
none of these approaches seems to be useful in addressing the global behavior of
second-order schemes.'' \parencite[p. 242]{nesterovLecturesConvexOptimization2018}.
This also includes trust bounds (Subsection~\ref{subsec: trust bounds}) and
damped Newton (Subsection~\ref{subsec: damped newton}).

\subsubsection{Hessian Free Optimization}

Since (\ref{eq: levenberg-marquard regularization}) still requires the inversion
of a matrix, which is infeasible in high dimension,
\textcite{martensDeepLearningHessianfree2010} suggests a Hessian free variant.
He notes that the bound
\begin{align}\label{eq: levenberg-marquard upper bound}
	B_{\Loss(\weights)}(\theta) := \Loss(\weights) + \langle \nabla\Loss(\weights), \theta-\weights \rangle
	+ \tfrac12\langle [\nabla^2\Loss(\weights)+M\identity](\theta-\weights), \theta-\weights \rangle
\end{align}
we are optimizing over is a quadratic, (and hopefully convex) function and
suggests using conjugate gradient descent (Section~\ref{sec: conjugate gradient
descent}) on this bound. Conjugate gradient descent does not require inversion
of the hessian and it does not even require the hessian itself. It only
requires the ability to evaluate \(\nabla^2\Loss(\weights)v\) for any vector
\(v\). For this evaluation \textcite{martensDeepLearningHessianfree2010} suggests
using either the finite differences approximation
\begin{align*}
	\nabla^2\Loss(\weights) v
	\approx \frac{\nabla\Loss(\weights + \epsilon v) - \nabla\Loss(\weights)}{\epsilon}.
\end{align*}
or adapting automatic differentiation frameworks to this task, as discussed in
\textcite{pearlmutterFastExactMultiplication1994}. This avoids storage cost
\(O(\dimension^2)\) of the Hessian. Stopping conjugate gradient descent early,
avoids most of the computational costs as well. This is justified with the fact
that the first few iterations of conjugate gradient achieve most of the
optimization.

\subsubsection{Krylov Subspace Descent}

Let us assume we are using Hesse Free Optimization and run conjugate gradient
descent for \(m\) steps on \(\weights_n\), then we know that conjugate gradient
will find the optimal \(\weights_{n+1}\) in the Krylov subspace \(\mathcal{K}_{m-1}\)
\begin{align}\label{eq: hessian free next}
	\tag{Hessian free}
	\weights_{n+1} = \arg\min \{ B_{\Loss(\weights_n)}(\weights) : \weights \in \weights_n + \mathcal{K}_{m-1} \},
\end{align}
where \(B_{\Loss(\weights_n)}\) is the upper bound defined in (\ref{eq:
levenberg-marquard upper bound}) and the Krylov subspace is accordingly
\begin{align*}
	\mathcal{K}_m
	&= \linSpan\{
		[\nabla^2\Loss(\weights_n) + M\identity]^0 \nabla\Loss(\weights_n), \dots,
		[\nabla^2\Loss(\weights_n) + M\identity]^m \nabla\Loss(\weights_n)
	\}	\\
	&= \linSpan\{
		[\nabla^2\Loss(\weights_n)]^0 \nabla\Loss(\weights_n), \dots,
		[\nabla^2\Loss(\weights_n)]^m \nabla\Loss(\weights_n)
	\},
\end{align*}
where we get the second equation as \(M\identity \nabla\Loss(\weights_n)\) is
just a multiple of the first element and therefore the linear hulls are the
same.

Noticing this fact, \textcite{vinyalsKrylovSubspaceDescent2012} argue that we
can avoid choosing the ``correct'' regularization parameter \(M\), if we
select
\begin{align}\label{eq: krylov subspace descent}
	\tag{Krylov subspace descent}
	\weights_{n+1}
	= \arg\min \{ \Loss(\weights) : \weights \in \weights_n + \mathcal{K}_{m-1} \}
\end{align}
instead of the (\ref{eq: hessian free next}) recursion, which requires \(B_{\Loss(\weights_n)}\)
and in turn the regularization parameter \(M\) for this upper bound. Of course
at this point we can not use conjugate gradient descent anymore. They then suggest
using L-BFGS to solve this optimization problem. Why not use L-BFGS in the first
place? Well, the problem
\begin{align*}
	\min \{ \Loss(\weights_n + \weights) : \weights \in \mathcal{K}_{m-1} \}
\end{align*}
can be reparametrized using an (orthogonal) basis \(d^{(1)}, \dots, d^{(m)}\)
of \(\mathcal{K}_{m-1}\) into
\begin{align*}
	\min_{\alpha} \Loss\left(\weights_n + \sum_{k=1}^m \alpha_k d^{(k)} \right)
\end{align*}
which is easier to solve as the dimensionality of the problem is greatly
reduced for \(m\ll \dimension\). So in some sense it is similar to
line search, but for \(m\) dimensions instead of one. Krylov subspace descent
essentially bets on the conjecture that the Krylov subspaces represent the
 ``important'' directions for optimization. Just
like the gradient direction is the optimal direction for line search.

\subsection{Trust Bounds}\label{subsec: trust bounds}

Another way to make sure that we can trust the Hessian is to have literal trust
bounds, i.e.
\begin{align*}
	\min_{ \theta\in B(\weights) }
	\underbrace{
		\Loss(\weights)
		+ \langle \nabla\Loss(\weights), \theta-\weights\rangle
		+ \langle \nabla^2\Loss(\weights)(\theta-\weights), \theta-\weights\rangle
	}_{= T_2\Loss(\theta)}
\end{align*}
where \(B(\weights)\) is some trust set around \(\weights\). For the euclidean
ball
\begin{align*}
	B(\weights) = \{ \theta : \|\theta - \weights\| \le \epsilon\}
\end{align*}
this restriction can be converted into the secondary condition
\begin{align*}
	g(\theta) := \frac{\|\theta - \weights\|^2 - \epsilon^2}{2} \le 0.
\end{align*}
The Lagrangian function to this constrained optimization problem is
\begin{align*}
	L(\theta, M) 
	&= T_2\Loss(\theta) + M g(\theta)\\
	&= \Loss(\weights) -\tfrac{\epsilon^2}{2}
	+ \langle \nabla\Loss(\weights), \theta-\weights\rangle
	+ \langle [\nabla^2\Loss(\weights) + M\identity](\theta-\weights), \theta-\weights\rangle.
\end{align*}
As the constant \(\frac{\epsilon}2\) does not change this optimization problem,
its solution is the same as in Levenberg-Marquard regularization. So trust
bounds are the more general concept.

In particular \textcite{dauphinIdentifyingAttackingSaddle2014} define the bound
\begin{align*}
	B(\weights) = \{ \theta : \|\theta-\weights\|_{|H|} \le \epsilon\}
\end{align*}
where \(|H|\) is the hessian with its eigenvalues replaced by their absolute values
and \(\|x\|_{|H|}:= x^T |H| x\). They show that this trust bound results in
\begin{align*}
	\tag{saddle-free Newton}
	\weights_{n+1}	= \weights_n - |\nabla^2\Loss(\weights_n)|^{-1}\nabla\Loss(\weights_n).
\end{align*}
This method avoids moving towards saddle points and maxima as it does not flip
the direction with negative eigenvalues in these eigenspaces. Since the
calculation of the eigenvalues of the Hessian is generally intractable for
high dimensional problems, \textcite{dauphinIdentifyingAttackingSaddle2014}
use the dimension reduction idea from Krylov subspace descent to calculate
and take the absolute value of the eigenvalues of the reduced Hessian.

\subsection{Damped Newton}\label{subsec: damped newton}

``Not moving too far away'' could also be phrased as a line search along the
direction given by the Newton-Raphson method, i.e.
\begin{align*}
	\lr_n := \arg\min_\lr \Loss\left(
		\weights_n - \lr [\nabla^2\Loss(\weights_n)]^{-1}\nabla\Loss(\weights_n)
	\right)	
\end{align*}
resulting in
\begin{align}
	\tag{damped Newton}
	\weights_{n+1}
	= \weights_n - \lr_n [\nabla^2\Loss(\weights_n)]^{-1}\nabla\Loss(\weights_n).
\end{align}

\section{Gauß-Newton Algorithm}

The Gauß-Newton Algorithm opens up the black box of our loss function, to create
a relatively cheap approximation of the Hessian \(\nabla^2\Loss\).
To motivate the (classical) Gauß-Newton algorithm, we are going to assume a
squared loss
\begin{align*}
	\loss(\weights, z) := \tfrac12(\model(x) - y)^2, \quad z=(x,y).
\end{align*}
Assuming differentiation and integration can be freely swapped, we have
\begin{align}
	\label{eq: derivative gauss newton}
	\nabla\Loss(\weights)
	&= \E[\nabla\tfrac12(\model[\weights](X)-Y)^2]
	= \E[(\model[\weights](X)-Y)\nabla\model[\weights](X)],\\
	\label{eq: hessian gauss newton}
	\nabla^2\Loss(\weights)
	&=\E[\nabla[(\model[\weights](X)-Y)\nabla\model[\weights](X)]]\\
	\nonumber
	&=\E[\nabla\model(X)\nabla\model(X)^T + (\model(X)-Y)\nabla^2\model(X)].
\end{align}
If our model is close to the truth, i.e. we are close to the minimum, then
\(\model(X)-Y\) should be small. Therefore we should approximately have
\begin{align*}
	\nabla^2\Loss(\weights) \approx \E[\nabla\model(X)\nabla\model(X)^T].
\end{align*}
A different way to motivate this approximation, is to use the first taylor
approximation of \(\model\) inside the loss function, i.e.
\begin{align*}
	\Loss(\theta)
	&= \E[\loss(\theta, Z)]
	\approx \E[\tfrac12(\model(X) + \langle \nabla\model(X), \theta-\weights\rangle - Y)^2]\\
	&= \begin{aligned}[t]
		&\E[\tfrac12(\model(X)-Y)^2]\\
		&+ \E[(\model(X)-Y)\langle \nabla\model(X), \theta-\weights\rangle]\\
		&+ \E[\tfrac12\langle \nabla\model(X), \theta-\weights\rangle^2]
	\end{aligned}\\
	&= \Loss(\weights) + \langle\nabla\Loss(\weights), \theta-\weights\rangle
	+ \tfrac12 (\theta-\weights)^T
	\underbrace{\E[\nabla\model(X)\nabla\model(X)^T]}_{\approx\nabla^2\Loss(\weights)}
	(\theta-\weights)
\end{align*}
The nice thing about this approximation is, that it is always positive definite
since we have
\begin{align*}
	 v^T \E[\nabla\model(X)\nabla\model(X)^T] v
	&= \E[ v^T\nabla\model(X) \nabla\model(X)^T v]\\
	&= \E[\langle \nabla\model(X), v\rangle^2]\ge 0.
\end{align*}

\subsection{Levenberg-Marquard Algorithm}

Unfortunately, it is not necessarily strictly positive definite. That is, if there
is a vector \(v\), which is orthogonal to \(\nabla\model(X)\) for all but a
zero probability set of \(X\), then this vector is in the eigenspace with
eigenvalue zero. For the expectation one might still be able to discuss that
away. But if one optimizes with regard to a finite sample
\(\sample=(Z_1,\dots,Z_\sampleSize)\), then we have
\begin{align*}
	\nabla^2\Loss_\sample(\weights)
	\approx \E_{Z\sim\dist_\sample}[\langle \nabla\model(X), v\rangle^2]
	= \frac1\sampleSize\sum_{k=1}^\sampleSize [\langle \nabla\model(X_k), v\rangle^2]
\end{align*}
and it is much easier to find a vector orthogonal to a finite set of gradients.
Especially if we are talking about small mini-batches and do not use full batch
gradient descent.

For this reason the Gauß-Newton algorithm is usually combined with
Levenberg-Marquard regularization (Subsection~\ref{subsec: levenberg-marquard
regularization}) to allow for inversion. With this regularization it is also
called Levenberg-Marquard algorithm going back to
\textcite{levenbergMethodSolutionCertain1944} and rediscovered by
\textcite{marquardtAlgorithmLeastSquaresEstimation1963}, since that
regularization technique was developed with the Gauß-Newton algorithm in
mind.

\subsection{Generalized Gauß-Newton}

While the classical Gauß-Newton algorithm seems to be tailored to squared losses
only, this approach can actually be generalized to more general loss functions
\parencite{schraudolphFastCurvatureMatrixVector2001}. Assuming the loss function
is of the form \(\loss(f,y)\), i.e with some abuse of notation
\begin{align*}
	\loss(\weights, z) := \loss(\model(x), y), \quad z=(x,y),
\end{align*}
then we obtain similarly to (\ref{eq: derivative gauss newton}) and (\ref{eq:
hessian gauss newton})
\begin{align}
	\label{eq: derivative generalized gauss newton}
	\tfrac{d}{d\weights}\Loss(\weights)
	&= \E[\tfrac{d}{d\weights}(\model(X), Y)]
	= \E[\tfrac{d}{df}\loss(\model(X),Y)\tfrac{d}{d\weights}\model(X)],\\
	\label{eq: hessian generalized gauss newton}
	\tfrac{d^2}{d\weights^2}\Loss(\weights)
	&=\E[\tfrac{d}{d\weights}[\tfrac{d}{df}\loss(\model(X),Y)\tfrac{d}{d\weights}\model(X)]]\\
	\nonumber
	&=\begin{aligned}[t]
		&\E[\tfrac{d}{d\weights}\model(X)^T\tfrac{d^2}{df^2}\loss(\model(X),Y)\tfrac{d}{d\weights}\model(X)]\\
		&+ \E[\tfrac{d}{df}\loss(\model(X),Y)\tfrac{d^2}{d\weights^2}\model(X)].
	\end{aligned}
\end{align}
Again we throw away the part with the Hessian of the model \(\model\),
to obtain the approximation
\begin{align*}
	\nabla^2\Loss(\weights)
	&\approx\E[\tfrac{d}{d\weights}\model(X)\tfrac{d^2}{df^2}^T\loss(\model(X),Y)\tfrac{d}{d\weights}\model(X)].
\end{align*}
The result is quite similar to the classical Gauß-Newton method, except for the
Hessian of the loss function in the middle. The calculation of which might be a
huge or tiny problem. If the model prediction is one dimensional it is likely
not an issue. Then it would only be a constant.

One should still come up with some justification, why the second part of the
Hessian part can be thrown away though, i.e. why is
\begin{align*}
	\E[\tfrac{d}{df}\loss(\model(X),Y)\tfrac{d^2}{d\weights^2}\model(X)]
\end{align*}
small? In the sqared loss case, we have used the fact that
\(\tfrac{d}{df}\loss(\model(X),Y)\) was small for good models.

\section{Natural Gradient Descent}

``Natural gradient descent'' would require a deep dive into ``information geometry''
for which we do not have time. But not mentioning this algorithm at all would do it
injustice. So to avoid playing a game of telephone, here is a direct quote of
the description by \textcite{bottouOptimizationMethodsLargeScale2018}:

``We have seen that Newton's method is invariant to linear transformations of
the parameter vector w. By contrast, the natural gradient method [5, 6] aims to
be invariant with respect to all differentiable and invertible transformations.
The essential idea consists of formulating the gradient descent algorithm in the
space of prediction functions rather than specific parameters. Of course, the
actual computation takes place with respect to the parameters, but accounts for
the anisotropic relation between the parameters and the decision function. That
is, in parameter space, the natural gradient algorithm will move the parameters
more quickly along directions that have a small impact on the decision function,
and more cautiously along directions that have a large impact on the decision
function.

We remark at the outset that many authors [119, 99] propose
quasi-natural-gradient methods that are strikingly similar to the quasi-Newton
methods described in §6.2. The natural gradient approach therefore offers a
different justification for these algorithms, one that involves qualitatively
different approximations. It should also be noted that research on the design of
methods inspired by the natural gradient is ongoing and may lead to markedly
different algorithms.''

\section{Coordinate Descent}\label{sec: coordinate descent}

Coordinate descent only ever updates one coordinate. ``[I]n situations in which
\(d\) coordinate updates can be performed at cost similar to the evaluation of
one full gradient, the method is competitive with a full gradient method both
theoretically and in practice.'' \parencite[7.3, p.
72]{bottouOptimizationMethodsLargeScale2018}. This might be interesting as the
Hessian is of course much easier to handle in one dimension. A line search might
not be the right thing to though, as
\textcite{powellSearchDirectionsMinimization1973} provides an example of a three
dimensional function on which coordinate descent jumps between the
corners of the cube, never converging. This behavior can be fixed
with a stochastic selection of the coordinates, instead of cycling through them.
In fact \textcite{bubeckConvexOptimizationAlgorithms2015} argues that coordinate
descent \emph{is} just SGD with a random filter, filtering out all
but one coordinate and disappearing in expectation. I.e. for a random index \(I\)
and standard basis vectors \(\stdBasis_1,\dots,\stdBasis_\dimension\) we can
define
\begin{align*}
	\nabla\loss(\weights, i)
	:= \tfrac1{P(I=i)} \langle \nabla\Loss(\weights), \stdBasis_i \rangle \stdBasis_i
\end{align*}
to get the original loss in expectation
\begin{align*}
	\E[\nabla\loss(\weights,I)] = \sum_{k=1}^\dimension P(I=k) \nabla\loss(\weights,k)
	= \sum_{k=1}^\dimension \langle \nabla\Loss(\weights), \stdBasis_k\rangle \stdBasis_k
	= \nabla\Loss(\weights).
\end{align*}
With conditional independence the same argument can be applied to already
stochastic losses. So we can apply our SGD convergence statements to prove
convergence of random coordinate descent.


}

	\clearpage
	\appendix

\chapter{Technical Proofs}

\section{Non-Trivial Basics}

\subsection{Convex Analysis}

\begin{lemma}\label{lem-appendix: lipschitz and bounded derivative}
	If \(f\) is differentiable, then the derivative \(\nabla f\) is
	bounded (w.r.t. the operator norm) by constant \(\lipConst\) if and only if the function
	\(f\) is \(\lipConst\)-Lipschitz continuous.
\end{lemma}
\begin{proof}
	Lipschitz continuity is implied by the mean value theorem
	\begin{align*}
		\|f(x_1) - f(x_0)\|
		&\le \|\nabla f(x_0 + \xi(x_1-x_0))\| \|x_1- x_0\|\\
		&\le \lipConst \|x_1-x_0\|.
	\end{align*}
	The opposite direction is implied by
	\begin{align*}
		\|\nabla f(x_0)\|
		&\equiv \sup_{v} \frac{\|\nabla f(x_0)v\|}{\|v\|}
		= \sup_{v} \lim_{\lambda\to 0}\frac{|\lambda|\|\nabla f(x_0) v\|}{|\lambda|\|v\|}\\
		&\lxle{\Delta}\sup_{v} \lim_{\lambda\to 0}
		(
			\underbrace{
				\tfrac{\|\nabla f(x_0)\lambda v + f(x_0) - f(x_0 +\lambda v)\|}{\|\lambda v\|}
			}_{
				\to 0 \text{ (derivative definition)}
			}
			+ \underbrace{
				\tfrac{\|f(x_0 + \lambda v) - f(x_0) \|}{\|\lambda v\|}
			}_{
				\le \lipConst 
			}
		)
	\end{align*}
	where we have used the scalability of the norm to multiply \(v\) with a
	decreasing factor both in the numerator and denominator in order to introduce
	the limit.
\end{proof}

\begin{lemma}
	\label{Appdx-lem: Lipschitz Gradient implies taylor inequality}
	If \(\nabla f\) is \(\ubound\)-Lipschitz continuous, then
	\begin{align*}
		|f(y) - f(x) - \langle \nabla f(x), y-x\rangle | \le \tfrac{\ubound}2 \|y-x\|^2
	\end{align*}
	If \(f\) is convex, then the opposite direction is also true.
\end{lemma}
\begin{proof}
	The first direction is taken from \textcite[Lemma
	1.2.3]{nesterovLecturesConvexOptimization2018}.
 \begin{align*}
		f(y) = f(x) + \int_0^1\langle\nabla f(x+\tau(y-x)), y-x \rangle d\tau
	\end{align*}
	implies (using the  Cauchy-Schwarz inequality)
	\begin{align*}
		&| f(y) - f(x) - \langle \nabla f(x), y-x\rangle | \\
		&\le \int_0^1 | \langle\nabla f(x+\tau(y-x))-\nabla f(x), y-x\rangle | d\tau \\
		&\lxle{\text{C.S.}}
		\int_0^1 \|\langle\nabla f(x+\tau(y-x))-\nabla f(x)\| \cdot \|y-x\| d\tau\\
		&\le \int_0^1 \ubound \|\tau(y-x)\|\cdot\|y-x\| d\tau
		= \tfrac{\ubound}2 \|y-x\|^2.
	\end{align*}
	The opposite direction is taken from \textcite[Lemma
	2.1.5]{nesterovLecturesConvexOptimization2018}.
	As we have only used
	\begin{align*}
		0\xle{\text{Convexity}} f(y) - f(x) - \langle \nabla f(x), y-x\rangle
		\le \tfrac{\ubound}2 \|y-x\|^2
	\end{align*}	
	in Lemma~\ref{lem: bregmanDiv lower bound} (and Lemma~\ref{lem: smallest upper
	bound} which was used in the proof), we know that equation (\ref{eq:
	bregmanDiv lower bound b}) follows from it and after applying Cauchy-Schwarz
	to this equation
	\begin{align*}
		\tfrac{1}\ubound \|\nabla f(x)-\nabla f(y)\|^2
		&\le \langle \nabla f(x) - \nabla f(y), x-y\rangle \\
		&\lxle{\text{C.S.}} \|\nabla f(x) - \nabla f(y)\| \|x-y\|,
	\end{align*}
	we only have to divide both sides by \(\tfrac{1}\ubound \|\nabla f(x)-\nabla
	f(y)\|\) to get \(\ubound\)-Lipschitz continuity back.
\end{proof}

\subsection{Sequences}

\begin{theorem}[Cesàro-Stolz]\label{thm-appendix: cesaro-stolz}
	Let \((x_n, n\in\naturals),(y_n, n\in\naturals)\) be real sequences with
	\(\lim_{n\to\infty}y_n\to\infty\), then for \(\Delta x_n:=x_{n+1}-x_n\)
	\begin{align}
		\liminf_n \frac{\Delta x_n}{\Delta y_n}
		\le \liminf_n \frac{x_n}{y_n} 
		\le \limsup_n \frac{x_n}{y_n} 
		\le \limsup_n \frac{\Delta x_n}{\Delta y_n}
	\end{align}
	In particular if the limit \(\lim_n\frac{\Delta x_n}{\Delta y_n}\) exists
	we have equality. Thereby it becomes a discrete version of L'Hôpital's
	rule.
\end{theorem}
\begin{proof}[{Proof \parencite{IMOmathHopitalTheorem}}]
	To show the first inequality let us assume we have
	\begin{align*}
		\gamma < \liminf_n \frac{\Delta x_n}{\Delta y_n}.
	\end{align*}
	Then by definition of the limes inferior there exists \(N\) such that
	\begin{align*}
		\gamma \Delta y_n < \Delta x_n \qquad \forall n\ge N
	\end{align*}
	Using telescoping sums this implies
	\begin{align*}
		\gamma (y_m- y_N)
		= \gamma \sum_{n=N}^{m-1} \Delta y_n
		< \sum_{n=N}^{m-1} \Delta x_n
		= x_m- x_N.
	\end{align*}
	Dividing both sides by \(y_m\) and taking the limes inferior keeping in
	mind that \(y_m\to\infty\) results in
	\begin{align*}
		\gamma = \liminf_m \gamma\left(\frac{y_m}{y_m} - \frac{y_N}{y_m}\right)
		< \liminf_m \frac{x_m}{y_m} - \frac{x_N}{y_m}
		= \liminf_m \frac{x_m}{y_m}.
	\end{align*}
	Since \(\gamma\) was arbitrary this proves the first inequality. The second
	inequality is true by definition of the limes inferior (and superior) and the
	third inequality can be proven in the same way as the first starting with
	\(\limsup_n \frac{\Delta x_n}{\Delta y_n} < \gamma\).
\end{proof}
\begin{lemma}
	\label{lem-appendix: diminishing contraction}
	Let \(a_0 \in [0, 1/q)\) for \(q>0\) and assume 
	\begin{align*}
		a_{n+1} = (1-q a_n)a_n \quad \forall n \ge 0,
	\end{align*}
	then we have
	\begin{align}
		\lim_{n\to\infty} n a_n = 1/q
	\end{align}
\end{lemma}
\begin{proof}
	The proof follows \textcite{israelHowWorkOut2012}. By induction we have
	\begin{align*}
		0\le a_n \le a_{n-1} < 1/q.
	\end{align*}
	Therefore \(a_n\) is a falling bounded sequences and thus converges to
	some \(a\in[0,1/q)\) which necessarily fulfills
		\(a(1-qa) = a\)
	which implies
	\begin{align*}
		\lim_{n\to\infty} a_n = a = 0
	\end{align*}
	Thus for \(y_n := 1/a_n\) and \(x_n:=n\) we have
	\begin{align*}
		\frac{\Delta x_n}{\Delta y_n}
		= \frac{1}{\frac{1}{a_{n+1}}-\frac1{a_n}}
		= \frac{a_{n+1}a_n}{a_n - a_{n+1}}
		= \frac{(1-q a_n)a_n}{1 - (1-q a_n)}
		= \frac{1-q a_n}{q} \to 1/q
	\end{align*}
	and since \(y_n\to\infty\) we can apply Theorem~\ref{thm-appendix: cesaro-stolz}
	to get
	\begin{align*}
		\lim_{n\to\infty} n a_n
		&= \lim_{n\to\infty} \frac{x_n}{y_n}
		= \lim_{n\to\infty} \frac{\Delta x_n}{\Delta y_n}
		= 1/q.
		\qedhere
	\end{align*}
\end{proof}

\begin{lemma}[Discrete Gr\"onwall]\label{lem-appendix: discrete gronwall}
	Let \(g_n, a_n, b_n \ge 0\), \(g_0=0\) and
	\begin{align*}
		g_{n+1} \le a_n g_n + b_n
	\end{align*}
	then we have
	\begin{align*}
		g_n \le \sum_{k=0}^{n-1}b_k\prod_{i=k+1}^{n-1}a_i
	\end{align*}
\end{lemma}
\begin{proof}
	By induction where \(n=0\) is trivial interpreting the empty sum as zero, we have	
	\begin{align*}
		g_{n+1}
		&\le a_n g_n + b_n
		\xle{\text{ind.}} a_n \left(\sum_{k=0}^{n-1}b_k\prod_{i=k+1}^{n-1}a_i\right) + b_n
		\le \sum_{k=0}^{n}b_k\prod_{i=k+1}^{n}a_i
		\qedhere
	\end{align*}
\end{proof}

\subsection{Operators}

\begin{lemma}[Absolute Values Inside Operator Norms]
	\label{lem-appdx: absolute value inside operator norms}
	For complex matrices \(A^{(1)},\dots,A^{(n)}\in\complex^{\dimension\times\dimension}\)
	we define \(B^{(1)},\dots,B^{(n)}\in\complex^{\dimension\times\dimension}\)
	with
	\begin{align*}
		B^{(k)}_{ij} := |A^{(k)}_{ij}| \qquad \forall k,i,j.
	\end{align*}
	Then we have
	\begin{align*}
		\Bigg\|\prod_{k=1}^n A^{(k)}\Bigg\| \le \Bigg\|\prod_{k=1}^n B^{(k)}\Bigg\|.
	\end{align*}
\end{lemma}
\begin{proof}
	Let \(x\in\complex^\dimension\) be an arbitrary vector and define \(y\in\complex^\dimension\)
	with
	\begin{align*}
		y_i = |x_i| \qquad \forall i=1,\dots,\dimension.
	\end{align*}
	Then we have
	\begin{align*}
		\Bigg\| \prod_{k=1}^n A^{(k)} x \Bigg\|
		&= \sum_{j=1}^{d} \Big| \sum_{1\le i_1,\dots, i_n \le d}
		A^{(1)}_{j i_1} A^{(2)}_{i_1 i_2} \dots A^{(n)}_{i_{n-1} i_n} x_{i_n} \Big|^2
		\\
		&\le \sum_{j=1}^{d} \Big| \sum_{1\le i_1,\dots, i_n \le d}
		B^{(1)}_{j i_1} B^{(2)}_{i_1 i_2} \dots B^{(n)}_{i_{n-1} i_n} y_{i_n} \Big|^2
		\\
		&= \Bigg\| \prod_{k=1}^n B^{(k)} y \Bigg\|
	\end{align*}
	Since we have \(\|x\|=\|y\|\) as the norm takes the absolute value anyway, we
	get our claim
	\begin{align*}
		\Bigg\|\prod_{k=1}^n A^{(k)}\Bigg\|
		&= \sup_{\|x\|=1}
		\Bigg\|\prod_{k=1}^n A^{(k)} x\Bigg\|
		\le \sup_{\|y\|=1}
		\Bigg\|\prod_{k=1}^n B^{(k)} y\Bigg\|
		= \Bigg\|\prod_{k=1}^n B^{(k)}\Bigg\|.
		\qedhere
	\end{align*}
\end{proof}

\section{Momentum Convergence}

\subsection{Eigenvalue Analysis}

\begin{theorem}[\cite{qianMomentumTermGradient1999}]
	\label{thm-appdx: momentum - stable set of parameters}
	Let
	\begin{align*}
		\momEV_{1/2}
		= \tfrac12 \left(
			1+\momCoeff-\lrSq\hesseEV \pm \sqrt{(1+\momCoeff-\lrSq\hesseEV)^2 - 4\momCoeff}
		\right)
	\end{align*}
	then 
	\begin{enumerate}
		\item \(\max\{|\momEV_1|,|\momEV_2|\}<1\) if and only if
		\begin{align*}
			0<\lrSq\hesseEV < 2(1+\momCoeff) \qquad \text{and} \qquad |\momCoeff|<1
		\end{align*}
		\item The complex case can be characterized by either
		\begin{align*}
			0<(1-\sqrt{\lrSq\hesseEV})^2 < \momCoeff < 1
		\end{align*}		
		or alternatively \(\momCoeff>0\) and
		\begin{align*}
			(1-\sqrt{\momCoeff})^2 < \lrSq\hesseEV < (1+\sqrt{\momCoeff})^2,
		\end{align*}
		for which we have \(|\momEV_1|=|\momEV_2|=\sqrt{\momCoeff}\).
		
		\item In the real case we have \(\momEV_1>\momEV_2\) and
		\begin{align}\label{eq: when does sigma_1 or sigma_2 dominate?}
			\max\{|\momEV_1|, |\momEV_2|\} = \begin{cases}
				|\momEV_1|=\momEV_1 & \lrSq\hesseEV < 1+\momCoeff \\
				|\momEV_2|=-\momEV_2 & \lrSq\hesseEV \ge 1+\momCoeff.
			\end{cases}
		\end{align}
		Restricted to \(1>\momCoeff>0\) this results in two different	
		behaviors. For
		\begin{align*}
			0<\lrSq\hesseEV \le (1-\sqrt{\momCoeff})^2 < 1+\momCoeff
		\end{align*}
		we have \(1>\momEV_1 > \momEV_2 > 0\). For
		\begin{align*}
			1+\momCoeff < (1+\sqrt{\momCoeff})^2\le \lrSq\hesseEV < 2(1+\momCoeff)
		\end{align*}
		on the other hand, we get \(-1 < \momEV_2 < \momEV_1 < 0\).
	\end{enumerate}
\end{theorem}
\begin{proof}
	Define
	\begin{align*}
		\Delta(\momCoeff)
		&:= (1+\momCoeff-\lrSq\hesseEV)^2 - 4\momCoeff\\
		&= \momCoeff^2 - 2(1+\lrSq\hesseEV)\momCoeff + (1-\lrSq\hesseEV)^2
	\end{align*}
	then \(\momEV_{1/2}\) is complex iff \(\Delta(\momCoeff)<0\). Since it is a
	convex parabola this implies that \(\momCoeff\) needs to be between the roots
	\begin{align*}
		\momCoeff_{1/2}
		= (1+\lrSq\hesseEV) \pm 
		\underbrace{
			\sqrt{(1+\lrSq\hesseEV)^2 - (1-\lrSq\hesseEV)^2}
		}_{=\sqrt{4\lrSq\hesseEV}=\mathrlap{2\sqrt{\lrSq\hesseEV}}}
		= (1\pm \sqrt{\lrSq\hesseEV})^2	
	\end{align*}
	of \(\Delta\). Assuming those roots are real, which requires \(\lrSq\hesseEV>0\).
	But if \(\lrSq\hesseEV\le0\) then we would have
	\begin{align*}
		\momEV_1
		= \tfrac12 \Big(
			\underbrace{1+\momCoeff+|\lrSq\hesseEV|}_{\ge 1+\momCoeff}
			+ \sqrt{\underbrace{(1+\momCoeff+|\lrSq\hesseEV|)^2 - 4\momCoeff}_{\smash{\ge (1-\momCoeff)^2}}}
		\Big)
		\ge \frac{1+\momCoeff + |1-\momCoeff|}2 \ge 1
	\end{align*}
	which means that \(\max\{|\momEV_1|,|\momEV_2|\}<1\) implies \(\lrSq\hesseEV>0\).
	The roots \(\momCoeff_{1/2}\) are therefore real no matter the direction we
	wish to prove.
	\begin{description}[wide, labelindent=0pt]
	\item[Complex Case:]
		\(\momEV_{1/2}\) are complex iff	
		\begin{align}\label{eq: complex case}
			0 < (1-\sqrt{\lrSq\hesseEV})^2 < \momCoeff < (1+\sqrt{\lrSq\hesseEV})^2.
		\end{align}
		In that case we have
		\begin{align*}
			|\momEV_1| = |\momEV_2|
			&= \sqrt{|\Re(\momEV_1)|^2 + |\Im(\momEV_1)|^2}\\
			&= \tfrac12 \sqrt{(1+\momCoeff-\lrSq\hesseEV)^2 + 4\momCoeff - (1+\momCoeff-\lrSq\hesseEV)^2}
			= \sqrt{\momCoeff}.
		\end{align*}
		So the condition \(\momCoeff<1\) is necessary and sufficient for
		\(\max\{|\momEV_1|,|\momEV_2|\} <1\) in the complex case.
		And since we have seen that \(\lrSq\hesseEV>0\) is necessary and \(0<\momCoeff\)
		from (\ref{eq: complex case}) implies the condition \(|\momCoeff|<1\), we only
		need to show that \(\lrSq\hesseEV <2(1+\momCoeff)\) is necessary in the complex
		case to have this case covered. Using \((\sqrt{\lrSq\hesseEV}-1)^2 < \momCoeff\)
		from (\ref{eq: complex case}) and \(ab \le a^2 + b^2\) we get
		\begin{align*}
			\sqrt{\lrSq\hesseEV} - 1
			< \momCoeff \implies \lrSq\hesseEV < (1+\sqrt{\momCoeff})^2
			\le 2(1+\momCoeff).
		\end{align*}

		The remaining characterization of the complex case follows from
		restricting (\ref{eq: complex case}) to \(\momCoeff<1\) and reordering.
		\item[Real Case:] Since we have \(\momEV_2 \le \momEV_1\),
		\(\max\{|\momEV_1|,|\momEV_2|\}<1\) is equivalent to
		\begin{align*}
			-1 < \momEV_2 \qquad \text{and} \qquad \momEV_1 < 1
		\end{align*}
		By subtracting the part before the "\(\pm\)" from the equation \(\momEV_1<1\)
		after multiplying it by two, we get
		\begin{align}\label{eq: sigma1 condition}
			0\xle{\text{Real Case}} \sqrt{(1+\momCoeff-\lrSq\hesseEV)^2 -4\momCoeff} < 1-\momCoeff +\lrSq\hesseEV.
		\end{align}
		And since \((1+\sqrt{\lrSq\hesseEV})^2 \le \momCoeff\) leads to a contradiction
		\begin{align*}
			0 \le 1-\momCoeff+\lrSq\hesseEV \le 1-(1+\sqrt{\lrSq\hesseEV})^2 +\lrSq\hesseEV
			= -2\sqrt{\lrSq\hesseEV} \ge 0,
		\end{align*}
		the real case is restricted to the case \(\momCoeff < (1-\sqrt{\lrSq\hesseEV})^2\)
		as we are otherwise in the complex case (\ref{eq: complex case}).
		The inequality (\ref{eq: sigma1 condition}) is (in the real case) equivalent to
		\begin{align*}
			&(1+\momCoeff-\lrSq\hesseEV)^2 - 4\momCoeff < (1-\momCoeff+\lrSq\hesseEV)^2\\
			&\iff \cancel{1^2} + 2(\momCoeff-\lrSq\hesseEV) + 
			\cancel{(\momCoeff-\lrSq\hesseEV)^2} - 4\momCoeff
			< \cancel{1^2} - 2(\momCoeff-\lrSq\hesseEV) + 
			\cancel{(\momCoeff-\lrSq\hesseEV)^2}\\
			&\iff 0 < 4\lrSq\hesseEV.
		\end{align*}
		Which is a given for positive learning rates in the convex	
		case. This is not surprising if one remembers the proof from gradient
		descent. The real issue is not overshooting \(-1\) with large learning
		rates. So multiplying \(-1<\momEV_2\) by two, adding 2 and moving the
		root to the other side we get
		\begin{align}\label{eq: from -1<sigma_2 followed}
			0 \xle{\text{Real Case}} \sqrt{(1+\momCoeff-\lrSq\hesseEV)^2 - 4\momCoeff}
			< 3+\momCoeff-\lrSq\hesseEV.
		\end{align}
		From this we get \(\lrSq\hesseEV < 3+\momCoeff\). We also get
		\begin{align}
			\nonumber
			&(1+\momCoeff-\lrSq\hesseEV)^2 - 4\momCoeff < (3+\momCoeff-\lrSq\hesseEV)^2\\
			\nonumber
			&\iff 1^2 + 2(\momCoeff-\lrSq\hesseEV) + \cancel{(\momCoeff-\lrSq\hesseEV)^2} - 4\momCoeff
			< 3^2 + 6(\momCoeff-\lrSq\hesseEV) + \cancel{(\momCoeff-\lrSq\hesseEV)^2}\\
			\label{eq: momentum upper bound derivation}
			&\iff 4\lrSq\hesseEV < 8 + 8\momCoeff
			\iff \lrSq\hesseEV < 2(1+\momCoeff)
		\end{align}
		which is the important requirement on the learning rate we are looking for.
		Now for \(\momCoeff<1\) this requirement is actually stronger than
		\(\lrSq\hesseEV <3+\momCoeff\) which means that we have proven that our
		requirements are sufficient for \(\max\{|\momEV_1|,|\momEV_2|\}<1\).
		What is left to show is that \(|\momCoeff|<1\) is also necessary in the real
		case. From (\ref{eq: momentum upper bound derivation}) we immediately
		get
		\begin{align*}
			0 \le \tfrac{\lrSq\hesseEV}2 < 1+\momCoeff
		\end{align*}
		which implies \(-1 < \momCoeff\). At first glance it might look like
		\(\momCoeff<(1-\sqrt{\lrSq\hesseEV})^2\) is already sufficient for the upper
		bound, but for \(\lrSq\hesseEV>4\) this bound is greater than one again.
		So let us assume \(\lrSq\hesseEV>2\) and \(\momCoeff>1\), then
		\(\momCoeff<(1-\sqrt{\lrSq\hesseEV})^2\) implies a contradiction:
		\begin{align*}
			\sqrt{\momCoeff} < \sqrt{\lrSq\hesseEV} -1
			&\implies \underbrace{(1+\sqrt{\momCoeff})^2}_{1+2\sqrt{\momCoeff} + \momCoeff}
			< \lrSq\hesseEV
			\stackrel{(\ref{eq: from -1<sigma_2 followed})}{<} 3+\momCoeff \\
			&\implies -2 +2\sqrt{\momCoeff} < 0
			\implies \sqrt{\momCoeff} < 1
		\end{align*}

		For the \emph{characterization of the real case} one only needs to observe
		that for any \(b>0\)
		\begin{align}
			\max|a\pm b| = \begin{cases}
				a + b & a \ge 0\\
				-(a-b) & a < 0
			\end{cases}.
		\end{align}
		In our case \(b\) is a square root in the real case, so by definition
		positive and \(a>0\) is equivalent to \(1+\momCoeff - \lrSq\hesseEV >0\)
		and thus implies (\ref{eq: when does sigma_1 or sigma_2 dominate?}).
		
		For \(\momCoeff>0\) one only needs to keep in mind that for
		 \(\lrSq\hesseEV < (1+\sqrt{\momCoeff})^2 <1+\momCoeff\)
		\begin{align*}
			1+\momCoeff-\lrSq\hesseEV = \sqrt{(1+\momCoeff-\lrSq\hesseEV)^2}
			> \sqrt{(1+\momCoeff-\lrSq\hesseEV)^2 - 4\momCoeff}
		\end{align*}
		ensures that \(\momEV_2 >0\). On the other hand for \(\momCoeff>0\)
		and \(\lrSq\hesseEV > (1-\sqrt{\momCoeff})^2> 1+\momCoeff\)
		\begin{align*}
			- (1+\momCoeff-\lrSq\hesseEV) = \sqrt{(1+\momCoeff-\lrSq\hesseEV)^2}
			> \sqrt{(1+\momCoeff-\lrSq\hesseEV)^2 - 4\momCoeff}
		\end{align*}
		implies \(\momEV_1 < 0\). \qedhere
 \end{description}
\end{proof}

\begin{lemma}\label{lem-appdx: just at the border of complex case is best beta}
	\begin{align*}
		\frac{d}{d\momCoeff}\max\{|\momEV_1|,|\momEV_2|\} < 0 \qquad
		\forall -1<\momCoeff <(1-\sqrt{\lrSq\hesseEV})^2
	\end{align*}
\end{lemma}
\begin{proof}
	Using the characterizations of the real case
	from Theorem~\ref{thm: momentum - stable set of parameters}, we get
	\begin{align*}
		\frac{d}{d\momCoeff}\max\{|\momEV_1|,|\momEV_2|\}
		= \begin{cases}
			\frac{d\momEV_1}{d\momCoeff} & \momCoeff \ge \lrSq\hesseEV -1 \\
			-\frac{d\momEV_2}{d\momCoeff} & \momCoeff \le \lrSq\hesseEV -1 
		\end{cases}
	\end{align*}
	This allows us to treat those cases separately. For the first case we use
	the same definition for \(\Delta\) as in Theorem~\ref{thm: momentum - stable
	set of parameters}. Calculating its derivative
	\begin{align*}
		\Delta'(\momCoeff) = 2(1+\momCoeff-\lrSq\hesseEV)	-4
	\end{align*}
	allows us to represent \(\momEV_1\) as
	\begin{align*}
		1>\momEV_1
		= 1 + \frac{\Delta'(\momCoeff)}{4} + \tfrac12 \sqrt{\Delta(\momCoeff)},
	\end{align*}
	which results in the inequality
	\begin{align*}
		0 > \tfrac14\Delta'(\momCoeff) + \tfrac12\sqrt{\Delta(\momCoeff)}
		\implies -\tfrac12 > \tfrac14\frac{\Delta'(\momCoeff)}{\sqrt{\Delta{\momCoeff}}}
	\end{align*}
	We can now use this inequality to show that\footnote{
		I want to apologize for the lack of insight in this proof upon which I
		stumbled by accident playing around with the \(\Delta(\momCoeff)\)
		representation of \(\momEV_{1/2}\). Its only redeeming merit is its brevity.
	}
	\begin{align*}
		\frac{d\momEV_1}{d\momCoeff}
		= \frac12 + \frac14 \frac{\Delta'(\momCoeff)}{\sqrt{\Delta(\momCoeff)}} < 0.
	\end{align*}
	The second case on the other hand requires us to show that
	\begin{align*}
		\frac{d\momEV_2}{d\momCoeff}
		= \frac12 - \frac14 \frac{\Delta'(\momCoeff)}{\sqrt{\Delta(\momCoeff)}} > 0
		\iff 2 > \frac{\Delta'(\momCoeff)}{\sqrt{\Delta(\momCoeff)}}.
	\end{align*}
	Which is true due to
	\begin{align*}
		2\sqrt{\Delta(\momCoeff)} > 0 > -4\sqrt{\lrSq\hesseEV}
		&= 2\overbrace{
			(1-\lrSq\hesseEV + (1-\sqrt{\lrSq\hesseEV})^2)
		}^{\smash{=2-2\sqrt{\lrSq\hesseEV}}} -4\\
		&= \Delta'((1-\sqrt{\lrSq\hesseEV})^2) > \Delta'(\momCoeff),
	\end{align*}
	since \(\momCoeff<(1-\sqrt{\lrSq\hesseEV})^2\) and \(\Delta'\) is monotonously
	increasing in \(\momCoeff\).
\end{proof}

\subsection{Nesterov's Momentum Convergence}

\begin{lemma}[Dumpster Dive]
	\label{lem-appendix: dumpster dive}
	The conditions
	\begin{align*}
		\ubound\gamma_n^2 \le \lbound_n
	\end{align*}
	and
	\begin{align*}
		y_n
		&= \tfrac{\lbound_n\weights_{n-1} + \lbound_{n-1}\gamma_n z_{n-1}}{\gamma_n\lbound + \lbound_{n-1}}
	\end{align*}
	are sufficient for lower bounding
	\begin{align*}
		J := (1-\gamma_n)\left(
			\langle \Loss(y_n), \weights_{n-1} - y_n\rangle
			+ \tfrac{\lbound_{n-1}}{2}\|y_n - z_{n-1}\|^2
		\right)
		- \tfrac{\lbound_n}{2}\|y_n - z_n\|^2
	\end{align*}
	by \(- \tfrac{1}{2\ubound}\|\nabla\Loss(y_n)\|^2\).
\end{lemma}
\begin{proof}
	By definition \(z_n\) is the unique minimum of \(\Phi_n\) so we can find
	it with
	\begin{align*}
		0\xeq{!}\Phi'_n(z_n)
		= \gamma_n\underbrace{\phi_n'(z_n)}_{
			= \nabla\Loss(y_n) \mathrlap{+ \lbound(z_n-y_n)}
		} +(1-\gamma_n)
		\underbrace{\Phi'_{n-1}(z_n)}_{
			=\lbound_{n-1}\mathrlap{(z_n-z_{n-1})}
		},
	\end{align*}
	where we have used the representation (\ref{eq: centered Phi representation})
	for \(\Phi_{n-1}\). It is also quite reassuring that the gradient at \(y_n\)
	which we ultimately need makes its first appearance here. Collecting all
	the \(z_n\) we get
	\begin{align}\label{eq: center of Phi recursion}
		\overbrace{\lbound_n}^{
			=\mathrlap{\gamma_n\lbound + (1-\gamma_n)\lbound_{n-1}}
		}z_n
		= \gamma_n\lbound y_n + (1-\gamma_n)\lbound_n z_{n-1}
		- \gamma_n\nabla\Loss(y_n).
	\end{align}
	Subtracting \(\lbound_n y_n\) results in
	\begin{align*}
		\lbound_n(z_n-y_n)
		= (1-\gamma_n)\lbound_{n-1}(z_{n-1}-y_n)
		- \gamma_n\nabla\Loss(y_n).
	\end{align*}
	Plugging this into the last part of our ``junk'', we get
	\begin{align*}
		\tfrac{\lbound_n}2\|y_n-z_n\|^2
		&= \tfrac1{2\lbound_n}
		\underbrace{
			\|{\scriptstyle (1-\gamma_n)\lbound_{n-1}}(z_{n-1}-y_n)
			- {\scriptstyle \gamma_n}\nabla\Loss(y_n)\|^2
		}_{
		\begin{aligned}
			=&{\scriptstyle (1-\gamma_n)^2\lbound_{n-1}^2}\|z_{n-1}-y_n\|^2 \\
			&- {\scriptstyle 2(1-\gamma_n)\lbound_{n-1}\gamma_n}\langle
			\nabla\Loss(y_n), z_{n-1}-y_n\rangle \\
			&+ {\scriptstyle\gamma_n^2}\|\nabla\Loss(y_n)\|^2
		\end{aligned}
		}
	\end{align*}
	Therefore our ``junk'' is equal to
	\begin{align*}
		-\tfrac{\gamma_n^2}{2\lbound_n}\|\nabla\Loss(y_n)\|^2
		\begin{aligned}[t]
			&+ {\scriptstyle(1-\gamma_n)}\left\langle \Loss(y_n),
			(\weights_n -y_n)
			+ \tfrac{\lbound_{n-1}\gamma_n}{\lbound_n}(z_{n-1} - y_n)
			\right\rangle\\
			&+\tfrac{(1-\gamma_n)\lbound_{n-1}}{2}
			\underbrace{\left(1-\tfrac{(1-\gamma_n)\lbound_{n-1}}{\lbound_n}\right)}_{
				= \tfrac{\gamma_n \lbound}{\lbound_n} \ge 0
			}
			\|y_n - z_{n-1}\|^2.
		\end{aligned}
	\end{align*}
	As the last summand is positive we can discard it and end up with two equations
	sufficient to achieve our lower bound
	\begin{align*}
		-\tfrac{\gamma_n^2}{\lbound_n}
		&\ge -\tfrac{1}{\ubound} \iff \ubound\gamma_n^2 \le \lbound_n
		\\
		0
		&= (\weights_{n-1} -y_n) + \tfrac{\lbound_n\gamma_n}{\lbound_n}(z_{n-1} - y_n)
		\xiff{\text{reorder}}
		y_n
		= \tfrac{\lbound_n\weights_{n-1} + \lbound_{n-1}\gamma_n z_{n-1}}{\gamma_n\lbound + \lbound_{n-1}}
		\qedhere
	\end{align*}
\end{proof}

\begin{lemma}\label{lem-appendix: streamline general schema of optimal methods}
	If we use the recursion
	\begin{align}
		\weights_n = y_n - \tfrac1\ubound\Loss(y_n)
	\end{align}
	in Algorithm~\ref{algo: general schema of optimal methods} line~\ref{line:
	selecting next position}, then we can simplify \(z_n\) and \(y_n\)
	to
	\begin{align}
		\label{eq: z_n recursion}
		z_n
		&= \weights_{n-1} + \tfrac1{\gamma_n}(\weights_n-\weights_{n-1})
		\qquad \forall n\ge 1\\
		y_{n+1}
		&=\weights_n + \momCoeff_n(\weights_n-\weights_{n-1}) \qquad \forall n\ge 0
	\end{align}
	for \(\weights_{-1}:=\weights_0=z_0\) and
	\begin{align}
		\momCoeff_n
		= \frac{\lbound_n \gamma_{n+1}(1-\gamma_n)}{\gamma_n(\gamma_{n+1}\lbound +\lbound_n)}
		= \frac{\gamma_n(1-\gamma_n)}{\gamma_n^2 + \gamma_{n+1}}
	\end{align}
\end{lemma}
\begin{proof}
	After the assignment in line~\ref{line: definition of y_n} we can reorder it
	as an equation to get
	\begin{align*}
		\lbound_{n-1} z_{n-1}
		= \tfrac{1}{\gamma_n} [(\gamma_n\lbound + \lbound_{n-1})y_n - \lbound_n\weights_{n-1}]
	\end{align*}
	plugging this into our definition of \(z_n\) results in our first claim
	\begin{align*}
		z_n
		&= \tfrac{1}{\lbound_n}
		\left[
			\gamma_n\lbound y_n
			+ \tfrac{(1-\gamma_n)}{\gamma_n}
			[(\gamma_n\lbound + \lbound_{n-1})y_n - \lbound_n\weights_{n-1}]
			- \gamma_n\nabla\Loss(y_n)
		\right]\\
		&=\tfrac{1}{\lbound_n} \underbrace{\Big[
			\underbrace{\gamma_n\lbound + (1-\gamma_n)\lbound}_{=\lbound}
			+ \tfrac{1-\gamma_n}{\gamma_n}\lbound_{n-1}
		\Big]}_{
			=\lbound_n/\gamma_n
		}y_n
		-\tfrac{1-\gamma_n}{\gamma_n}\weights_{n-1}
		-\underbrace{\tfrac{\gamma_n}{\lbound_n}}_{
			=\frac{1}{\ubound\gamma_n} \mathrlap{\ (L\gamma_n^2 = \lbound_n)}
		}\nabla\Loss(y_n)\\
		&= \weights_{n-1} + \tfrac{1}{\gamma_n}(y_n - \weights_{n-1})
		- \tfrac{1}{\ubound\gamma_n}\nabla\Loss(y_n)\\
		&\lxeq{(\ref{eq-lem-assmpt: weigth recursion})}
		\weights_{n-1} + \tfrac{1}{\gamma_n}(\weights_n-\weights_{n-1})
	\end{align*}
	Now we will rewrite \(y_{n+1}\):
	\begin{align*}
		y_{n+1}
		&= \frac{
			\overbrace{[\gamma_{n+1} \lbound + (1-\gamma_{n+1})\lbound_n]}^{\lbound_{n+1}}\weights_n
			+ \lbound_n \gamma_{n+1} z_n
		}{\gamma_{n+1}\lbound +\lbound_n}
		=\weights_n
		+ \frac{\lbound_n \gamma_{n+1} (z_n - \weights_n)}{\gamma_{n+1}\lbound +\lbound_n}
	\end{align*}
	For \(n\ge1\) we can apply our previous result (\ref{eq: z_n recursion})
	\begin{align*}
		z_n - \weights_n
		= \left(\tfrac1{\gamma_n} -1\right)(\weights_n-\weights_{n-1}),
	\end{align*}
	for \(n=0\) both sides are zero and we therefore obtain our second recursion
	\begin{align*}
		y_{n+1}
		=\weights_n
		+ \underbrace{
			\frac{\lbound_n \gamma_{n+1}(1-\gamma_n)}{\gamma_n(\gamma_{n+1}\lbound +\lbound_n)}
		}_{=\momCoeff_n}
		(\weights_n-\weights_{n-1}).
	\end{align*}
	Now we just use the identity \(\gamma_{n+1}\lbound = \ubound\gamma_{n+1}^2 -
	(1-\gamma_{n+1})\lbound_n\) from the definition of \(\gamma_{n+1}\) to simplify our
	momentum coefficient
	\begin{align*}
		\momCoeff_n
		&= \frac{\lbound_n \gamma_{n+1}(1-\gamma_n)}{
			\gamma_n([\ubound\gamma_{n+1}^2 - (\cancel{1}-\gamma_{n+1})\lbound_n] +\cancel{\lbound_n})
		}
		= \frac{\lbound_n\cancel{\gamma_{n+1}}(1-\gamma_n)}{
			\gamma_n\cancel{\gamma_{n+1}}(\ubound\gamma_{n+1} + \underbrace{\lbound_n}_{=\ubound\gamma_n^2})
		}\\
		&\lxeq{\bcancel{\ubound\gamma_n}} \frac{\gamma_n(1-\gamma_n)}{\gamma_{n+1} + \gamma_n^2}
		\qedhere
	\end{align*}
\end{proof}

\begin{lemma}[{\cite[Lemma 2.2.4]{nesterovLecturesConvexOptimization2018}}]
	\label{lem-appendix: convergence rate bounds for estimating sequences}
	For \(\condition_0^{-1}\in (\condition^{-1}, 3 + \condition^{-1}]\) we have
	\begin{align*}
		\Gamma^n
		\le \frac{4\condition^{-1}}{(\condition_0^{-1}-\condition^{-1})\left[
			\exp\left(\frac{n+1}{2\sqrt{\condition}}\right)
			-\exp\left(-\frac{n+1}{2\sqrt{\condition}}\right)
		\right]^2}
		\le \frac{4}{(\condition_0^{-1}-\condition^{-1})(n+1)^2}
	\end{align*}
\end{lemma}
\begin{proof}
	We are interested in finding a function fulfilling
	\begin{align}\label{eq: bounding function properties}
		\Gamma^n \le \frac{1}{f(n)^2} \iff f(n) \le \frac1{\sqrt{\Gamma^n}}
	\end{align}
	Since the weaker inequality would require \(f(n)\sim n\) it might be a
	good idea to look at the increments.
	In particular it would be sufficient if we had \(f(0) \le
	\frac1{\sqrt{\Gamma^0}}=1\) and
	\begin{align*}
		f(k) - f(k-1) \le \frac{1}{\sqrt{\Gamma^k}}-\frac{1}{\sqrt{\Gamma^{k-1}}}
		\qquad \forall k > 0.
	\end{align*}
	For our weaker inequality this requires the right side to be lower bounded
	by a constant. Now if we had not started out looking for a squared \(f\), we
	would have instead ended up with the increment
	\begin{align*}
		\frac{1}{\Gamma^k}-\frac{1}{\Gamma^{k-1}}.
	\end{align*}
	This provides some intuition how Nesterov might have come up with the following	
	idea. Let us consider the increment above and multiply it by \(\Gamma^k\).
	Then we get
	\begin{align*}
		1-\frac{\Gamma^k}{\Gamma^{k-1}}
		&\lxeq{\text{def.}} \gamma_k
		\xeq{(\ref{eq: morphing speed equation})}
		\sqrt{\frac{\lbound_k}{\ubound}} \xeq{(\ref{eq: lbound recursion})}
		\sqrt{\frac{\Gamma^k \lbound_0 + (1-\Gamma^k)\lbound}{\ubound}}\\
		&\lxeq{\text{def.}} \sqrt{\condition^{-1} + \Gamma^k(\condition_0^{-1}-\condition^{-1})}.
	\end{align*}
	Now we not only need to return to the increments by dividing by \(\Gamma^k\)
	again, we also need to somehow add roots. To do that we
	use the third binomial formula
	\begin{align}
		\nonumber
		\tfrac1{\sqrt{\Gamma^k}}
		\sqrt{\tfrac{\condition^{-1}}{\Gamma^k} + \condition_0^{-1}-\condition^{-1}}
		&= \frac{1}{\Gamma^k}-\frac{1}{\Gamma^{k-1}}\\
		\nonumber
		&\lxeq{\text{bin.}} 
		\left(
			\frac{1}{\sqrt{\Gamma^k}}+\frac{1}{\sqrt{\Gamma^{k-1}}}
		\right)
		\left(
			\frac{1}{\sqrt{\Gamma^k}}-\frac{1}{\sqrt{\Gamma^{k-1}}}
		\right) \\
		&\le \frac{2}{\sqrt{\Gamma^k}}
		\left(
			\frac{1}{\sqrt{\Gamma^k}}-\frac{1}{\sqrt{\Gamma^{k-1}}}
		\right)
		\label{eq: inverse gamma increment}
	\end{align}
	where we have used \(\Gamma^k\le\Gamma^{k-1}\) for the inequality. From
	this we get the sufficient condition
	\begin{align*}
		f(k) - f(k-1)
		\le \tfrac12\sqrt{\tfrac{\condition^{-1}}{\Gamma^k} + \condition_0^{-1}-\condition^{-1}}
		= \tfrac12\underbrace{\sqrt{\condition_0^{-1}-\condition^{-1}}}_{=:c}
		\sqrt{\tfrac{\condition^{-1}}{\Gamma^k}c^{-2} + 1}
	\end{align*}
	In particular we could throw away the complicated positive part in the root
	to get the sufficient condition
	\begin{align*}
		f(k) - f(k-1) := \tfrac{c}2\quad \forall k>0 \quad \text{and}
		\quad f(0):=\tfrac{c}2 = \tfrac{\sqrt{\condition_0^{-1}-\condition^{-1}}}2
		\le \tfrac{\sqrt{3}}{2}<1.
	\end{align*}
	And this immediately leads to \(f(n) = \tfrac{c}{2}(n+1)\) which provides us
	with the weaker second bound. Note that the ``complicated'' part we have
	thrown away is zero in case of \(\condition^{-1}\) i.e.\ when we do not
	have strong convexity this is the best we can do.
	
	But for strong convexity we need to find a tighter upper bound \(f\). 
	If we shuffle around (\ref{eq: inverse gamma increment}) a bit we get
	\begin{align*}
		\frac{1}{\sqrt{\Gamma^{k-1}}}
		\le \frac{1}{\sqrt{\Gamma^k}}
		- \frac{c}{2}\sqrt{\tfrac{\condition^{-1}}{c^2}\left(\tfrac{1}{\sqrt{\Gamma^k}}\right)^2+1}
		=: g\left(\tfrac1{\sqrt{\Gamma^k}}\right)
	\end{align*}
	Now here is the trick: \(g\) is strictly monotonously increasing (its
	derivative is positive). So if \(n\) was the first number violating
	our requirement (\ref{eq: bounding function properties}), (i.e.
	\(\tfrac{1}{\sqrt{\Gamma^n}}< f(n)\)), then due to
	\begin{align*}
		f(n-1)
		\xle{\text{(\ref{eq: bounding function properties})}} \frac{1}{\sqrt{\Gamma^{n-1}}}
		\le g\left(\tfrac1{\sqrt{\Gamma^n}}\right) 
		< g(f(n))
	\end{align*}
	it would be enough for
	\begin{align}\label{eq: weird contradiction requirement}
		g(f(n))\le f(n-1)
	\end{align}
	to be true, to obtain a
	contradiction. So if \(f(0)\le 1\) is a given, \(f(n)\) would then fulfill
	(\ref{eq: bounding function properties}) by induction\footnote{Unfortunately,
	I have no idea how Nesterov came up with this trick, nor how he eventually
	found the function \(f\) that satisfies the requirements}. As we know the
	result we can define
	\begin{align*}
		f(n)
		:=\tfrac{c}{2\sqrt{\condition^{-1}}}\left[
			\exp\left(\tfrac{n+1}{2\sqrt{\condition}}\right)
			-\exp\left(-\tfrac{n+1}{2\sqrt{\condition}}\right)
		\right]
		=\frac{c}{4\delta}\left[
			e^{(n+1)\delta}- e^{-(n+1)\delta}
		\right]
	\end{align*}
	for \(\delta:=\sqrt{\condition^{-1}}/2\le 1/2\). Since \(f(0)\) is increasing in \(\delta\)
	we have the induction start 
	\begin{align*}
		f(0)=\tfrac{c}{4\delta}\left[e^\delta-e^{-\delta}\right]
		\le \tfrac{c}2\left[e^{1/2}-e^{-1/2}\right]
		\le \tfrac{c}{\sqrt{3}} \le 1.
	\end{align*}
	Now we only need to show (\ref{eq: weird contradiction requirement}) to be
	true to finish the proof:
	\begin{align*}
		g(f(n))
		&= f(n) - \frac{c}{2}\sqrt{\tfrac{\condition^{-1}}{c^2}f(n)^2 + 1}\\
		&= f(n) - \frac{c}{2}\sqrt{
			\tfrac{1}{2^2}\left[e^{(n+1)\delta}- e^{-(n+1)\delta}\right]^2
			+ 1
		} \\
		&= f(n) - \frac{c}{2}\sqrt{
			\tfrac{1}{2}\left[
				\tfrac12 e^{2(n+1)\delta}
				-\smash{\underbrace{e^{(n+1)\delta}e^{-(n+1)\delta}}_{=1}}
				+\tfrac12e^{-2(n+1)\delta}
			\right]
			+ 1
		}\\
		&= f(n) - \underbrace{\frac{c}{4}\left[e^{(n+1)\delta}+ e^{-(n+1)\delta}\right]}_{=f'(n)}\\
		&= f(n) + \langle f'(n), n - (n-1) \rangle \xle{f\text{ convex}} f(n-1)
		\qedhere
	\end{align*}
\end{proof}
\begin{lemma}\label{lem-appendix: derivative loss difference constant in gamma}
	\begin{align*}
		c(\gamma_1)
		=2\ubound\frac{1-\gamma_1(1+\condition^{-1}) + \gamma_1^2}{\gamma_1^2 -\condition^{-1}}
		\qquad \gamma_1\in(\sqrt{\condition^{-1}}, 1).
	\end{align*}
	is monotonously decreasing.
\end{lemma}
\begin{proof}
	We discard \(2\ubound\) without loss of generality to get	
	\begin{align*}
		\frac{dc}{d\gamma_1}
		&= \frac{
			(-(1+\condition^{-1})+2\gamma_1)(\gamma_1^2 -\condition^{-1})
			- (1-\gamma_1(1+\condition^{-1})+\gamma_1^2)2\gamma_1
		}{(\gamma_1^2-\condition^{-1})^2}.
	\end{align*}
	Since we are interested in the sign of the derivative we can discard the
	denominator and look at just the enumerator
	\begin{align*}
		&2\gamma_1(\cancel{\gamma_1^2}-\condition^{-1}) - (1+\condition^{-1})(\gamma_1^2-\condition^{-1})
		- (1-\gamma_1(1+\condition^{-1})+\cancel{\gamma_1^2})2\gamma_1\\
		&= -2\gamma_1\condition^{-1} - (1+\condition^{-1})(\gamma_1^2 + \condition^{-1})
		-2\gamma_1 + 2\gamma_1^2(1+\condition^{-1})\\
		&= (1+\condition^{-1})\underbrace{(\condition^{-1}-2\gamma_1 + 2\gamma_1^2)}_{=:f(\gamma_1)}.
	\end{align*}
	In which we only need to consider the parabola \(f\). Now since the parabola
	is convex, it is negative for all points between two ends where it
	is negative. And we have
	\begin{align*}
		f(\sqrt{\condition^{-1}}) &= 2(\condition^{-1}-\sqrt{\condition^{-1}}) < 0\\
		f(1) &= \condition^{-1} -1 < 0
		\qedhere
	\end{align*}
\end{proof}

\chapter{Code for Visualizations}

The Programming Language Julia \parencite{bezansonJuliaFreshApproach2017} was
used to generate all the visualizations without citation in this work. The code
can be found in various .jl files on GitHub at
\url{https://github.com/FelixBenning/masterthesis/tree/main/media}. While
these are normal Julia files which can be run as-is, they were written for
the \href{https://github.com/fonsp/Pluto.jl}{Pluto.jl} notebook, which is
required for the interactive features of these plots.


	\backmatter
	\printbibliography[heading=bibintoc]

@unpublished{ankirchnerApproximatingStochasticGradient2021,
  title = {Approximating Stochastic Gradient Descent with Diffusions: Error Expansions and Impact of Learning Rate Schedules},
  shorttitle = {Approximating Stochastic Gradient Descent with Diffusions},
  author = {Ankirchner, Stefan and Perko, Stefan},
  date = {2021-06},
  url = {https://hal.archives-ouvertes.fr/hal-03262396},
  urldate = {2021-09-22}
}

@article{armijoMinimizationFunctionsHaving1966,
  title = {Minimization of Functions Having {{Lipschitz}} Continuous First Partial Derivatives.},
  author = {Armijo, Larry},
  date = {1966-01},
  journaltitle = {Pacific Journal of Mathematics},
  volume = {16},
  number = {1},
  pages = {1--3},
  publisher = {{Pacific Journal of Mathematics, A Non-profit Corporation}},
  issn = {0030-8730},
  url = {https://projecteuclid.org/journals/pacific-journal-of-mathematics/volume-16/issue-1/Minimization-of-functions-having-Lipschitz-continuous-first-partial-derivatives/pjm/1102995080.full},
  urldate = {2021-09-27}
}

@online{bachNonstronglyconvexSmoothStochastic2013,
  title = {Non-Strongly-Convex Smooth Stochastic Approximation with Convergence Rate {{O}}(1/n)},
  author = {Bach, Francis and Moulines, Eric},
  date = {2013-06-10},
  eprint = {1306.2119},
  eprinttype = {arxiv},
  primaryclass = {cs, math, stat},
  url = {http://arxiv.org/abs/1306.2119},
  urldate = {2021-07-15},
  archiveprefix = {arXiv}
}

@article{bezansonJuliaFreshApproach2017,
  title = {Julia: {{A Fresh Approach}} to {{Numerical Computing}}},
  shorttitle = {Julia},
  author = {Bezanson, Jeff and Edelman, Alan and Karpinski, Stefan and Shah, Viral B.},
  date = {2017-01-01},
  journaltitle = {SIAM Review},
  shortjournal = {SIAM Rev.},
  volume = {59},
  number = {1},
  pages = {65--98},
  publisher = {{Society for Industrial and Applied Mathematics}},
  issn = {0036-1445},
  doi = {10.1137/141000671}
}

@online{bottouOptimizationMethodsLargeScale2018,
  title = {Optimization {{Methods}} for {{Large-Scale Machine Learning}}},
  author = {Bottou, Léon and Curtis, Frank E. and Nocedal, Jorge},
  date = {2018-02-08},
  eprint = {1606.04838},
  eprinttype = {arxiv},
  primaryclass = {cs, math, stat},
  url = {http://arxiv.org/abs/1606.04838},
  urldate = {2021-06-08},
  archiveprefix = {arXiv}
}

@article{bouttierConvergenceRateSimulated2019,
  title = {Convergence {{Rate}} of a {{Simulated Annealing Algorithm}} with {{Noisy Observations}}},
  author = {Bouttier, Clément and Gavra, Ioana},
  date = {2019},
  journaltitle = {Journal of Machine Learning Research},
  volume = {20},
  number = {4},
  pages = {1--45},
  issn = {1533-7928},
  url = {http://jmlr.org/papers/v20/16-588.html},
  urldate = {2021-07-05}
}

@book{bovierMetastabilityPotentialTheoreticApproach2015,
  title = {Metastability: {{A Potential-Theoretic Approach}}},
  shorttitle = {Metastability},
  author = {Bovier, Anton and den Hollander, Frank},
  date = {2015},
  series = {Grundlehren Der Mathematischen {{Wissenschaften}}},
  publisher = {{Springer International Publishing}},
  doi = {10.1007/978-3-319-24777-9},
  isbn = {978-3-319-24775-5},
  langid = {english}
}

@article{brayStatisticsCriticalPoints2007,
  title = {The Statistics of Critical Points of {{Gaussian}} Fields on Large-Dimensional Spaces},
  author = {Bray, Alan J. and Dean, David S.},
  date = {2007-04-10},
  journaltitle = {Physical Review Letters},
  shortjournal = {Phys. Rev. Lett.},
  volume = {98},
  number = {15},
  eprint = {cond-mat/0611023},
  eprinttype = {arxiv},
  issn = {0031-9007, 1079-7114},
  doi = {10.1103/PhysRevLett.98.150201},
  archiveprefix = {arXiv}
}

@article{broydenConvergenceClassDoublerank1970,
  title = {The {{Convergence}} of a {{Class}} of {{Double-rank Minimization Algorithms}} 1. {{General Considerations}}},
  author = {Broyden, C. G.},
  date = {1970-03-01},
  journaltitle = {IMA Journal of Applied Mathematics},
  shortjournal = {IMA Journal of Applied Mathematics},
  volume = {6},
  number = {1},
  pages = {76--90},
  issn = {0272-4960},
  doi = {10.1093/imamat/6.1.76}
}

@online{bubeckConvexOptimizationAlgorithms2015,
  title = {Convex {{Optimization}}: {{Algorithms}} and {{Complexity}}},
  shorttitle = {Convex {{Optimization}}},
  author = {Bubeck, Sébastien},
  date = {2015-11-16},
  eprint = {1405.4980},
  eprinttype = {arxiv},
  primaryclass = {cs, math, stat},
  url = {http://arxiv.org/abs/1405.4980},
  urldate = {2021-07-12},
  archiveprefix = {arXiv}
}

@online{bubeckGeometricAlternativeNesterov2015,
  title = {A Geometric Alternative to {{Nesterov}}'s Accelerated Gradient Descent},
  author = {Bubeck, Sébastien and Lee, Yin Tat and Singh, Mohit},
  date = {2015-06-26},
  eprint = {1506.08187},
  eprinttype = {arxiv},
  primaryclass = {cs, math},
  url = {http://arxiv.org/abs/1506.08187},
  urldate = {2021-07-07},
  archiveprefix = {arXiv},
  langid = {english}
}

@article{cauchyMethodeGeneralePour1847,
  title = {Méthode Générale Pour La Résolution Des Systemes d’équations Simultanées},
  author = {Cauchy, Augustin},
  date = {1847},
  journaltitle = {Comptes rendus de l'Académie des Sciences},
  volume = {25},
  pages = {536--538}
}

@unpublished{chenLargeScaleOptimizationData2019,
  title = {Large-{{Scale Optimization}} for {{Data Science}}: {{Mirror Descent}}},
  author = {Chen, Yuxin},
  date = {2019},
  url = {http://www.princeton.edu/~yc5/ele522_optimization/lectures/mirror_descent.pdf},
  urldate = {2021-07-07}
}

@online{choromanskaLossSurfacesMultilayer2015,
  title = {The {{Loss Surfaces}} of {{Multilayer Networks}}},
  author = {Choromanska, Anna and Henaff, Mikael and Mathieu, Michael and Arous, Gérard Ben and LeCun, Yann},
  date = {2015-01-21},
  eprint = {1412.0233},
  eprinttype = {arxiv},
  primaryclass = {cs},
  url = {http://arxiv.org/abs/1412.0233},
  urldate = {2021-06-16},
  archiveprefix = {arXiv}
}

@inproceedings{darkenFasterStochasticGradient1991,
  title = {Towards Faster Stochastic Gradient Search},
  booktitle = {{{NIPs}}},
  author = {Darken, Christian and Moody, John},
  date = {1991},
  volume = {91},
  pages = {1009--1016}
}

@online{dauphinIdentifyingAttackingSaddle2014,
  title = {Identifying and Attacking the Saddle Point Problem in High-Dimensional Non-Convex Optimization},
  author = {Dauphin, Yann and Pascanu, Razvan and Gulcehre, Caglar and Cho, Kyunghyun and Ganguli, Surya and Bengio, Yoshua},
  date = {2014},
  eprint = {1406.2572},
  eprinttype = {arxiv},
  archiveprefix = {arXiv}
}

@article{devolderFirstorderMethodsSmooth2014,
  title = {First-Order Methods of Smooth Convex Optimization with Inexact Oracle},
  author = {Devolder, Olivier and Glineur, François and Nesterov, Yurii},
  date = {2014-08-01},
  journaltitle = {Mathematical Programming},
  shortjournal = {Math. Program.},
  volume = {146},
  number = {1},
  pages = {37--75},
  issn = {1436-4646},
  doi = {10.1007/s10107-013-0677-5},
  langid = {english}
}

@online{dontlooWhatDifferenceMomentum2016,
  title = {What's the Difference between Momentum Based Gradient Descent and {{Nesterov}}'s Accelerated Gradient Descent?},
  author = {{dontloo}, (https://stats.stackexchange.com/users/95569/dontloo)},
  date = {2016-01-21},
  url = {https://stats.stackexchange.com/q/191727},
  howpublished = {Cross Validated},
  organization = {{Cross Validated Stack Exchange}}
}

@article{dozatIncorporatingNesterovMomentum2016,
  title = {Incorporating {{Nesterov Momentum}} into {{Adam}}},
  author = {Dozat, Timothy},
  date = {2016-02-18},
  url = {https://openreview.net/forum?id=OM0jvwB8jIp57ZJjtNEZ},
  urldate = {2021-11-16},
  langid = {english}
}

@article{duchiAdaptiveSubgradientMethods2011,
  title = {Adaptive {{Subgradient Methods}} for {{Online Learning}} and {{Stochastic Optimization}}},
  author = {Duchi, John and Hazan, Elad and Singer, Yoram},
  date = {2011-07-01},
  journaltitle = {The Journal of Machine Learning Research},
  shortjournal = {J. Mach. Learn. Res.},
  volume = {12},
  pages = {2121--2159},
  issn = {1532-4435},
  issue = {null}
}

@online{flammarionAveragingAccelerationThere2015,
  title = {From {{Averaging}} to {{Acceleration}}, {{There}} Is {{Only}} a {{Step-size}}},
  author = {Flammarion, Nicolas and Bach, Francis},
  date = {2015-04-07},
  eprint = {1504.01577},
  eprinttype = {arxiv},
  primaryclass = {math, stat},
  url = {http://arxiv.org/abs/1504.01577},
  urldate = {2021-06-21},
  archiveprefix = {arXiv},
  langid = {english}
}

@article{fletcherNewApproachVariable1970,
  title = {A New Approach to Variable Metric Algorithms},
  author = {Fletcher, R.},
  date = {1970-01-01},
  journaltitle = {The Computer Journal},
  shortjournal = {The Computer Journal},
  volume = {13},
  number = {3},
  pages = {317--322},
  issn = {0010-4620},
  doi = {10.1093/comjnl/13.3.317}
}

@online{garipovLossSurfacesMode2018,
  title = {Loss {{Surfaces}}, {{Mode Connectivity}}, and {{Fast Ensembling}} of {{DNNs}}},
  author = {Garipov, Timur and Izmailov, Pavel and Podoprikhin, Dmitrii and Vetrov, Dmitry and Wilson, Andrew Gordon},
  date = {2018-10-30},
  eprint = {1802.10026},
  eprinttype = {arxiv},
  primaryclass = {cs, stat},
  url = {http://arxiv.org/abs/1802.10026},
  urldate = {2021-06-15},
  archiveprefix = {arXiv}
}

@article{gelfandNormierteRinge1941,
  title = {Normierte Ringe},
  author = {Gelfand, Israel},
  date = {1941},
  journaltitle = {Matematicheskii Sbornik},
  shortjournal = {Mat. Sb.},
  volume = {9(51)},
  number = {1},
  pages = {3--24},
  url = {http://www.mathnet.ru/links/78bac43c7b4d5544986886ed8be19701/sm6046.pdf},
  urldate = {2021-11-08},
  langid = {german}
}

@inproceedings{ghadimiGlobalConvergenceHeavyball2015,
  title = {Global Convergence of the {{Heavy-ball}} Method for Convex Optimization},
  booktitle = {2015 {{European Control Conference}} ({{ECC}})},
  author = {Ghadimi, Euhanna and Feyzmahdavian, Hamid Reza and Johansson, Mikael},
  date = {2015-07},
  pages = {310--315},
  doi = {10.1109/ECC.2015.7330562},
  eventtitle = {2015 {{European Control Conference}} ({{ECC}})}
}

@article{gohWhyMomentumReally2017,
  title = {Why {{Momentum Really Works}}},
  author = {Goh, Gabriel},
  date = {2017-04-04},
  journaltitle = {Distill},
  shortjournal = {Distill},
  volume = {2},
  number = {4},
  pages = {e6},
  issn = {2476-0757},
  doi = {10.23915/distill.00006},
  langid = {english}
}

@article{goldfarbFamilyVariablemetricMethods1970,
  title = {A Family of Variable-Metric Methods Derived by Variational Means},
  author = {Goldfarb, Donald},
  date = {1970},
  journaltitle = {Mathematics of Computation},
  shortjournal = {Math. Comp.},
  volume = {24},
  number = {109},
  pages = {23--26},
  issn = {0025-5718, 1088-6842},
  doi = {10.1090/S0025-5718-1970-0258249-6},
  langid = {english}
}

@book{golubMatrixComputations2013,
  title = {Matrix Computations},
  author = {Golub, Gene H. and Van Loan, Charles F.},
  date = {2013},
  series = {Johns {{Hopkins}} Studies in the Mathematical Sciences},
  edition = {Fourth edition},
  publisher = {{The Johns Hopkins University Press}},
  location = {{Baltimore}},
  url = {http://math.ecnu.edu.cn/~jypan/Teaching/books/2013%20Matrix%20Computations%204th.pdf},
  urldate = {2021-11-10},
  isbn = {978-1-4214-0794-4},
  langid = {english},
  pagetotal = {756}
}

@unpublished{guptaAdvancedAlgorithmsFall2020,
  title = {Advanced {{Algorithms}}, {{Fall}} 2020, {{Lecture}} 19},
  author = {Gupta, Anupam},
  date = {2020},
  url = {http://www.cs.cmu.edu/~15850/notes/lec19.pdf},
  urldate = {2021-07-07}
}

@online{hardtTrainFasterGeneralize2016,
  title = {Train Faster, Generalize Better: {{Stability}} of Stochastic Gradient Descent},
  shorttitle = {Train Faster, Generalize Better},
  author = {Hardt, Moritz and Recht, Benjamin and Singer, Yoram},
  date = {2016-02-07},
  eprint = {1509.01240},
  eprinttype = {arxiv},
  primaryclass = {cs, math, stat},
  url = {http://arxiv.org/abs/1509.01240},
  urldate = {2021-04-19},
  archiveprefix = {arXiv}
}

@article{heConditionsConvergenceEvolutionary2001,
  title = {Conditions for the Convergence of Evolutionary Algorithms},
  author = {He, Jun and Yu, Xinghuo},
  date = {2001-07-01},
  journaltitle = {Journal of Systems Architecture},
  shortjournal = {Journal of Systems Architecture},
  series = {Evolutionary Computing},
  volume = {47},
  number = {7},
  pages = {601--612},
  issn = {1383-7621},
  doi = {10.1016/S1383-7621(01)00018-2},
  langid = {english}
}

@unpublished{hintonNeuralNetworksMachine2012,
  type = {Massive Open Online Course},
  title = {Neural {{Networks}} for {{Machine Learning}}},
  author = {Hinton, Geoffrey},
  date = {2012},
  url = {https://www.cs.toronto.edu/~hinton/coursera_lectures.html},
  urldate = {2021-11-16},
  editora = {Sirvastava, Nitish and Swersky, Kevin},
  editoratype = {collaborator},
  venue = {{Coursera}}
}

@article{hochreiterFlatMinima1997,
  title = {Flat {{Minima}}},
  author = {Hochreiter, Sepp and Schmidhuber, Jürgen},
  date = {1997-01-01},
  journaltitle = {Neural Computation},
  shortjournal = {Neural Computation},
  volume = {9},
  number = {1},
  pages = {1--42},
  issn = {0899-7667},
  doi = {10.1162/neco.1997.9.1.1}
}

@online{IMOmathHopitalTheorem,
  title = {{{IMOmath}}: {{L}}’{{Hopital}}’s {{Theorem}}},
  url = {https://www.imomath.com/index.php?options=686},
  urldate = {2021-09-15}
}

@online{israelHowWorkOut2012,
  title = {How to Work out This Problem Using {{Stolz}}'s Theorem?},
  author = {Israel, Robert (https://math.stackexchange.com/users/8508)},
  date = {2012-10-05},
  url = {https://math.stackexchange.com/q/207596},
  organization = {{Mathematics Stack Exchange}}
}

@online{izmailovAveragingWeightsLeads2019,
  title = {Averaging {{Weights Leads}} to {{Wider Optima}} and {{Better Generalization}}},
  author = {Izmailov, Pavel and Podoprikhin, Dmitrii and Garipov, Timur and Vetrov, Dmitry and Wilson, Andrew Gordon},
  date = {2019-02-25},
  eprint = {1803.05407},
  eprinttype = {arxiv},
  primaryclass = {cs, stat},
  url = {http://arxiv.org/abs/1803.05407},
  urldate = {2021-06-15},
  archiveprefix = {arXiv}
}

@online{karimiLinearConvergenceGradient2020,
  title = {Linear {{Convergence}} of {{Gradient}} and {{Proximal-Gradient Methods Under}} the Polyak-\L{}ojasiewicz {{Condition}}},
  author = {Karimi, Hamed and Nutini, Julie and Schmidt, Mark},
  date = {2020-09-12},
  eprint = {1608.04636},
  eprinttype = {arxiv},
  primaryclass = {cs, math, stat},
  url = {http://arxiv.org/abs/1608.04636},
  urldate = {2021-09-21},
  archiveprefix = {arXiv}
}

@online{kingmaAdamMethodStochastic2017,
  title = {Adam: {{A Method}} for {{Stochastic Optimization}}},
  shorttitle = {Adam},
  author = {Kingma, Diederik P. and Ba, Jimmy},
  date = {2017-01-29},
  eprint = {1412.6980},
  eprinttype = {arxiv},
  primaryclass = {cs},
  url = {http://arxiv.org/abs/1412.6980},
  urldate = {2021-04-19},
  archiveprefix = {arXiv}
}

@article{kozyakinAccuracyApproximationSpectral2009,
  title = {On Accuracy of Approximation of the Spectral Radius by the {{Gelfand}} Formula},
  author = {Kozyakin, Victor},
  date = {2009-11-01},
  journaltitle = {Linear Algebra and its Applications},
  shortjournal = {Linear Algebra and its Applications},
  volume = {431},
  number = {11},
  pages = {2134--2141},
  issn = {0024-3795},
  doi = {10.1016/j.laa.2009.07.008},
  langid = {english}
}

@book{kushnerStochasticApproximationAlgorithms1997,
  title = {Stochastic Approximation Algorithms and Applications},
  author = {Kushner, Harold J.},
  date = {1997},
  series = {Applications of Mathematics; 35},
  publisher = {{Springer}},
  location = {{Berlin, et al.}},
  editora = {Yin, George},
  editoratype = {collaborator},
  isbn = {978-0-387-94916-1},
  langid = {english}
}

@article{lessardAnalysisDesignOptimization2016,
  title = {Analysis and {{Design}} of {{Optimization Algorithms}} via {{Integral Quadratic Constraints}}},
  author = {Lessard, Laurent and Recht, Benjamin and Packard, Andrew},
  date = {2016-01-01},
  journaltitle = {SIAM Journal on Optimization},
  shortjournal = {SIAM J. Optim.},
  volume = {26},
  number = {1},
  pages = {57--95},
  publisher = {{Society for Industrial and Applied Mathematics}},
  issn = {1052-6234},
  doi = {10.1137/15M1009597}
}

@article{levenbergMethodSolutionCertain1944,
  title = {A Method for the Solution of Certain Non-Linear Problems in Least Squares},
  author = {Levenberg, Kenneth},
  date = {1944},
  journaltitle = {Quarterly of Applied Mathematics},
  shortjournal = {Quart. Appl. Math.},
  volume = {2},
  number = {2},
  pages = {164--168},
  issn = {0033-569X, 1552-4485},
  doi = {10.1090/qam/10666},
  langid = {english}
}

@inproceedings{liStochasticModifiedEquations2017,
  title = {Stochastic {{Modified Equations}} and {{Adaptive Stochastic Gradient Algorithms}}},
  booktitle = {Proceedings of the 34th {{International Conference}} on {{Machine Learning}}},
  author = {Li, Qianxiao and Tai, Cheng and E, Weinan},
  date = {2017-07-17},
  pages = {2101--2110},
  publisher = {{PMLR}},
  issn = {2640-3498},
  url = {https://proceedings.mlr.press/v70/li17f.html},
  urldate = {2021-09-23},
  eventtitle = {International {{Conference}} on {{Machine Learning}}},
  langid = {english}
}

@online{liVisualizingLossLandscape2018,
  title = {Visualizing the {{Loss Landscape}} of {{Neural Nets}}},
  author = {Li, Hao and Xu, Zheng and Taylor, Gavin and Studer, Christoph and Goldstein, Tom},
  date = {2018-11-07},
  eprint = {1712.09913},
  eprinttype = {arxiv},
  primaryclass = {cs, stat},
  url = {http://arxiv.org/abs/1712.09913},
  urldate = {2021-06-16},
  archiveprefix = {arXiv}
}

@article{marquardtAlgorithmLeastSquaresEstimation1963,
  title = {An {{Algorithm}} for {{Least-Squares Estimation}} of {{Nonlinear Parameters}}},
  author = {Marquardt, Donald W.},
  date = {1963-06-01},
  journaltitle = {Journal of the Society for Industrial and Applied Mathematics},
  volume = {11},
  number = {2},
  pages = {431--441},
  publisher = {{Society for Industrial and Applied Mathematics}},
  issn = {0368-4245},
  doi = {10.1137/0111030}
}

@inproceedings{martensDeepLearningHessianfree2010,
  title = {Deep Learning via {{Hessian-free}} Optimization},
  booktitle = {{{ICML}}},
  author = {Martens, James},
  date = {2010},
  volume = {27},
  pages = {8},
  langid = {english}
}

@online{metzGradientsAreNot2021,
  title = {Gradients Are {{Not All You Need}}},
  author = {Metz, Luke and Freeman, C. Daniel and Schoenholz, Samuel S. and Kachman, Tal},
  date = {2021-11-10},
  eprint = {2111.05803},
  eprinttype = {arxiv},
  primaryclass = {cs, stat},
  url = {http://arxiv.org/abs/2111.05803},
  urldate = {2021-11-16},
  archiveprefix = {arXiv}
}

@article{nemirovskiRobustStochasticApproximation2009,
  title = {Robust {{Stochastic Approximation Approach}} to {{Stochastic Programming}}},
  author = {Nemirovski, A. and Juditsky, A. and Lan, G. and Shapiro, A.},
  date = {2009-01-01},
  journaltitle = {SIAM Journal on Optimization},
  shortjournal = {SIAM J. Optim.},
  volume = {19},
  number = {4},
  pages = {1574--1609},
  publisher = {{Society for Industrial and Applied Mathematics}},
  issn = {1052-6234},
  doi = {10.1137/070704277}
}

@book{nesterovLecturesConvexOptimization2018,
  title = {Lectures on {{Convex Optimization}}},
  author = {Nesterov, Yurii Evgen'evič},
  date = {2018},
  series = {Springer Optimization and {{Its}} Applications; Volume 137},
  edition = {Second edition},
  publisher = {{Springer}},
  location = {{Cham}},
  doi = {10.1007/978-3-319-91578-4},
  isbn = {978-3-319-91578-4},
  langid = {english}
}

@article{nesterovMethodSolvingConvex1983,
  title = {A Method for Solving the Convex Programming Problem with Convergence Rate \(O(1/k^2)\)},
  author = {Nesterov, Yurii Evgen'evič},
  date = {1983},
  journaltitle = {Dokl. Akad. Nauk SSSR},
  volume = {269},
  pages = {543--547},
  url = {http://m.mathnet.ru/php/archive.phtml?wshow=paper&jrnid=dan&paperid=46009&option_lang=eng},
  urldate = {2021-07-17}
}

@online{nguyenFirstExitTime2019,
  title = {First {{Exit Time Analysis}} of {{Stochastic Gradient Descent Under Heavy-Tailed Gradient Noise}}},
  author = {Nguyen, Thanh Huy and Şimşekli, Umut and Gürbüzbalaban, Mert and Richard, Gaël},
  date = {2019-06-21},
  eprint = {1906.09069},
  eprinttype = {arxiv},
  primaryclass = {cs, stat},
  url = {http://arxiv.org/abs/1906.09069},
  urldate = {2021-09-24},
  archiveprefix = {arXiv},
  langid = {english}
}

@incollection{nocedalQuasiNewtonMethods2006,
  title = {Quasi-{{Newton Methods}}},
  booktitle = {Numerical {{Optimization}}},
  author = {Nocedal, Jorge and Wright, Stephen J.},
  date = {2006},
  series = {Springer {{Series}} in {{Operations Research}} and {{Financial Engineering}}},
  pages = {135--163},
  publisher = {{Springer}},
  location = {{New York, NY}},
  doi = {10.1007/978-0-387-40065-5_6},
  isbn = {978-0-387-40065-5},
  langid = {english}
}

@online{pascanuSaddlePointProblem2014,
  title = {On the Saddle Point Problem for Non-Convex Optimization},
  author = {Pascanu, Razvan and Dauphin, Yann N. and Ganguli, Surya and Bengio, Yoshua},
  date = {2014-05-27},
  eprint = {1405.4604},
  eprinttype = {arxiv},
  primaryclass = {cs},
  url = {http://arxiv.org/abs/1405.4604},
  urldate = {2021-09-16},
  archiveprefix = {arXiv}
}

@article{pearlmutterFastExactMultiplication1994,
  title = {Fast {{Exact Multiplication}} by the {{Hessian}}},
  author = {Pearlmutter, Barak A.},
  date = {1994-01},
  journaltitle = {Neural Computation},
  volume = {6},
  number = {1},
  pages = {147--160},
  issn = {0899-7667},
  doi = {10.1162/neco.1994.6.1.147},
  eventtitle = {Neural {{Computation}}}
}

@inproceedings{pesmeConvergenceDiagnosticBasedStep2020,
  title = {On {{Convergence-Diagnostic}} Based {{Step Sizes}} for {{Stochastic Gradient Descent}}},
  booktitle = {Proceedings of the 37th {{International Conference}} on {{Machine Learning}}},
  author = {Pesme, Scott and Dieuleveut, Aymeric and Flammarion, Nicolas},
  date = {2020-11-21},
  pages = {7641--7651},
  publisher = {{PMLR}},
  issn = {2640-3498},
  url = {https://proceedings.mlr.press/v119/pesme20a.html},
  urldate = {2021-09-20},
  eventtitle = {International {{Conference}} on {{Machine Learning}}},
  langid = {english}
}

@book{polyakIntroductionOptimization1987,
  title = {Introduction to Optimization},
  author = {Polyak, Boris T.},
  date = {1987},
  publisher = {{New York Optimization Software}},
  isbn = {978-0-911575-14-9},
  langid = {english}
}

@article{polyakMethodsSpeedingConvergence1964,
  title = {Some Methods of Speeding up the Convergence of Iteration Methods},
  author = {Polyak, Boris T.},
  date = {1964-01-01},
  journaltitle = {USSR Computational Mathematics and Mathematical Physics},
  shortjournal = {USSR Computational Mathematics and Mathematical Physics},
  volume = {4},
  number = {5},
  pages = {1--17},
  issn = {0041-5553},
  doi = {10.1016/0041-5553(64)90137-5},
  langid = {english}
}

@article{powellSearchDirectionsMinimization1973,
  title = {On Search Directions for Minimization Algorithms},
  author = {Powell, M. J. D.},
  date = {1973-12-01},
  journaltitle = {Mathematical Programming},
  shortjournal = {Mathematical Programming},
  volume = {4},
  number = {1},
  pages = {193--201},
  issn = {1436-4646},
  doi = {10.1007/BF01584660},
  langid = {english}
}

@article{qianMomentumTermGradient1999,
  title = {On the Momentum Term in Gradient Descent Learning Algorithms},
  author = {Qian, Ning},
  date = {1999-01-01},
  journaltitle = {Neural Networks},
  volume = {12},
  number = {1},
  pages = {145--151},
  publisher = {{Pergamon}},
  issn = {0893-6080},
  doi = {10.1016/S0893-6080(98)00116-6},
  langid = {english}
}

@online{radfordVisualizingOptimizationAlgos2014,
  title = {Visualizing {{Optimization Algos}} - {{GIFs}}},
  author = {Radford, Alec},
  date = {2014-09-18},
  url = {https://imgur.com/a/Hqolp},
  urldate = {2021-11-16},
  langid = {english},
  organization = {{Imgur}}
}

@unpublished{rechtOptimization2013,
  title = {Optimization {{I}}},
  author = {Recht, Benjamin},
  date = {2013-04-09},
  url = {https://simons.berkeley.edu/talks/ben-recht-2013-09-04},
  urldate = {2021-09-03},
  eventtitle = {Big {{Data Boot Camp}}},
  langid = {english},
  venue = {{Simons Institute for the Theory of Computing, UC Berkeley}}
}

@online{reddiConvergenceAdam2019,
  title = {On the {{Convergence}} of {{Adam}} and {{Beyond}}},
  author = {Reddi, Sashank J. and Kale, Satyen and Kumar, Sanjiv},
  date = {2019-04-19},
  eprint = {1904.09237},
  eprinttype = {arxiv},
  primaryclass = {cs, math, stat},
  url = {http://arxiv.org/abs/1904.09237},
  urldate = {2021-06-08},
  archiveprefix = {arXiv}
}

@online{ruderOverviewGradientDescent2017,
  title = {An Overview of Gradient Descent Optimization Algorithms},
  author = {Ruder, Sebastian},
  date = {2017-06-15},
  eprint = {1609.04747},
  eprinttype = {arxiv},
  primaryclass = {cs},
  url = {http://arxiv.org/abs/1609.04747},
  urldate = {2021-06-07},
  archiveprefix = {arXiv}
}

@online{schmidtDescendingCrowdedValley2021,
  title = {Descending through a {{Crowded Valley}} -- {{Benchmarking Deep Learning Optimizers}}},
  author = {Schmidt, Robin M. and Schneider, Frank and Hennig, Philipp},
  date = {2021-02-11},
  eprint = {2007.01547},
  eprinttype = {arxiv},
  primaryclass = {cs, stat},
  url = {http://arxiv.org/abs/2007.01547},
  urldate = {2021-06-15},
  archiveprefix = {arXiv}
}

@inproceedings{schraudolphFastCurvatureMatrixVector2001,
  title = {Fast {{Curvature Matrix-Vector Products}}},
  booktitle = {Artificial {{Neural Networks}} — {{ICANN}} 2001},
  author = {Schraudolph, Nicol N.},
  editor = {Dorffner, Georg and Bischof, Horst and Hornik, Kurt},
  date = {2001},
  series = {Lecture {{Notes}} in {{Computer Science}}},
  pages = {19--26},
  publisher = {{Springer}},
  location = {{Berlin, et al.}},
  doi = {10.1007/3-540-44668-0_4},
  isbn = {978-3-540-44668-2},
  langid = {english}
}

@article{sejnowskiUnreasonableEffectivenessDeep2020,
  title = {The Unreasonable Effectiveness of Deep Learning in Artificial Intelligence},
  author = {Sejnowski, Terrence J.},
  date = {2020-12-01},
  journaltitle = {Proceedings of the National Academy of Sciences},
  shortjournal = {PNAS},
  volume = {117},
  number = {48},
  eprint = {31992643},
  eprinttype = {pmid},
  pages = {30033--30038},
  publisher = {{National Academy of Sciences}},
  issn = {0027-8424, 1091-6490},
  doi = {10.1073/pnas.1907373117},
  langid = {english}
}

@article{shannoConditioningQuasiNewtonMethods1970,
  title = {Conditioning of Quasi-{{Newton}} Methods for Function Minimization},
  author = {Shanno, D. F.},
  date = {1970},
  journaltitle = {Mathematics of Computation},
  shortjournal = {Math. Comp.},
  volume = {24},
  number = {111},
  pages = {647--656},
  issn = {0025-5718, 1088-6842},
  doi = {10.1090/S0025-5718-1970-0274029-X},
  langid = {english}
}

@book{shewchukIntroductionConjugateGradient1994,
  title = {An Introduction to the Conjugate Gradient Method without the Agonizing Pain},
  author = {Shewchuk, Jonathan Richard},
  date = {1994},
  publisher = {{Carnegie-Mellon University. Department of Computer Science}},
  url = {https://www.cs.cmu.edu/~quake-papers/painless-conjugate-gradient.pdf},
  urldate = {2021-12-13}
}

@online{simsekliTailIndexAnalysisStochastic2019,
  title = {A {{Tail-Index Analysis}} of {{Stochastic Gradient Noise}} in {{Deep Neural Networks}}},
  author = {Simsekli, Umut and Sagun, Levent and Gurbuzbalaban, Mert},
  date = {2019-01-17},
  eprint = {1901.06053},
  eprinttype = {arxiv},
  primaryclass = {cs, stat},
  url = {http://arxiv.org/abs/1901.06053},
  urldate = {2021-05-27},
  archiveprefix = {arXiv}
}

@online{smithDonDecayLearning2018,
  title = {Don't {{Decay}} the {{Learning Rate}}, {{Increase}} the {{Batch Size}}},
  author = {Smith, Samuel L. and Kindermans, Pieter-Jan and Ying, Chris and Le, Quoc V.},
  date = {2018-02-23},
  eprint = {1711.00489},
  eprinttype = {arxiv},
  primaryclass = {cs, stat},
  url = {http://arxiv.org/abs/1711.00489},
  urldate = {2021-10-06},
  archiveprefix = {arXiv}
}

@online{suDifferentialEquationModeling2015,
  title = {A {{Differential Equation}} for {{Modeling Nesterov}}'s {{Accelerated Gradient Method}}: {{Theory}} and {{Insights}}},
  shorttitle = {A {{Differential Equation}} for {{Modeling Nesterov}}'s {{Accelerated Gradient Method}}},
  author = {Su, Weijie and Boyd, Stephen and Candes, Emmanuel J.},
  date = {2015-10-27},
  eprint = {1503.01243},
  eprinttype = {arxiv},
  primaryclass = {math, stat},
  url = {http://arxiv.org/abs/1503.01243},
  urldate = {2021-07-12},
  archiveprefix = {arXiv}
}

@article{sunNonErgodicConvergenceAnalysis2019,
  title = {Non-{{Ergodic Convergence Analysis}} of {{Heavy-Ball Algorithms}}},
  author = {Sun, Tao and Yin, Penghang and Li, Dongsheng and Huang, Chun and Guan, Lei and Jiang, Hao},
  date = {2019-07-17},
  journaltitle = {Proceedings of the AAAI Conference on Artificial Intelligence},
  volume = {33},
  number = {01},
  pages = {5033--5040},
  issn = {2374-3468},
  doi = {10.1609/aaai.v33i01.33015033},
  langid = {english}
}

@inproceedings{sutskeverImportanceInitializationMomentum2013,
  title = {On the Importance of Initialization and Momentum in Deep Learning},
  booktitle = {International {{Conference}} on {{Machine Learning}}},
  author = {Sutskever, Ilya and Martens, James and Dahl, George and Hinton, Geoffrey},
  date = {2013-05-26},
  pages = {1139--1147},
  publisher = {{PMLR}},
  issn = {1938-7228},
  url = {http://proceedings.mlr.press/v28/sutskever13.html},
  urldate = {2021-06-07},
  eventtitle = {International {{Conference}} on {{Machine Learning}}},
  langid = {english}
}

@book{suttonReinforcementLearningIntroduction2018,
  title = {Reinforcement Learning: An Introduction},
  shorttitle = {Reinforcement Learning},
  author = {Sutton, Richard S. and Barto, Andrew G.},
  date = {2018},
  series = {Adaptive Computation and Machine Learning Series},
  edition = {Second edition},
  publisher = {{The MIT Press}},
  location = {{Cambridge, Massachusetts}},
  isbn = {978-0-262-03924-6},
  langid = {english}
}

@online{truongBacktrackingGradientDescent2019,
  title = {Backtracking Gradient Descent Method for General $C^1$ Functions, with Applications to {{Deep Learning}}},
  author = {Truong, Tuyen Trung and Nguyen, Tuan Hang},
  date = {2019-04-04},
  eprint = {1808.05160},
  eprinttype = {arxiv},
  primaryclass = {cs, math, stat},
  url = {http://arxiv.org/abs/1808.05160},
  urldate = {2021-09-27},
  archiveprefix = {arXiv}
}

@article{vapnikOverviewStatisticalLearning1999,
  title = {An Overview of Statistical Learning Theory},
  author = {Vapnik, V.N.},
  date = {1999-09},
  journaltitle = {IEEE Transactions on Neural Networks},
  volume = {10},
  number = {5},
  pages = {988--999},
  issn = {1941-0093},
  doi = {10.1109/72.788640},
  eventtitle = {{{IEEE Transactions}} on {{Neural Networks}}}
}

@inproceedings{vinyalsKrylovSubspaceDescent2012,
  title = {Krylov {{Subspace Descent}} for {{Deep Learning}}},
  booktitle = {Proceedings of the {{Fifteenth International Conference}} on {{Artificial Intelligence}} and {{Statistics}}},
  author = {Vinyals, Oriol and Povey, Daniel},
  date = {2012-03-21},
  pages = {1261--1268},
  publisher = {{PMLR}},
  issn = {1938-7228},
  url = {https://proceedings.mlr.press/v22/vinyals12.html},
  urldate = {2021-09-29},
  eventtitle = {Artificial {{Intelligence}} and {{Statistics}}},
  langid = {english}
}

@online{wilsonMarginalValueAdaptive2018,
  title = {The {{Marginal Value}} of {{Adaptive Gradient Methods}} in {{Machine Learning}}},
  author = {Wilson, Ashia C. and Roelofs, Rebecca and Stern, Mitchell and Srebro, Nathan and Recht, Benjamin},
  date = {2018-05-21},
  eprint = {1705.08292},
  eprinttype = {arxiv},
  primaryclass = {cs, stat},
  url = {http://arxiv.org/abs/1705.08292},
  urldate = {2021-04-19},
  archiveprefix = {arXiv}
}

@online{zeilerADADELTAAdaptiveLearning2012,
  title = {{{ADADELTA}}: {{An Adaptive Learning Rate Method}}},
  shorttitle = {{{ADADELTA}}},
  author = {Zeiler, Matthew D.},
  date = {2012-12-22},
  eprint = {1212.5701},
  eprinttype = {arxiv},
  primaryclass = {cs},
  url = {http://arxiv.org/abs/1212.5701},
  urldate = {2021-11-15},
  archiveprefix = {arXiv}
}
\end{document}